\documentclass[11pt,english]{smfart}

\usepackage{amsmath,sfheaders,leftidx,amsthm}
\usepackage[all]{xy}
\usepackage{xspace,smfthm}
\usepackage[psamsfonts]{amssymb}
\usepackage[latin1]{inputenc}
\usepackage{graphicx,color,mathrsfs}
\usepackage{fancyhdr}
\usepackage{appendix}
\usepackage{hyperref}
\usepackage{xcolor}

\usepackage{enumitem}

\makeatletter
\def\th@plain{%
  \let\thm@indent\noindent
  \thm@headfont{\bfseries\smf@boldmath\itshape}%
  \thm@notefont{\bfseries\smf@boldmath\upshape}%
  \thm@preskip.5\linespacing \@plus .5\linespacing
  \thm@postskip\thm@preskip
  \thm@headpunct{\MakePointrait}
  \itshape
}
\def\th@definition{\th@plain\upshape}
\def\th@remark{\th@definition}
\makeatother

\newtheorem{theorem}{\bf Theorem}[section]
\newtheorem{proposition}[theorem]{\bf Proposition}
\newtheorem{definition}[theorem]{\bf Definition}
\newtheorem{Theorem}{\bf Theorem}

\newtheorem*{Conjecture}{\bf Conjecture}

\newtheorem{lemma}[theorem]{\bf Lemma}
\newtheorem{corollary}[theorem]{\bf Corollary}
\newtheorem{conjecture}[theorem]{\bf Conjecture}
\newtheorem{Corollary}[Theorem]{\bf Corollary}

\theoremstyle{definition}
\newtheorem{rk}[theorem]{\bf Remark}

\def\C{{\mathbb C}}

\def\R{{\mathbb R}}

\def\Q{{\mathbb Q}}

\def\p{\mathbb{P}}

\def\supp{\textup{supp}}
\def\id{\mathrm{id}}

\newcommand{\Bb}{\mathbb{B}}
\newcommand{\Db}{\mathbb{D}}
\newcommand{\Pb}{\mathbb{P}}
\newcommand{\Cb}{\mathbb{C}}
\newcommand{\Nb}{\mathbb{N}}
\newcommand{\Zb}{\mathbb{Z}}
\newcommand{\Qb}{\mathbb{Q}}
\newcommand{\Rb}{\mathbb{R}}
\newcommand{\Sb}{\mathbb{S}}

\newcommand{\codim}{\text{codim}}

\newcommand{\ddc}{{dd^c}}
\newcommand{\diag}{{\rm diag}}
\newcommand{\Ab}{\mathbb{A}}

\newcommand{\End}{\operatorname{End}}
\newcommand{\Poly}{\operatorname{Poly}}
\newcommand{\PGL}{\operatorname{PGL}}
\newcommand{\Aff}{\operatorname{Aff}}

\title[Sparsity and uniformity in dynamical Manin--Mumford problems]{Sparsity and uniformity in dynamical Manin--Mumford problems}
\author{Thomas Gauthier}
\address{Laboratoire de Math\'ematiques d'Orsay, B\^atiment 307, Universit\'e Paris-Saclay, 91405 Orsay Cedex, France}
\email{thomas.gauthier@universite-paris-saclay.fr}
\author{Johan Taflin}
\address{Universit\'e Bourgogne Europe,
Institut de Math\'ematiques de Bourgogne,
Facult\'e des Sciences Mirande,
9 avenue Alain Savary,
21000 Dijon, France}
\email{johan.taflin@ube.fr}

\begin{document}

\begin{abstract} For endomorphisms of the complex projective plane of a given degree, we prove that the existence of a special curve of bounded degree, in the sense of the dynamical Manin--Mumford conjecture, is confined to a proper algebraic subvariety of the parameter space. This confirms a conjecture of DeMarco and Mavraki in dimension two.

For general regular polynomial endomorphisms of $\mathbb A^k$, we obtain a uniform dynamical Manin--Mumford bound outside an exceptional hypersurface.

These results arise from a common strategy, valid in every dimension and codimension, which reduces such questions to the geometry of exceptional loci in the space of pairs formed by subvarieties and maps. We control these loci by combining arithmetic equidistribution, blender constructions, dynamical directions, monodromy arguments, and degenerations at the boundary of parameter space.
\end{abstract}

\maketitle
\tableofcontents

\section{Introduction}
For many algebraic dynamical systems, periodic points, or more generally preperiodic points, are Zariski dense. It is therefore natural to study their algebraic distribution.

This question was formalized in the setting of polarized endomorphisms by Zhang~\cite{zhang} in the so-called dynamical Manin-Mumford conjecture (see also~\cite{DTZ-dmm,ght} for an updated version of the conjecture), which is a central question in arithmetic dynamics.
To be more precise, let $X$ be a complex projective variety and let $L$ be an ample line bundle on $X$. An endomorphism $f:X\to X$ is said to be \emph{polarized} by $L$ of degree $d\geq 2$ if
\[f^*L\simeq L^{\otimes d}.\]
By a theorem of Fakhruddin~\cite{Fakhruddin}, for such a polarized endomorphism the set
\[\mathrm{Preper}(f)
:=\{x\in X\; ;\; f^m(x)=f^n(x) \text{ for some } m>n\geq 0\}\]
of \emph{preperiodic points} of $f$ is countable and Zariski dense in $X$. On the other hand, a proper subvariety $Z\subsetneq X$ is said to be \emph{$f$-special} if one of the following conditions holds:
\begin{enumerate}
\item either $Z$ is preperiodic under iteration of $f$;
\item or there exist a closed subvariety $Y\subseteq X$, an endomorphism $g:Y\to Y$ polarized, and an integer $n\geq 1$ such that $Z\subseteq Y$, $f^n(Y)=Y$, the maps $f^n$ and $g$ commute on $Y$, and $Z$ is preperiodic under iteration of $g$.
\end{enumerate}
We also say that the pair $(Z,f)$ is special. Observe that the first case is included in the second, but not conversely (see~\cite{DTZ-dmm}). The two are distinguished because one might, at first sight, expect most examples to fall into the first case.

The conjecture can then be formulated as follows.
\begin{Conjecture}[Dynamical Manin-Mumford]
Let $f:X\to X$ be a polarized endomorphism and let $Z\subsetneq X$ be a closed subvariety. If
\[Z\cap \mathrm{Preper}(f)\]
is Zariski dense in $Z$, then $Z$ is $f$-special.
\end{Conjecture}
The conjecture has been established only in a limited number of cases. The first and most emblematic one is the case of abelian varieties. If $A$ is an abelian variety and $f\colon A\to A$ is multiplication by an integer $N\geq 2$, then
\[\mathrm{Preper}(f)=A_{\mathrm{tor}},\]
the set of torsion points of $A$. In this setting, the dynamical Manin-Mumford conjecture reduces to the original Manin-Mumford conjecture, proved by Raynaud~\cite{raynaud1,raynaud2}.

Beyond this classical case, the conjecture was proved by Ghioca, Nguyen and Ye for split maps
\[f:(\mathbb{P}^1)^N\to(\mathbb{P}^1)^N,\qquad f(z_1,\ldots,z_N)=(f_1(z_1),\ldots,f_N(z_N)),\]
where each $f_i$ is a rational map of degree $d$ (see~\cite{GNY-split-2,GNKY-Maninmumford}), which extends the theorem of Ihara-Serre-Tate (see \cite{Lang-Division}) for power maps. Apart from these two families, another case is that of regular polynomial endomorphisms of $\mathbb{C}^2$, proved by Dujardin, Favre, and Ruggiero under an additional mild assumption (see~\cite{DFR}).

We also mention that, for dissipative polynomial automorphisms of $\mathbb{C}^2$, which are not polarized, an analogue of the conjecture was proved by Dujardin and Favre~\cite{DF-dmm}.

Finally, Fakhruddin~\cite[Theorem~1.3]{fakhruddin-generic} proved that, for very general endomorphisms of $\mathbb{P}^k$, every proper subvariety contains only finitely many preperiodic points. Passing from very general to general endomorphisms, even for subvarieties of bounded degree, is however a fundamentally different problem.

\begin{center}
\rule{0.5\linewidth}{0.4pt}
\end{center}

\bigskip

A complementary direction, which does not directly settle the dynamical Manin-Mumford conjecture but rather completes the picture, is to study how such pairs $(Z,f)$, for which $Z\cap \mathrm{Preper}(f)$ is Zariski dense in $Z$, are distributed in families. Assuming the dynamical Manin-Mumford conjecture, this amounts to studying the repartition of pairs $(Z,f)$ such that $Z$ is $f$-special.

An illustration of this theme is the following conjecture of DeMarco and Mavraki (see \cite[Conjecture 4.1]{udmm} and also \cite[Conjecture 2.7]{demarco-cirm} for a weaker formulation). Here, for $k\geq 2$ and $d\geq 2$, $\mathrm{End}_d^k$ denotes the space of all endomorphisms of $\mathbb{P}^k$ of degree $d$, which is naturally an affine variety.
\begin{Conjecture}[DeMarco-Mavraki]
For each $D \geq 1$, there exists a proper closed subvariety $V_D \subsetneq \mathrm{End}_d^k$ such that any $f \in \mathrm{End}_d^k$ admitting a positive-dimensional $f$-special subvariety of degree at most $D$ belongs to $V_D$.
\end{Conjecture}
The main purpose of this article is to develop a general strategy for problems of this kind. Our contribution is twofold:
\begin{itemize}
\item a general mechanism reducing these questions to the study of \emph{exceptional loci} in the product $\mathcal X_D^u\times\mathrm{End}_d^k$, where $\mathcal X_D^u$ denotes the Chow variety of effective $u$-cycles of degree $D$ in $\mathbb P^k$,
\item geometric control of these loci, yielding sparsity and uniformity results and, in particular, a proof of the above conjecture in dimension two.
\end{itemize}
The first part is a far-reaching generalization of the methods developed in~\cite{gtv}. Its distinctive feature is that it treats arbitrary subvarieties
in every dimension and codimension within a common framework, without relying on the connection between critical dynamics and stability that was essential in \cite{gtv}, see Section~\ref{sec-strat}. 

The resulting geometric reformulation opens the way to deformation and degeneration arguments, with degenerations at the boundary of $\mathrm{End}_d^2$ providing the decisive step in dimension two.

In higher dimensions, the same framework suggests a route towards resolving the conjecture in full generality, which would call for a more systematic understanding of degenerations in higher-dimensional dynamics, extending tools that are now well developed in dimension one, such as rescaling limits and non-Archimedean methods.

\bigskip

In another direction, our strategy also yields uniformity results. The expected non-density of preperiodic points for a general pair $(Z,f)$ suggests asking for a uniform quantitative refinement, in the spirit of recent results on the Manin--Mumford, Bogomolov and Mordell--Lang conjectures \cite{DKY1,Dimitrov-Gao-Habegger,Kuhne,Yuan-Bogomolov,Poineau-BFT,GGK}. Can the quantity
\[
\#\bigl(Z\cap\operatorname{Preper}(f)\bigr)
\]
be bounded uniformly for general pairs $(Z,f)$? Our second main result establishes such a bound on the space $\mathrm{Poly}_d^k$ of regular polynomial endomorphisms of $\mathbb A^k$ of degree $d$, after removing a hypersurface of uniformly bounded degree depending only on $f$.

In what follows, we first present our main results. We then describe the general strategy underlying them in Section~\ref{sec-strat}. In Section~\ref{sec-coarse}, we explain how this strategy yields a first coarse control of the exceptional locus and leads to Theorem~\ref{th-special-hypersurface}. We next discuss the additional ingredients needed to prove Theorem~\ref{th-demarco2} and Theorem~\ref{th-unif} in Sections~\ref{sec-thb} and~\ref{sec-thc}, respectively. Finally, we discuss further developments and conjectures suggested by our approach.

\subsection{Main results}
We start with a first consequence of our approach, which gives directly the sparsity of special pairs. Here and throughout the paper, all maps are defined over the field $\C$ of complex numbers, and elements of $\mathcal X_D^u$ are identified with their supports.
\begin{Theorem}\label{th-sparsity}
Let $M$ be either $\mathrm{End}_d^k$ or $\mathrm{Poly}_d^k$. Then, for every $D\geq 1$ and $1 \leq u \leq k-1$, there exists a proper closed subvariety $\mathcal Z_D^u\subsetneq \mathcal X_D^u \times M$ such that, if $Y \in \mathcal X_D^u$ is $f$-special for some $f \in M$, then $(Y,f) \in \mathcal Z_D^u$.
\end{Theorem}
Note that the extension from preperiodic subvarieties to special pairs relies in part on a recent result of Yugang Zhang~\cite{zhang-iterated-centralizer} on centralizers of general endomorphisms.

In what follows, we informally call \emph{exceptional locus} the smallest subvariety $\mathcal Z_D^u$ satisfying the above statement.
The proofs of our other results then reduce to obtaining further control of this locus. By showing that its general fibers cannot cover the whole projective space, we obtain the following weak form of the DeMarco--Mavraki conjecture.
\begin{Theorem}\label{th-special-hypersurface}
Let $M$ and $D\geq1$ as above. There exist a dense Zariski open subset
$\mathcal U_D\subset M$, an integer $B\geq1$, and a morphism
\[
\mathcal H_D:\mathcal U_D\longrightarrow\mathcal X_B^{k-1}
\]
such that, for every $f\in\mathcal U_D$ and every
$1\leq u\leq k-1$, each $f$-special
$Y\in\mathcal X_D^u$ is contained in $\mathcal H_D(f)$.
\end{Theorem}
Of course, we expect that $\mathcal H_D(f)$ could be replaced by the hyperplane at infinity $H_\infty$ when $M=\mathrm{Poly}_d^k$, and by the empty set when
$M=\mathrm{End}_d^k$. Still, this already provides uniform geometric control of special subvarieties of bounded degree as the endomorphism varies, in arbitrary dimension and codimension. In particular, it gives a uniform bound on the number of irreducible $f$-special hypersurfaces of bounded degree, since each such hypersurface must be an irreducible component of $\mathcal H_D(f)$. Related finiteness and rigidity results for periodic curves of regular polynomial endomorphisms of $\mathbb A^2$ were obtained by Xie~\cite{Xie-algebraicity}, with no restriction on the degree $D$ but without uniformity, and Zhong~\cite{Zhong-periodic-curves}.

\bigskip

Starting from this, our two other main results provide refinements of a rather different nature, depending on whether we focus on endomorphisms of $\mathbb{P}^k$  (in which case we work in $M=\mathrm{End}_d^k$) or on regular polynomial endomorphisms of $\mathbb{A}^k$ (in which case we work in $M=\mathrm{Poly}_d^k$).

In the former case, we restrict ourselves to $k=2$ and study the degenerations of the set $\mathcal Z_D^1$ at infinity inside $\mathcal X_D^1 \times \mathrm{End}_d^2$. This allows us to exclude the possibility that $\mathcal Z_D^1$ surjects onto $\mathrm{End}_d^2$. We thereby prove the DeMarco-Mavraki conjecture in dimension $k=2$.
\begin{Theorem}\label{th-demarco2}
For every $d\geq 2$ and $D\geq 1$, there exists a proper closed subvariety $V_D \subsetneq \mathrm{End}_d^2$ such that any $f \in \mathrm{End}_d^2$ admitting an $f$-special curve of degree at most $D$ belongs to $V_D$.
\end{Theorem}
The main ideas of the proof of Theorem~\ref{th-demarco2} apply in arbitrary dimension and codimension. Controlling the degenerations in this generality is, however, considerably more technical, especially in codimension greater than $1$. We nevertheless expect the case of codimension $1$ to be within reach in arbitrary dimension using the same approach.

\bigskip

We now return to the uniformity question raised above. For regular polynomial endomorphisms of $\mathbb A^k$, the canonical height decomposes as a sum of non-negative local Green functions, and this extra positivity yields a uniform statement of dynamical Manin--Mumford type, strengthening Theorem~\ref{th-special-hypersurface}.

\begin{Theorem}\label{th-unif}
For every $d\geq2$ and $D\geq1$, there exist a dense Zariski open subset $\mathcal U_D\subset\mathrm{Poly}_d^k$, an integer $B\geq1$ and a morphism
\[
\mathcal H_D\colon\mathcal U_D\to\mathcal X_B^{k-1}
\]
such that, for all $f\in\mathcal U_D$ and all $1\leq u\leq k-1$, each $Y\in\mathcal X_D^u$ satisfies
\[
\#\bigl((\mathrm{Preper}(f)\cap Y)\setminus\mathcal H_D(f)\bigr)\leq B.
\]
\end{Theorem}

For curves, B\'ezout's theorem turns it into a bound on the full intersection since an irreducible curve of degree at most $D$ not contained in $\mathcal H_D(f)$ meets this hypersurface in at most $DB$ points. In dimension two, the curves contained in $\mathcal H_D(f)$ are among its at most $B$ components, so that only finitely many curves escape the bound.

\begin{Corollary}\label{cor-unif-dim2}
For every $d\geq2$ and $D\geq1$, there exist a dense Zariski open subset $\mathcal U_D\subset\mathrm{Poly}_d^2$ and an integer $B\geq1$ such that, for every $f\in\mathcal U_D$, all but at most $B$ irreducible curves $Y\subset\Pb^2$ of degree at most $D$ satisfy
\[
 \#\bigl(Y\cap\mathrm{Preper}(f)\bigr)\leq B.
\]
\end{Corollary}

One of the exceptional curves is of course the line at infinity, and we expect it to be the only one. More generally, we expect that $\mathcal H_D(f)$ can be replaced by $H_\infty$ in Theorem~\ref{th-unif} (see Conjecture~\ref{conj-unif}). Controlling preperiodic points inside $H_\infty$ itself would require uniform results for general endomorphisms of $\Pb^{k-1}$, which remain beyond the scope of our present arithmetic methods.

Corollary~\ref{cor-unif-dim2} complements the work of Dujardin, Favre and Ruggiero~\cite{DFR}, beyond the uniform aspect. Indeed, they prove the dynamical Manin--Mumford conjecture in $\mathrm{Poly}_d^2$ for curves in $\Pb^2$ whose intersection with the line at infinity is not contained in the set of points eventually mapped to a superattracting periodic point. For a fixed map $f\in\mathrm{Poly}_d^2$ with a non-exceptional superattracting periodic point at infinity, this condition is not Zariski generic in $\mathcal X_D^1$, for any $D\geq1$. On the other hand, for a very general map $f\in\mathrm{Poly}_d^2$, there is no superattracting periodic point at infinity, and their result holds for all $Y\in\mathcal X_D^1$.

This illustrates the complementarity of the two approaches. Ours yields uniform results for general pairs, but gives only limited control over the exceptional locus. By contrast, the approach of \cite{DFR} is more bottom-up, with the potential to lead to a complete solution of the conjecture in the two-dimensional affine case.

\subsection{The general mechanism}\label{sec-strat}
We now describe the mechanism underlying our results. It extends the approach of~\cite{gtv}, which was tailored to the critical hypersurface, to arbitrary families of subvarieties of any codimension. Its arithmetic part carries over to this setting, but its complex-dynamical part requires substantial new developments.

We choose a suitable subvariety $\mathcal{S}_d^k$ of $\mathrm{End}_d^k$ (respectively of $\mathrm{Poly}_d^k$) such that the induced map to the moduli space $\mathscr{M}_d^k$ of degree $d$ endomorphisms of $\mathbb{P}^k$ (respectively to the moduli space $\mathscr{P}_d^k$ of regular polynomial endomorphisms of $\mathbb{A}^k$) is dominant and generically finite. Then, for any closed subvariety $\mathcal{Y}\subseteq \mathcal{X}_D^u\times \mathcal{S}_d^k$ that projects dominantly onto $\mathcal{S}_d^k$ and that is defined over a number field, we obtain a family
\[
\hat{f}:\mathbb{P}^k\times\mathcal{Y}\to\mathbb{P}^k\times\mathcal{Y},\quad \hat f(z,Y,f):=(f(z),Y,f),
\]
together with a family $\hat{\mathcal{Y}}$ of $u$-dimensional subvarieties of $\mathbb{P}^k$, both parametrized by $\mathcal{Y}$. The strategy rests on four main ingredients:
\begin{enumerate}
\item a dynamical natural height function on $\mathcal{Y}$ vanishing on special pairs,
\item an equidistribution result towards a bifurcation measure (or, alternatively, an arithmetic positivity criterion for regular polynomial endomorphisms),
\item a positivity criterion for the associated bifurcation measure,
\item a dynamical mechanism allowing one to verify this criterion in a robust way.
\end{enumerate}
{The first two ingredients were already available, in the required generality, in \cite{gtv}, relying in particular on work of Call-Silverman~\cite{CS-height}, Zhang~\cite{Zhang-adelic}, Gauthier~\cite{Good-height} and Yuan-Zhang~\cite{YZ-adelic}. Let us explain briefly the context: any map $f:\p^k\to\p^k$ of degree $d$ defined over a number field has a \emph{canonical height function} $\hat{h}_f:\p^k(\bar{\mathbb{Q}})\to\R_+$ given by
\[\hat{h}_f:=\lim_{n\to\infty}\frac{1}{d^n}h_{\mathrm{nv}}\circ f^n,\]
where $h_{\mathrm{nv}}$ is the na\"ive logarithmic Weil height function. In particular, it is characterized by the properties $\hat{h}_f=h_{\mathrm{nv}}+O(1)$ and $\hat{h}_f\circ f=d\cdot \hat{h}_f$. In particular, by the Northcott property, the set $\{\hat{h}_f=0\}$ coincides with $\mathrm{Preper}(f)$. As this canonical height function is induced by an adelic line bundle, this can also be evaluated on subvarieties.  By Zhang's inequality, $\hat{h}_f(Y)=0$ whenever $Y$ is $f$-special, since the $f$-preperiodic points
are Zariski dense in $Y$.

When coming in a family $\hat{f}$ as above, the canonical height $\hat{h}_{f_t}$ depends on the parameter $t$ and this allows us to define a dynamical height function $\hat{h}_\mathcal{Y}:\mathcal{Y}(\bar{\mathbb{Q}})\to\R_+$ given by
\[\hat{h}_\mathcal{Y}(t):=\hat{h}_{f_t}(Y_t), \quad t\in \mathcal{Y}(\bar{\mathbb{Q}}).\]
Again, this height function is induced by an adelic line bundle on $\mathcal{Y}$ and the equidistribution theorem of small points of Yuan and Zhang~\cite{YZ-adelic} applies (see also~\cite{Good-height}).

We now want to describe the limit measure $\mu_\mathcal{Y}$ (which is a complex Monge-Amp\`ere measure on $\mathcal{Y}(\C)$). When acting on the complex space $\p^k(\C)\times\mathcal{Y}(\C)$, $\hat{f}$ admits a fibered Green current defined as
\[\hat{T}_{\hat{f}}:=\lim_{n\to\infty}\frac{1}{d^n}(\hat{f}^n)^*(\hat{\omega}),\]
where $\hat{\omega}$ is the $(1,1)$-form on $\p^k(\C)\times\mathcal{Y}(\C)$ defined as the pull-back by the first projection of the Fubini-Study form. The \emph{bifurcation current} $T_{\mathcal{Y}}$ of the pair $(\hat{f},\hat{\mathcal{Y}})$ is given by 
\[
T_{\mathcal{Y}}
:=(\pi_{\mathcal{Y}})_*
\left(
  \hat{T}_{\hat{f}}^{\wedge(u+1)}
  \wedge[\hat{\mathcal{Y}}(\C)]
\right).
\]
A similar construction in the fibered space $(\p^k)^{\dim\mathcal{Y}}\times\mathcal{Y}$ gives a bifurcation measure $\mu_\mathcal{Y}$ (which coincides with $T_\mathcal{Y}^{\wedge \dim\mathcal{Y}}$ when $u=k-1$), see~Section~\ref{sec-courant} for more details. This is the limit measure appearing in the equidistribution result of points of small height.
}

Starting from these ingredients, the proof of Theorem~\ref{th-sparsity} proceeds by contradiction. Suppose that the special pairs $(Z,f)$ defined over $\overline{\Qb}$ form a Zariski dense subset of the relevant parameter space. These pairs have height zero, since $Z\cap\operatorname{Preper}(f)$ is Zariski dense in $Z$. By equidistribution, a generic sequence of Galois orbits of such pairs must converge towards the normalized bifurcation measure $\mu_{\mathcal Y}$.

The key point is therefore to find a Euclidean open set $\Omega\subseteq\mathcal{Y}(\C)$ contained in the support of this bifurcation measure, but disjoint from the locus of pairs $(Z,f)$ for which $Z$ is $f$-special. This is incompatible with the equidistribution of height-zero pairs, and gives the desired sparsity statement.

It remains to construct such an open set $\Omega$. This is where the last two ingredients above enter.

\bigskip

\noindent{\bf A positivity criterion for bifurcation currents.}
Positivity criteria for bifurcation currents and measures have a long history. They rely on a local similarity between the dynamical space and the parameter space near special parameters. In complex dynamics, such phase-parameter correspondences appear in the work of Tan Lei~\cite{Tan-similarity} and Shishikura~\cite{Shishikura2}. Buff and Epstein~\cite{buffepstein} then interpreted this similarity property in terms of bifurcation current and measure, leading to many different variations of this similarity result in many different contexts, see~e.g.~\cite{Article1,AGMV,BBD,gauthier-abscont,DeMarco-Mavraki,gtv}.

In the situations considered so far, such criteria usually involve a proper intersection between the motion of a hyperbolic repeller, often a repelling periodic point, and the motion of a hypersurface.

In higher codimension, we need a more general version involving hyperbolic saddle sets. The following statement gives a simple form of the criterion for bifurcation currents. See Section~\ref{sec-misiu} for the corresponding statement for measures.

\begin{theorem}\label{th-misu-intro}
Suppose that there exist $(Y_0,f_0) \in \mathcal Y$ and a saddle point $r(f_0)$ of index $(s,u+1)$, with $s=k-u-1$, such that:
\begin{itemize}
\item[\textbf{(1)}] $T_{f_0}^{\wedge(u+1)}\wedge[W^u_{r(f_0),loc}]\neq 0$;
\item[\textbf{(2)}] there exists $n_0 \geq 0$ such that
\[f_0^{n_0}(Y_0)\cap W^s_{r(f_0),loc}\neq\varnothing;\]
\item[\textbf{(3)}] there exist pairs $(Y,f)\in\mathcal Y$, arbitrarily close to $(Y_0,f_0)$, such that
\[f^{n_0}(Y)\cap W^s_{r(f),loc}=\varnothing.\]
\end{itemize}
Then $(Y_0,f_0)$ belongs to the support of the current $T_{\mathcal Y}$.
\end{theorem}
Although the proof is not difficult, this criterion naturally generalizes the notion of \emph{Misiurewicz parameters} developed in \cite{BBD}, and may be useful beyond the present setting whenever higher-codimensional subvarieties arise.

It also highlights the role played by the \emph{intermediate Julia sets}
\[J_p(f):=\operatorname{supp}(T_f^{\wedge p}),\]
introduced by Forn\ae ss and Sibony in \cite{fs-montreal} and Hubbard and Papadopol in \cite{HP}. Apart from the two extremal cases, namely $J_1(f)$, the usual Julia set, and $J_k(f)$, the \emph{small Julia set}, which is the support of the unique measure of maximal entropy, these sets have so far been much less used and studied (see \cite{dujardin-directions} and \cite{t-attractor} for two exceptions).

\bigskip

\noindent{\bf Open sets of mixed-type maximal bifurcations.}
Once this criterion is established, it remains to verify it in a sufficiently strong form to obtain the desired arithmetic consequences. For Theorem~\ref{th-sparsity}, this amounts to proving that it holds densely in an open set $\Omega \subset \mathcal X_D^u\times M$ with the additional property that, for every $(Z,f)\in \Omega$, the subvariety $Z$ is not $f$-special. Observe that this last property relies crucially on a recent result of Yugang Zhang~\cite{zhang-iterated-centralizer}, combined with a criterion for the Zariski density of unstable manifolds (see Proposition~\ref{prop-zariski-dense-unstable}).

As in \cite{gtv}, the construction used here relies on a robust heterodimensional cycle between a periodic point and a hyperbolic set with a blender property.

Let $1\leq u\leq k-1$ be the dimension of the subvarieties under consideration, and set $s:=k-u-1$. We consider a hyperbolic set $\Lambda(f)$ for $f$, of index $(s,u+1)$, such that:
\begin{itemize}
\item its local stable set
\[W^s_{\Lambda(f),loc}:=\bigcup_{x\in\Lambda(f)}W^s_{x,loc}\]
intersects every element of an open family $\mathcal F$ of local submanifolds of dimension $u$;
\item this property is robust under small perturbations of $f$.
\end{itemize}
Each local stable manifold $W^s_{x,loc}$ has dimension $s$, and therefore cannot robustly intersect elements of $\mathcal F$, which have dimension $u$, since $s+u=k-1<k$. The robustness of the intersection must therefore come from a special feature of the whole hyperbolic set $\Lambda(f)$. As explained by Dujardin, this may be due either to the topology of $\Lambda(f)$, see \cite[Section~3]{dujardin-bif}, or to a subtler fractal-geometric property, see \cite[Section~4]{dujardin-bif}, which was introduced by Bonatti and D\'iaz in \cite{bonatti-diaz-blender} under the name of \emph{blender}. Since then, blenders have been used in many contexts, both in real dynamics (see e.g. \cite{bonatti-diaz-tan}, \cite{berger-parablender} and \cite{Li-Turaev-cycles}) and in complex dynamics (see e.g. \cite{t-blender} and \cite{biebler-hetero}).

Although blenders may at first seem rather intricate, the construction we use here is simple and can be viewed as a perturbation of an affine iterated function system with a covering property. This elementary and affine nature is crucial in our approach, especially when we study degenerations at infinity in the proof of Theorem~\ref{th-demarco2}, see Section~\ref{sec-thb}.

In addition to the blender $\Lambda(f)$, we also consider a saddle fixed point $p(f)$ of index $(s+1,u)$ such that:
\begin{itemize}
\item its local unstable manifold $W^u_{p(f),loc}$ belongs to the special family $\mathcal F$ associated with $\Lambda(f)$,
\item its local stable manifold $W^s_{p(f),loc}$ robustly intersects the unstable manifolds of $\Lambda(f)$.
\end{itemize}
The second property is easy to achieve, since $\dim W^s_{p(f),loc}+\dim W^u_{x,loc}=(s+1)+(u+1)=k+1>k.$
On the other hand, the blender property of $\Lambda(f)$ ensures that $W^u_{p(f),loc}$ intersects $W^s_{\Lambda(f),loc}$ robustly under perturbations of $f$.
In this situation, $p(f)$ and $\Lambda(f)$ form a \emph{robust heterodimensional cycle}, which is a source of strong bifurcations.

In Section~\ref{sec-blender}, for each $1\leq u\leq k-1$, we construct an open subset $\Omega_u$ of $\mathrm{End}_d^k$ (or of $\mathrm{Poly}_d^k$) such that each $f\in\Omega_u$ has such a cycle and with the following additional property.

\begin{theorem}\label{th-ouvert-int}
Let $(Y,f) \in \mathcal{Y}$ with $f\in \Omega_u$. If there exists $n\geq 0$ such that
\[f^n(Y) \text{ intersects } W^s_{p(f),loc} \text{ transversely},\]
then $(Y,f)$ belongs to the interior of the support of $\mu_{\mathcal Y}$.
\end{theorem}

We refer to Section~\ref{sec-blender} for discussions about the proof strategy, which has strong similarities with \cite[Section~3]{gtv}. Here, we emphasize the main differences and additional difficulties. The most immediate one is that we now work in arbitrary codimension. In this setting, the periodic points in $\Lambda(f)$ are no longer repelling, but saddle points. As a consequence, one cannot rely on holomorphic linearization in the same way, and this creates serious difficulties in controlling the images, under large iterates, of subvarieties transverse to $W^s_{p(f),loc}$.

We overcome these difficulties by developing a notion of \emph{dynamical direction}, extending in arbitrary dimension and codimension the \emph{dynamical slope} introduced by Dujardin \cite{dujardin-homoclinic} in dimension $2$, itself inspired by ideas of Takens \cite{takens}.

Two further difficulties already appear in codimension $1$ in this new framework. The first one is that, in \cite{gtv}, we only considered the critical hypersurface. By the work of Berteloot, Bianchi and Dupont \cite{BBD}, the absence of bifurcations could then be related to a strong form of dynamical stability, implying in particular the existence of an \emph{equilibrium lamination}. This lamination played a crucial role in \cite{gtv}. In the present setting, however, no such stability statement is available for an arbitrary hypersurface. The higher-codimension case is even more problematic since, in that case, the criterion coming from Theorem~\ref{th-misu-intro} no longer involves the small Julia set, but rather an intermediate Julia set, for which no bifurcation theory has yet been established. We instead rely on the recent result of the second author with Gorbovickis \cite{t-gorbovickis} on the algebraic independence of multipliers.

A second difficulty comes from the fact that maximal bifurcation in the product space $\mathcal X_D^u\times\mathrm{End}_d^k$ requires a mechanism rich enough to produce bifurcations not only when $f$ varies, but also when $f$ is fixed and the subvariety $Z$ varies in $\mathcal X_D^u$. This becomes particularly delicate in higher codimension where the stable directions of the hyperbolic set $\Lambda(f)$ leave additional degrees of freedom, which have to be controlled.

\subsection{Coarse control of the exceptional locus}\label{sec-coarse}
To prove Theorem~\ref{th-special-hypersurface}, we show that
the subvarieties parametrized by a general fiber of the
minimal exceptional locus do not cover $\Pb^k$.
More precisely, let $\mathcal Z_D^u\subset\mathcal X_D^u\times M$
be the smallest proper closed subvariety satisfying the
conclusion of Theorem~\ref{th-sparsity}. We prove that, for
general $f\in M$, the set
\[
H_f:=
\bigcup_{\substack{Y\in\mathcal X_D^u\\
                  (Y,f)\in\mathcal Z_D^u}}Y
\]
is a proper algebraic subvariety of $\Pb^k$, possibly empty.
This comes from the following more general
transverse-intersection statement.

\begin{theorem}\label{th-mono-gorbo}
Let $D\geq1$, let $1\leq u\leq k-1$, and let $M$ be either
$\mathrm{End}_d^k$ or $\mathrm{Poly}_d^k$.
Let $\mathcal Z\subset\mathcal X_D^u\times M$ be a closed
algebraic subvariety such that the projection
\[
\rho\colon
\{(x,Y,f)\in\Pb^k\times\mathcal Z:x\in Y\}
\longrightarrow\Pb^k\times M,
\qquad
(x,Y,f)\longmapsto(x,f),
\]
is dominant.
Then there exist $(Y,f)\in\mathcal Z$ with $f\in\Omega_u$
and an integer $n\geq0$ such that
\[
f^n(Y)\text{ intersects }W^s_{p(f),loc}
\text{ transversely}.
\]
\end{theorem}
The proof relies on the transitivity of the monodromy obtained by Gorbovickis and the second author \cite[Theorem~1.7]{t-gorbovickis}, not on the set of periodic points but on their eigendirections.

The next step is to show that $H_f$ is generically empty. This is much more delicate, and we achieve it only in dimension $2$ in the next section.
\subsection{Degenerations at infinity in $\mathcal X_D^{1}\times\mathrm{End}_d^2$}\label{sec-thb}
In this section, we only work in dimension $2.$ Therefore, $k=2$ and $u=1.$ 

To pass from a statement for general pairs in $\mathcal X_D^1\times\mathrm{End}_d^2$, as in Theorem~\ref{th-sparsity}, to a statement for general maps
in $\mathrm{End}_d^2$, as in Theorem~\ref{th-demarco2}, we need finer control of the exceptional locus $\mathcal Z_D^1$ than that obtained in Section~\ref{sec-coarse}.
More precisely, we show that, when this locus is chosen minimal, its projection to $\mathrm{End}_d^2$ is not dominant. To this end, we combine Theorem~\ref{th-ouvert-int}
with the following result.
\begin{theorem}\label{th-chasse-int}
Let $\mathcal Z \subset \mathcal X_D^{1} \times \mathrm{End}_d^2$ be an irreducible algebraic subvariety which projects dominantly onto $\mathrm{End}_d^2$. Then there exist $(Z,f)\in\mathcal Z$ and $n \geq 0$ such that $f \in \Omega_{1}$ and $f^n(Z)$ intersects $W_{p(f),loc}^s$ transversely.
\end{theorem}
Applying the strategy explained in Section~\ref{sec-strat} to each irreducible component of $\mathcal Z_D^{1}$ which dominates $\mathrm{End}_d^2$, this theorem allows us to replace $\mathcal Z_D^{1}$ by a proper algebraic subvariety. We can then repeat the argument as long as the resulting set still dominates $\mathrm{End}_d^2$. By Noetherianity, this process must stop, and we eventually obtain Theorem~\ref{th-demarco2}.

The proof is much more involved than that of Theorem~\ref{th-mono-gorbo}. The main difficulty is that the open parameter set $\Omega_1$ is a priori localized
within $\mathrm{End}_d^2$, and, for a given pair $(Z,f)$ with $f\in\Omega_1$, the stable manifold of $p(f)$ occupies only a limited region of $\Pb^2$, possibly far from $Z$.

To overcome this difficulty, we exploit the flexibility of the blender construction to extend $\Omega_1$ towards the boundary of $\mathrm{End}_d^2$. Under a suitable
degeneration, the stable manifold $W^s_{p(f)}$ becomes a Zariski open subset of a line $L$. B\'ezout's theorem then forces every curve $Z$ to meet $L$.

The failure of proper intersection between $Z$ and the preimages of $W^s_{p(f)}$ can then be recast as a Manin--Mumford type problem, which can be solved in this restricted setting. Using results of Bedford--Lyubich--Smillie~\cite{BedfordLyubichSmillie2} and Fakhruddin~\cite{fakhruddin-generic}, one can then upgrade this intersection to a transverse one.

A further difficulty arises in this approach. In the present setting, the positivity criterion of Proposition~\ref{prop-positivevolume} requires the repelling
point under consideration to belong to the small Julia set $J_2(f)$. This is not automatic, and we have so far relied on skew-product structures to ensure this property. The perturbations considered here, however, no longer preserve a fibration, making this membership substantially more delicate to establish. Doing so is an important step in the proof of Theorem~\ref{th-chasse-int}, see Section~\ref{sec-chasse-julia}.

To this end, we combine an entropy criterion due to de Th\'elin~\cite{deThelin05} and Dinh~\cite{Dinh07} with a rescaling argument in the spirit of rescaling limits. The mechanism is similar to that of the iterated function system used to construct the blender in Section~\ref{sec-existence}. Here, however, the limiting object is no longer locally invertible
and can instead be interpreted as a correspondence on $\Pb^1$. By analyzing its inverse branches, we obtain sufficiently many such branches for the approximating maps before passing to the limit.

\subsection{From subvariety control to uniform bounds in the affine case}
\label{sec-thc}

A first observation about Theorem~\ref{th-unif} is that
the most important case is that of hypersurfaces, since
every proper subvariety of degree $D$ is contained in
a hypersurface of the same degree. Nevertheless, our proof
requires induction and uses the result in every codimension.
Indeed, the arithmetic result~\cite[Theorem~7.2]{gtv}
does not directly bound the number of preperiodic points.
Instead, following~\cite{GGK}, it places them in a proper
subvariety of uniformly bounded degree. The important point
is that there is no a priori control on the position of
this subvariety, nor on its dependence on the parameter.

Let us briefly explain the arithmetic argument.
For a family $\mathcal Y$ satisfying our positivity
criterion, set $m:=\dim\mathcal Y$ and consider the fibered
power $\hat{\mathcal Y}^{[m+1]}$, endowed with the sum
of the canonical heights of its coordinates.
A generic sequence of affine tuples whose heights tend
to zero would equidistribute, through its Galois orbits,
towards the normalization of
\[
\eta:=
\hat T_{\hat f^{[m+1]}}^{\wedge(u(m+1)+m)}
\wedge[\hat{\mathcal Y}^{[m+1]}].
\]
The non-negativity of the local Green functions forces
the Galois averages of the complex Green function to tend
to zero. However, our positivity criterion gives a strictly
positive integral of this function against $\eta$,
yielding a contradiction.

Thus, sufficiently small tuples lie in a proper closed
subset $\mathcal V\subsetneq\hat{\mathcal Y}^{[m+1]}$.
For a general pair $(Y,f)\in\mathcal Y$, the set $\Sigma_{Y,f}$
of closed points in $Y$ of sufficiently small height therefore satisfies
\[
\Sigma_{Y,f}^{m+1}\subseteq\mathcal V_{Y,f}\subsetneq Y^{m+1}.
\]
By~\cite[Lemma~4.3]{GGK}, this places $\Sigma_{Y,f}$ in a proper
subvariety $W_{Y,f}\subsetneq Y$ of uniformly bounded degree.
In particular, $W_{Y,f}$ contains all preperiodic points
lying in $Y$. An important point here is that the dependence
of $W_{Y,f}$ on $(Y,f)$ can be very wild. However, we would like
to apply induction to its irreducible components.

By Theorem~\ref{th-mono-gorbo}, we can repeat the arithmetic argument on any exceptional family whose members cover
$\Pb^k$ for general $f$. Continuing in this way, the remaining exceptional subvarieties are all contained, for general $f$, in a proper algebraic subvariety of $\Pb^k$, which we may assume to be a hypersurface $\mathcal H_D(f)$.

For each positive-dimensional component of $W_{Y,f}$, we thus either continue the induction or include it in $\mathcal H_D(f)$. Only finitely many degree bounds arise, so a single algebraically varying hypersurface of uniformly bounded degree, also containing $H_\infty$, suffices.
Outside it, the induction ends in dimension zero and gives the required uniform bound. A density argument extends
the conclusion to complex parameters (see Section~\ref{sec-unif}).

\subsection{Further developments and conjectures}

As we have seen, finer control of the exceptional locus $\mathcal Z_D^u$ in Theorem~\ref{th-sparsity} leads to stronger conclusions, and we believe there is plenty of work to be done in this direction. A natural way to obtain such control is to extend the transverse-intersection statement of Theorem~\ref{th-chasse-int} to arbitrary dimensions and codimensions, both in the projective and polynomial settings.

\begin{conjecture}\label{conj-dominant}
Let $D\geq1$, let $1\leq u\leq k-1$, and let $M$ be either
$\mathrm{End}_d^k$ or $\mathrm{Poly}_d^k$.
Let $\mathcal Z\subset\mathcal X_D^u\times M$ be an
irreducible algebraic subvariety whose projection to $M$
is dominant. If $M=\mathrm{Poly}_d^k$, assume in addition
that $Y\not\subset H_\infty$ for a general
$(Y,f)\in\mathcal Z$.
Then there exist $(Y,f)\in\mathcal Z$ with $f\in\Omega_u$
and an integer $n\geq0$ such that
\[
f^n(Y)\text{ intersects }W^s_{p(f),loc}
\text{ transversely}.
\]
\end{conjecture}
Although the transversality result of Bedford--Lyubich--Smillie~\cite[Lemma~6.4]{BedfordLyubichSmillie2}
is not available in higher dimensions, a proper intersection is actually sufficient for this statement. However, upgrading a non-proper intersection to a proper one is a more serious difficulty in this problem.

For $M=\mathrm{End}_d^k$, this statement would imply that the minimal
exceptional locus $\mathcal Z_D^u$ does not project
dominantly onto $M$, thereby establishing the
DeMarco--Mavraki sparsity conjecture in all dimensions.

In the polynomial setting, it would imply the following uniform dynamical Manin-Mumford result.
\begin{conjecture}\label{conj-unif}
For every $D\geq1$, there exist an integer $B\geq0$
and a dense Zariski open subset
$\mathcal U_D\subset\mathrm{Poly}_d^k$ such that, for every
$f\in\mathcal U_D$, every $1\leq u\leq k-1$, and every
$Y\in\mathcal X_D^u$, we have
\[
\#\bigl(\mathrm{Preper}(f)\cap Y\cap\Cb^k\bigr)\leq B.
\]
\end{conjecture}
These statements leave open the geometry of the projected
exceptional loci. In particular, the Zariski open sets
in these conclusions are allowed to depend on $D$.
We expect that they could be chosen independently of $D$,
with the bound $B$ still depending on $D$, but such a
strengthening lies beyond the reach of our present methods.

Finally, these questions should be distinguished from
the much broader conjecture of DeMarco and
Mavraki~\cite[Conjecture~1.1]{udmm}, which characterizes
the Zariski density of preperiodic points in arbitrary
algebraic families in terms of dynamical rank and
relative special dimension.
Our approach relies crucially on the freedom to vary
the endomorphism over the full parameter space.
The study of arbitrary families, including the dynamical
Manin--Mumford conjecture for a fixed map, lies beyond
the scope of the methods developed here.

\subsection{Organization of the paper}
The paper is organized as follows. Section~\ref{sec-dir} is dedicated to studying (non-holomorphic) linearization for saddle periodic points and their dependence to a parameter. In Section~\ref{sec-graph}, we state basic properties about the geometric framework of horizontal/vertical relations and families of horizontal-like maps. In Section~\ref{sec-courant}, we recall the construction and the properties of families of maps endowed with a family of subvarieties, as well as their arithmetic properties and their bifurcation currents and measure, and we show a general positivity criterion for bifurcation currents based on intersections with stable manifolds.

In Section~\ref{sec-blender}, we show that suitable blender configurations yield a transverse-intersection criterion
for maximal bifurcations. Section~\ref{sec-existence} constructs such configurations in every dimension and codimension.

Section~\ref{sec-chasse} treats the case of curves in $\Pb^2$. We prove Theorem~\ref{th-chasse-int} by combining degeneration, rigidity, and monodromy arguments to obtain the required transverse intersections. A rescaling toward correspondences, developed in Section~\ref{sec-chasse-julia}, establishes the small Julia set condition needed to apply our positivity criterion.

Section~\ref{sec-size} studies the size of the exceptional set in arbitrary dimension and codimension. Finally, Section~\ref{sec-mainresult} brings these ingredients
together to prove the main results.

\subsection{Conventions}

Unless otherwise specified, we work over $\C$ and identify
algebraic varieties with their sets of complex points.
Dimensions and codimensions are complex. Topological notions
such as openness, closure and density refer to the usual
complex topology, unless the Zariski topology is explicitly
specified.

On an irreducible algebraic parameter space, a property holds
for a \emph{general} or \emph{generic} parameter if it holds
on a nonempty Zariski open subset. It holds for a
\emph{very general} parameter if it holds outside a countable
union of proper Zariski closed subsets.
Other notions of genericity are explicitly specified.

\subsection*{Acknowledgement}
We heartily thank Gabriel Vigny for very fruitful discussions at an early stage of this project. The first author's research is partially supported by the Institut Universitaire de France. This work was partially funded by the French National Research Agency under the project DynAtrois (ANR-24-CE40-1163).

\subsection*{AI disclosure} This project began at the \'Etats de la recherche "Arithmetic, Algebraic and Analytic Dynamics", held at CIRM in January 2025. The overall strategy behind Theorems A-D was already in place by the summer of 2025. Since then, we have used OpenAI's ChatGPT to assist in polishing the presentation of the paper and checking the proofs. Notably, in spring 2026, GPT-5.5 identified gaps in our arguments, since filled by the authors.

\section{Dynamical directions}\label{sec-dir}
We emphasize that, in this section, $s$ and $u$ denote integers satisfying $s+u=k$. This differs from the convention used in the rest of the article, in particular in
Section~\ref{sec-main-blender}, where $s+u+1=k$. The results of
this section will be applied to the periodic points $p(f)$ and
$r_0(f)$, of indices $(s+1,u)$ and $(s,u+1)$, respectively.

\bigskip

Since holomorphic linearizations are not always available for saddle fixed points, the purpose of this section is to construct linearization maps whose differentials are $\Cb$-linear along the stable and unstable manifolds, under suitable bunching and non-resonance assumptions. A key intermediate step will be the construction of distinguished invariant directions along these manifolds, which we call \emph{dynamical directions}. This construction is inspired by Dujardin's notion of \emph{zero dynamical slope} \cite[Proposition~4.5]{dujardin-homoclinic}, which itself is based on ideas of Takens \cite{takens}.

Along the stable manifold, these directions coincide with the unstable eigendirections coming from the linearization. The crucial point, however, is that they are complex directions depending holomorphically on both the base point and the parameter. They will play an important role in the sequel, in particular for producing bifurcations when the submanifold varies while the dynamics is kept fixed (see Section~\ref{sec-subvar}).

In order to state the result, we first need to introduce the main assumptions. Throughout this section, we consider a complex manifold $M$, that we assume to be connected and simply connected, in order to avoid monodromy issues. We view $M$ as a parameter space for a holomorphic family of germs $(f)_{f \in M}$
of holomorphic maps on $\mathbb{C}^k$, each preserving a family of fixed points $(p(f))_{f \in M}$. We assume that these fixed points are all of saddle type, of index $(s,u)$ with $s+u=k$. That is, we assume that the multipliers, i.e., the eigenvalues of $D_{p(f)}f$, satisfy
\begin{equation*}\label{eq-multi}
0<|\lambda_1(f)|\leq\cdots\leq|\lambda_s(f)|<1<|\rho_1(f)|\leq\cdots\leq|\rho_u(f)|.
\end{equation*}

The main assumptions for this section depend on an integer $\ell\geq4$ and are the following. The first one is a bunching-type condition.
\begin{equation}\tag{$\mathcal B_\ell$}\label{eq-bun}
\begin{minipage}{0.9\linewidth}
There exist constants $0<a<b<1$ such that $b^{\ell-1}<a^2$, and such that for all
$1\leq i\leq s-1$, $1\leq j\leq u-1$, and $f\in M,$
\[|\lambda_1(f)|,\ |\lambda_s(f)|,\
\left|\frac{\lambda_i(f)}{\lambda_{i+1}(f)}\right|,\
|\rho_1(f)|^{-1},\  |\rho_u(f)|^{-1},\
\left|\frac{\rho_{j}(f)}{\rho_{j+1}(f)}\right|
\in [a,b].\]
\end{minipage}
\end{equation}
The second assumption concerns resonance relations among the multipliers.
\begin{equation}\tag{$\mathcal R_\ell$}\label{eq-res}
\begin{minipage}{0.9\linewidth}
For any multi-index
$N=(m_1,\ldots,m_s,n_1,\ldots,n_u)\in\Zb^{s+u}$ with
\[2\leq |N|:=|m_1|+\cdots+|m_s|+|n_1|+\cdots+|n_u| \leq 2\ell^2+1,\]
we have for all $f\in M$
\[|\lambda_1(f)^{m_1}\cdots\lambda_s(f)^{m_s}\,
\rho_1(f)^{n_1}\cdots\rho_u(f)^{n_u}|\neq1.\]
\end{minipage}
\end{equation}
Notice that we allow negative coefficients in the multi-indices $N$. There are two reasons for this.
First, in Proposition \ref{prop-dynaslope}, we will consider fixed points in the projectivized tangent bundle over
$W^s_{p(f),loc}$ (or $W^u_{p(f),loc}$), where the multipliers are quotients of the form
$\lambda_i/\rho_j$ or $\rho_j/\rho_i$.
Second, the proof of Proposition~\ref{prop-bunching} relies on the fact that
\[|\lambda^m\rho_i|\neq|\lambda_j|
\quad\text{and}\quad
|\lambda^m|\neq|\lambda^{m'}|\]
for all multi-indices $m,m'\in\Nb^{s}$ with $|m|,|m'|\leq\ell$ and $m\neq m'$.

However, condition~\eqref{eq-res} ensures that the standard \emph{strong Sternberg condition of order}
$2\ell^2$ (see \cite[Section~VII]{sell}) is satisfied. Moreover, condition~\eqref{eq-bun}
implies that the \emph{$2\ell^2$-smoothness} of $D_{p(f)}f$ in the sense of Sell is greater than $\ell$ (see the proof of Proposition \ref{prop-cderiv} for the details).
As a consequence, by Sell's results \cite[Theorems~5 and~9]{sell}, the fixed point $p(f)$
is $C^\ell$-linearizable, with a linearization depending continuously on $f$.

\begin{proposition}\label{prop-cderiv}
Let $(p(f))_{f\in M}$ be a family of saddle fixed points satisfying conditions \eqref{eq-bun} and \eqref{eq-res}. There exists a family of $C^2$-linearization maps $(\theta_f)_{f\in M}$, depending continuously on $f$ in the $C^2$ topology, such that for every $y\in W^u_{p(f),loc}\cup W^s_{p(f),loc},$ the differential $D_y\theta_f$ is $\Cb$-linear and depends holomorphically on both $y$ and $f$.
\end{proposition}

There are two proofs of this proposition, following parallel but conceptually different approaches. The first one uses the non-resonance assumptions to eliminate low-order nonlinear terms of $f$ by a holomorphic change of coordinates, and then applies a $C^\ell$ linearization for $\ell\geq1$ sufficiently large. Arguing as in \cite[Lemma~3.10]{gtv}, the bunching condition then implies that the differential is $\Cb$-linear along the invariant manifolds, with the required dependence properties.

The second proof adopts a more geometric point of view. It first constructs invariant complex directions along the stable manifold by adapting Dujardin's zero dynamical slope argument to higher dimension, using the non-resonance assumptions. This step directly yields the correct holomorphic dependence on $y$ and $f$. One then shows that, under a suitable bunching condition, these dynamical directions must coincide with those arising from a $C^\ell$ linearization.

In this section, we follow this second approach, which is geometrically more natural.
We first restrict the problem to $W^s_{p(f),loc}$. Following Dujardin~\cite[Proposition~4.5]{dujardin-homoclinic}, we consider the dynamics on the projectivized tangent bundle $\Pb T\Cb^k$ over this manifold. This bundle is canonically isomorphic to
\[E:=W^s_{p(f),loc}\times\Pb^{k-1},\]
and we denote by $f_*$ the induced action of $f$ on $E$, defined by
\[f_*(x,v)=(f(x),D_xf(v)).\]
The action on the fiber above $p(f)$ is given by $D_{p(f)}f$, and therefore has exactly $k$ fixed points, corresponding to the $k$ eigendirections of $D_{p(f)}f$. Observe that,  condition \eqref{eq-bun} implies that the eigenvalues satisfy
\[0<|\lambda_1(f)|<\cdots<|\lambda_s(f)|<1<|\rho_1(f)|<\cdots<|\rho_u(f)|,\]
and we denote by $e_1^s(f),\ldots,e_s^s(f)$ and $e_1^u(f),\ldots,e_u^u(f)$ the corresponding eigendirections in $\Pb^{k-1}$.
\begin{proposition}\label{prop-dynaslope}
Under conditions \eqref{eq-bun} and \eqref{eq-res}, for each $1\leq i\leq u$, there exists a holomorphic section $\mathcal E_f^{u,i}$ of the bundle $E$, invariant under $f_*$, depending holomorphically on $f$, and satisfying
\[\mathcal E_f^{u,i}(p(f))=e_i^u(f).\]
\end{proposition}
We shall refer to these sections as the \emph{unstable dynamical directions} along $W^s_{p(f),loc}$. Since there also exist natural stable directions along $W^s_{p(f),loc}$, the full set of \emph{dynamical directions} consists of the collection of these stable and unstable sections. In what follows, we view them either as sections of $E$ or as rank-$1$ subbundles of the tangent bundle restricted to $W^s_{p(f),loc}.$

\begin{proof}
Let us fix an unstable eigenvalue $\rho_i(f)$ for some $1\leq i\leq u$. We choose homogeneous coordinates on $\Pb^{k-1},$ depending holomorphically on $f,$ so that
\[e_i^u(f)=[1:0:\cdots:0],\]
while the remaining eigendirections are represented by the other standard coordinate points, indexed according to the ordering of the remaining eigenvalues. It follows that the differential of $f_*$ at the fixed point
\[\hat e_i^u(f):=(p(f),e_i^u(f))\]
is conjugate, after a suitable linear change of coordinates, to the diagonal linear map
\[(x_1,\ldots,x_s,v_1,\ldots,v_{k-1})
\mapsto
(\lambda_1(f)x_1,\ldots,\lambda_s(f)x_s,\delta_1(f)v_1,\ldots,\delta_{k-1}(f)v_{k-1}),\]
where the coefficients $\delta_j(f)$ are given by
\begin{equation}\label{eq-valeur}
\delta_j(f)= \begin{cases}
\dfrac{\lambda_j(f)}{\rho_i(f)} & \text{for } 1\leq j\leq s,\\[0.4em]
\dfrac{\rho_{j-s}(f)}{\rho_i(f)} & \text{for } s+1\leq j< s+i,\\[0.4em]
\dfrac{\rho_{j-s+1}(f)}{\rho_i(f)} & \text{for } s+i\leq j\leq k-1.
\end{cases}
\end{equation}
Hence, $\hat e_i^u(f)$ is a fixed point of index $(2s+i-1,u-i).$ Its local stable manifold $W^s_{\hat e_i^u(f),loc}$ has thus dimension $2s+i-1.$ A resonance among the corresponding multipliers would imply a relation of the form
\[\lambda_1^{m_1}\cdots\lambda_s^{m_s}\,
\delta_1^{m_{s+1}}\cdots\delta_{s+i-1}^{m_{2s+i-1}}
=\lambda_{j_1}^{n_1}\delta_{j_2}^{n_2},\]
where $m=(m_1,\ldots,m_{2s+i-1})$ and $n=(n_1,n_2)$ are tuples of nonnegative integers satisfying $|m|\geq2$ and $|n|=1$.

Observe that, by \eqref{eq-bun} and \eqref{eq-valeur}, each multiplier appearing in this expression has a modulus in $[a^2,b].$ Since $b^{\ell-1}<a^2,$ this implies that if such an equality holds then $|m|\leq\ell-2.$  Hence, \eqref{eq-res} ensures that such relations are not possible.

Therefore, by the Poincar\'e-Dulac theorem, the dynamics on $W^s_{\hat e_i^u(f),loc}$ can be linearized. In particular, there exists an invariant manifold $W_{\hat e_i^u(f),loc},$ corresponding to the eigenvalues $(\lambda_1(f),\ldots,\lambda_s(f))$. This manifold has dimension $s$ and is transverse to all the remaining eigendirections, associated with the multipliers $\delta_1(f),\ldots,\delta_{k-1}(f)$, which generate the tangent space to the fiber $\{p(f)\}\times\Pb^{k-1}$ at $\hat e_i^u(f)$.

As a consequence, $W_{\hat e_i^u(f),loc}$ is transverse to the fibers in a neighborhood of $\hat e_i^u(f)$, and the natural projection onto $W^s_{p(f),loc}$ restricts to a local biholomorphism. Invariance under the dynamics then ensures that $W_{\hat e_i^u(f),loc}$ defines a holomorphic section over $W^s_{p(f),loc}$. The holomorphic dependence on $f$ follows directly from the fact that the Poincar\'e-Dulac linearization depends holomorphically on the map.
\end{proof}
It then follows that, after a suitable holomorphic change of coordinates which depends holomorphically on $f$, we may assume that $f$ is defined and
injective on $B_p:=B^s\times B^u$ with
\begin{itemize}
\item the fixed point $p(f)$ is constantly identified with $0$;
\item $B^s\times\{0\}$ is the local stable manifold of $f$ in $B_p$,
on which $f$ acts as an attracting diagonal matrix;
\item $\{0\}\times B^u$ is the local unstable manifold of $f$ in $B_p$,
on which $f$ acts as an expanding diagonal matrix;
\item along the stable and unstable manifolds, the dynamical directions
defined in Proposition~\ref{prop-dynaslope} coincide with the standard
coordinate directions in $\Pb^{k-1}$.
\end{itemize}
We then say that $p(f)$ is in \emph{standard form} on $B_p$.
This terminology will be used again in Section~\ref{sec-standard},
in a slightly adapted setting.
\begin{proposition}\label{prop-bunching}
Under conditions \eqref{eq-bun} and \eqref{eq-res}, assume that $p(f)$
is in standard form on $B_p$. If $\theta_f$ is a $C^\ell$-linearization normalized by
$D_{p(f)}\theta_f=\id$, then its differential is complex diagonal along the local stable and
unstable manifolds, and depends holomorphically on both the base point and $f$.
\end{proposition}
\begin{rk}\label{rk-dynaslope}
Actually, the proof gives in particular that the section $\mathcal E_f^{u,i}$ defined in Proposition  \ref{prop-dynaslope} is unique.
\end{rk}

\begin{proof}
Since the dependence on $f$ plays no role in the statement, we simplify notation by omitting most references to $f$.

The main step is to show that the bunching condition forces any $C^{\ell-1}$-field of real $2$-planes defined along the stable manifold, invariant under the dynamics and equal to $\Cb\cdot e_i^u$ at $0$, to coincide with the field associated to $e_i^u,$ induced by $\theta_f$.

To that end, we work in the coordinates defined by $\theta_f$ where the dynamics on the tangent bundle above the local stable manifold, which can be identified to $B^s\times\Cb^k$ is simply
\[(x_1,\ldots,x_s,v_1^s,\ldots,v_s^s,v_1^u,\ldots,v_u^u)\mapsto(\lambda_1x_1,\ldots,\lambda_sx_s,\lambda_1v_1^s,\ldots,\lambda_sv_s^s,\rho_1v_1^u,\ldots,\rho_uv_u^u).\]
Let $1\leq i\leq u$ and let $(P_x)$ be a $C^{\ell-1}$-field of real $2$-planes defined for all $x$ in a neighborhood of $0$, which is invariant by the dynamics and with $P_0=\Cb\cdot e_i^u.$ In particular, it is locally transverse to
\begin{equation}\label{eq-decom}
F := E_1^s \oplus \cdots \oplus E_s^s \bigoplus_{j\neq i} E_j^u,
\end{equation}
where $E_j^s := \Cb\cdot e_j^s$ and $E_j^u := \Cb\cdot e_j^u$. Since $F\oplus E_i^u=\Cb^k$, any real $2$-plane transverse to $F$ can be written as the graph of a real linear map
\[L\colon E_i^u\to F.\]
Accordingly, the Grassmannian of such planes can be identified with the space of these linear maps. Using the decomposition~\eqref{eq-decom}, each map $L$ splits as
\[L_j^s\colon E_i^u\to E_j^s
\quad\text{and}\quad
L_j^u\colon E_i^u\to E_j^u.\]
As a consequence, in a neighborhood of $0$, the family $(P_x)$ can be represented by a $C^{\ell-1}$ family of maps $(L_x)$, with $L_0=0$. The invariance under the dynamics of both the family $(P_x)$ and the decomposition~\eqref{eq-decom} implies that
\[L_{j,x}^s(y)=\lambda_j^{-1} L_{j,f(x)}^s(\rho_i y),
\qquad
L_{j,x}^u(y)=\rho_j^{-1} L_{j,f(x)}^u(\rho_i y),\]
where we still denote by $f$ the induced action on the stable manifold. Here, multiplication by a complex number is understood as multiplication by the corresponding $2\times2$ real matrix. Iterating these identities and writing $\rho_i=|\rho_i|\xi$ we have
\[L_{j,x}^s(y)=\lambda_j^{-n}|\rho_i|^nL_{j,f^n(x)}^s(\xi^n y),
\qquad
L_{j,x}^u(y)=\rho_j^{-n}|\rho_i|^nL_{j,f^n(x)}^u(\xi^n y).\]
For $i<j\leq u$, the inequality $|\rho_i|<|\rho_j|$ already implies that $L_{j,x}^u$ vanishes identically. For the remaining factors, we use the bunching condition. Fix $1\leq j\leq s$ and consider $L_{j,x}^s$. The other $u$-components can be treated in the same way.

Since $x\mapsto L_{j,x}^s$ is $C^{\ell-1}$ and real-valued, its Taylor
expansion has the form
\[
L_{j,x}^s(y)=
\sum_{|\alpha|+|\beta|\leq\ell-2}
(A_{\alpha,\beta}y+B_{\alpha,\beta}\bar y)
x^\alpha\bar x^\beta
+O(\|x\|^{\ell-1})\|y\|.
\]
Comparing coefficients in the invariance relation
\[
L_{j,x}^s(y)=\lambda_j^{-1}L_{j,f(x)}^s(\rho_i y),
\]
any nonzero $A_{\alpha,\beta}$ or $B_{\alpha,\beta}$ would give, after
taking moduli,
\[
|\lambda_j|
=
|\rho_i|\prod_{m=1}^s|\lambda_m|^{\alpha_m+\beta_m},
\]
which is excluded by \eqref{eq-res} and hyperbolicity. Hence all Taylor
coefficients of order at most $\ell-2$ vanish, so
\[
\|L_{j,f^n(x)}^s\|=O(b^{(\ell-1)n}).
\]
Therefore, by invariance and \eqref{eq-bun},
\[
\|L_{j,x}^s\|
\leq C(a^{-2}b^{\ell-1})^n\to0.
\]
The same argument applies to $L_{j,x}^u$ for $j<i$, while the case
$j>i$ was already treated above. Thus $L_x=0$ and $(P_x)$ is a constant distribution, as desired.

Returning to the original holomorphic coordinates of $f$, we observe that the images under $\theta_f$ of the dynamical directions along $W^s_{p(f),loc}$ define invariant $C^{\ell-1}$-fields of real $2$-planes. By the previous step, these fields must be constant, hence equal to the standard coordinate axes. Consequently, if $f$ is in standard form, the (real) differential of $\theta_f$ must be block diagonal, with real $2\times2$ blocks corresponding to the restrictions to the dynamical directions.

We first focus on the unstable directions. Since $\theta_f$ conjugates $f$ to its differential, the restriction of $D_x\theta_f$ to $E_i^u$, for $x\in W^s_{0,loc}$, satisfies
\begin{align*}
D_x\theta(v)&=\rho_i^{-n} D_{f^n(x)}\theta\circ D_xf^n(v)=\xi^{-n}D_{f^n(x)}\theta\circ D_xf^n(|\rho_i|^{-n}v)\\
&=\xi^{-n}D_{0}\theta\circ D_xf^n(|\rho_i|^{-n}v)+D_xf^n(|\rho_i|^{-n}v)O(\|f^nx\|)\\
&=\rho_i^{-n}D_xf^n(v)+D_xf^n(|\rho_i|^{-n}v)O(\|f^nx\|),
\end{align*}
where the last equality follows from the fact that $D_0\theta_f=\id$ and $\rho_i=\xi|\rho_i|$.

Since $v$ belongs to the invariant complex line $E_i^u$, the derivative $D_xf(v)$ acts as multiplication by a complex number $\rho_i(x)=\rho_i+O(x).$ Therefore,
\[\rho_i^{-n}D_xf^n=\prod_{j=0}^{n-1}\frac{\rho_i(f^j(x))}{\rho_i}=\prod_{j=0}^{n-1}(1+O(\|f^j(x)\|)),\]
and this product converges to a complex number $\delta_i(f,x)$. This proves that the restriction $D_x\theta_f|_{E_i^u}$ is given by multiplication by $\delta_i(f,x)$. Moreover, the convergence of this infinite product is locally uniform in $f$ and $x$. Hence, $\delta_i$ is holomorphic in both these variables.

The same argument applies to the stable directions, in a simpler form. Indeed, for
$x\in W^s_{0,loc}$ and $v\in E^s_j$, we have
\[
D_x\theta(v)=\lambda_j^{-n}D_{f^n(x)}\theta(\lambda_j^n v)\to v,
\]
since $f^n(x)\to0$ and $D_0\theta=\id$.

This concludes the proof on $W^s_{0,loc}$ and the result on $W^u_{0,loc}$ is obtained by passing to $f^{-1}$.\end{proof}

Propositions~\ref{prop-dynaslope} and~\ref{prop-bunching} together directly imply Proposition~\ref{prop-cderiv}.
\begin{proof}[Proof of Proposition \ref{prop-cderiv}]
First, note that both conditions~\eqref{eq-bun} and~\eqref{eq-res} are symmetric and therefore also hold for $f^{-1}$.
Applying Proposition~\ref{prop-dynaslope} to both $f$ and $f^{-1}$, we obtain dynamical directions along
$W^s_{p(f),loc}\cup W^u_{p(f),loc}$ that depend holomorphically on both the base point and the parameter $f$.
As a consequence, there exist holomorphic coordinates in which $f$ is in standard form, as defined above
Proposition~\ref{prop-bunching}.

Moreover, the bunching condition~\eqref{eq-bun} implies that
\[\frac{\big|\log|\lambda_1|\big|}{\big|\log|\lambda_s|\big|}<\ell\quad\text{and}\quad
\frac{\big|\log|\rho_u|\big|}{\big|\log|\rho_1|\big|}<\ell.\]
Setting $Q:=2\ell^2$, it follows from the trivial inequality $\ell^2-\ell\times\ell\geq0$ that the
$Q$-smoothness of $D_{p(f)}f$ in the sense of Sell \cite[Section VII]{sell} is at least $\ell$.
On the other hand, the non-resonance condition~\eqref{eq-res} ensures that $D_{p(f)}f$ satisfies the
strong Sternberg condition of order $Q$.
Therefore, by \cite[Theorems~5 and~9]{sell}, the map $f$ is $C^\ell$-linearizable, with linearization maps
$(\theta_f)_{f\in M}$ depending continuously on $f$ in the $C^\ell$-topology.
Finally, Proposition~\ref{prop-bunching} guarantees that the differential $D\theta_f$ is $\Cb$-linear
along $W^s_{p(f),loc}\cup W^u_{p(f),loc}$ and holomorphic in the base point and in $f$.
\end{proof}

For almost all the arguments in the rest of the paper, the existence of dynamical directions is sufficient. However, in order to prove that, for certain Euclidean open sets $\Omega$ of endomorphisms $f$ of $\Pb^k$, the unstable manifold $W^u_{p(f)}$ is Zariski dense for every $f\in\Omega$, we are led to consider higher-order jets of the linearization $\theta_f$.

Actually, under a sufficiently strong bunching assumption, Proposition~\ref{prop-cderiv} would already be enough. In the setting considered below, however, the presence of a weakly attracting direction forces us to control higher-order terms. More precisely, we need higher-order jets of $\theta_f$ to be holomorphic along the invariant manifolds. This allows us to work with expansions involving only the complex multipliers, rather than with the real multipliers, for which resonances are inevitable.

For the remainder of this section, we fix $f\in M$ and write
$p:=p(f)$. We assume that $p$ is in standard form on $B_p$.

\begin{proposition}
\label{prop-holomorphic-jets}
Let $1\leq r<\ell$ and assume moreover that, for the constants
$0<a<b<1$ appearing in \eqref{eq-bun},
\begin{equation}\label{eq-higher-bunching}
    b^{\ell-r}<a^{r+1}.
\end{equation}
Let $\theta$ be a $C^\ell$-linearization normalized by $D_{p}\theta=\id.$ Then, the $r$-jet of $\theta$ is holomorphic at every point of
\[
W^s_{p,loc}\cup W^u_{p,loc}.
\]

\end{proposition}

\begin{proof}
For $r=1$, the result and the assumptions are exactly as in Proposition~\ref{prop-bunching}. We may therefore assume that $r\geq2$.

Let
\[
K=\diag(\lambda_1,\ldots,\lambda_s,\rho_1,\ldots,\rho_u)
\]
and let $\psi:=\theta^{-1}$. Set
\[
P:=\partial\psi,\qquad Q:=\bar\partial\psi,
\qquad \mu:=P^{-1}Q.
\]
The map $P$ is invertible in a neighborhood of the invariant
manifolds. Since
\[
\psi\circ K=f\circ\psi
\]
and $f$ is holomorphic, the chain rule gives
\[
P_{Kz}K=D_{\psi(z)}f\,P_z,
\qquad
Q_{Kz}\bar K=D_{\psi(z)}f\,Q_z.
\]
Consequently,
\begin{equation}
\label{eq-mu-invariance}
        \mu_z=K^{-1}\mu_{Kz}\bar K.
\end{equation}

By Proposition~\ref{prop-bunching}, $D\psi$ is complex linear
along the local stable and unstable manifolds. Hence
\begin{equation}
\label{eq-mu-zero}
        \mu=0
        \quad\text{on}\quad
        W^s_{0,loc}\cup W^u_{0,loc}.
\end{equation}

We treat the stable manifold; the unstable case follows by applying
the same argument to $K^{-1}$. Write
\[
\Lambda:=K|_{E^s}.
\]
For $1\leq q\leq r-1$, set
\[
T_q(x):=D^q\mu_{(x,0)},
\qquad x\in W^s_{0,loc}.
\]
Differentiating \eqref{eq-mu-invariance}, and using that $K$ is
linear, gives
\begin{equation}
\label{eq-Tq-mu}
T_q(x)(v_1,\ldots,v_q)
 =
K^{-1}T_q(\Lambda x)
        (Kv_1,\ldots,Kv_q)\bar K .
\end{equation}

We claim that
\[
D^dT_q(0)=0,
\qquad 0\leq d\leq\ell-2-q.
\]
Indeed, after differentiating \eqref{eq-Tq-mu} $d$ times at $0$
and complexifying the relevant tensor space, each component of
$D^dT_q(0)$ is fixed by multiplication by a quotient of products
of multipliers involving at most
\[
q+d+2\leq\ell
\]
factors. Such a multiplier cannot be equal to $1$ by
\eqref{eq-res} and hyperbolicity. Hence every component vanishes.

Since $T_q$ is of class $C^{\ell-1-q}$, Taylor's formula now gives
\[
\|T_q(\Lambda^n x)\|
 =
O\bigl(\|\Lambda^n x\|^{\ell-1-q}\bigr)
 =
O\bigl(b^{(\ell-1-q)n}\bigr).
\]
On the other hand, iterating \eqref{eq-Tq-mu} in an adapted norm
yields
\[
\|T_q(x)\|
 \leq
C a^{-(q+2)n}\|T_q(\Lambda^n x)\|.
\]
Therefore
\[
\|T_q(x)\|
 \leq
C\left(
a^{-(q+2)}b^{\ell-1-q}
\right)^n.
\]
For $q\leq r-1$, we have
\[
b^{\ell-1-q}\leq b^{\ell-r}
<a^{r+1}\leq a^{q+2},
\]
by \eqref{eq-higher-bunching}. Letting $n$ tend to infinity gives
\[
T_q(x)=0.
\]
Together with \eqref{eq-mu-zero}, this proves
\[
D^q\mu=0
\quad\text{on }W^s_{0,loc},
\qquad 0\leq q\leq r-1.
\]

Since
\[
Q=P\mu,
\]
Leibniz's formula gives
\[
D^qQ=0
\quad\text{on }W^s_{0,loc},
\qquad 0\leq q\leq r-1.
\]
Thus the $r$-jet of $\psi$ is holomorphic along the stable manifold.
Its linear part being invertible, the same holds for the $r$-jet
of its inverse $\theta$.

Applying the argument to $K^{-1}$ proves the assertion along the
unstable manifold.
\end{proof}

We now record a consequence that will be useful later. Roughly speaking,
the higher-order holomorphicity obtained above prevents a proper analytic
subset transverse to the stable manifold from containing backward orbit
segments accumulating on a sufficiently generic point of the unstable
manifold.

\begin{proposition}
\label{prop-asymptotic-rigidity}
Under the assumptions and notation of
Proposition~\ref{prop-holomorphic-jets}, write the unstable
multipliers of $p$ as
\[
1<|\rho_1|<\cdots<|\rho_u|
\]
and assume in addition that
\[
|\rho_u|<|\rho_1|^{r+1}.
\]

Let $\Gamma\subset B_p$ be a complex submanifold of dimension $u$
intersecting $W^s_{p,loc}$ transversely at $w_0$, and let
$H\subset\Gamma$ be an analytic subset which is smooth at
$w_0$.

Assume that there exist $n_k\to+\infty$ and $w_k\in H$, with
\[
w_k\to w_0,
\qquad
y_k:=f^{n_k}(w_k)\to y_\infty\in W^u_{p,loc},
\]
such that the orbit segment
\[
w_k,f(w_k),\ldots,f^{n_k}(w_k)
\]
is contained in $B_p$, and all the unstable coordinates of
$\theta(y_\infty)$ are nonzero.

Then
\[
(H,w_0)=(\Gamma,w_0).
\]
In particular, $H$ contains a neighborhood of $w_0$ in $\Gamma$.
\end{proposition}

\begin{proof}
By Proposition~\ref{prop-holomorphic-jets}, the $r$-jet of $\theta$
is holomorphic at $w_0$.

Let
\[
R:=\diag(\rho_1,\ldots,\rho_u)
\]
be the unstable part of $D_pf$, and let $\pi_u$ denote the projection
onto the unstable coordinates in the linearizing space. Set
\[
z_k:=\pi_u\theta(w_k),
\qquad
\eta_k:=\pi_u\theta(y_k).
\]
Since the whole orbit segment from $w_k$ to $y_k$ is contained in
$B_p$, the conjugacy relation gives
\[
z_k=R^{-n_k}\eta_k.
\]
Moreover,
\[
\eta_k\to
\eta_\infty:=\pi_u\theta(y_\infty)\in(\Cb^*)^u.
\]
Since the first coordinate of $\eta_\infty$ is nonzero and
$|\rho_1|<\cdots<|\rho_u|$, there exist constants $c,C>0$ such that
\begin{equation}
\label{eq-asymptotic-size}
c|\rho_1|^{-n_k}
\leq \|z_k\|
\leq C|\rho_1|^{-n_k}
\end{equation}
for all sufficiently large $k$.

Consider
\[
g:=\pi_u\circ\theta|_\Gamma:
(\Gamma,w_0)\to(\Cb^u,0).
\]
Since $\theta$ sends the local stable manifold onto the stable
linear subspace,
\[
\ker D_{w_0}(\pi_u\circ\theta)
=
T_{w_0}W^s_{p,loc}.
\]
As $\Gamma$ is transverse to $W^s_{p,loc}$ and has dimension $u$,
the differential $D_{w_0}g$ is an isomorphism. Thus $g$ is a local
$C^\ell$-diffeomorphism. By Proposition~\ref{prop-holomorphic-jets},
its $r$-jet is holomorphic, and therefore so is the $r$-jet of its
local inverse.

Let
\[
d:=\dim_{w_0}H
\]
and assume by contradiction that $d<u$. Since $H$ is smooth at
$w_0$, there exists a holomorphic submersion
\[
h:(\Gamma,w_0)\to(\Cb^{u-d},0)
\]
whose zero set is $(H,w_0)$. Choose a complex linear form $L$ on
$\Cb^{u-d}$ such that
\[
d(L\circ h\circ g^{-1})_0\neq0,
\]
and set
\[
\varphi:=L\circ h\circ g^{-1}.
\]
The map $\varphi$ is of class $C^\ell$, vanishes on $g(H)$, and has
a holomorphic $r$-jet at $0$. Let $P$ be its Taylor polynomial of
order $r$. Then $P$ is a holomorphic polynomial satisfying
\[
P(0)=0,
\qquad
dP_0\neq0,
\]
and
\[
\varphi(z)=P(z)+O(\|z\|^{r+1}).
\]
Since $z_k=g(w_k)$ and $w_k\in H$, we obtain
\begin{equation}
\label{eq-polynomial-small}
        P(z_k)=O(\|z_k\|^{r+1}).
\end{equation}

Write
\[
P(z)=\sum_{1\leq|\alpha|\leq r}c_\alpha z^\alpha,
\qquad
\rho^\alpha:=\rho_1^{\alpha_1}\cdots\rho_u^{\alpha_u}.
\]
For two distinct multi-indices $\alpha$ and $\beta$ in the support
of $P$, the numbers $|\rho^\alpha|$ and $|\rho^\beta|$ are distinct.
Indeed, otherwise
\[
|\rho^{\alpha-\beta}|=1.
\]
This is impossible by hyperbolicity if $|\alpha-\beta|=1$, and by
condition~\eqref{eq-res} otherwise, since
\[
2\leq|\alpha-\beta|
\leq|\alpha|+|\beta|
\leq2r
<2\ell^2+1.
\]

Let $\alpha_0$ be the unique multi-index in the support of $P$ for
which $|\rho^{\alpha_0}|$ is minimal. Using
$z_k=R^{-n_k}\eta_k$, we obtain
\[
P(z_k)
=
(\rho^{-\alpha_0})^{n_k}
\left(
c_{\alpha_0}\eta_k^{\alpha_0}
+
\sum_{\alpha\neq\alpha_0}
c_\alpha
\left(
\frac{\rho^{\alpha_0}}{\rho^\alpha}
\right)^{n_k}
\eta_k^\alpha
\right).
\]
By the choice of $\alpha_0$, every term in the sum tends to zero,
whereas
\[
c_{\alpha_0}\eta_k^{\alpha_0}
\to
c_{\alpha_0}\eta_\infty^{\alpha_0}\neq0.
\]
Hence
\begin{equation}
\label{eq-polynomial-main-term}
        |P(z_k)|\asymp|\rho^{\alpha_0}|^{-n_k}.
\end{equation}

Finally, since $dP_0\neq0$, the support of $P$ contains a multi-index
of degree one. By the minimality of $|\rho^{\alpha_0}|$,
\[
|\rho^{\alpha_0}|
\leq|\rho_u|
<|\rho_1|^{r+1}.
\]
Combining this with \eqref{eq-asymptotic-size} and
\eqref{eq-polynomial-main-term}, we obtain
\[
\frac{|P(z_k)|}{\|z_k\|^{r+1}}
\geq
c'
\left(
\frac{|\rho_1|^{r+1}}{|\rho^{\alpha_0}|}
\right)^{n_k}
\to+\infty,
\]
for some $c'>0$, contradicting \eqref{eq-polynomial-small}.

Therefore $d=u$, and since $H$ is smooth at $w_0$, the germ
$(H,w_0)$ is open in $(\Gamma,w_0)$. Hence
\[
(H,w_0)=(\Gamma,w_0).
\]
\end{proof}

We now apply this result to $f^{-1}$ to obtain a simple criterion ensuring
that an unstable manifold is Zariski dense in $\Pb^k.$ A very similar argument will
be applied later to $f$ itself in Lemma~\ref{le-zariski}, where the
assumptions of Section~\ref{sec-ass} allow us to take $r=1$.

\begin{proposition}
\label{prop-zariski-dense-unstable}
Under the assumptions of Proposition~\ref{prop-holomorphic-jets}, assume
that $f$ is an endomorphism of $\Pb^k$. Write
\[
0<|\lambda_1|<\cdots<|\lambda_s|<1
\]
for the stable multipliers of $p$, and assume that
\[
|\lambda_s|^{r+1}<|\lambda_1|.
\]

Assume moreover that there exist $N\geq1$ and
$\tilde q\in W^u_{p,loc}$ such that $D_{\tilde q}f^N$ is
invertible and, setting $q:=f^N(\tilde q)$,
\[
q\in W^s_{p,loc}\setminus\{p\},
\]
the germ $(f^N(W^u_{p,loc}),q)$ is transverse to $W^s_{p,loc}$ at $q$. Assume finally that all the
stable coordinates of $\theta(q)$ are nonzero.

Then $W^u_p$ is Zariski dense in $\Pb^k$.
\end{proposition}
\begin{rk}
A natural strategy to prove the Zariski density of $W^u_p$ would be to
consider a H\'enon-like map of main dynamical degree $2$, obtained by
 the homoclinic transition through $q$ for a sufficiently large
iterate of $f$. The assumption that $q$ does not lie on any coordinate
axis of $W^s_{p,loc}$ suggests, possibly with additional assumptions, that the main dynamical degree should
strictly dominate all the other dynamical degrees, so that the Zariski density would follow from the results
of Bianchi-Dinh-Rakhimov \cite{BDR-continuity-super-pot}. However, we did not find a simple
proof of this domination.\end{rk}

\begin{proof}
Let $Y$ be a sufficiently small neighborhood of $q$ in
$W^s_{p,loc}$. In particular, the germ
$f^N(W^u_{p,loc})$ is transverse to $Y$ at $q$. Let $\mathcal P$
be the set of points $x\in Y$ for which there exist $n\geq1$ and
$\tilde x\in W^u_{p,loc}$ such that
\[
x=f^n(\tilde x),\qquad D_{\tilde x}f^n\ \text{is invertible},
\]
and the germ $f^n(W^u_{p,loc})$ is transverse to $Y$ at $x$.

Assume that $\mathcal P$ is contained in a proper analytic subset of
$Y$, and let $H\subsetneq Y$ be an analytic subset of minimal
dimension containing $\mathcal P$. By minimality,
\[
\mathcal P\cap H_{\rm reg}\neq\varnothing,
\]
and let $x\in\mathcal P\cap H_{\rm reg}.$ Choose $n$ and $\tilde x$ as above, and let $\phi$ be the local
inverse of $f^n$ sending $x$ to $\tilde x$. Set
\[
\Gamma:=\phi(Y),\qquad H':=\phi(H),\qquad w_0:=\tilde x.
\]
Then $H'$ is smooth at $w_0$ and $\Gamma$ is transverse to
$W^u_{p,loc}$ there.

We now apply Proposition~\ref{prop-asymptotic-rigidity} to the local
inverse $g:=f^{-1}$. By the inclination lemma, $g^m(\Gamma)$ converges
to $W^s_{p,loc}$. Hence, for $m$ large, it intersects
$f^N(W^u_{p,loc})$ transversely at a point $y_m\to q$. Writing
\[
w_m:=g^{-m}(y_m)\in\Gamma,
\]
we have $w_m\to w_0$, and, by construction,
$\phi^{-1}(w_m)\in\mathcal P$, hence $w_m\in H'$.

The unstable multipliers of $g$ are
\[
|\lambda_s|^{-1}<\cdots<|\lambda_1|^{-1},
\]
and the assumed inequality
\[
|\lambda_s|^{r+1}<|\lambda_1|
\]
is precisely the assumption required in
Proposition~\ref{prop-asymptotic-rigidity}. Moreover, the unstable
coordinates of $q$ for $g$ are the stable coordinates of $q$ for $f$,
which are all nonzero. Proposition~\ref{prop-asymptotic-rigidity}
therefore gives
\[
(H',w_0)=(\Gamma,w_0),
\]
a contradiction. Thus $H=Y$ and  $\mathcal P$ is not contained in any proper analytic subset of
$Y$.

This implies that $W^u_p$ is Zariski dense in $\Pb^k.$ Indeed, let $F$ be a
homogeneous polynomial vanishing on $W^u_p$. Since
$\mathcal P\subset W^u_p$, the restriction of $F$ to $Y$ vanishes
identically.

Choose local holomorphic coordinates $(x,z)$ near $Y$ such that
$Y=\{z=0\}$, and write
\[
F(x,z)=\sum_{\beta} F_\beta(x)z^\beta.
\]
We claim by induction on $|\beta|$ that all the functions $F_\beta$
vanish identically. This is clear for $|\beta|=0$. Now, for
$q\in\mathcal P$, let $\Gamma_q\subset W^u_p$ be the corresponding
local branch transverse to $Y$. It can be written near $q$ as
\[
x=q+\varphi_q(z).
\]
Since $F$ vanishes on $\Gamma_q$, we have
\[
0=F(q+\varphi_q(z),z)
=\sum_{\beta}F_\beta(q+\varphi_q(z))z^\beta.
\]
Assuming that $F_\beta\equiv0$ for $|\beta|<m$, and using
$\varphi_q(z)=O(\|z\|)$, the homogeneous part of degree $m$ is
\[
\sum_{|\beta|=m}F_\beta(q)z^\beta.
\]
It therefore vanishes identically, so
\[
F_\beta(q)=0
\qquad\text{for all }|\beta|=m.
\]
Since $\mathcal P$ is not
contained in any proper analytic subset of $Y$, we get
$F_\beta\equiv0$ on $Y$.

Thus $F$ vanishes on a neighborhood of $Y$, hence identically.
Therefore no nonzero homogeneous polynomial vanishes on $W^u_p$, so
$W^u_p$ is Zariski dense in $\Pb^k$.
\end{proof}

\section{Horizontal/vertical relations and horizontal-like maps}\label{sec-graph}

In Section~\ref{sec-blender}, we consider holomorphic horizontal-like dynamical systems.
Such systems have been extensively studied in dimension~$2$ (see~\cite{dujardin-hl}) and in higher dimension (see, for instance,~\cite{dns}).
We will only use basic properties of these maps, which are classical. The purpose of the present section is simply to provide a geometric framework for our semi-local dynamics.

A simple but important feature in our setting is the following.
We will work in a partially hyperbolic situation where (strong) stable manifolds appear as \emph{horizontal} submanifolds of dimension~$s$, while (strong) unstable manifolds are \emph{vertical} submanifolds of dimension~$u$, inside an ambient space of dimension~$k=u+s+1$.
In particular, these manifolds do not intersect generically.
However, when the dynamics depends on parameters, relations between stable and unstable manifolds give rise to hypersurfaces in parameter space, which will play a crucial role in what follows.

\subsection{Special graphs and relations}\label{sec-graph-spe}

Let $k\geq2$ and $1\leq u\leq k-1$, and set $s:=k-u-1$.
We consider three bounded polydiscs $N_c$, $N_s$, and $N_u$ of dimensions $1$, $s$, and $u$, respectively, and define
\[\mathcal N := N_c \times N_s \times N_u \ \text{ and }\ N_{cs}:=N_c\times N_s,\ N_{cu}:=N_c\times N_u.\]
These decompositions induce natural projections
\[\pi_u:\mathcal N\to N_u,\quad
\pi_s:\mathcal N\to N_s,\quad
\pi_{cs}:\mathcal N\to N_{cs},\quad
\pi_{cu}:\mathcal N\to N_{cu},\]
and a natural decomposition of the boundary $\partial\mathcal N$ as
\[\partial_c\mathcal N := (\partial N_c)\times N_s\times N_u,\quad
\partial_s\mathcal N :={N_c}\times(\partial N_s)\times{N_u},\quad
\partial_u\mathcal N :={N_c\times N_s}\times(\partial N_u).\]
Observe that these sets are not closed in $\overline{\mathcal N}.$ We also set
\[\partial_{cs}\mathcal N := \partial_c\mathcal N \cup \partial_s\mathcal N,
\qquad
\partial_{cu}\mathcal N := \partial_c\mathcal N \cup \partial_u\mathcal N.\]
A closed complex submanifold $X\subset\mathcal N$ is called a \emph{$u$-graph} if
$\pi_{u|X}:X\to N_u$ is a biholomorphism and $\pi_{cs}(X)\Subset N_{cs}$.
The notions of $s$-graphs, $cs$-graphs, and $cu$-graphs are defined analogously.

It is immediate that a $u$-graph intersects (transversely) any $cs$-graph in exactly one point.
\begin{lemma}\label{le-inter}
Let $X$ be a $u$-graph (resp.\ $s$-graph) and $Y$ a $cs$-graph (resp.\ $cu$-graph) in $\mathcal N$.
Then $X\cap Y$ consists of a single point, and the intersection is transverse.
\end{lemma}
\begin{proof}
There exist two holomorphic maps $h\colon N_u\to N_{cs}$ and $g\colon N_{cs}\to N_u$ such that
\[X=\{(h(z),z)\ ;\ z\in N_u\},\ Y=\{(y,g(y))\ ;\ y\in N_{cs}\},\]
with $h(N_u)\Subset N_{cs}$ and $g(N_{cs})\Subset N_u$. In particular, $G:=g\circ h$ is a holomorphic self-map of $N_u$ with $G(N_u)\Subset N_u$. Therefore, it must be contracting, with a unique attracting fixed point $z_0\in N_u$. Since the points in $X\cap Y$ correspond exactly to fixed points of $G$, this intersection is equal to $x_0:=(h(z_0),z_0)$. Moreover, an intersection between the tangent spaces of $X$ and $Y$ corresponds to a fixed point of $D_{z_0}G$. As $z_0$ is attracting for $G$, this intersection is reduced to $\{0\}$.
\end{proof}

In what follows, $u$-graphs will be regarded as \emph{vertical} and $s$-graphs as \emph{horizontal}.
Although such graphs do not generically intersect, parameterized families of them naturally define hypersurfaces in parameter space.

More precisely, let $M$ be a connected complex manifold of dimension $m$ which is relatively compact in a larger manifold.
Let $(V_\lambda)_{\lambda\in M}$ and $(H_\lambda)_{\lambda\in M}$ be holomorphic families of $u$-graphs and $s$-graphs, respectively. Their \emph{total spaces} are defined by
\[V := \{(\lambda,z)\in M\times\mathcal N \; ; \; z\in V_\lambda\}\ \text{ and }\ H := \{(\lambda,z)\in M\times\mathcal N \; ; \; z\in H_\lambda\}.\]
A simple dimension count yields the following result.

\begin{lemma}\label{le-dim-relation}
Let $(V_\lambda)_{\lambda\in M}$ and $(H_\lambda)_{\lambda\in M}$ be as above and assume there exist compact sets $K_{cu}\subset N_c\times N_u$ and $K_{cs}\subset N_c\times N_s$ such that for all $\lambda\in M$,
\[\pi_{cs}(V_\lambda)\subset K_{cs}, \qquad \pi_{cu}(H_\lambda)\subset K_{cu}.\]
Then
\[M':=\{\lambda\in M\ ;\ V_\lambda\cap H_\lambda\neq\varnothing\}\]
is a closed analytic subset of $M,$ whose irreducible components have dimension $m$ or $m-1$.
\end{lemma}
\begin{proof}
Since $\pi_{cs}(V_\lambda)\subset K_{cs}$ and $\pi_{cu}(H_\lambda)\subset K_{cu}$, the closure of $V\cap H$ is disjoint from
$\overline{M}\times\partial\mathcal N$. Thus, its projection on $M$ has to be a closed analytic subset of $M$ and it corresponds exactly to $M'.$

Moreover, let $Z$ be an irreducible component of $V\cap H$. Then
\[\dim(Z)\geq \dim(V)+\dim(H)-\dim(M\times\mathcal N)
= (u+m)+(s+m)-(u+s+1+m)=m-1.\]
On the other hand, using again that $\overline{V\cap H}$ is disjoint from
$\overline{M}\times\partial\mathcal N$, we have that $\dim(Z)\leq m$.
\end{proof}

Notice that when $\dim(Z)=m$, we have $M'=M$, and therefore $Z$ is a graph over $M$, since for each $\lambda$ there exists at most one point in $V_\lambda\cap H_\lambda$. In this case, we say that the intersection between $(V_\lambda)_{\lambda\in M}$ and $(H_\lambda)_{\lambda\in M}$ can be \emph{followed holomorphically} over $M$.

On the other hand, if $\dim(Z)=m-1$, then $Z$ is a graph over an irreducible component of $M'$, and the intersection between $(V_\lambda)_{\lambda\in M}$ and $(H_\lambda)_{\lambda\in M}$ defines a \emph{nontrivial horizontal/vertical relation}.

Observe that it may also happen that $V\cap H=\varnothing$.

The stability of proper intersections yields the following result, fundamental for what follows.
\begin{lemma}\label{le-proper}
Let $(V_\lambda)_{\lambda\in M},$ $(H_\lambda)_{\lambda\in M}$ and $M'$ be as in Lemma \ref{le-dim-relation} and assume that $M'$ has dimension $m-1.$ Then, if $(\tilde V_\lambda)_{\lambda\in M}$ and $(\tilde H_\lambda)_{\lambda\in M}$ are other holomorphic families of $u$-graphs and $s$-graphs sufficiently close to $(V_\lambda)_{\lambda\in M}$ and  $(H_\lambda)_{\lambda\in M}$ respectively, then their total spaces
\[\tilde V := \{(\lambda,z)\in M\times\mathcal N \; ; \; z\in \tilde V_\lambda\}\ \text{ and }\ \tilde H := \{(\lambda,z)\in M\times\mathcal N \; ; \; z\in \tilde H_\lambda\}\]
satisfy $\dim(\tilde V\cap\tilde H)=m-1$ and this intersection is arbitrarily close to $V\cap H$.
\end{lemma}
\begin{proof}
The assumption $\dim(M')=m-1$ exactly says that the intersection $V\cap H$ is proper. Thus, the result is a direct consequence of the stability of proper intersections (see e.g.\cite[\textsection 12.3]{Chirka}).
\end{proof}

\subsection{Degree~$1$ and injective horizontal-like maps}

In our partially hyperbolic setting, we introduce horizontal-like maps adapted to
the decomposition of the boundary of $\mathcal N$.
We refer to~\cite[Chapter 1]{dujardin-thesis} for a more detailed discussion on this concept.

Let $N_u'\subset\Cb^u$ be a larger unstable polydisc and set
\[\mathcal N' := N_c\times N_s\times N_u',
\qquad
\mathcal N\subset\mathcal N'.\]
This enlargement implies that $u$-graphs cross a region larger than the one where
the horizontal-like map is defined.

\begin{definition}\label{def-hl}
An injective holomorphic map defined on a neighborhood of $\overline{\mathcal N}$
is called a \emph{horizontal-like map of index $(s+1,u)$} from $\mathcal N$
to $\mathcal N'$ if
\begin{itemize}
\item $f(\mathcal N)\cap\mathcal N\neq\varnothing$,
\item $f(\partial_u\mathcal N)\cap\overline{\mathcal N'}=\varnothing$,
\item $f(\overline{\mathcal N})\cap\partial\mathcal N'\subset\partial_u\mathcal N'$.
\end{itemize}
\end{definition}

Since $\partial_{cs}\mathcal N\subset\partial_{cs}\mathcal N'$, such a map is also
horizontal-like from $\mathcal N$ to itself.

The \emph{degree} of a horizontal-like map is the cardinality of
$f(X)\cap Y$, where $X$ and $Y$ are generic $u$-graph and $cs$-graph in
$\mathcal N$ and $\mathcal N'$, respectively.
Here, we only consider degree~$1$ maps, which is a particularly simple setting.

For such maps, $f(X)\cap\mathcal N$ remains a $u$-graph in $\mathcal N$,
called the \emph{graph transform} of $X$ by $f$.
On the other hand, since $f$ is injective, if $Y$ is a $cs$-graph of $\mathcal N$ then the image of $f(\mathcal N)\cap Y$ by $f^{-1}$ is again a $cs$-graph in $\mathcal N,$ which corresponds to the \emph{graph transform} of $Y$ by $f^{-1}.$

%

It is easy to see that maximal invariant set of a degree $1$ horizontal-like map of index $(s+1,u)$ is reduced to a single point, which has to be a saddle fixed point $p(f)$ of index $(s+1,u).$ Its local unstable manifold $W^u_{p(f),loc}$ is a $u$-graph whereas its local stable manifold $W^s_{p(f),loc}$ is a $cs$-graph. They are obtained by
\begin{equation}\label{eq-manifolds}
W^u_{p(f),loc}=\bigcap_{n\geq0} f^n(\mathcal N) \text{ and } W^s_{p(f),loc}=\bigcap_{n\geq0}f^{-n}(\mathcal N).
\end{equation}

In Section~\ref{sec-blender}, the saddle point appearing in
Assumption~\ref{saddle} will be obtained in this way.
The second main dynamical object, namely a hyperbolic saddle set
$\Lambda(f)$ of index $(s,u+1)$ in Assumption~\ref{hyperbolic},
arises from a similar construction using two horizontal-like maps with
the same target polydisc.

More precisely, a \emph{horizontal-like map of index $(s,u+1)$} is defined
as above, replacing $\partial_u$ with $\partial_{cu}$ in
Definition~\ref{def-hl}.
Let $N_{u,+},N_{u,-}\subset N_u'$ be two disjoint polydiscs and define
\[\mathcal N_\pm := N_c\times N_s\times N_{u,\pm}\ \text{ and } \tilde{\mathcal N}:=\mathcal N_+\cup\mathcal N_-.\]

\begin{definition}\label{def-hl2}
A holomorphic map $f$ defined on a neighborhood of
$\tilde{\mathcal N}$ is called a \emph{degree~$2$ horizontal-like
map of index $(s,u+1)$} from $\tilde{\mathcal N}$ to $\mathcal N'$
if the restrictions $f|_{\mathcal N_+}$ and $f|_{\mathcal N_-}$ are
degree~$1$ horizontal-like maps of index $(s,u+1)$ into $\mathcal N'$.
\end{definition}
%
In Section~\ref{sec-ass}, we consider a degree~$2$ horizontal-like
map $f$ of index $(s,u+1)$ from $\tilde{\mathcal N}$ to $\mathcal N'$
such that its maximal invariant set
\[\Lambda := \bigcap_{n \in \Zb} f^n(\tilde{\mathcal N})\]
is hyperbolic of index $(s,u+1)$. Note that, using contractions of the Kobayashi metrics, one can prove that this always holds.

In this setting, for $x \in \Lambda$, we define its local stable manifold as follows. For each $y \in \tilde{\mathcal N}$, let $L_y^s$ denote the horizontal
linear subspace passing through $y$, namely
\[L_y^s := \{\pi_c(y)\} \times N_s \times \{\pi_u(y)\}.\]
Then, for $x \in \Lambda$, we set
\[W^s_{x,loc}
:= \lim_{n \to \infty} f^{-n}\big(L^s_{f^n(x)}\big).\]
Observe that here (and throughout the paper in the same context),
$f^{-n}\big(L^s_{f^n(x)}\big)$ denotes the successive graph transforms
under $f^{-1}$. As there are two inverse branches of $f,$ at each step, we choose the one passing from $f^{i+1}(x)$ to $f^{i}(x).$

On the other hand, since the maps considered here are not invertible, unstable objects are
naturally defined on the natural extension
\[
\hat\Lambda
:=
\left\{(x_j)_{j\leq0}\in\Lambda^{\mathbb Z_{\leq0}}
\ ;\ f(x_j)=x_{j+1}\right\}.
\]
Then, the local unstable manifold of a prehistory $\hat x=(x_j)_{j\leq0}\in\hat\Lambda$ is
\[W^u_{\hat x,loc}
:= \lim_{n \to+\infty} f^{n}\big(L^{cu}_{x_{-n}}\big),\]
where $L_y^{cu}:=N_c\times\{\pi_s(y)\}\times N_u'.$ Note that, whereas the local stable manifold only depends on the
coordinate \(x_0\), the local unstable manifold may depend on the whole
prehistory \(\hat x=(x_j)_{j\leq0}\). We refer to Przytycki~\cite{przytycki-anosov} for the study of natural extensions of hyperbolic saddle sets for non-invertible maps, and to Jonsson~\cite{jonsson-phd} for an exposition particularly suited to our purposes.

In a holomorphic family, Jonsson's theorem
\cite[Theorem~B]{jonsson-motion_hyperbolic} shows that prehistories in
the natural extension admit holomorphic continuations. After projection
onto the zeroth coordinate, this gives in general only a branched
holomorphic motion of \(\Lambda\).

To lighten notation, in Section~\ref{sec-blender} we shall often suppress the natural extension whenever no confusion can arise.
Thus, when a point
\(x(f)\in\Lambda(f)\) is followed holomorphically, a prehistory is
implicitly chosen. Similarly, an unstable object such as
\(W^u_{x(f),loc}\) is understood to be associated with this chosen
prehistory. Stable objects, which only depend on the zeroth coordinate,
are unaffected by this convention. 

\begin{proposition}\label{prop-hyp-graphs}
In the situation above, we have that
\begin{itemize}
\item the local stable manifolds $W^s_{x,loc}$, for $x\in\Lambda$, are $s$-graphs in $\tilde{\mathcal N}$,
\item for each $\hat x$ in the natural extension $\hat\Lambda$ of $\Lambda$, the local unstable manifold $W^u_{\hat x,loc}$ is a $cu$-graph in $\mathcal N'$.
\end{itemize}
Moreover, these invariants manifolds are contained in the standard stable and unstable manifolds respectively and they depends holomorphically on $f.$
\end{proposition}
\begin{proof}
The first part follows directly from the fact that the graph transform
of an $s$-graph (resp. a $cu$-graph) by $f^{-1}$ (resp. by $f$)
is again an $s$-graph (resp. a $cu$-graph).

For the second point, observe first that, by definition of a
horizontal-like map, there exist compact sets
$K_s \Subset N_s$ and $K_{cu,\pm} \Subset N_{cu,\pm}$ such that
\[\pi_{cu}\big(\tilde{\mathcal N} \cap f^{-1}(\mathcal N')\big)
\subset K_{cu,+} \cup K_{cu,-}
\quad \text{and} \quad
\pi_s\big(\mathcal N' \cap f(\tilde{\mathcal N})\big)
\subset K_s.\]
In particular, for every $x \in \Lambda$, $\pi_{cu}(W^s_{x,loc})\subset K_{cu,+} \cup K_{cu,-},$
so the family of such $s$-graphs is normal. Moreover, by definition of the graph transform, if $x \in \Lambda$ then $f^{-1}(W^s_{f(x),loc}) = W^s_{x,loc},$
and hence $f(W^s_{x,loc}) \subset W^s_{f(x),loc}.$ By the second inclusion above, for each $y \in \Lambda$ we have $\pi_s(f(W^s_{y,loc})) \subset K_s.$
For $n \ge 0$, define
\[g_n := \pi_s \circ f \circ
\big(\pi_{s|W^s_{f^n(x),loc}}\big)^{-1}.\]
Then $g_n(N_s) \subset K_s$, and therefore the family $(g_n)_{n \ge 0}$
uniformly contracts the Kobayashi metric on $N_s$.
It follows that the diameter of the image of $G_n := g_{n-1} \circ \cdots \circ g_0$
converges to $0$.

Since
\[\pi_s\big(f^n(W^s_{x,loc})\big)
\subset G_n(N_s),\]
and since $f^n(W^s_{x,loc}) \subset W^s_{f^n(x),loc}$,
which belong to a normal family of $s$-graphs,
we conclude that for every $z \in W^s_{x,loc}$, $d(f^n(x),f^n(z))$ converges to $0.$
Hence $W^s_{x,loc}$ coincides with the standard local stable
manifold of $x$. The unstable case is proved in the same way.

If we have a holomorphic family of degree $2$ horizontal-like maps $(f_\lambda)_{\lambda\in M}$ parametrized by $M,$ then by \cite[Theorem~B]{jonsson-motion_hyperbolic} the natural extension $\hat\Lambda(f_\lambda)$ (resp. $\Lambda(f_\lambda)$) can be followed by a holomorphic motion (resp. branched holomorphic motion). The construction above directly implies that if $\hat x(\lambda)\in\hat\Lambda(\lambda)$ (resp. $x(\lambda)\in\Lambda(\lambda)$) then $W^u_{\hat x(\lambda),loc}$ and $W^s_{x(\lambda),loc}$ both depend holomorphically on $\lambda.$
\end{proof}

\subsection{Case of fibered products with parameters}\label{sec-hl-renor}

Observe that all the results above (except
Lemma~\ref{le-dim-relation} and Lemma~\ref{le-proper}) are
independent of the dimension of the central direction, which
was equal to $1$ in the previous setting. Throughout the paper,
we mainly work in this situation, except when dealing with
fibered products in Section~\ref{sec-courant}, where the central dimension increases. However, in that section, we will only deal with the dynamics near a saddle point where the central direction plays no special role.

Hence, we keep the notation of Section~\ref{sec-graph-spe}, except that we ignore the central direction. That is
\[\mathcal N := N_s \times N_u,\quad \partial_s\mathcal N :=(\partial N_s)\times{N_u},\quad \text{ and }\quad \partial_u\mathcal N :=N_s\times(\partial N_u).\]
We also fix a parameter space $M$, which is a complex manifold
of dimension $m$. The integers $s$ and $u$ here will correspond,
in the proof of Proposition~\ref{prop-positivevolume}, to $ms$ and $m(u+1)$, respectively.

We consider a holomorphic family of maps
$(f_\lambda)_{\lambda\in M}$, parametrized by $M$, consisting of
horizontal-like maps of index $(s,u)$ from $\mathcal N$ to
itself. More precisely, since there is no central direction this means that for each $\lambda\in M$, the map
$f_\lambda$ satisfies Definition~\ref{def-hl}.
Moreover, we assume that these maps have degree~$1$, so that
if $X$ is a $u$-graph in $\mathcal N$ (resp.\ an
$s$-graph), then $f_\lambda(X)\cap\mathcal N$ (resp.\
$f_\lambda^{-1}(X)\cap\mathcal N$) is again a $u$-graph
(resp.\ an $s$-graph).

The maximal invariant set of $f_\lambda$ in $\mathcal N$ reduces
to a single fixed point $r(f_\lambda)$, which is a saddle of
index $(s,u)$ and whose local stable and unstable manifolds
are defined as in~\eqref{eq-manifolds}. Moreover, denoting by
\[F\colon\mathcal N\times M\to\Cb^{s+u}\times M,
\qquad
F(x,\lambda)=(f_\lambda(x),\lambda),\]
we have
\[\mathcal W^u_{r,loc}=
\bigcap_{n\geq0} F^n(\mathcal N\times M),
\qquad
\mathcal W^s_{r,loc}=
\bigcap_{n\geq0} F^{-n}(\mathcal N\times M),\]
which assemble the invariant manifolds in the parameter space.

In the proof of Proposition~\ref{prop-positivevolume}, we will
consider the graph transform of certain subvarieties in
$\mathcal N\times M$. More precisely, let
$\Gamma_0\subset\mathcal N\times M$ be a closed irreducible
subvariety of dimension $u$, intersecting
$\mathcal W^s_{r,loc}$ at a single point and satisfying
\begin{equation}\label{eq-renor}
\Big(\overline{\partial_s\mathcal N\times M
\cup \mathcal N\times\partial M}\Big)\cap\overline{\Gamma_0}=\varnothing.
\end{equation}
\begin{proposition}\label{prop-renor}
Let $(f_\lambda)_{\lambda\in M}$ and $\Gamma_0$ as above. Define by induction, for $n\geq0$
\[\Gamma_{n+1}:=F(\Gamma_n)\cap(\mathcal N\times M).\]
Then, $(\Gamma_n)_{n\geq0}$ is a sequence of closed analytic subsets of $\mathcal N\times M$ converging (possibly with some multiplicity) to
\[W^u_{r(f_\lambda),loc}\times\{\lambda\},\]
for some $\lambda\in M.$
\end{proposition}
\begin{proof}
Observe first that, since for each $\lambda\in M$,
\[f_\lambda(\overline{\mathcal N})\cap\partial\mathcal N
\subset \partial_{u}\mathcal N,\]
each $\Gamma_n$ also satisfies~\eqref{eq-renor}.

On the other hand, again by~\eqref{eq-renor}, the restriction of
the projection
\[\pi\colon\mathcal N\times M\to N_u\]
to $\Gamma_n$ is a ramified covering. Its degree coincides
with the number of intersection points (counted with multiplicity)
between $\Gamma_n$ and
\[H := \{(z,\lambda)\in \mathcal N\times M \; ; \; z\in H_\lambda\},\]
where $(H_\lambda)_{\lambda\in M}$ is a family of $s$-graphs.
Since each $f_\lambda$ has degree~$1$, the sets $H'_\lambda:=f_\lambda^{-1}(H_\lambda)\cap \mathcal N$
also form a family of $s$-graphs, whose total space is denoted by $H'$. Then the number of intersection points (with multiplicity)
of $\Gamma_n\cap H'$ coincides with that of
$\Gamma_{n+1}\cap H$. It follows that the degree of
$\Gamma_n$ is independent of $n$.

Hence $(\Gamma_n)_{n\geq0}$ is a sequence of analytic sets of
dimension $u$ in $\mathcal N\times M$ with uniformly bounded
mass. By Bishop's theorem, it converges (up to extraction) to an
analytic set $\Gamma_\infty$ of dimension $u$. Moreover,
by construction, $\Gamma_n\subset
\bigcap_{i=0}^n F^i(\mathcal N\times M)$, hence
\[\Gamma_\infty\subset \mathcal W^u_{r,loc}.\]

Finally, by assumption, $\Gamma_0$ intersects
$\mathcal W^s_{r,loc}$ at a unique point
$(x_0,\lambda_0)$. If $\lambda\neq\lambda_0$, then
\[\Gamma_0\cap (\mathcal N\times\{\lambda\})=
\Gamma_0(\lambda)\times\{\lambda\},\]
where $\Gamma_0(\lambda)$ is a closed subset of $\mathcal N$ disjoint
from $W^s_{r(\lambda),loc}$. It follows that there exists
$n\geq0$ such that
\[\bigcap_{i=0}^n f_\lambda^{-i}(\mathcal N)\cap\Gamma_0(\lambda)=\varnothing,\]
and therefore
\[\Gamma_\infty\cap (\mathcal N\times\{\lambda\})=\varnothing.\]
Hence,
\[\Gamma_\infty\subset
W^u_{r(\lambda_0),loc}\times\{\lambda_0\}.\]
By dimensional considerations, this inclusion is in fact an equality.
Since the limit does not depend on the subsequence, the sequence
itself converges.
\end{proof}

\section{Dynamical pairs, bifurcations and parameter spaces}\label{sec-courant}
\subsection{Dynamical pairs}
For the material of this paragraph, we refer to \cite{gtv} and \cite{GV_Northcott}. Note that we do not follow the notations from \cite{gtv}. When studying algebraic families of objects, if $S$ is the parameter space, we use the notation $\hat{S}$ for the total space of this family in what follows.

 Fix an integer $k\geq2$ and let $\hat{f}:\mathbb{P}^k\times \mathcal{Y}\to\mathbb{P}^k\times \mathcal{Y}$ be a family of endomorphisms of degree $d$ parametrized by a quasi-projective complex manifold $\mathcal Y$, i.e. $\hat{f}$ is a morphism and $\hat{f}(x,t)=(f_t(x),t)$ for all $(x,t)\in \mathbb{P}^k\times \mathcal{Y}$, where $f_t:\mathbb{P}^k\to\mathbb{P}^k$ is an endomorphism whose degree $d\geq2$ is independent of $t$.
Let $\pi_\mathcal{Y}:\mathbb{P}^k\times \mathcal{Y}\to \mathcal{Y}$ and $\pi_{\mathbb{P}^k}:\mathbb{P}^k\times \mathcal{Y}\to \mathbb{P}^k$ be the two canonical projections.

When $\hat{f}$ and $\mathcal{Y}$ are defined over $\C$, the \emph{fibered Green current} of the family $\hat{f}$ is
\[\hat{T}_{\hat{f}}:=\lim_{n\to\infty}d^{-n}(\hat{f}^n)^*(\pi_{\mathbb{P}^k}^*(\omega)),\]
where $\omega$ is the Fubini-Study form of $\mathbb{P}^k$ of mass $1$, i.e. $\omega\in c_1(\mathcal{O}(1))$.
Let now $p:\hat{\mathcal{Y}}\to \mathcal{Y}$ be a family of subvarieties, i.e. $\hat{\mathcal{Y}}\subset\mathbb{P}^k\times \mathcal{Y}$ is a closed subvariety and $p:=\pi_\mathcal{Y}|_{\hat{\mathcal{Y}}}$ is projective and flat. Assume $\hat{\mathcal{Y}}$ has relative dimension $u$ with $1\leq u\leq k-1$.

\begin{definition}
The \emph{bifurcation current} of the pair $(\hat{f},\hat{\mathcal{Y}})$ is the closed positive current $T_{\mathcal{Y}}$ on $\mathcal{Y}(\C)$ defined as $T_{\mathcal{Y}}:=(\pi_\mathcal{Y})_*(\hat{T}_{\hat{f}}^{\wedge(u+1)}\wedge[\hat{\mathcal{Y}}])$.
\end{definition}

\par\noindent For any $\ell\geq1$, let $\hat{f}^{[\ell]}:(\mathbb{P}^k)^\ell\times \mathcal{Y}\to (\mathbb{P}^k)^\ell\times \mathcal{Y}$ be the fibered action of $\hat{f}$ given by
\[\hat{f}^{[\ell]}(z_1,\ldots,z_\ell,t):=(f_t(z_1),\ldots,f_t(z_\ell),t), \quad (z_1,\ldots,z_\ell,t)\in (\mathbb{P}^k)^\ell\times \mathcal{Y}.\]
When everything is defined over $\C$, the fibered Green current $\hat{T}_{\hat{f}^{[\ell]}}$ of $\hat{f}^{[\ell]}$ is given by $\hat{T}_{\hat{f}^{[\ell]}}=\sum_{i=1}^\ell p_i^*(\hat{T}_{\hat{f}})$, where $p_i:(\mathbb{P}^k)^\ell\times \mathcal{Y}\to \mathbb{P}^k\times \mathcal{Y}$ is the projection onto the $i$-th factor and let $\pi_{\mathcal{Y},\ell}:(\mathbb{P}^k)^\ell\times \mathcal{Y}\to \mathcal{Y}$ be the canonical projection. For any $\ell\geq1$, denote by
\[\hat{\mathcal{Y}}^{[\ell]}:=\hat{\mathcal{Y}}\times_\mathcal{Y}\cdots\times_\mathcal{Y}\hat{\mathcal{Y}}\]
the $\ell$-th fibered power of $p:\hat{\mathcal{Y}}\to \mathcal{Y}$. This defines a family $p_\ell:\hat{\mathcal{Y}}^{[\ell]}\to \mathcal{Y}$ of varieties of relative dimension $u\cdot \ell$.
\begin{definition}
The \emph{bifurcation measure} of the pair $(\hat{f},\hat{\mathcal{Y}})$ is the positive measure $\mu_{\mathcal{Y}}$ on $\mathcal{Y}(\C)$ defined as $\mu_{\mathcal{Y}}=(\pi_{\mathcal{Y},\dim \mathcal{Y}})_*\left(\hat{T}_{\hat{f}^{[\dim \mathcal{Y}]}}^{\wedge(u+1)\dim \mathcal{Y}}\wedge[\hat{\mathcal{Y}}^{[\dim \mathcal{Y}]}]\right)$.
\end{definition}
By \cite[Proposition 1.4]{gtv}, we have the following.
\begin{lemma}\label{lm:formulafibered}
For any $m\geq\dim \mathcal{Y}$, there is $C(m)\geq1$ depending only on $m$ such that
\[(\pi_{\mathcal{Y},m})_*\left(\hat{T}_{\hat{f}^{[m]}}^{\wedge( um+\dim\mathcal Y)}\wedge[\hat{\mathcal{Y}}^{[m]}]\right)\geq C(m)\cdot\mu_{\mathcal{Y}}\geq T_{\mathcal{Y}}^{\wedge \dim \mathcal{Y}}.\]
Moreover, if $\hat{\mathcal{Y}}$ is a family of hypersurfaces, then $\mu_{\mathcal{Y}}$ and $ T_{\mathcal{Y}}^{\wedge \dim \mathcal{Y}}$ are proportional.
\end{lemma}

When $\hat{f}$ is a family of regular polynomial endomorphisms of $\mathbb{A}^k$, i.e. when for all $t\in \mathcal Y(\C)$, $f_t:\mathbb{A}^k\to\mathbb{A}^k$ is a polynomial mapping which extends to an endomorphism of $\mathbb{P}^k$, we can define the Green function of $f_t$ as
\[G_{f_t}:=\lim_{n\to\infty}\frac{1}{d^n}\log^+\|f_t^n\|,\]
where $\|\cdot\|$ is the euclidean norm on $\mathbb{C}^k$. Note that it satisfies
\[dd^cG_{f_t}=T_{f_t}-[H_\infty] \quad \text{as currents on} \ \mathbb{P}^k.\]
We also can define a parametric Green function associated to $(\hat{f},\hat{\mathcal{Y}})$ on $\mathcal{Y}(\C)$ by letting
\[G_{\mathcal{Y}}(t):=\int_{\C^k}G_{f_t}(dd^c G_{f_t}(z))^{\wedge u}\wedge[Y_t], \qquad t\in \mathcal{Y}(\C).\]
\begin{definition}
The \emph{polynomial bifurcation measure} of the pair $(\hat{f},\hat{\mathcal{Y}})$ is the positive measure on $\mathcal{Y}(\C)$ defined by $\mu_{\mathcal{Y}}^{\mathrm{pol}}:=(dd^cG_{\mathcal{Y}})^{\wedge\dim \mathcal{Y}}$.
\end{definition}

\subsection{Canonical heights, equidistribution and a height gap criterion}
For a number field $\mathbb{K}$, denote by $M_\mathbb{K}$ the set of places of $\mathbb{K}$, i.e. the set of equivalence classes of absolute values on $\mathbb{K}$. The naive height function $h_{\mathrm{nv}}:\mathbb{P}^k(\bar{\mathbb{Q}})\to\mathbb{R}$ is defined as
\[h_{\mathrm{nv}}([x_0:\cdots:x_k]):=\frac{1}{[\mathbb{K}:\mathbb{Q}]}\sum_{v\in M_{\mathbb{K}}}N_v\log\max_{0\leq j\leq k}|x_j|_v,\]
for any number field $\mathbb{K}$ such that $x_0,\ldots,x_k\in \mathbb{K}$ and $N_v=[\mathbb{K}_v:\mathbb{Q}_v]$. This quantity is well-defined and independent of $\mathbb{K}$ and of $(x_0,\ldots,x_k)$, by the product formula (see, e.g.,~\cite{Silvermandiophantine}).

~

Let $\mathbb{K}$ be a number field and let $f:\mathbb{A}^k\to\mathbb{A}^k$ be a regular polynomial endomorphism of the affine space, i.e. a polynomial map which extends to a map $f:\mathbb{P}^k\to\mathbb{P}^k$. Following Call and Silverman \cite{CS-height}, one can define the \emph{canonical height} $\hat{h}_f:\mathbb{P}^k(\bar{\mathbb{Q}})\to\R_+$ of the endomorphism $f$ as
 \[\hat{h}_f=\lim_{n\to\infty}\frac{1}{d^n}h_{\mathrm{nv}}\circ f^n,\]
 where $h_{\mathrm{nv}}$ is the naive logarithmic Weil height function on $\mathbb{P}^k(\bar{\mathbb{Q}})$. This height function is characterized by the following properties on $\mathbb{P}^k(\bar{\mathbb{Q}})$:
 \[\left\{\begin{array}{ll}
 \hat{h}_f=h_{\mathrm{nv}}+O(1),& \\
 \hat{h}_f\circ f = d \hat{h}_f. &
 \end{array}\right.\]
By construction, $\hat{h}_f\geq0$ and by a famous result of Northcott~\cite{northcott}, the set $\{x\in \mathbb{P}^k(\bar{\mathbb{Q}})\,; \ \hat{h}_f(x)=0\}$ of points of height zero coincides with the set $\mathrm{Preper}(f)$ of preperiodic points of $f$.

We will rely in the sequel on the next result extracted from~\cite{gtv}:

\begin{theorem}[{\cite[Theorem~7.2]{gtv}}]\label{prop-empty-interior}
Let $\hat{f}:\mathbb{P}^k\times \mathcal{Y}\to \mathbb{P}^k\times \mathcal{Y}$ be a family of regular polynomial endomorphisms of degree $d$ of the affine space parametrized by a quasiprojective variety $\mathcal{Y}$ and let $\hat{\mathcal{Y}}\subset\mathbb{P}^k\times \mathcal{Y}$ be an irreducible closed subvariety,  all defined over a number field. Assume that $p:\hat{\mathcal{Y}}\to \mathcal{Y}$ is surjective and that $\hat{\mathcal{Y}}$ intersects properly $H_\infty\times \mathcal{Y}$.
Assume in addition that
\[\int_{\mathcal{Y}(\mathbb{C})}G_{\mathcal{Y}}\cdot\mu^{\mathrm{pol}}_{\mathcal{Y}}>0.\]
Then there exists $Z\subsetneq \mathcal Y$ Zariski-closed, $\varepsilon>0$ and an integer $N\geq1$ such that for any $t\in (\mathcal{Y}\setminus Z)(\bar{\mathbb{Q}})$, there is a proper subvariety $W_t\subsetneq Y_t$ with $\deg(W_t)\leq N$ such that
\[\{z\in Y_t(\bar{\mathbb{Q}})\, : \ \hat{h}_{f_t}(z)\leq \varepsilon\}\subseteq W_t.\]
\end{theorem}

Note that this result is stated only when $\hat{\mathcal{Y}}$ is a \emph{hypersurface} of $\mathbb{P}^k\times \mathcal{Y}$ in \cite{gtv}. However, the proof does not rely on this restriction and can be applied verbatim to the more general setting we use here.

\bigskip

Recall that a sequence $(F_i)_i$ of finite Galois-invariant subsets of
$\mathcal{Y}(\overline{\mathbb{Q}})$ is \emph{generic} if, for every proper
closed subvariety $W\subsetneq\mathcal{Y}$ defined over $\mathbb{K}$, there
exists $i_0\geq1$ such that $F_i\cap W=\varnothing$ for all $i\geq i_0$.
We have the fundamental equidistribution theorem for such sequences which can be deduced from \cite[Theorem~5.4.3]{YZ-adelic} (see \cite[Corollary~6.3]{gtv} in our setting).

\begin{theorem}\label{tm:small-distrib}
Let $\hat{f}:\mathbb{P}^k\times \mathcal{Y}\to \mathbb{P}^k\times\mathcal{Y}$ be a family of endomorphisms of degree $d$ parametrized by a quasiprojective variety $\mathcal{Y}$, and let $\hat{\mathcal{Y}}\subset\mathbb{P}^k\times \mathcal{Y}$ be an irreducible closed proper subvariety, all defined over a number field $\mathbb{K}$. Assume $\mu_{\mathcal{Y}}$ is non-zero. Assume also there is a generic sequence $\{F_i\}_i$ of finite Galois invariant subsets of $\mathcal{Y}(\bar{\mathbb{Q}})$ such that for all $t\in \bigcup_iF_i$, the subvariety $Y_t$ of $\mathbb{P}^k$ is $f_t$-special. Then for any $\varphi\in\mathscr{C}^0_c(\mathcal{Y}(\mathbb{C}),\mathbb{R})$, we have
\[\lim_{i\to\infty}\frac{1}{\# F_i}\sum_{t\in F_i}\varphi(t)=\int_{\mathcal{Y}(\mathbb{C})}\varphi \, \frac{\mu_{\mathcal{Y}}}{\mu_{\mathcal{Y}}(\mathcal{Y}(\C))}.\]
\end{theorem}

\subsection{A positivity sufficient criterion}\label{sec-misiu}
Let $(f_t)_{t\in M}$ be a holomorphic family of endomorphisms of $\Pb^k$ of degree $d$ parametrized by a complex manifold $M$, and let $\tilde m:=\dim(M)$. We do not assume that the family is algebraic, although it is
only used in the algebraic setting in the proof of Theorem~\ref{th-main-blender}. We expect that Proposition~\ref{prop-positivevolume} below can apply in many other situations, including non-algebraic ones.

Let $\hat{\mathcal Y}\subset\p^k\times M$ be an irreducible closed analytic subset of dimension $\tilde m+u$ which surjects onto $M$. We denote by $Y_t$ the fiber of $\hat{\mathcal Y}$ above $t$ and assume that  the canonical projection $\hat{\mathcal Y}\to M$ is proper, projective, flat and holomorphic, so that $Y_t$ has dimension $u$ for all $t\in M$.

\medskip

Fix an integer $1\leq m\leq\tilde m$. As above, let $\hat{\mathcal Y}^{[m]}:=\hat{\mathcal Y}\times_M\cdots\times_M\hat{\mathcal Y}$ be the $m$-fiber power of $\hat{\mathcal Y}$ and let $\pi_{[m]}:\hat{\mathcal{Y}}^{[m]}\to M$ be the projection induced by $\pi$. Again, we define the $m$-fiber power of the family $\hat f$ as the family of maps $\hat f^{[m]}:(\p^k)^m\times M\to(\p^k)^m\times M$ given by
\[\hat f^{[m]}(x,t)=\big(f_t^{[m]}(x),t\big), \quad \text{ with }\quad f_t^{[m]}(x):=(f_t(x_1),\ldots,f_t(x_m)),\]
where $x=(x_1,\ldots,x_m).$ We also denote by $p_{j}: (\p^k)^m\times M\to \p^k\times M$ the projection given by
\[p_{j}(x_1,\ldots,x_m,t)=(x_j,t),\]
for $(x_1,\ldots,x_m,t)\in (\p^k)^m\times M$ and by $\tau_j:(\p^k)^m\to\p^k$ the projection onto the $j$th coordinate.

\bigskip

We aim here to generalize the now classical positivity criteria for repelling points (see, e.g.,~\cite{BBD,AGMV,DeMarco-Mavraki,gtv})
to the saddle case. In the repelling setting, corresponding to families of hypersurfaces $\hat{\mathcal Y}$, a parameter $t_0$ is said to be \emph{Misiurewicz} if there exists a repelling periodic point $r(f_{t_0})$ in the small Julia set of $f_{t_0}$ such that $r(f_{t_0})\in f_{t_0}^n(Y_{t_0})$ for some $n\geq0$, in a
\emph{non-persistent} way, meaning that there exist small perturbations $t$ of $t_0$ such that $r(f_t)\notin f_t^n(Y_t)$. Such parameters belong to the support
of the bifurcation current of bidegree $(1,1)$ associated with the family. By passing to fibered products, one can obtain analogous criteria in higher bidegrees.

To adapt this approach to the saddle case, we make two simple observations:
\begin{itemize}
\item Since the unstable manifold of a repelling point is a neighborhood of that point, the fact that it belongs to the small Julia set is equivalent to
\[T_{f_{t_0}}^{\wedge k}\wedge[W^u_{r(f_{t_0}),loc}]\neq0.\]
\item The local stable manifold of a repelling point reduces to that point. Thus, the second condition is equivalent to a proper intersection between the graph of this point and the family
$\hat{\mathcal Y}$.
\end{itemize}

These two properties admit natural generalizations to saddle points, which turn out to be sufficient to obtain positivity. However, since the stable manifold now has positive dimension, we work with a local version of the proper intersection criterion.

\bigskip

Let $(r_j(f_t))_{t\in M}$, $1\leq j\leq m$, be holomorphic families of saddle points of index $(s,u+1)$, where $s:=k-u-1$. The point
\[r(f_t):=(r_1(f_t),\ldots,r_m(f_t))\]
then defines a holomorphic family of saddle periodic points of $f_t^{[m]}$ of index $(m s,m(u+1))$. Their local stable and unstable manifolds $,W^s_{r(f_t),loc}$ and $W^u_{r(f_t),loc},$ form submanifolds in the product space $(\Pb^k)^m\times M$ which can be defined by
\begin{align*}
\mathcal W_{r,loc}^s:=\left\{(z_1,\ldots,z_m,t)\in(\Pb^k)^m\times M\, : \ \forall 1\leq j\leq m, \  z_j\in W_{r_j(f_t),loc}^s\right\},
\end{align*}
and
\begin{align*}
\mathcal W_{r,loc}^u:=\left\{(z_1,\ldots,z_m,t)\in(\Pb^k)^m\times M\, : \ \forall 1\leq j\leq m, \  z_j\in W_{r_j(f_t),loc}^u\right\}.
\end{align*}

We now consider a parameter $(z_0,t_0)\in \hat{\mathcal Y}^{[m]}$ such that
\begin{enumerate}[label=\textbf{(\arabic*)}]
\item \label{it:cond1}
$\bigwedge_{j=1}^m \tau_j^*\left({T}_{f_{t_0}}^{\wedge(u+1)}\right)\wedge[W^u_{r(f_{t_0}),loc}]\neq 0$,
\item \label{it:cond2}
there exist an open neighborhood $U$ of $(z_0,t_0)$ in $\hat{\mathcal Y}^{[m]}$ and an integer $n_0\geq0$ such that $(\hat f^{[m]})^{n_0}(U)$ intersects properly $\mathcal W^s_{r,loc}$ in a neighborhood of $(\hat f^{[m]})^{n_0}(z_0,t_0).$

\end{enumerate}

A first observation is that these conditions are independent of the choice of the local manifolds, provided $n_0$ is taken large enough and $U$ is reduced accordingly when the local stable manifold $\mathcal W^s_{r,loc}$ is shrunk. A second observation is that both conditions can be checked outside the fibered product framework, using the individual
saddle points $r_j$ (see Lemma~\ref{le-cond1} and Lemma~\ref{le-cond2}).  This is actually the approach we will use, although in this setting condition~\ref{it:cond2} becomes more difficult to formulate.

Adapting the classical phase-parameter transfer argument, we have the following.
\begin{proposition}\label{prop-positivevolume}
Under the assumptions \textup{\ref{it:cond1}} and \textup{\ref{it:cond2}} above, we have
\[(z_0,t_0)\in \mathrm{supp}\left(\bigwedge_{j=1}^m p_j^*\left(\hat{T}_{\hat f}^{\wedge(u+1)}\right)\wedge[\hat{\mathcal Y}^{[m]}]\right).\]
In particular, if $m=\dim M$ and if $(f_t)_{t\in M}$ is a family of regular polynomial endomorphisms of $\mathbb{C}^k$, and if $r_j(f_{t_0})\in \C^k$ for all $j$, then we have
\[t_0\in \mathrm{supp}\left(\mu_{\mathcal{Y}}^{\mathrm{pol}}\right).\]
\end{proposition}
\begin{proof}
For simplicity, we set $F:=\hat f^{[m]}$ and
\[T:=\bigwedge_{j=1}^m p_j^*\left(\hat{T}_{\hat f}^{\wedge(u+1)}\right).\]
Since $\hat{T}_{\hat f}$ has continuous potentials, for any analytic submanifold $M'\subset M$ of dimension $m$ with ${t_0}\in M'$, we have
\[\mathrm{supp}\left(T|_{\pi_{[m]}^{-1}(M')}\wedge [\hat{\mathcal Y}^{[m]}\cap\pi_{[m]}^{-1}(M')]\right)\subset \mathrm{supp}\left(T\wedge [\hat{\mathcal Y}^{[m]}]\right)\]
by e.g.~\cite[Lemma~6.3]{Article1}. Hence, we may replace $M$ by such a submanifold $M'$ of dimension $m$, and $\hat{\mathcal Y}^{[m]}$ by
$\hat{\mathcal Y}^{[m]}\cap\pi_{[m]}^{-1}(M')$, so as to assume that $\mathcal W_{r,loc}^s\cap F^n(U)$ reduces to a single point, provided $U$ is chosen sufficiently small in $\hat{\mathcal Y}^{[m]}.$

On the other hand, recall that $r(f_{t_0}):=(r_1(f_{t_0}),\ldots,r_m(f_{t_0}))$ is a saddle periodic point of index $(sm,(u+1)m)$ for $f_{t_0}^{[m]}.$ Its period is the lcm of the periods of the saddle points $r_j(f_{t_0}).$ As the Green current of $F$ is the same that for its iterates, we can actually assume that this period is $1.$ On the other hand, as the index of $r(f_{t_0})$ is $(m s,m(u+1))$, there exist two polydiscs $N_s\subset\Cb^{m s}$ and $N_u\subset\Cb^{m(u+1)}$ and a neighborhood $\mathcal N$ of $r(f_{t_0})$ biholomorphic to $N_s\times N_u$ such that $f_{t_0}^{[m]}$ is a horizontal-like map on $\mathcal N$ of index $(m s,m(u+1))$ and of degree $1$ in these coordinates. As this property is stable under perturbations and that the statement is purely local, possibly by shrinking $M$, we could assume that the map $f^{[m]}_t$ is also horizontal-like on $\mathcal N$ of index $(m s,m(u+1))$ and of degree $1$ for all $t\in M.$

We now restrict the local stable and unstable manifolds of $r(f_t)$ to $\mathcal N$, and adjust the notation
\[\mathcal W_{r,loc}^s:=\left\{(x,t)\in \mathcal N\times M\, : \ x\in W_{r(f_t),loc}^s\right\}\]
and
\[\mathcal W_{r,loc}^u:=\left\{(x,t)\in \mathcal N\times M\, : \ x\in W_{r(f_t),loc}^u\right\}.\]
By the discussion above, after possibly increasing $n_0$ and
shrinking $U$, we may assume that
$\Gamma_0:=F^{n_0}(U)$ intersects
$\mathcal W_{r,loc}^s$ at a unique point. Moreover, since points outside $\mathcal W^s_{r,loc}$ escape from $\mathcal N$ under the dynamics, we may, after possibly
adjusting $n_0$ and $U$ again, assume that $\Gamma_0$ is closed in $\mathcal N\times M$ and
\[\Big(\overline{\partial_s\mathcal N\times M \cup \mathcal N\times\partial M}\Big) \cap\overline{\Gamma_0} = \varnothing.\]
Here, the boundary $\partial_s\mathcal N$ is given as in Section~\ref{sec-hl-renor} via the biholomorphism $\mathcal N\simeq N_s\times N_u.$ From this, by Proposition \ref{prop-renor}, the sequence $(\Gamma_n)_{n\geq0}$ of closed analytic subsets of $\mathcal N\times M$ defined by
\[\Gamma_{n+1}:=F(\Gamma_n)\cap(\mathcal N\times M),\]
converges to $W^u_{r(f_{t_0}),loc}\times\{t_0\},$ with a multiplicity $c,$ corresponding to the intersection number between $\Gamma_0$ and $\mathcal W^s_{r,loc}.$

On the other hand, since $p_j\circ F=\hat f\circ p_j,$ we have $F^*(p_j^*(\hat T_{\hat f}^{\wedge(u+1)}))=d^{u+1}p_j^*(\hat T_{\hat f}^{\wedge(u+1)})$ thus
\[F^*T=F^*\left(\bigwedge_{j=1}^m p_j^*\left(\hat{T}_{\hat f}^{\wedge(u+1)}\right)\right)=\bigwedge_{j=1}^m (p_j\circ F)^*\left(\hat{T}_{\hat f}^{\wedge(u+1)}\right)=d^{m(u+1)}T.\]
Hence,
\begin{align*}
I_n:=\int T\wedge[\Gamma_n]\leq d^{nm(u+1)}\int T\wedge\left[\Gamma_0\right]\leq d^{(n+n_0)m(u+1)}\int T\wedge [\hat{\mathcal Y}^{[m]}].
\end{align*}
Since $[\Gamma_n]$ converges weakly to $c[W^u_{r(f_{t_0}),loc}\times\{t_0\}]$ as $n\to\infty$ and, since the current $T=\bigwedge_{j=1}^m p_j^*\left(\hat{T}_{\hat f}^{\wedge(u+1)}\right)$ is the wedge product of $m(u+1)$ closed positive $(1,1)$-currents with continuous potential, we have
\[T\wedge\left[\Gamma_n\right]\to c\left(T\wedge\left[W^u_{r(f_{t_0}),loc}\times\{t_0\}\right]\right)\]
as $n\to+\infty$. To conclude, it suffices to show that this last measure is non-zero. This is precisely condition~\ref{it:cond1}, since, by the continuity of the potential again, the slice of $T$ by $(\Pb^k)^{m}\times\{t_0\}$ coincides with $\bigwedge_{j=1}^m \tau_j^*\left({T}_{f_{t_0}}^{\wedge(u+1)}\right)$.
\end{proof}
In what follows, we verify conditions~\ref{it:cond1} and \ref{it:cond2} outside the fibered product framework. The second
condition is harder to formulate in this setting, and we will use the framework of horizontal-like maps together with
dominated splittings in Section~\ref{sec-blender} for this purpose (see Lemma~\ref{le-cond2}). It is also much more delicate to verify and actually it occupies the whole of that section.

By contrast, the first condition can be easily reformulated. Moreover, it is stronger than a possible more natural condition.
\begin{lemma}\label{le-cond1}
The condition \textup{\ref{it:cond1}} is equivalent to
\begin{enumerate}[label=\textbf{(\arabic*')}]
\item\label{it:cond1'} for each $1\leq j\leq m,$ we have $T_{f_{t_0}}^{\wedge(u+1)}\wedge[W^u_{r_j(f_{t_0}),loc}]\neq 0$.
\end{enumerate}
Furthermore, it implies the condition
\begin{enumerate}[label=\textbf{(\arabic*'')}]
\item\label{it:cond1''} $(T_{f^{[m]}_{t_0}})^{\wedge m(u+1)}\wedge[W^u_{r(f_{t_0}),loc}]\neq 0$.
\end{enumerate}
\end{lemma}
\begin{proof}
Recall that the saddle periodic point $r(f_{t_0})$ for $f^{[m]}_{t_0}$ is defined by
\[r(f_{t_0})=(r_1(f_{t_0}),\ldots,r_m(f_{t_0})).\]
Its local unstable manifold can be given by
\[W^u_{r(f_{t_0}),loc}:=W^u_{r_1(f_{t_0}),loc}\times\cdots\times W^u_{r_m(f_{t_0}),loc}.\]

Thus, if we denote by $\tau_j:(\p^k)^m\to\p^k$  the projection onto the $j$th coordinate, we have
\[[W^u_{r(f_{t_0}),loc}]=\tau_1^*[W^u_{r_1(f_{t_0}),loc}]\wedge\cdots\wedge\tau_m^*[W^u_{r_m(f_{t_0}),loc}].\]
Hence, each $T_{f_{t_0}}^{\wedge(u+1)}\wedge[W^u_{r_j(f_{t_0}),loc}]$ is non-zero if and only if
\begin{align*}\bigwedge_{j=1}^m \tau_j^*\left({T}_{f_{t_0}}^{\wedge(u+1)}\right)\wedge & [W^u_{r(f_{t_0}),loc}]\\
& =\tau_1^*\left({T}^{\wedge(u+1)}_{f_{t_0}}\wedge [W^u_{r_1(f_{t_0}),loc}]\right)\wedge\cdots\wedge\tau_m^*\left({T}^{\wedge(u+1)}_{f_{t_0}}\wedge[W^u_{r_m(f_{t_0}),loc}]\right)
\end{align*}
is non-zero. More precisely, the mass of this measure is given by
\[\left\|\bigwedge_{j=1}^m \tau_j^*\left({T}_{f_{t_0}}^{\wedge(u+1)}\right)\wedge[W^u_{r(f_{t_0}),loc}]\right\|=\prod_{j=1}^m\left\|T_{f_{t_0}}^{\wedge(u+1)}\wedge[W_{r_j(f_{t_0}),loc}^u]\right\|.\]
On the other hand, we have
\begin{align}
(T_{f^{[m]}_{t_0}})^{\wedge m(u+1)} =
& \left(\tau_1^*\left({T}_{f_{t_0}}\right)+\cdots+\tau_m^*\left({T}_{f_{t_0}}\right)\right)^{\wedge m(u+1)}\nonumber\\
& \geq\tau_1^*\left({T}^{\wedge(u+1)}_{f_{t_0}}\right)\wedge\cdots\wedge\tau_m^*\left({T}^{\wedge(u+1)}_{f_{t_0}}\right).\label{eq-fibered}
\end{align}
This gives
\begin{align*}
(T_{f^{[m]}_{t_0}})^{\wedge m(u+1)}\wedge & [W^u_{r(f_{t_0}),loc}]\\
&\geq\tau_1^*\left({T}^{\wedge(u+1)}_{f_{t_0}}\wedge [W^u_{r_1(f_{t_0}),loc}]\right)\wedge\cdots\wedge\tau_m^*\left({T}^{\wedge(u+1)}_{f_{t_0}}\wedge[W^u_{r_m(f_{t_0}),loc}]\right).
\end{align*}
Hence, condition \ref{it:cond1'} implies \ref{it:cond1''}.
\end{proof}

Observe that the inequality \eqref{eq-fibered}, together with Proposition~\ref{prop-positivevolume}, shows that conditions~\ref{it:cond1} and~\ref{it:cond2} imply that the
parameter $t_0$ lies in the support of the $m$-bifurcation current associated with $\hat{\mathcal Y}$. The same conclusion holds under the weaker condition~\ref{it:cond1''}, by the same argument. Although we do not use this refinement, it may be of independent interest.
\begin{proposition}\label{prop-positivevolume2}
Under the assumptions \textup{\ref{it:cond1''}} and \textup{\ref{it:cond2}} above, we have
\[(z_0,t_0)\in \mathrm{supp}\left(\hat{T}_{\hat f^{[m]}}^{\wedge m(u+1)}\wedge[\hat{\mathcal Y}^{[m]}]\right).\]
In particular, $t_0$ belongs to the support of the $m$-bifurcation current associated to $(\hat{f},\hat{\mathcal Y})$:
\[t_0\in \mathrm{supp}\left((\pi_{[m]})_*\left(\hat{T}_{\hat f^{[m]}}^{\wedge m(u+1)}\wedge[\hat{\mathcal Y}^{[m]}]\right)\right).\]
\end{proposition}

\subsection{Moduli spaces}
We now turn to the parameter spaces used in this paper. Before discussing dynamical moduli spaces, we briefly recall some facts about Chow varieties.
\subsubsection{Chow varieties and families of components}\label{sec-chow}
Let $\mathcal X_D^u$ denote the Chow variety of effective cycles of dimension $u$ and degree $D$ in $\Pb^k$.
We use it to parametrize the subvarieties considered in our main results. For dynamical purposes, we identify cycles with their supports.
The fact that $\mathcal X_D^u$ is defined over $\Qb$ and projective (see e.g. \cite[Theorem~I.3.21]{kollar-rational}) is nevertheless
important since it gives proper projections and allows us to consider degenerations.

The geometry of Chow varieties is complicated in general and not fully understood. In particular, they need not be irreducible,
and some components parametrize generically reducible cycles. A simple example is the Chow variety of degree-two curves
in $\Pb^3$, where one component parametrizes pairs of lines. When these lines are disjoint, such cycles cannot be approximated by irreducible conics and thus form a component.

More generally, for positive integers satisfying $D=\sum_{j=1}^r m_je_j$, we have a map
\begin{equation}\label{eq-chow}
 \prod_{j=1}^r\mathcal X_{e_j}^u 
 \longrightarrow\mathcal X_D^u,\qquad
 (C_1,\ldots,C_r)\longmapsto\sum_{j=1}^r m_jC_j.
\end{equation}
This map is proper with finite fibers. In the example above, with $e_1=e_2=1$ and $D=2,$ an irreducible component of $\mathcal X_2^1$ is the image of such a map.

We do not need a detailed description of these spaces. However, the families $\mathcal Y\subset\mathcal X_D^u\times M$ we consider, where $M=\mathrm{End}_d^k$ or $\mathrm{Poly}_d^k$, are only assumed to be irreducible and dominant over $M$. Thus their general members need not be irreducible, even in the hypersurface case $u=k-1$.

We call such a family \emph{generically irreducible} if its general member is reduced and irreducible. The following argument allows us to reduce to this case
by passing to a finite ramified cover and selecting a family of irreducible components, possibly of smaller degree.

First, for an irreducible family $\mathcal Y\subset\mathcal X_D^u\times M$, write $(Y_t,f_t)_{t\in \mathcal Y}$. On a dense Zariski open subset of $\mathcal Y$, the number, degrees and multiplicities of the irreducible components are constant. Pulling back by the corresponding map in \eqref{eq-chow} gives a finite, possibly ramified covering $\pi:\mathcal Y'\to \mathcal Y$ on which these
components can be followed algebraically, i.e.,
\[
 Y_{\pi(t')}=\sum_{j=1}^r m_j[Y_{j,t'}],
 \qquad
 \deg Y_{j,t'}=e_j,\qquad
 \sum_{j=1}^r m_je_j=D.
\]
The $Y_{j,t'}$ are generically distinct irreducible subvarieties. For a fixed irreducible component of $\mathcal Y'$ dominating $\mathcal Y$ and for each $j$, we then consider the closure of the image of
\[
 t'\longmapsto(Y_{j,t'},f_{\pi(t')})
 \quad\text{in}\quad\mathcal X_{e_j}^u\times M,
\]
which also dominates $M$.

We will systematically use this reduction to pass from a family $\mathcal Y\subset\mathcal X_D^u\times M$ to a family $\mathcal Y_{e_j}\subset\mathcal X_{e_j}^u\times M$
with irreducible general members. Remark~\ref{rk-irreducible} explains why this is important in what follows.


\subsubsection{Transverse families}
\label{sec:moduli-space}

We recall here the moduli spaces used below and a few elementary facts about
transverse families.

As an endomorphism $f$ of $\Pb^k$ of degree $d$ is given by $k+1$ homogeneous polynomials of degree $d$, the coefficients of these polynomials allow us to see $f$ as a point in $\Pb^{N}$ where $N:= (k+1)\binom{k+d}{d}-1$. The condition on the coefficients to ensure that the associated map is an endomorphism of $\Pb^k$ is algebraic so there exists a Zariski open set $\mathrm{End}_d^k\subset\Pb^{N}$ corresponding to degree $d$ endomorphisms. See, e.g.,~\cite[Proposition~1.1]{BB1} for more details, see also \cite[\S~3]{szpiro}.

In particular, the variety $\mathrm{End}_d^k$ is an irreducible smooth quasi-projective variety defined over $\mathbb{Q}$.
We thus have a family
\[f\colon\mathbb{P}^k\times \mathrm{End}_d^k\to\mathbb{P}^k\times\mathrm{End}_d^k\]
 of degree $d$ endomorphisms of $\mathbb{P}^k$ parametrized by $\mathrm{End}_d^k$ which is defined over $\mathbb{Q}$.

Since the dynamics is unchanged under holomorphic conjugacy, the space naturally adapted to the study of the family of dynamical systems is the moduli space $\mathscr M_d^k$, obtained as the quotient space of the space $\mathrm{End}_d^k$ by the action by conjugacy of $\mathrm{PGL}(k+1)$: 
\[\Pi\colon\mathrm{End}_d^k\to \mathscr{M}_d^k.\]
It is an irreducible affine variety defined over $\mathbb{Q}$, of dimension
$(k+1)\binom{k+d}{d} - (k+1)^2$, and the stabilizer of every element of
$\End_d^k$ is finite, see~\cite{Silverman} for $k=1$
and~\cite[\S~3.2-3.3]{szpiro} for $k>1$. In particular, the critical set $V$ of $\Pi$ is defined over $\mathbb{Q}$ and $\Pi$ is a submersion on $\mathrm{End}_d^k\setminus V$.

 As $\mathscr{M}_d^k$ is merely a coarse moduli space, there is no universal family. However, we can cook up a good family which can play a sufficiently similar role. The following is Lemma~2.1 from \cite{gtv}.
 \begin{lemma}\label{good-family-end}
There exists an irreducible subvariety
$\mathcal S_d^k\subset\End_d^k$, defined over $\Q$, such that
\[
 \dim\mathcal S_d^k=\dim\mathscr M_d^k,
\]
$\Pi(\mathcal S_d^k)$ contains a dense Zariski open subset of
$\mathscr M_d^k$, and $\Pi|_{\mathcal S_d^k}$ is finite onto its image.
\end{lemma}

We next consider regular polynomial endomorphisms. Let $\Poly_d^k$ be the
space of polynomial maps $f:\Ab^k\to\Ab^k$ of degree $d$ which extend to
endomorphisms of $\Pb^k$ of degree $d$. It is a smooth irreducible
quasi-projective variety defined over $\Q$, of dimension
\[
 \dim\Poly_d^k=k\binom{k+d}{d}.
\]
Since $\Aff(\Ab^k)$ is not reductive, we do not use the quotient by affine
conjugacy. Following \cite{t-gorbovickis}, we instead set
\[
 \mathscr P_d^k:=\Pi(\Poly_d^k)\subset\mathscr M_d^k.
\]
This is the closed irreducible subvariety of $\mathscr M_d^k$ formed by
classes admitting a totally invariant hyperplane. The natural map from
affine conjugacy classes in $\Poly_d^k$ to $\mathscr P_d^k$ is generically
injective, and
\[
 \dim\mathscr P_d^k
 =k\binom{k+d}{d}-(k^2+k).
\]

The construction in \cite{t-gorbovickis} also gives the polynomial analogue
of Lemma~\ref{good-family-end}.

\begin{lemma}\label{good-family-poly}
There exists an irreducible subvariety
$\mathcal T_d^k\subset\Poly_d^k$, defined over $\Q$, such that
\[
 \dim\mathcal T_d^k=\dim\mathscr P_d^k,
\]
$\Pi(\mathcal T_d^k)$ contains a dense Zariski open subset of
$\mathscr P_d^k$, and $\Pi|_{\mathcal T_d^k}$ is finite onto its image.
\end{lemma}

For the local arguments in Section~\ref{sec-blender}, we use
analytic slices. The guiding idea is to work with a parameter
space $\mathcal Y$ that maps generically finite-to-one to the
``moduli space'' of pairs, which we do not formally introduce,
while retaining enough dynamical variation.

The first condition is necessary for the bifurcation measure $\mu_{\mathcal Y}$ to be nonzero. Indeed, if a positive-dimensional local family of conjugate pairs
$(\phi^{-1}(Y),\phi^{-1}\circ f\circ\phi)$ passes through a general point of $\mathcal Y$, the wedge product defining $\mu_{\mathcal Y}$ has too many factors and vanishes.
On the other hand, a family whose projection to the moduli space of maps is too small may miss the blender region of Section~\ref{sec-blender}, or meet it without providing
enough variation to satisfy the required hypotheses.

Let $\mathcal E$ denote
either $\End_d^k$ or $\Poly_d^k$, and let $\mathscr{M}$ and $G$ denote respectively
$\mathscr M_d^k$ and $\PGL(k+1)$ or $\mathscr P_d^k$ and $\Aff(\Ab^k)$. Recall that $V$ is the ramification locus of the projection $\Pi\colon\End_d^k\to\mathscr M_d^k$ and we keep the same notation for its counterpart in $\Poly_d^k.$ Since $\Pi$ is a submersion outside $V$ and the stabilizers are finite the holomorphic submersion theorem yields the following
local product decomposition on $\mathcal E\setminus V$.
\begin{lemma}
\label{lem:local-slice}
Let $f_0\in\mathcal E\setminus V.$ Then, there exists a local complex submanifold
$\tilde\Omega\subset\mathcal E\setminus V$ containing $f_0$ such that the conjugacy map
\[
 G\times\tilde\Omega\to\mathcal E,
 \qquad
 (\varphi,f)\mapsto\varphi^{-1}\circ f\circ\varphi,
\]
is a local biholomorphism between neighborhoods of $(\id,f_0)$ and $f_0$,
respectively. In particular, $\Pi_{|\tilde\Omega}\colon\tilde\Omega\to\mathscr M$ is a local biholomorphism.
\end{lemma}
Such local submanifold $\tilde\Omega$ is called a \emph{local slice at $f_0$}.

\medskip

Let now
\[
 \mathcal Y\subset\mathcal X_D^u\times\mathcal E
\]
be irreducible.
\begin{definition}
We say that $\mathcal Y$ is \emph{dynamically full} if
\[
 \Pi\circ\pi_2:\mathcal Y\to\mathscr M
\]
is dominant. We say that $\mathcal Y$ is \emph{maximally varying} if the restriction to $G\times\mathcal Y$ of the conjugacy map
\begin{align*}
 \mathrm{Conj}\colon G\times\mathcal X_D^u\times\mathcal E
 &\to \mathcal X_D^u\times\mathcal E,\\
 (\varphi,(Y,f))&\mapsto
 \bigl(\varphi^{-1}(Y),\varphi^{-1}\circ f\circ\varphi\bigr)
\end{align*}
is generically finite onto its image.
\end{definition}
As already mentioned, this latter condition rules out redundancies in the family and is necessary for the bifurcation
measure $\mu_{\mathcal Y}$ to be nonzero. On the other hand, this can always be achieved by taking an appropriate slice.

The two conditions above allow us to put $\mathcal Y$ into a convenient
local normal form. Roughly speaking, the following proposition shows
that, near a generic point and after simultaneous conjugacy, $\mathcal Y$
may be viewed locally as a complex submanifold of $\mathcal X_D^u\times\tilde\Omega$ whose projection onto $\tilde\Omega$ is a surjective submersion.
\begin{proposition}
\label{prop:local-normalization}
Assume that $\mathcal Y$ is dynamically full and maximally varying.
Then there exists a dense Zariski open subset
$\mathcal Y^\circ\subset\mathcal Y$ such that the following holds.

For every $(Y_0,f_0)\in\mathcal Y^\circ$ and every local slice
$\tilde\Omega$ through $f_0$, after possibly shrinking
$\tilde\Omega$ around $f_0$, there exist a complex submanifold containing $(Y_0,f_0)$
\[ \tilde{\mathcal Y}
 \subset
 \mathcal X_D^u\times\tilde\Omega
\]
and a holomorphic map
\[
 \Psi\colon\tilde{\mathcal Y}\to G,\text{ with }
 \Psi(Y_0,f_0)=\id,
\]
such that the projection $ \pi_{\tilde\Omega}\colon \tilde{\mathcal Y}\to\tilde\Omega$
is a surjective submersion and the map
\[
 \tilde{\mathcal Y}\to
 \mathcal X_D^u\times\mathcal E,
\]
\[
 z=(Z,g)\mapsto
 \bigl(\psi_z^{-1}(Z),
       \psi_z^{-1}\circ g\circ\psi_z\bigr),
 \qquad
 \psi_z:=\Psi(z),
\]
is a biholomorphism from $\tilde{\mathcal Y}$ onto a neighborhood
of $(Y_0,f_0)$ in $\mathcal Y$.
\end{proposition}

\begin{proof}
Since $\mathcal Y$ is
dynamically full, the map
\[
 p:=\Pi\circ\pi_2|_{\mathcal Y}
 \colon\mathcal Y\to\mathscr M
\]
is dominant. Hence, there exists a dense Zariski open subset
\[
 \mathcal Y_{\mathrm{full}}
 \subset\mathcal Y_{\mathrm{reg}}
\]
on which $p$ is a submersion. In particular, if
$y=(Y,f)\in\mathcal Y_{\mathrm{full}}$, then $f\notin V$, since
$\Pi$ does not have maximal rank on $V$.

On the other hand, maximal variation means that the conjugacy map
\[
 \mathrm{Conj}\colon
 G\times\mathcal Y
 \to
 \mathcal X_D^u\times\mathcal E
\]
is generically finite onto its image, and hence generically immersive.
Therefore, by equivariance of this map, there exists a dense Zariski
open subset
\[
 \mathcal Y_{\mathrm{mv}}
 \subset\mathcal Y_{\mathrm{reg}}
\]
such that $\mathrm{Conj}$ is immersive at $(\id,y)$ for every
$y\in\mathcal Y_{\mathrm{mv}}$. We then define
\[
 \mathcal Y^\circ
 :=
 \mathcal Y_{\mathrm{full}}
 \cap
 \mathcal Y_{\mathrm{mv}}.
\]

Let now $(Y_0,f_0)\in\mathcal Y^\circ$, and let
$\tilde\Omega$ be a local slice at $f_0$. By
Lemma~\ref{lem:local-slice}, there exist neighborhoods
$\tilde G$ of $\id$ in $G$ and $\Omega$ of $f_0$ in
$\mathcal E$ such that
\[
 (\varphi,f)\mapsto\varphi^{-1}\circ f\circ\varphi
\]
is a biholomorphism from $\tilde G\times\tilde\Omega$ onto
$\Omega$. Therefore,
\begin{align*}
 F\colon
 \mathcal X_D^u\times\tilde G\times\tilde\Omega
 &\to
 \mathcal X_D^u\times\Omega,\\
 (Y,\varphi,f)
 &\mapsto
 \bigl(\varphi^{-1}(Y),
       \varphi^{-1}\circ f\circ\varphi\bigr)
\end{align*}
is a biholomorphism.

After shrinking around $(Y_0,\id,f_0)$, set
\[
 \check{\mathcal Y}
 :=
 F^{-1}\bigl(
 \mathcal Y\cap(\mathcal X_D^u\times\Omega)
 \bigr),
\]
and let
\[
 \pi\colon
 \mathcal X_D^u\times\tilde G\times\tilde\Omega
 \to
 \mathcal X_D^u\times\tilde\Omega
\]
be the natural projection. In the coordinates given by $F$, the
conjugacy map acts only on the $G$-coordinate:
\[
 F^{-1}\bigl(
 \mathrm{Conj}(\eta,F(Y,\varphi,f))
 \bigr)
 =
 (Y,\varphi\eta,f).
\]
Consequently, the immersivity of $\mathrm{Conj}$ at
$(\id,(Y_0,f_0))$ exactly implies that
$\pi|_{\check{\mathcal Y}}$ is immersive at
$(Y_0,\id,f_0)$. Hence, after shrinking the neighborhoods,
$\pi|_{\check{\mathcal Y}}$ is a biholomorphism onto the complex
submanifold
\[
 \tilde{\mathcal Y}
 :=
 \pi(\check{\mathcal Y})
 \subset
 \mathcal X_D^u\times\tilde\Omega.
\]

We may thus write
\[
 \bigl(\pi|_{\check{\mathcal Y}}\bigr)^{-1}(Y,f)
 =
 \bigl(Y,\Psi(Y,f),f\bigr),
\]
where
\[
 \Psi\colon\tilde{\mathcal Y}\to G
\]
is holomorphic and satisfies $\Psi(Y_0,f_0)=\id$. Setting
$\psi_y:=\Psi(y)$, the map
\[
 \tilde{\mathcal Y}\to\mathcal Y,
 \qquad
 y=(Y,f)\mapsto
 \bigl(\psi_y^{-1}(Y),
       \psi_y^{-1}\circ f\circ\psi_y\bigr)
\]
is a biholomorphism onto a neighborhood of $(Y_0,f_0)$ in
$\mathcal Y$.

Finally, under this biholomorphism, one has
\[
 p
 =
 \bigl(\Pi|_{\tilde\Omega}\bigr)
 \circ
 \pi_{\tilde\Omega},
\]
where $\pi_{\tilde\Omega}\colon \tilde{\mathcal Y} \to \tilde\Omega$
 is the projection. Since $p$ is a submersion and
$\Pi|_{\tilde\Omega}$ is a local biholomorphism, the projection $\pi_{\tilde\Omega}$ is a submersion. After shrinking $\tilde\Omega$ if necessary, it
is surjective.
\end{proof}

\section{Dynamical consequences of blenders in arbitrary codimension}\label{sec-blender}

In this section we introduce a framework leading to maximal
bifurcations. Section~\ref{sec-ass} presents the assumptions
under which we will work. In Section~\ref{sec-main-blender} we
state that, within this framework, the support of the bifurcation
measure $\mu_{\mathcal Y}$ has nonempty interior, and we outline
the strategy of the proof together with the organization of the
rest of the section.

\subsection{Assumptions}\label{sec-ass}
To obtain maximal bifurcations, our approach relies on a robust heterodimensional cycle between a saddle fixed point and a hyperbolic set possessing the blender property.
Within the geometric framework provided by Assumption~\ref{verti}, these dynamical objects are introduced in Assumptions~\ref{hyperbolic} and~\ref{saddle}, while the blender property itself is formulated in Assumption~\ref{blender}.

We further impose partially hyperbolic conditions ensuring the existence of a one-dimensional central direction, which will play a crucial role in the analysis of bifurcations.

Among the remaining hypotheses, Assumption~\ref{misu} is essential for applying Proposition~\ref{prop-positivevolume}, whereas Assumptions~\ref{decalage} and~\ref{grille} exclude certain product-type configurations that would otherwise prevent maximal bifurcations (see Section~\ref{sec-subvar}). Finally, Assumption~\ref{init} provides the initialization step for the induction on the dimension.

Accordingly, we fix a constant $\eta>10$ controlling the size of the cone fields, a constant $\tilde\chi>1$ controlling the multipliers, and a constant $R>2$ controlling the size of the center direction.

\bigskip

\noindent\emph{\textbf{Setting.}} In what follows, let $\Omega_u$ be a complex analytic open subset of $\mathrm{End}_d^k$ or $\mathrm{Poly}_d^k$ such that every $f\in\Omega_u$ satisfies all the assumptions described below.

Observe that most of the objects introduced here are assumed to depend holomorphically on $f\in\Omega_u$, and our notation reflects this dependence. For instance, $p\colon f\mapsto p(f)$ denotes the holomorphic motion of the saddle point given in Assumption~\ref{saddle}.

An important exception is the saddle hyperbolic set $\Lambda(f)$. As explained in Section~\ref{sec-graph}, its holomorphic continuation is naturally formulated on the natural extension. Throughout this section, we use the notational convention introduced there and often suppress the natural extension whenever no confusion can arise.

Note that, in the unstable direction, we use the polydisc $\Bb_{3/2}^u$ of radius $3/2$, since the construction in Section~\ref{sec-existence} requires allowing points whose norm may be slightly larger than $1$.

\bigskip

We now list the assumptions for this section.
\begin{enumerate}
\item\label{verti}
There exist two disjoint holomorphic polydiscs $U_+^u,U_-^u\subset\Bb_{3/2}^u$, such that if
\[\mathcal U_\pm:=\Db_{2R}\times \Bb^s\times U_\pm^u\]
then $f|_{\mathcal U_-}$ (resp. $f|_{\mathcal U_+}$) is an injective degree $1$ horizontal-like map of index $(s+1,u)$ (resp. $(s,u+1)$) from $\mathcal U_-$ (resp. $\mathcal U_+$) to $\Db_{2R}\times \Bb^s\times \Bb_{3/2}^u.$ Moreover, these two maps contract the cone field $C^{uu}_{\eta}$ and expand the cone field $C^{ss}_{\eta}.$\item\label{hyperbolic} There exist two holomorphic polydiscs $V_+^u\Subset U_+^u$ and $V_-^u\Subset U_-^u$ such that, if we set
\[\mathcal V_+:=\Db_R\times \Bb^s\times V_+^u,\ \ \mathcal V_-:=\Db_R\times \Bb^s\times V_-^u\ \text{ and }\ \mathcal V:=\mathcal V_+\cup\mathcal V_-\]
then  $f(\mathcal V_-)\Subset \mathcal U_+,$ $f(\mathcal V_+)\Subset \mathcal U_-,$ and $f^2|_{\mathcal V}$ is a degree $2$ horizontal-like map of index $(s,u+1)$ from $\mathcal V$ to $\Db_{R}\times \Bb^s\times \Bb_{3/2}^u$.
Moreover,
\[\Lambda(f):=\bigcap_{n\in\Zb}f^{-2n}(\mathcal V)\]
is a basic set of index $(s,u+1)$ such that the unstable bundle admits a dominated splitting. More precisely, for each $\hat x(f)$ in the natural extension $\hat\Lambda(f)$ of $\Lambda(f)$ we have
\[E^u_{\hat x(f)}=E^{uu}_{\hat x(f)}\oplus E^c_{\hat x(f)},\]
where $E^{uu}_{\hat x(f)}\subset C^{uu}_{\eta},$ and for each $v^s\in E^s_{\hat x(f)},$ $v^u\in E^{uu}_{\hat x(f)},$ $v^c\in E^c_{\hat x(f)}$ and $n$ sufficiently large, we have
\[\|D_{\hat x(f)}f^{2n}(v^s)\|\leq\tilde\chi^{-4n}\|v^s\|,\quad\tilde\chi^{4n}\|v^u\|\leq\|D_{\hat x(f)}f^{2n}(v^u)\|\]
and
\[\|D_{\hat x(f)}f^{2n}(v^c)\|\leq\tilde\chi^n\|v^c\|.\]
\item\label{saddle} $f$ has a saddle fixed point $p(f)$ of index $(s+1,u)$ in $\mathcal U_-$ that satisfies the conditions \eqref{eq-bun} and \eqref{eq-res} for some $\ell\geq4.$ Moreover, its multipliers, $\lambda_1(f),\ldots,\lambda_s(f)$, $\chi_{p(f)}$ and $\rho_1(f),\ldots,\rho_u(f)$, satisfy
\begin{gather*}
\tilde\chi^{-3}<|\lambda_1(f)|<\cdots<|\lambda_s(f)|<\tilde\chi^{-2},\\
\tilde\chi^{-1}<|\chi_{p(f)}|<1,\\
\tilde\chi^3<|\rho_1(f)|<\cdots<|\rho_u(f)|<\tilde\chi^{4}.
\end{gather*}
We also ask that its local unstable manifold $W^u_{p(f),loc}$ is a $u$-graph in $\Db\times \Bb^s\times \Bb_{3/2}^u$ tangent to $C^{uu}_{\eta}$.
\item\label{blender} Each $u$-graph in $\Db\times \Bb^s\times V_+^u$ (resp. $\Db\times \Bb^s\times V_-^u$)  tangent to $C^{uu}_{\eta}$ intersects the stable set of $\Lambda(f)$
\[W^s_{\Lambda(f),loc}:=\bigcup_{x\in\Lambda}W^s_{x(f),loc}.\]
\item\label{free} The intersections between $W^u_{p,loc}$ and $W^s_{\Lambda(f),loc}$ are not persistent in $\Omega_u$ (i.e. if $x\in\hat\Lambda$ and $f\in\Omega_u$ satisfy $W^s_{x(f),loc}\cap W^u_{p(f),loc}\neq\varnothing$ then there exists $f'\in\Omega_u$ close to $f$ such that $W^s_{x(f'),loc}\cap W^u_{p(f'),loc}=\varnothing$).
\item\label{repelling} There exists a saddle $2$-periodic point $r_0(f)\in\Lambda(f)\cap\Db_{1/2}\times \Bb^s\times \Bb_{3/2}^u.$ We denote by $\chi_{r_0(f)}$ its multiplier in the central direction.
\item\label{misu} For each $\hat x(f)\in\hat\Lambda(f),$ we have
\begin{equation*}\label{eq-misiu}
[W^u_{\hat x(f),loc}]\wedge T_f^{\wedge(u+1)}\neq0.
\end{equation*}
\item\label{homoclinic}
There exist $N\geq1$ and
$\tilde q(f)\in W^u_{p(f),loc}$ such that
$q(f):=f^N(\tilde q(f))$, the point is a transverse homoclinic intersection in $W^s_{p(f),loc}\setminus\{p(f)\}$, $D_{\tilde q(f)}f^N$ is invertible,
and all its stable coordinates in the linearizing chart of $p(f)$ are nonzero.
Moreover, if $a$ and $b$ are the constants used in \eqref{eq-bun} in Assumption~\ref{saddle} then there is  $1\leq r<\ell$ with
\[
b^{\ell-r}<a^{r+1}
\qquad\text{and}\qquad
|\chi_{p(f)}|^{r+1}<|\lambda_1(f)|.
\]
\item\label{decalage}
If $y(f)\in W^s_{\Lambda(f),loc}\cap W^u_{p(f),loc}$ lies in $\mathcal U_+$, then all its coordinates in the linearizing chart of $W^u_{p(f),loc}$ are nonzero.
\item\label{grille} For each $1 \le i \le s$, let $\mathcal E^{s,i}_{f}$ denote the dynamical direction
given by Proposition~\ref{prop-cderiv} along $W^u_{p(f),loc}$
associated with the eigenvalue $\lambda_i(f)$.
Let $G_+ \colon f(\mathcal U_+) \to \mathcal U_+$ be the inverse branch of $f$.
Then, for every $x \in f(\mathcal U_+)$ and every
$y \in W^u_{p(f),loc}$,
\[D_x G_+\big(C^{ss}_\eta(x)\big) \cap \big(T_y W^u_{p(f),loc} \oplus\mathcal E^{s,i}_{f,y}\big) = \varnothing.\]
\item\label{init} For every non-empty open subset $\Omega'\subset\Omega_u$, there exists $f\in\Omega'$, $m_0\in\Nb$ and $\tilde r_0(f_0)\in W^s_{\Lambda(f_0),loc}$ such that
\begin{itemize}
\item[i)] $f^{m_0}(\tilde r_0(f))\in W^s_{r_0(f),loc},$
\item[ii)] $\tilde r_0(f)\in W^u_{p(f),loc},$
\item[iii)] the subgroup $\langle\chi_{p(f)},\chi_{r_0(f)}\rangle$ of $\Cb^*$ generated by $\chi_{p(f)}$ and $\chi_{r_0(f)}$ is dense.
\end{itemize}
\end{enumerate}

\bigskip

\noindent\emph{\textbf{Comments.}} This set of hypotheses is close to the one used in~\cite{gtv}.
Nevertheless, working in codimension possibly different from $1$ introduces new technical difficulties and additional delicate cases that must be handled.

Roughly speaking, the global strategy is the following. Consider a pair $(Y,f)$ in  $\mathcal X_D^u \times \mathrm{End}_d^k$ with $f \in \Omega_u$ and such that
$f^{n_0}(Y)$ intersects $W^s_{p(f),loc}$ transversely for some
$n_0\ge0$. Then, by the inclination lemma and the blender property, the iterates $f^n(Y)$ have arbitrarily many intersection points with the stable set of $\Lambda(f)$. It is \emph{very likely} that these intersections are sufficiently independent to allow one to apply, via Assumption~\ref{misu}, Proposition~\ref{prop-positivevolume} and thus conclude that $(Y,f)$ belongs to the support of the bifurcation
measure $\mu_{\mathcal Y}$ relative to any dynamically full and maximally
varying family
$(Y,f)\in\mathcal Y\subset\mathcal X_D^u\times\End_d^k$. Since the above condition on $(Y,f)$ is robust, this yields an open subset in this support.

A substantial part of the assumptions is devoted to turning this ``very likely" into a rigorous statement. To achieve this, we need a fine control of the dynamics, provided by the partially hyperbolic structure and the different bunching assumptions, together with genericity hypothesis, namely Assumptions~\ref{decalage}, \ref{grille}, and \ref{init}. On the other hand, Assumption~\ref{homoclinic} is mainly designed to ensure, via Proposition~\ref{prop-zariski-dense-unstable}, that $W^u_{p(f)}$ is Zariski dense in $\Pb^k$. This will be crucial for extending our main results from preperiodic subvarieties to $f$-special subvarieties.

Section~\ref{sec-existence} will be devoted to constructing examples satisfying all these hypotheses. In these examples, several assumptions will be verified thanks to the discrepancy between the dynamics on $\mathcal U_-$ and on $\mathcal U_+$. In particular, this will allow us to obtain hyperbolic sets, $p(f)$ and $\Lambda(f)$, with different indices, as well as the shifts between the strongly stable cone fields in Assumption~\ref{grille} and the coordinate shift in Assumption~\ref{decalage}.

Observe that the first two points of Assumption~\ref{init} are automatic under Assumption~\ref{free} together with the hyperbolicity properties of $\Lambda(f)$ (see Lemma~\ref{le-shadowing}).
The delicate issue is to obtain the three points simultaneously.
However, this difficulty already appeared in~\cite{gtv}, and the solution presented in Section~\ref{sec-existence} is an adaptation of that argument.

An important technical point is to control graph transforms of $u$-graphs in the polydiscs $\mathcal U_\pm$ and $\mathcal V_\pm$ (see Section~\ref{sec-control}). Observe that in $\mathcal U_-$ the graph transform is always well defined, since $f|_{\mathcal U_-}$ is horizontal-like of index $(s+1,u)$. In the other cases, there is an additional central repelling direction. Nevertheless, contraction of the cone field guarantees that a suitable restriction of the image of a $u$-graph tangent to $C^{uu}_\eta$ remains such a graph, provided it does not approach too closely the boundary coming from the central direction.
In particular, the strongly unstable manifolds arising from the dominated splitting in Assumption~\ref{hyperbolic} are $u$-graphs tangent to $C^{uu}_\eta$.

\subsection{Main results and general strategy}
\label{sec-main-blender}

We now state the main results of this section, which imply Theorem~\ref{th-ouvert-int} and also provide the version that we need in the polynomial setting.
\begin{theorem}
\label{th-main-blender}
Let
\[
\mathcal{Y}\subset \mathcal{X}_D^u\times \End_d^k
\]
be an irreducible, dynamically full and maximally varying family.
If $(Y,f)\in\mathcal{Y}$ with $f\in\Omega_u$ and $Y$ reduced and irreducible satisfies
\begin{equation}\tag{$\pitchfork$}\label{eq-pitch}
f^n(Y)\text{ intersects }W^s_{p(f),loc}\text{ transversely
for some }n\geq0,
\end{equation}
then $(Y,f)$ belongs to the support of $\mu_{\mathcal Y}$.
\end{theorem}

\begin{theorem}
\label{th-blender-poly}
The same statement holds for
\[
\mathcal{Y}\subset \mathcal{X}_D^u\times \Poly_d^k,
\]
with $\mu_{\mathcal{Y}}$ replaced by $\mu_{\mathcal{Y}}^{\rm pol}$,
provided that the transverse intersection in~\eqref{eq-pitch} lies in
$\Ab^k$.
\end{theorem}

The proof consists in showing that the assumptions of
Section~\ref{sec-ass}, together with~\eqref{eq-pitch}, imply the criterion
of Proposition~\ref{prop-positivevolume}. A key ingredient is the
independence of multipliers established by the second author with Igors
Gorbovickis~\cite{t-gorbovickis}, both on the moduli space
$\mathscr M_d^k$ and on its polynomial counterpart $\mathscr P_d^k$.
The arguments in both settings are very similar, so we focus on the
projective case.

\begin{rk}\label{rk-irreducible}
Observe that, since the member $Y$ in Theorem~\ref{th-main-blender} is reduced and irreducible, the family $\mathcal Y$ is generically irreducible. Without this assumption, a general member $(Y,f)$ could have a decomposition $Y=\sum_{j=1}^r m_j[Y_j]$, with some components $Y_j$ failing to satisfy \eqref{eq-pitch}.
Variations affecting only these components would then not be detected by our transversality argument, and the measure $\mu_{\mathcal Y}$ could therefore vanish.
\end{rk}

\bigskip

\noindent\emph{\textbf{Proof strategy for Theorem~\ref{th-main-blender}.}}
Let $\mathcal Y$ satisfy the assumptions of the theorem. Since \eqref{eq-pitch}, the condition $f\in\Omega_u$, and the condition that $Y$ is reduced and irreducible are stable under perturbation, we fix a small open subset $\mathcal Y_0\subset\mathcal Y$ on which all three hold.
Our goal is to prove
that
\begin{equation}\label{eq-inter-non-vide}
\supp(\mu_{\mathcal Y})\cap\mathcal Y_0\neq\varnothing.
\end{equation}

By Proposition~\ref{prop:local-normalization}, we may choose
\[
(Y_0,f_0)\in\mathcal Y_0\cap\mathcal Y^\circ
\]
and a local slice
\[
f_0\in\tilde\Omega\subset\Omega_u
\]
such that, after a local simultaneous conjugacy and possibly shrinking
$\mathcal Y_0$, a neighborhood of $(Y_0,f_0)$ in $\mathcal Y$ is
identified with a submanifold
\[
\tilde{\mathcal Y}\subset\mathcal X^u_d\times\tilde\Omega,
\]
whose projection on $\tilde\Omega$ is surjective. In particular, by dynamical fullness, Assumptions~\ref{free}
and~\ref{init} remain valid on $\tilde\Omega$. Observe also that we will
freely shrink $\mathcal Y_0$ and $\tilde\Omega$ in what follows.

We now let $m\geq0$ be maximal among the integers for which there exist
periodic points
\[
x_1(f_0),\ldots,x_m(f_0)\in\Lambda(f_0),
\]
in the hyperbolic set given by Assumption~\ref{hyperbolic}, satisfying
the two properties below. For any such choice of periodic points,
condition~\ref{it:cond1} of Proposition~\ref{prop-positivevolume}
follows directly from Lemma~\ref{le-cond1} and Assumption~\ref{misu}.
It remains to establish condition~\ref{it:cond2}.

\begin{enumerate}[label=\textbf{(2.\arabic*)}]
\item \label{it:cond21}
For each $1\leq j\leq m$, there are an open subset $U_j$ of
$\hat{\mathcal Y_0}$ and an integer $n_j$ such that
$\hat f^{n_j}(U_j)$ defines a family of $u$-graphs in
$\Db_{2R}\times\Bb^s\times\Bb_{3/2}^u$ tangent to $C^{uu}_\eta$.
More precisely, for each $t=(Y_t,f_t)\in\mathcal Y_0$, the fiber
$U_j(t)$ of $U_j$ above $t$ is an open subset of $Y_t$ such that
$f_t^{n_j}(U_j(t))$ is a $u$-graph in
$\Db_{2R}\times\Bb^s\times\Bb_{3/2}^u$ tangent to $C^{uu}_\eta$.

\item \label{it:cond22}
If we set
\[
M_j:=
\left\{
t=(Y_t,f_t)\in\mathcal Y_0
\,:\,
f_t^{n_j}(U_j(t))\cap W^s_{x_j(f_t),loc}\neq\varnothing
\right\},
\]
then, possibly after reducing $\mathcal Y_0$,
\[
M^{(m)}:=M_1\cap\cdots\cap M_m
\]
is an irreducible analytic subvariety of $\mathcal Y_0$ of codimension
$m$, containing $(Y_0,f_0)$.
\end{enumerate}

\begin{lemma}\label{le-cond2}
Possibly after shrinking $\mathcal Y_0$, conditions
\textup{\ref{it:cond21}} and \textup{\ref{it:cond22}} imply
condition~\textup{\ref{it:cond2}} from Section~\ref{sec-misiu}.
\end{lemma}

Roughly speaking, the set $U$ appearing in~\ref{it:cond2} corresponds
to the fibered product of the sets $U_j$ given by~\ref{it:cond21}.
However, a few technical issues arise: the integers $n_j$ need not
coincide, and $\Lambda(f)$ is only invariant under $f^2$. We therefore
postpone the proof of Lemma~\ref{le-cond2} to Section~\ref{sec-standard}
and continue with the proof strategy.

By condition~\ref{it:cond22}, necessarily
\[
m\leq\dim\mathcal Y,
\]
since $M^{(m)}$ is a non-empty analytic subvariety of $\mathcal Y_0$ of
codimension $m$. It remains to show that
\[
m=\dim\mathcal Y.
\]
Assume that $m<\dim\mathcal Y$. Then $M^{(m)}$ has positive dimension.
By maximality of $m$, no new independent stable/unstable relation
satisfying \ref{it:cond21} and \ref{it:cond22} can be produced along
$M^{(m)}$.

The argument has two main steps.
\begin{itemize}
\item First, we show that the projection of $M^{(m)}$ onto
$\tilde\Omega$ is constant. In other words, the $m$ relations defining
$M^{(m)}$ exhaust all degrees of freedom coming from the dynamics
(see Theorem~\ref{th-point}).

\item Now that the dynamics is fixed along $M^{(m)}$, we show, using
the genericity assumptions~\ref{decalage} and~\ref{grille}, that the
intersection between $f^n(Y)$ and $W^s_{\Lambda(f),loc}$ is rich enough
to also exclude variations of $Y$ (see Theorem~\ref{th-subvar}).
\end{itemize}

Our proof follows in fact a slightly different order, and this variation is
crucial. The first relation we impose is between the local unstable
manifold of the saddle point $p(f)$ defined in Assumption~\ref{saddle},
and a preimage of the local stable manifold of $r_0(f)$ defined in
Assumption~\ref{repelling}. This allows us to fully exploit the
heterodimensional cycle between $p(f)$ and $\Lambda(f)$.

This initial step relies on Assumption~\ref{init}, which also ensures that
the two central multipliers $\chi_{p(f)}$ and $\chi_{r_0(f)}$ generate a
dense subgroup of $\Cb^*$, a crucial ingredient for the first step above
(see Section~\ref{sec-relation} for details).

\bigskip

\noindent\emph{\textbf{Organisation of this section.}}
In the following subsection, possibly after shrinking $\tilde\Omega$
several times, we derive several consequences of the assumptions in
Section~\ref{sec-ass} and set up the geometric framework needed for the
rest of the proof.

Next, in Section~\ref{sec-control}, we establish two graph transform
control results, respectively near the saddle point $p(f)$ and near
saddle points in $\Lambda(f)$.

At that stage, Section~\ref{sec-relation} establishes the first step
of the proof of Theorem~\ref{th-main-blender} mentioned above, while
Section~\ref{sec-subvar} deals with the second step of the proof.

Finally, Section~\ref{sec-preuve-main} completes the proof of
Theorem~\ref{th-main-blender}.

\subsection{First consequences of the assumptions and standard forms for saddle fixed points}\label{sec-standard}
We use here the open set $\tilde\Omega$ introduced above. The assumptions of Section~\ref{sec-ass} have several consequences. The first one is the
absence of $f$-special subvarieties among the pairs $(Y,f)\in\mathcal Y_0$ satisfying~\eqref{eq-pitch}. This relies on a recent result of Y.~Zhang
\cite{zhang-iterated-centralizer}, according to which there exists a dense Zariski open subset $\mathcal U_{\mathrm{cent}}$ in $\End_d^k$ (or $\Poly_d^k$) such that, for every
$f\in\mathcal U_{\mathrm{cent}}$, any endomorphism of $\Pb^k$
commuting with an iterate of $f$ is itself an iterate of $f$. Since the image of $\tilde\Omega$ in $\mathscr M_d^k$ (or $\mathscr P_d^k$) is open, after shrinking
$\tilde\Omega$ if necessary, we may assume that it is contained in $\mathcal U_{\mathrm{cent}}$.
\begin{proposition}\label{prop-special}
If $f$ is in $\tilde\Omega$ and if $Y\in\mathcal X_D^u$ satisfies \eqref{eq-pitch} then $Y$ is not $f$-special.
\end{proposition}
\begin{proof}
By~\eqref{eq-pitch}, there exists $n\geq0$ such that $f^n(Y)$
intersects $W^s_{p(f),loc}$ transversely. Fix $m\geq1$. Applying the
inclination lemma to $f^m$, we obtain
\[
W^u_{p(f)}
\subset
\overline{\bigcup_{j\geq0}f^{mj}(Y)}^{\,\mathrm{Zar}}.
\]
On the other hand, Assumption~\ref{homoclinic} and
Proposition~\ref{prop-zariski-dense-unstable} give
\[
\overline{W^u_{p(f)}}^{\,\mathrm{Zar}}=\Pb^k.
\]
Hence
\begin{equation}\label{eq-dense-shifted-orbit}
\overline{\bigcup_{j\geq0}f^{mj}(Y)}^{\,\mathrm{Zar}}=\Pb^k.
\end{equation}
In other words, for every $m\geq1$, the $f^m$-orbit of $Y$ is Zariski dense.

In particular, $Y$ is not preperiodic under $f$. Suppose now, by
contradiction, that $Y$ is $f$-special. By the definition of
$f$-specialness, there exist a closed subvariety $Z\subset\Pb^k$, an
integer $m\geq1$, and a polarized endomorphism $g:Z\to Z$ such that
\[
Y\subset Z,\qquad f^m(Z)=Z,
\]
the maps $f^m$ and $g$ commute on $Z$, and $Y$ is preperiodic under
$g$. Since $Z$ contains the whole $f^m$-orbit of $Y$, the density
proved above implies
\[
Z=\Pb^k.
\]
Thus $g$ is an endomorphism of $\Pb^k$ commuting with $f^m$.
By the choice of $\tilde\Omega$ coming from \cite{zhang-iterated-centralizer}, there exists $j\geq1$ such that
\[
g=f^j.
\]
But then the preperiodicity of $Y$ under $g$ implies its
preperiodicity under $f$, a contradiction.
\end{proof}

We now temporarily fix an element $f' \in \tilde\Omega$. As the periodic points are dense in $\Lambda(f')$ there exist $N := \dim \mathscr{M}_d^k$ periodic points $r_1(f'),\ldots,r_N(f')$ in $\Lambda(f')\cap\Db_{1/2}\times\Bb^s\times \Bb_{3/2}^u$, corresponding to distinct cycles of period greater than $4$. Possibly after shrinking $\tilde\Omega,$ we can assume that the holomorphic continuation of these points stay in $\Db_{1/2}\times\Bb^s\times \Bb_{3/2}^u$ for all $f\in\tilde\Omega.$

The dominated splitting in Assumption~\ref{hyperbolic} ensures that each of
these points possesses a central (unstable) direction.
We denote by $\chi_{r_1(f)},\ldots,\chi_{r_N(f)}$ the associated multipliers.

A crucial fact is that, by~\cite[Theorem~1.2]{t-gorbovickis}, these
multipliers are independent. Consequently, after shrinking  once again
$\tilde\Omega$, we may assume that the map
\[
f \mapsto \big(\chi_{r_1(f)},\ldots,\chi_{r_N(f)}\big)
\]
is a biholomorphism from $\tilde\Omega$ onto its image.
In particular, its Jacobian does not vanish on $\tilde\Omega$.

From now on, in order to simplify notation, we replace the family
$(f)_{f\in\tilde\Omega}$ by $(f^n)_{f\in\tilde\Omega}$, where
$n$ is the least common multiple of the periods of the cycles
$r_0(f),\ldots,r_N(f)$.
This allows us to assume that each of the points
$p(f), r_0(f), \ldots, r_N(f)$ is fixed and that $\Lambda(f)$ is $f$-invariant. Notice that with this change,
the constant $\tilde\chi$ in Section~\ref{sec-ass} has to be replaced
by $\tilde\chi^n$. Once this is done, Assumption~\ref{hyperbolic}
implies that the central multiplier $\chi_{r_i(f)}$ of $r_i(f)$ satisfies
\[1 <|\chi_{r_i(f)}| < \tilde\chi,\]
whereas the stable and strongly unstable multipliers have modulus
respectively smaller than $\tilde\chi^{-2}$ and larger than $\tilde\chi^{2}$.

Moreover, since the $k$ multipliers of a periodic point are free
parameters in $\mathcal M_d^k$, there exists an integer $\ell \ge 4$
such that the multipliers of the fixed points
$r_0(f), \ldots, r_N(f)$ satisfy
conditions~\eqref{eq-bun} and~\eqref{eq-res} on a dense open subset of
$\tilde\Omega$. Observe that by Assumption~\ref{saddle} $p(f)$ also satisfies such conditions and, after possibly shrinking $\tilde\Omega$ again, we may assume
\begin{equation}\tag{$\star$}\label{star} \begin{minipage}{0.9\linewidth} For every $f \in \tilde\Omega$, the points $p(f), r_0(f), \ldots, r_N(f)$ are fixed points of $f$ satisfying conditions~\eqref{eq-bun} and~\eqref{eq-res}. Moreover, the Jacobian of the map \[ f \mapsto \big(\chi_{r_1(f)},\ldots,\chi_{r_N(f)}\big) \] does not vanish on $\tilde\Omega$. \end{minipage} \end{equation}
We emphasize that the pair of integers $(s,u)$ has a different meaning here than in Section~\ref{sec-dir}, where it refers to either $(s+1,u)$ or $(s,u+1)$.

Another important feature of the dynamics on $\Lambda(f)$ is that it is topologically transitive and satisfies the shadowing lemma. In particular, this yields the following lemma which, in the repelling case, essentially asserts that the preimages of any point in $\Lambda(f)$ are dense in $\Lambda(f)$.
For our purposes, however, we will need to obtain a preimage outside the strong unstable manifold.
In a naïve product situation, this could be not possible, but the blender property easily excludes this obstruction.
\begin{lemma}\label{le-shadowing}
Let $f\in\tilde\Omega.$ Let $x,y\in\hat\Lambda(f)$ and $\epsilon>0.$ Then, there exists $n\geq0$ and $z\in\Lambda(f)$ such that
\[z\in W^u_{x,loc}\cap B(x,\epsilon)\ \text{ and }\ f^n(z)\in W^s_{y,loc}\cap B(y,\epsilon).\]
Moreover, if $x$ corresponds to the periodic point $r_0(f)$, then we can assume that $z\notin W^{uu}_{x,loc}.$
\end{lemma}
\begin{proof}
Write $x=(x_j)_{j\leq0}$ and $y=(y_j)_{j\leq0}$.
Since $\Lambda(f)$ is locally maximal, its natural extension $\hat\Lambda(f)$ has local product structure, and therefore the shadowing lemma holds on $\hat\Lambda(f)$ (see, for instance, \cite[Paper I, Theorem~2.4]{jonsson-phd}).
Hence there exists $\eta>0$ such that any $\eta$-pseudo-orbit in $\hat\Lambda(f)$ stays within distance $\epsilon$ of a true orbit in $\hat\Lambda(f)$.

On the other hand, by topological transitivity there exist $\tilde z\in\Lambda(f)\cap B(x,\eta)$ and $n\geq0$ such that
\[f^n(\tilde z)\in B(y,\eta).\]
Applying the shadowing lemma yields a point $z=(z_j)_{j\leq0}\in\hat\Lambda(f)$ such that for all $j\leq0$ and $i\geq0$,
\[d(z_j,x_j)<\epsilon
\quad\text{and}\quad
d\big(f^{n+i}(z_0),f^i(y_0)\big)<\epsilon.\]
Possibly after reducing $\epsilon$, this proves the first part of the statement.

Assume now that $x=r_0(f)$.
By Assumption~\ref{repelling}, the strong unstable manifold $W^{uu}_{x,loc}$ is a $u$-graph tangent to the cone field $C^{uu}_\eta$ inside the blender region $\Db\times\Bb^s\times\Bb_{3/2}^u$.
Let $L$ be a leaf of the strong unstable foliation inside $W^u_{x,loc}$ that is close to, but distinct from, $W^{uu}_{x,loc}$.
Then, Assumption~\ref{blender} implies that $L$ intersects $W^s_{\Lambda(f),loc}$ at some point $x'$.
By local maximality of $\Lambda(f)$, we have $x'\in\Lambda(f)$.

Applying the first part of the lemma to $x'$ (viewed as a point of $\hat\Lambda(f)$ whose backward orbit converges to $x$) and choosing $\epsilon<d(L,W^{uu}_{x,loc})$, we obtain $n'\geq0$ and $z'\in\Lambda(f)$ such that
\[z'\in W^u_{x',loc}\cap B(x',\epsilon)
\quad\text{and}\quad
f^{n'}(z')\in W^s_{y,loc}\cap B(y,\epsilon).\]
Since $W^u_{x',loc}\cap B(x',\epsilon)$ is an open subset of $W^u_{x,loc}$ for $\epsilon$ small enough, the conclusion follows by taking preimages of $z'$ inside $W^u_{x,loc}$.
\end{proof}

We finish this part with the proof of Lemma~\ref{le-cond2}.
\begin{proof}[Proof of Lemma~\ref{le-cond2}]
In this proof, set $\mathcal D:=\Db_{2R}\times\Bb^s\times\Bb_{3/2}^u$.

Observe that if $f\in\tilde\Omega$ and $\Gamma$ is a $u$-graph in
$\mathcal U_-$ tangent to $C^{uu}_\eta$, then
$f(\Gamma)\cap\mathcal D$ still has the same properties, since
$f|_{\mathcal U_-}$ is a degree~$1$ horizontal-like map of index
$(s+1,u)$ from $\mathcal U_-$ to $\mathcal D$. This may fail if
$\mathcal U_-$ is replaced by $\mathcal U_+$, since
$f(\Gamma)\cap\mathcal D$ could a priori intersect the central
boundary $(\partial\Db_{2R})\times\Bb^s\times\Bb_{3/2}^u$. However,
this cannot occur if $\Gamma$ intersects $W^s_{x,loc}$ for some
$x\in\Lambda(f)$, since
$\Lambda(f)\subset\Db_R\times\Bb^s\times\Bb_{3/2}^u$ and
$W^s_{x,loc}$ is tangent to $C^{ss}_\eta$.

Therefore, using~\ref{it:cond22} and possibly replacing
$\mathcal Y_0$ by a smaller neighborhood of $(Y_0,f_0)$, we may
replace $n_j$ by $n_j+1$ and $U_j$ by a smaller set $U'_j$ such
that
\[\hat f^{n_j+1}(U'_j)=
\hat f^{n_j+1}(U_j)\cap(\mathcal D\times\mathcal Y_0)\]
defines a family of $u$-graph in $\mathcal D.$

Iterating this procedure, we may replace each $n_j$ by
$n_0:=\max_j n_j$ and assume that all these integers coincide.
We then define $U$ as the fibered product of
$U_1,\ldots,U_m$ over $\mathcal Y_0$.

We claim that $(\hat f^{[m]})^{n_0}(U)$ intersects properly
$\mathcal W^s_{x,loc}$, where
$x(f)=(x_1(f),\ldots,x_m(f))$. Fix a parameter
$t=(Y_t,f_t)\in\mathcal Y_0$. Then each
$f_t^{n_0}(U_j(t))$ is a $u$-graph in $\mathcal D$, hence is
contained in a $cu$-graph. By Lemma~\ref{le-inter}, it
intersects the $s$-graph $W^s_{x_j(f_t),loc}$ in at most one
point. Therefore, the set
\[\left\{(y_j,t)\in\Pb^k\times\mathcal Y_0 \; ; \;
y_j\in f_t^{n_0}(U_j(t))\cap W^s_{x_j(f_t),loc}\right\}\]
projects biholomorphically onto
\[M_j:=\left\{t=(Y_t,f_t)\in\mathcal Y_0 \, ; \,
f_t^{n_0}(U_j(t))\cap W^s_{x_j(f_t),loc}\neq\varnothing\right\},\]
defined in~\ref{it:cond22}. It follows that
$(\hat f^{[m]})^{n_0}(U)\cap\mathcal W^s_{x,loc}$ is
biholomorphic to
\[M^{(m)}:=M_1\cap\cdots\cap M_m,\]
and therefore has dimension $\tilde m - m$, where
$\tilde m:=\dim(\mathcal Y_0)$.

On the other hand, $(\hat f^{[m]})^{n_0}(U)$ has dimension
$\tilde m+um$, while $\mathcal W^s_{x,loc}$ has dimension
$\tilde m+sm$. Their codimensions in
$(\Pb^k)^m\times\mathcal Y_0$ add up to
\[(s+1)m+(u+1)m=(k+1)m,\]
which coincides with the codimension
\[\tilde m+km-(\tilde m-m)\]
of $(\hat f^{[m]})^{n_0}(U)\cap\mathcal W^s_{x,loc}$.
\end{proof}

 \noindent\emph{\textbf{Standard forms.}} Since the saddle fixed points $p(f), r_0(f), \ldots, r_N(f)$ satisfy
conditions~\eqref{eq-bun} and~\eqref{eq-res}, Proposition~\ref{prop-cderiv} gives special local coordinates near each of these points.

Actually, these coordinates will be adapted to our partially hyperbolic setting.
More precisely, we denote by $\Cb^c$ a copy of $\Cb$, and by $B^c$, $B^s$, and $B^u$ the unit polydiscs of
$\Cb^c$, $\Cb^s$, and $\Cb^u$, respectively. Observe that these domains coincide with
$\Db$, $\Bb^s$, and $\Bb^u$, respectively. However, we use the notation $B^\ast,$ $\ast\in\{c,s,u\},$ 
only in the local setting, while the notation $\Bb^\ast$ is reserved for the
original dynamical space.

Moreover, we also set
\[B^{cs} := B^c \times B^s \subset \Cb^{cs} := \Cb^c \times \Cb^s\quad \text{and} \quad
B^{cu} := B^c \times B^u \subset \Cb^{cu} := \Cb^c \times \Cb^u.\]
We will always use notation that illustrates this decomposition for objects in
$B^c \times B^s \times B^u$.
For example, a point $x \in B^c \times B^s \times B^u$ will be decomposed as
\[x = (x^c, x^s, x^u),\]
and similarly for tangent vectors or constants appearing in this context.

From this, combining Proposition~\ref{prop-cderiv} and Proposition~\ref{prop-bunching}, applied to the point $p(f)$, we choose a neighborhood $B_p$
of $W^u_{p(f),loc}$ contained in $\Db \times \Bb^s \times \Bb_{3/2}^u$
(the neighborhood may depend slightly on $f$, although we do not
make this dependence explicit),
together with a holomorphic chart, depending holomorphically on $f$,
from $B_p$ to $B^c \times B^s \times B^u$,
in which $p(f)$ is in \emph{standard form}. That is,
\begin{itemize}
\item the fixed point $p(f)$ is constantly $0$;
\item $B^{cs} \times \{0\}$ is the local stable manifold of $f$, on which $f$ acts as the attracting diagonal matrix
\[(x^c, x^s) \mapsto \big(\chi_{p(f)} x^c, A_f x^s\big)
\quad \text{where} \quad
A_f := \mathrm{diag}\big(\lambda_1(f), \ldots, \lambda_s(f)\big);\]
\item $\{0\} \times B^u$ is the local unstable manifold of $f$, on which $f$ acts as an expanding diagonal matrix
\[x^u \mapsto E_f x^u,
\quad \text{where} \quad
E_f := \mathrm{diag}\big(\rho_1(f), \ldots, \rho_u(f)\big);\]
\item along the stable and unstable manifolds, the dynamical directions defined in Proposition~\ref{prop-dynaslope}
coincide with the standard coordinate points in $\Pb^{k-1}$;
\item there exists a family of $C^2$ linearization maps $(\theta_f)_{f \in\tilde\Omega}$ on $B_p$, depending continuously on $f$
in the $C^2$ topology, such that for every
$y \in W^u_{p(f),loc} \cup W^s_{p(f),loc}$, the differential $D_y \theta_f$ is diagonal and $\theta_f$ conjugates $f$ to
\[K_f \colon (x^c, x^s, x^u) \mapsto \left(\chi_{p(f)} x^c, A_f x^s, E_f x^u\right).\]
\end{itemize}
When $p(f)$ is in standard form in $B_p$, we identify $B_p$ with
$B^c \times B^s \times B^u$, where the notions of $u$-, $s$-, $cs$- and
$cu$-graphs are defined. We still denote by $f$ the germ of the
dynamics on this set. Moreover, by definition, such a standard form
comes together with the linearization $\theta_f$ and the linear map $K_f$.

If $(r(f))_{f \in \tilde{\Omega}}$ corresponds to one of the families
$(r_i(f))_{f \in \tilde{\Omega}}$, $i = 0, \ldots, N$, we define the
\emph{standard form} for $r(f)$ on a polydisc $B_r$ in the same way,
except that $B^{cs}$ and $B^u$ are replaced by $B^s$ and $B^{cu}$, respectively,
and that $B_r$ is a small neighborhood of $W^{uu}_{r(f),loc}$ in
$\Db \times \Bb^s \times \Bb_{3/2}^u$, which we identify with
$B^c \times B^s \times B^u$ via appropriate coordinates.

This is possible since we have assumed that $r(f)$ lies in
$\Db_{1/2} \times \Bb^s \times \Bb_{3/2}^u$, and since
$W^{uu}_{r(f),loc}$ is a $u$-graph tangent to $C^{uu}_\eta$ in that set.
This property is important, as it ensures that if $B_r$ is chosen
to be a sufficiently small neighborhood of $W^{uu}_{r(f),loc}$, then
any $u$-graph in $B_r$ must intersect the stable set of $\Lambda(f)$
by Assumption~\ref{blender}.

We still use the notation $\theta_f$ for the $C^2$ linearization of $r(f)$ in $B_r$.
No ambiguity will arise, since we always use these different coordinate systems
separately. The same convention applies to the linear map $K_f$, knowing that in $B_r$
\[K_f(x^c, x^s, x^u) = \big(\chi_{r(f)} x^c, A_f x^s, E_f x^u\big),\]
where $A_f$ and $E_f$ are diagonal, attracting and expanding respectively,
and $|\chi_{r(f)}| > 1$.

Observe that, by definition of these standard forms, the invariant manifolds
$W^u_{p(f),loc}$ and $W^{uu}_{r(f),loc}$ correspond to the same objects,
whether they are viewed in $\Db \times \Bb^s \times \Bb_{3/2}^u$ or in $B_p$ and $B_r$,
respectively. This is not a priori the case for the other invariant manifolds,
in particular for $W^s_{p(f),loc}$ (resp. $W^u_{r(f),loc}$), which may be much
smaller in $B_p$ (resp. $B_r$).

On the other hand, in order to control the graph transform near $r(f)$ in
Lemma~\ref{le-control2}, we consider two different ``local center-stable manifolds'' in $B_r$.
More precisely, we set
\[\mathsf W_{r(f)}^{cs} := B^{cs} \times \{0\} \subset B_r,\]
which is a complex manifold but \emph{not invariant} under the dynamics.
By contrast, the linearization $\theta_f$ provides a $C^2$ local center-stable
manifold $W^{cs}_{r(f),loc} \subset B_r$ that is invariant but \emph{not holomorphic}. Notice that, possibly after shrinking $B_r$ in these directions, we may assume
that $W^{cs}_{r(f),loc}$ intersects every $u$-graph in $B_r$.

An important observation here is that, thanks to the dynamical directions, these two
center-stable manifolds are tangent along $W^s_{r(f),loc}$.
If this were not the case - that is, if the tangent space to
$\mathsf W_{r(f)}^{cs}$ along $W^s_{r(f),loc}$ had a component in the unstable direction -
then the estimates obtained in Lemma~\ref{le-control2} could break down.

\bigskip

We conclude this part with an elementary result that will be used
only for fixed $f \in \tilde{\Omega}$ in Section \ref{sec-subvar}. Observe that it may be viewed
as a counterpart of Assumption~\ref{grille}, but involving the central
multiplier of $p(f)$. We do not include this statement among the
assumptions in order to emphasize the contrast between this simple
consequence of cone contraction and Assumption~\ref{grille}, which requires a more delicate control.
\begin{lemma}\label{le-centrer-unstable}
For each $y\in W^{u}_{p(f),loc}$ let denote by $\mathcal E^c_y$ the dynamical direction associated to the central eigenvalue $\chi_{p(f)}.$ Then, the center-unstable bundle along $W^u_{p(f),loc}$ defined by
\[E^{cu}:=TW^u_{p(f),loc}\oplus\mathcal E^c_y\]
is disjoint from the stable cone field $C^{ss}_\eta.$
\end{lemma}
\begin{proof}
Let $y \in W^{u}_{p(f),loc}$ and let $v \in E^{cu}_y \setminus \{0\}$.
By invariance of the dynamical directions,
\[D_y f^{-n}(v) \in E^{cu}_{f^{-n}(y)} \quad \text{for all } n \ge 0.\]
Since $f^{-n}(y)$ converges to $p(f),$ any limit value $v_\infty$ of
\[v_n := \frac{D_y f^{-n}(v)}{\|D_y f^{-n}(v)\|}\]
belongs to $E^{cu}_{p(f)}$.

By contradiction, if $v \in C^{ss}_\eta$, then, using that this cone field is contracted by $f^{-1}$, we should have $v_\infty \in C^{ss}_\eta$.
This is impossible since,
\[E^{cu}_{p(f)}\cap C^{ss}_\eta=\varnothing,\]
completing the proof.
\end{proof}

%
%
%
%
%
%

\subsection{Control Lemmas}\label{sec-control}
This section is devoted to controlling graph transforms near the
fixed points $p(f)$ and $r_i(f)$, $i=0,\ldots,N$. Observe that
Lemma~\ref{le-control1} closely follows the proof of
\cite[Lemma~3.10]{gtv}, although the $\Cb$-linearity used there
does not follow as directly in our setting. Instead, we use
Proposition~\ref{prop-cderiv}, which relies on the conditions
\eqref{eq-bun} and \eqref{eq-res}. On the other hand,
Lemma~\ref{le-control2} is similar to
\cite[Lemma~3.11]{gtv}, but its proof in the saddle case is more
delicate.

We will apply these
results to parameter spaces $M$ which are submanifolds of $\mathcal X_D^u \times \tilde\Omega$. To simplify notation,
if $f=(Y,\tilde f)$ is a point of $M$, we view $f$ as the map $\tilde f$
from the dynamical point of view. In particular, we write
$p(f)=p(\tilde f)$ and $r_i(f)=r_i(\tilde f)$. Note that even if $\tilde f$ is constant, the sets $(\Gamma(f))_{f\in M}$ may vary.
\begin{lemma}\label{le-control1}
Let $(p(f))_{f\in M}$ be in standard form in $B_p$ as in Section~\ref{sec-standard}.
Let $\Delta=(\Delta_f)_{f\in M}$ be a family of $cs$-graphs given by injective holomorphic maps $h_f\colon B^{cs}\to B_p$ preserving the central direction, i.e., such that $D_0h_f(1,0)$ belongs to $E^{cu}$ as defined in Lemma \ref{le-centrer-unstable}.
Then, there exists a holomorphic map $\alpha\colon M\to\Cb^*$ with the following property.

Let $\Gamma=(\Gamma(f))_{f\in M}$ be a family of $u$-graphs in $B_p$ intersecting $W^s_{p(f),loc}$ at $(b^c(f),b^s(f))$, with $b^c(f)\neq0$ on $M.$ Then the graph transform $f^n(\Gamma(f))$ satisfies
\[f^n(\Gamma(f))\cap\Delta_f=\{h_f(a_n(f))\}\text{ where } a_n(f)=\chi_{p(f)}^n\Big[b^c(f)\alpha(f)(1,0)+R_n(f)\Big],\]
 where $R_n\colon M\to\Cb^{cs}$ are holomorphic maps converging to $0,$ locally uniformly with respect to $f,$ $\Gamma$ and $\Delta$.
 \end{lemma}
 \begin{rk}\label{rk-control1}
 By taking $h_f\colon (x^c,x^s)\mapsto(x^c,x^s,z)$ for a fixed $z\in B^u,$ we obtain that the graph transform $f^n(\Gamma(f))$ is given by
\[f^n(\Gamma(f))=\left\{\left(\chi_{p(f)}^{n}b^c(f)\Big[\alpha(f)(1,0)+R_{n}(f,z)\Big],z\right)\,;\,z\in B^u\right\}.\]
 \end{rk}

\begin{proof}
Let $\pi\colon \Cb^{cs}\times \Cb^u \to \Cb^{cs}$ denote the natural projection, and define $\tilde\pi$ by 
\[\tilde\pi(y,z) = (y,0), \quad (y,z)\in \Cb^{cs}\times \Cb^u.\]
Our goal is to compare,  near the unstable manifold, the map $\tilde\pi$ with the projection
\[G_f = \pi \circ \theta_f^{-1} \circ \tilde\pi \circ \theta_f\]
along the unstable foliation of the $C^2$-linearization. 

A key fact is that, by Proposition \ref{prop-cderiv} the map $G_f$ is $\Cb$-differentiable on $W_{p(f),loc}^u$. In particular, the map $G_f\circ h_f,$ which is a local diffeomorphism at $0$, is $\Cb$-differentiable at this point. We denote its inverse by $H_{f}:=(G_f\circ h_f)^{-1}$.

For $n\geq 0$, the graph transform $f^n(\Gamma(f))$ is a $u$-graph 
which intersects $\Delta_f$ in a single point. We consider $(x_n(f))_{f\in M}$ the family of points on $\Gamma(f)$ and $(a_n(f))_{f\in M}$ in $B^{cs}$ such that $f^n(x_n(f))=h_f(a_n(f))$ corresponds to this intersection point. Then
\[G_f(f^n(x_n(f))) 
= \pi\circ\theta_f^{-1}\circ\tilde\pi\circ\theta_f(f^n(x_n(f))) 
= \pi\circ\theta_f^{-1}\circ\tilde\pi\circ K_f^n\circ\theta_f(x_n(f)).\]
Recall that if $(x^c,x^s,x^u)\in\Cb^{cs}\times\Cb^u$ then $K_f(x^c,x^s,x^u)=(\chi_{p(f)}x^c,A_fx^s,E_fx^u),$ where $A_f$ is contracting, $E_f$ is expanding, and 
\[\|A_f\|,\ \|E_f^{-1}\|<\tilde\chi^{-2}<\tilde\chi^{-1}<|\chi_{p(f)}|.\] 
Writing $\theta_f(x_n(f))=(x^c_n(f),x^s_n(f),x^u_n(f))$, we obtain
\[K_f^n\circ\theta_f(x_n(f))
= (\chi^n_{p(f)}x^c_n(f),A_f^n x^s_n(f),E_f^n x^u_n(f)).\]
In particular, $x_n^u(f)=o(\tilde\chi^{-2n})$. Since $\Gamma(f)$ intersects $W^s_{p(f),loc}$ at $(b^c(f),b^s(f))$, it follows that
\[x_n^c(f)=b^c(f)+o(\tilde\chi^{-2n}), 
\qquad 
x^s_n(f)=b^s(f)+o(\tilde\chi^{-2n}).\]
Combining this with $\|A_f^n\|=o(\tilde\chi^{-2n})$ yields
\begin{equation}\label{eq-erreur1}
G_f(f^n(x_n(f)))=\chi_{p(f)}^n b^c(f)(1,0)+o(\tilde\chi^{-2n}).
\end{equation}
On the other hand, since $f^n(x_n(f))=h_f(a_n(f))$, we have
\begin{align*}
a_{n}(f) 
&= H_{f}(G_f(f^n(x_n(f)))) = H_{f}\!\left(\chi_{p(f)}^n b^c(f)(1,0)+o(\tilde\chi^{-2n})\right) \\
&= 0+\chi_{p(f)}^n b^c(f)\, D_0H_{f}(1,0)+o(\chi^n_{p(f)}).
\end{align*}
By assumption, $D_0H_{f}(1,0)$ is equal to $\alpha(f)(1,0)$ for some normalization constant $\alpha(f)\in\Cb^*.$ This function $f\mapsto\alpha(f)$ is holomorphic, since it is the uniform limit of 
\[f\mapsto \frac{a_{n}(f)}{\chi_{p(f)}^n b^c(f)}.\]
It follows that the remainder
\[R_n(f):=\chi_{p(f)}^{-n}a_n(f)-b^c(f)\alpha(f)(1,0)\]
is holomorphic as well. Moreover, $R_n(f)=o(1)$ as $n\to\infty$. In fact, this remainder splits into two contributions. The first one comes from the error term in~\eqref{eq-erreur1}, which, after factoring out $\chi_{p(f)}^n$, is $o(\tilde\chi^{-n})$ uniformly in $f$ and $\Gamma$. The second one arises from the Taylor expansion of $H_f$, which is uniform in $f$ and $\Delta$, since $\theta_f$ is of class $C^2$ and depends continuously on $f$ in the $C^2$ topology.
\end{proof}


We now turn to the dynamics near a point
\[(r(f))_{f\in M}:=(r_i(f))_{f\in M}\]
for some $i=0,\ldots,N$. Recall that $W^{cs}_{r(f),loc}$ denotes the local center-stable
manifold of $r(f)$, and that $\mathsf W_{r(f)}^{cs}$ is simply the
plane $\{z^u = 0\}$ in standard coordinates $B_r$ near $r(f)$. We emphasize
again that $\mathsf W_{r(f)}^{cs}$ is not invariant in general.


\begin{lemma}\label{le-control2}
Let $(\Gamma_l)_{l\geq0}$ be a sequence of families parametrized by $M$ of $u$-graph in $B_r$ intersecting $\mathsf W_{r(f)}^{cs}$ at $(b^c_l(f),b^s_l(f)).$ Assume that this sequence converges to $\Gamma_\infty:=(\Gamma_\infty(f))_{f\in M}$ satisfying
\begin{equation}\label{eq-gammainfty}
\forall f\in M,\ \ \Gamma_\infty(f)\cap W^s_{r(f),loc}\neq\varnothing.
\end{equation}
Then, for each sequence $(n_l)_{l\geq0}$ and for all $f\in M$ such that
\begin{equation}\label{eq-hypcontrol2}
\tilde\chi^{-2n_l}=o(b_l^c(f))\ \text{ and }\ |\chi_{r(f)}^{n_l}b^c_l(f)\beta(f)|<1/4,
\end{equation}
the graph transform $f^{n_l}(\Gamma_l(f))$ is a subset of the form
\begin{equation}\label{eq-controle2}
\left\{\left(\chi_{r(f)}^{n_l}b^c_l(f)\Big[\beta(f)(1,0)+S_{n_l}(f,z)\Big],z\right)\,;\,z\in B^u\right\},
\end{equation}
where the error term $S_{n_l}$ is holomorphic and converges to $0$ locally uniform with respect to $f \in M$ and $z\in B^u$.
\end{lemma}
The purpose of the second assumption in~\eqref{eq-hypcontrol2} is to ensure that the graph transform $f^{n_l}(\Gamma_l(f))$ is well defined, that is, does not escape from $B_r$, while the first assumption guarantees that $n_l$ is sufficiently large with respect to $b^c_l.$ These two constraints will also appear in the applications of this result, namely in Lemmas~\ref{le-doublecontrole} and~\ref{le-triplecontrole}.
\begin{proof}
We first obtain estimates on the point $\Gamma_l(f)\cap W^{cs}_{r(f),loc}$, which provide control on the integers $n$ for which $f^n(\Gamma_l(f))$ is well defined. In a second step, we obtain control of the graph $f^n(\Gamma_l(f))$, which could in fact be achieved by applying Lemma~\ref{le-control1} to $f^{-1}$.

Actually, as we have seen in Section~\ref{sec-standard}, there exists a unique intersection point
\[z_l(f)=(z_l^c(f),z_l^s(f),z_l^u(f))\]
between $W^{cs}_{r(f),loc}$ and $\Gamma_l(f)$. We use the same notations with $l=\infty,$ with the remark that $z_\infty^c(f)$ and $z_\infty^u(f)$ both vanish by \eqref{eq-gammainfty}.

If $\check{z}_l(f):=\theta_f(z_l(f))=(\check{z}^c_l(f),\check{z}^s_l(f),0),$ then using that $\theta_f$ is the identity on $W^{s}_{r(f),loc}$, we have
\begin{align*}
z_l(f)&=(0,\check{z}^s_l(f),0)+D_{(0,\check{z}^s_l(f),0)}\theta_f^{-1}(\check{z}^c_l(f),0,0)+o(\check{z}^c_l(f))\\
&=(0,\check{z}^s_l(f),0)+D_{(0,\check{z}^s_\infty(f),0)}\theta_f^{-1}(\check{z}^c_l(f),0,0)+o(\check{z}^c_l(f)),
\end{align*}
where the second equality follows from the fact that $\check{z}^s_l(f)$ converges to $\check{z}^s_\infty(f)$. On the other hand, since $D\theta_f$ preserves the dynamical directions and is $\Cb$-linear along $W^s_{r(f),loc}$, we have
\[\check{z}^c_l(f)(\gamma(f),0,0):=D_{(0,\check{z}^s_\infty(f),0)}\theta_f^{-1}(\check{z}^c_l(f),0,0),\]
where $\gamma(f)\neq0$. Here we make two observations. First, the dependence of these objects
(such as $z_l(f)$, $\gamma(f)$, etc.) on $f$ is not a priori holomorphic.
Second, if
$D_{(0,\check{z}^s_\infty(f),0)} \theta_f^{-1}(1,0,0)$
had a nonzero component in the strong unstable direction, then this could
break all the estimates that follow. Thus, Proposition~\ref{prop-cderiv}
plays an essential role here.

To summarize, so far we have
\[z^u_l(f)=o(\check{z}^c_l(f))\ \text{ and } \ z^c_l(f)=\check{z}^c_l(f)(\gamma(f)+o(1)).\]
Thus
\[\check{z}^c_l(f)=z^c_l(f)(\beta(f)+o(1))\ \text{ and }\ z^u_l(f)=o({z}^c_l(f)),\]
where $\beta(f):=\gamma(f)^{-1}.$

The last observation for this part is that, writing $\Gamma_l(f)$ as a $u$-graph
\[\Gamma_l(f)=\left\{(h^c_{l,f}(z^u),h^s_{l,f}(z^u),z^u)\ ; \ z^u\in B^u\right\},\]
gives that
\[z^c_l(f)=h^c_{l,f}(z^u_l(f))=h^c_{l,f}(0)+O(z^u_l(f))=h^c_{l,f}(0)+o(z^c_l(f)).\]
Since by definition $h^c_{l,f}(0)=b^c_l(f)$, we obtain $z^c_l(f)=b^c_l(f)(1+o(1))$ and thus
\[\check{z}^c_l(f)=b^c_l(f)\beta(f)(1+o(1)).\]

As we have $f^n(z_l(f))=\theta_f^{-1}(K^n_f(\check z_l(f)))=\theta_f^{-1}(\chi_{r(f)}^n\check{z}^c_l(f),A^n_f\check{z}^s_l(f),0),$ this point stays in $1/2B_{r}$ as long as
\[|\chi_{r(f)}^n\check{z}^c_l(f)|<1/2\ \text{ and thus if } |\chi_{r(f)}^nb^c_l(f)\beta(f)|<1/4,\]
for $l$ large enough. Thus, using that the strong unstable eigenvalues dominated $\chi_{r(f)}$, for such pair $(n,l),$ the graph transform $f^n(\Gamma_l(f))$ is a well-defined $u$-graph in $B_{r}.$

In this situation, observe that all this discussion above also holds when replacing $z_l(f)$ by any point $x_l(f)\in\Gamma_l(f)$ such that $y_l(f):=f^n(x_l(f))$ stays on the graph transform $f^n(\Gamma_l(f)).$ That is, $\check{x}^c_l(f)=b^c_l(f)\beta(f)(1+o(1)).$ In particular, if we fix $z\in B^u$ and $y_l(f)$ corresponds to the point of intersection of $f^n(\Gamma_l(f))$ with the plane $\{z^u=z\}$ then
\begin{align*}
y_l(f)&=(y_l^c(f),y_l^s(f),z)=\theta^{-1}_f\circ K^n_f\circ\theta_f(x_l(f))\\
&=\theta_f^{-1}\Big(\chi_{r(f)}^n\check{x}^c_l(f),A^n_f\check{x}^s_l(f),E^n_f\check{x}^u_l(f)\Big).
\end{align*}

Thus, using that $\theta_f$ is the identity on $W^u_{r(f),loc}$, we obtain $y_l^s(f)=o(\tilde\chi^{-2n})$ and
\[y_l^c(f)
=\chi_{r(f)}^n\check{x}^c_l(f)+o(\tilde\chi^{-2n})
=\chi_{r(f)}^nb^c_l(f)\beta(f)(1+o(1))+o(\tilde\chi^{-2n}),\]
where $o(1)$ is with respect to $l$ and $o(\tilde\chi^{-2n})$ with respect to $n$. The formulation~\eqref{eq-controle2} then follows from the assumption that $n_l$ is chosen so that
\[\tilde\chi^{-2n_l}=o(b_l^c(f)).\]
It remains to check that $\beta$ and $S_l$ are holomorphic. This follows, using arguments similar to those at the end of the proof of Lemma~\ref{le-control1}, from the fact that, for $n=n_l$, the intersection point $y_l(f)$ between $f^{n_l}(\Gamma_l(f))$ and the plane $\{z^u=z\}$ depends holomorphically on $f$.
\end{proof}

\subsection{Families without new stable/unstable relations}\label{sec-relation}
In this section, we focus on relations between the stable set of $\Lambda$ and the images of a piece $\Gamma$ of unstable manifold $W^u_p$ transverse to $W^s_{p,loc}$, which we call \emph{stable/unstable relations}. By themselves, such relations cannot yield maximal bifurcation in  $\mathcal{Y}$ (recall that $\mathcal{Y}\subset\mathcal X^u_D\times\mathrm{End}_d^k$ is introduced in Section~\ref{sec-main-blender}), since they do not involve the degrees of freedom coming from $\mathcal X^u_D$. The purpose of this section is to show that they nevertheless suffice to eliminate all degrees of freedom arising from $\mathrm{End}_d^k$.

Observe that if $M \subset \mathcal X^u_D \times \mathrm{End}_d^k$
is a connected family containing a point $(Y_0, f_0)$ with
$f_0 \in \Omega_u$, and if $Y_0$ intersects
$W^s_{p(f_0),loc}$ transversely, then the inclination lemma implies
that some iterate $f_0^n(Y_0)$ has a subset arbitrarily close to
$\Gamma(f_0)$.
By persistence of proper intersections, the results of this section
therefore also apply to the relations between $f^n(Y)$ and
$W^s_{\Lambda(f),loc}$. In particular, if no new such relation
appears along $M$, then $M$ must project to a single point in
$\mathscr{M}_d^k$, by Theorem \ref{th-point} below.

\bigskip

\noindent\emph{\textbf{Setting and main result of this section.}}
As no subvariety $Y\in\mathcal X_D^u$ is involved here, let $M$
be a connected analytic submanifold of $\tilde\Omega$. We assume
that $M$ contains a map $f_0$ satisfying the following
properties.
\begin{itemize}
\item The unstable manifold $W^u_{p(f_0),loc}$ contains a point
$\tilde r_0(f_0)\in W^s_{\Lambda(f_0),loc}$ such that
$f_0^{m_0}(\tilde r_0(f_0))\in W^s_{r_0(f_0),loc}$ for some
$m_0\geq0$.
\item The subgroup
$\langle\chi_{p(f_0)},\chi_{r_0(f_0)}\rangle$ is dense in $\Cb^*$.
\end{itemize}
On the other hand, by Assumption~\ref{homoclinic}, there exists a family of local
manifolds $(\Gamma(f))_{f \in M}$, of dimension $u$, contained in
$W^u_{p(f)}$ and intersecting transversely $W^s_{p(f),loc}$.
By the inclination lemma, possibly after replacing $\Gamma(f)$ by
$f^n(\Gamma(f))$ for some $n \ge 0$, we may assume that this family
consists of $u$-graphs in $B_p$.

From this, the central assumption of this section is that
there exist no $n \ge 0$, $x \in \hat\Lambda$ and sub-family $(\Gamma'(f))_{f\in M'}$ where $\Gamma'(f)$ is an open subset of $\Gamma(f)$ and $M'$ a neighborhood of $f_0$, such that
\[\{ f \in M' \,;\,
W^s_{x(f),loc} \cap f^n(\Gamma'(f)) \neq \varnothing \}\]
defines a (possibly singular) germ of a hypersurface at $f_0$.
In other words, for every such $n$ and $x$, one of the following holds:
\begin{itemize}
\item the map $f_0$ does not belong to this set, i.e. $W^s_{x(f_0),loc}\cap f_0^n(\Gamma(f_0))=\varnothing$,
\item the relation persists on a neighborhood of $f_0$ in $M'.$
\end{itemize}
Observe first that, since $W^s_{x(f),loc}$ is only a local object,
an intersection with $f^n(\Gamma(f))$ may persist on a proper
neighborhood of $f_0$ without extending to the whole parameter set.

However, in what follows we restrict to situations where
$f^n(\Gamma(f))$ is a $u$-graph in the polydisc $B_{r_i}$
associated with some saddle point $r_i(f)$,
for some $0 \le i \le N$, defined in Section \ref{sec-standard}.
In that case, such intersections correspond to
horizontal/vertical relations in the sense of
Section~\ref{sec-graph}.

Therefore, by Lemma~\ref{le-proper}, if such a relation
persists locally near $f_0$, then it must in fact persist
on the whole defining set $M'$.

The goal of this part is to show that in this situation, $M$ must reduce to a single point.
\begin{theorem}\label{th-point}
If $M$ is as above, then for each $1\leq i\leq N$ the differential $D_{f_0}\chi_{r_i}$ vanishes on $T_{f_0}M$. In particular, by \eqref{star}, $M$ reduces to a point.
\end{theorem}

\bigskip

\noindent\emph{\textbf{Outline of the proof of Theorem \ref{th-point}.}} The proof proceeds by contradiction, assuming that $M$ is not reduced to a point. We can then choose a holomorphic disc in $M$ centered at $f_0$, and, by \eqref{star}, it suffices to show that $D_{f_0}\chi_{r_i}$ vanishes in the tangent direction of this disc. In other words, we may assume that $M$ itself is a disc and prove that $\chi_{r_i}'(f_0)=0$ for all $1\leq i\leq N$. 

The first step is to show that $\frac{\log|\chi_{r_0(f)}|}{\log|\chi_{p(f)}|}$ is independent of $f$. This is equivalent to saying that there exist $\zeta\in\Sb^1$ and $\omega\in\Rb$ such that, for all $f\in M$,
\begin{equation}\label{eq-relation}
\chi_{r_0(f)}=\zeta\chi_{p(f)}^{-\omega}.
\end{equation}
This is precisely the counterpart of Lemma~3.12 and Proposition~3.13 in \cite{gtv}. The former establishes the normality of a family $(\Xi_{j_n,l_n})_{n\geq0}$, where for $j,l\geq 0$ we denote by $\Xi_{j,l}$ the holomorphic function on $M$ defined by
\[\Xi_{j,l}(f)=\chi_{p(f)}^j \chi_{r_0(f)}^l.\]
We will provide a proof of this result in the present setting (see Lemma~\ref{le-normal}). The key ingredient is the blender property -- forcing intersections between $f^n(\Gamma(f))$ and $W^s_{\Lambda(f),loc}$ -- together with the persistence of these intersections over $M$. The relation~\eqref{eq-relation} then follows exactly as in \cite[Proposition~3.13]{gtv}.

In the second step, the fact that the subgroup $\langle \chi_{p(f_0)}, \chi_{r_0(f_0)} \rangle$ is dense in $\Cb^*$, is crucial. It implies that, if we fix a preimage $\tilde r_i(f_0)\in\Lambda(f_0)$ sufficiently close to $r_0(f_0)$ of a point on $W^s_{r_i(f_0),loc}$, then we can find arbitrarily large $n$ such that some portion of $f_0^n(\Gamma(f_0))$ passes arbitrarily close to $W^s_{r_i(f_0),loc}$. This in turn implies that a family of maps of the form
\[\Big(\Xi_{j_n,l_n}(f)-t_i(f)\Big)\chi_{r_i(f)}^{m_n},\]
where $t_i$ is a non-vanishing holomorphic function, is normal. As a consequence, the derivative of $\chi_{r_i(f)}$ must vanish at $f_0$.

\bigskip

\noindent\emph{\textbf{Steps towards Theorem \ref{th-point}.}} As explained above, we assume that $M$ is a disc in $\tilde\Omega$, containing $f_0$ which satisfies $\overline{\langle \chi_{p(f_0)}, \chi_{r_0(f_0)} \rangle}=\Cb^*$. Moreover, we also assume that there exists an intersection $\tilde r_0(f_0)$ between ${W}^u_{p(f_0),loc}$ and $W^s_{\Lambda(f_0),loc}$ such that $f_0^{m_0}(\tilde r_0(f_0))\in W^s_{r_0(f_0),loc}.$ As all these intersersections persist along $M,$ there is a holomorphic map $\tilde r_0\colon M\to\Pb^k$ such that $\tilde r_0(f)$ is in $W^u_{p(f),loc}\cap W^s_{\Lambda(f),loc}$ and $f^{m_0}(\tilde r_0(f))\in W^s_{r_0(f),loc}.$ Moreover, since $f$ contracts to cone fields $C^{uu}_\eta$ on $\mathcal U_+\cup\mathcal U_-,$ possibly by increasing $m_0$ and shrinking $M,$ we can assume that $f^{m_0}(W^u_{p(f),loc})$ is arbitrarily close to $W^{uu}_{r_0(f),loc}$ and thus, is a $u$-graph in $B_{r_0}.$

On the other hand, there exist $n_0\geq m_0$ and a $cs$-graph $\Delta_{f_0}$ in $B_p$ such that $f_0^{n_0}(\Delta_{f_0})$ is an open subset of $\mathsf W_{r_0(f_0)}^{cs}$. Possibly after shrinking again $M$, it extends to a family of $cs$-graphs $(\Delta_f)_{f\in M}$. Roughly speaking, using this $n_0$, we have that
\begin{itemize}
\item for $j\geq0$ large enough, the graph transform $f^j(\Gamma(f))$ in $B_p$ is close to $W^u_{p(f),loc}$;
\item thus, some part of $f^{j+n_0}(\Gamma(f))$ is close to a $u$-graph in $B_{r_0}$ intersecting $W^s_{r_0(f),loc}$;
\item applying an appropriated large iterate $f^l$ gives a $u$-graph (in $f^{j+n_0+l}(\Gamma(f))$) close to a leaf of the strong unstable foliation of $W^{u}_{r_0(f),loc},$ whose first coordinates is essentially
\[\Xi_{j,l}(f)=\chi_{p(f)}^j \chi_{r_0(f)}^l.\]
\end{itemize}
The following result makes this final claim more precise, using Lemma \ref{le-control1} and Lemma~\ref{le-control2}.
\begin{lemma}\label{le-doublecontrole}
There exists a holomorphic function $\gamma\colon M\to\Cb^*$  such that if $(j_n)$ and $(l_n)$ are sequences then for all $f\in M$ such that
\begin{equation}\label{eq-hypdoublecontrole}
\tilde\chi^{-2l_n}=o(\chi_{p(f)}^{j_n})\ \text{ and }\ |\Xi_{j_n,l_n}(f)\gamma(f)|\leq 1/4,
\end{equation}
the set $f^{j_n+n_0+l_n}(\Gamma(f))$ contains, for all $n$ large enough, a $u$-graph $\Gamma_{j_n,l_n}(f)$ in $B_{r_0}$ of the form
\begin{equation}\label{eq-doublecontrole}
\Gamma_{j_n,l_n}(f)=\left\{\left(\Xi_{j_n,l_n}(f)\gamma(f)\Big[(1,0)+S_{n}(f,z)\Big],z\right)\,;\,z\in B^u\right\},
\end{equation}
where the remainders $S_{n}(f,z)$ converges to $0$ locally uniformly.
\end{lemma}
In other words, the $u$-graph $\Gamma_{j_n,l_n}(f)$ is arbitrarily close to the vertical space 
\[\{z^c=\Xi_{j_n,l_n}(f)\gamma(f),z^s=0\}\subset W^u_{r_0(f),loc}.\]

\begin{proof}
Let $h_f\colon B^{cs}\to B_p$ be the family of injective holomorphic maps defining the $cs$-graphs $(\Delta_f)_{f\in M}$ such that $D_0h_f(1,0)$ belongs to $E^{cu}$ defined along $W^u_{p(f),loc}$.

By Lemma~\ref{le-control1}, there exists a holomorphic function $\alpha\colon M\to\Cb^*$ such that, for every $j\geq0$, the graph transform $f^j(\Gamma_0(f))$ intersects $\Delta_f$ at the point $h_f(a_j(f))$, where
\[a_j(f)=\chi_{p(f)}^j\Big(b^c(f)\alpha(f)(1,0)+R_j(f)\Big),\]
with $R_j(f)$ converging locally uniformly to zero.

By definition of $\Delta_f$, we have $f^{n_0}(h_f(0))\in W^s_{r_0(f),loc}$. Moreover, since $D_0h_f(1,0)\in E_{h_f(0)}^{cu}$ does not belong to the strong stable cone field, which is expanding, it follows that $D_0(f^{n_0}\circ h_f)(1,0)$ is not tangent to $W^s_{r_0(f),loc}$. In other words, in the coordinates of $B_{r_0}$, we have
\begin{equation}\label{eq-v}
D_0(f^{n_0}\circ h_f)(1,0)=(v^c(f),v^s(f),0),
\end{equation}
with $v^c(f)\neq 0$.

Hence, $f^{j+n_0}(\Gamma_0(f))$ contains a vertical graph $\Gamma_j(f)$ in $B_{r_0}$ intersecting $\mathsf W_{r_0(f)}^{cs}$ at $(b^c_j(f),b^s_j(f))$, where
\begin{equation}\label{eq-control-b}
b^c_j(f)=\chi_{p(f)}^j\big(b^c(f)\alpha(f)v^c(f)+\tilde R_j(f)\big),
\end{equation}
and where $\tilde R_j(f)$ converges locally uniformly to zero.

Therefore, we can apply Lemma~\ref{le-control2}, which in particular provides a function $\beta\colon M\to\Cb^*$ such that, if $l(j)$ is chosen so that
\begin{equation}\label{eq-control3}
\tilde\chi^{-2l(j)}=o(\chi_{p(f)}^{j})\ \text{ and }\ \big|\chi_{r_0(f)}^{l(j)}b^c_j(f)\beta(f)\big|<1/4,
\end{equation}
then the graph transform $f^{l(j)}(\Gamma_j(f))$ in $B_{r_0}$ is a subset of the form
\begin{equation*}
\left\{\left(\chi_{r_0(f)}^{l(j)}b^c_j(f)\Big[\beta(f)(1,0)+S_{l(j)}(f,z)\Big],z\right)\,;\,z\in B^u\right\}.
\end{equation*}
This applies in particular to the two sequences, $(j_n)$ and $(l_n),$ in the statement since \eqref{eq-control3} is satisfied by \eqref{eq-hypdoublecontrole}. Hence, combining \eqref{eq-control-b} and \eqref{eq-control3}, we obtain~\eqref{eq-doublecontrole} by setting
\[\gamma(f):=\alpha(f)\beta(f)v^c(f)b^c(f),\]
and using the fact that
\[\Xi_{j_n,l_n}(f)\tilde R_{j_n}(f)\]
converges locally uniformly to $0$ as a consequence of~\eqref{eq-control-b}.
\end{proof}

\begin{lemma}\label{le-normal}
Let $(j_n)_{n\geq0}$ and $(l_n)_{n\geq0}$ be two increasing sequences such that $\Xi_{j_n,l_n}(f_0)$ converges in $\Db^*$. Then the family $\{f\mapsto\Xi_{j_n,l_n}(f)\}_{n\geq 0}$ is normal in a neighborhood $M_0\subset M$ of $f_0$.
\end{lemma}

\begin{proof}
Let $\gamma$ be the holomorphic function given by Lemma~\ref{le-doublecontrole}. By multiplying by a suitable power of $\chi_{p(f_0)}$, we may assume that $\Xi_{j_n,l_n}(f_0)$ converges to some $\Xi_\infty$ satisfying
\[0<|\Xi_\infty\gamma(f_0)|<1/4.\]
On the other hand, since $\Xi_\infty\neq0$ we have that for all $f\in M,$ $\tilde\chi^{-2l_n}=o(\chi_{p(f)}^{j_n}).$

By Lemma~\ref{le-doublecontrole}, for all sufficiently large $n$, the iterate $f_0^{j_n+n_0+l_n}(\Gamma(f_0))$ contains a $u$-graph $\Gamma_{j_n,l_n}(f_0)$ in $B_{r_0}$ arbitrarily close to the vertical space
\[\{z^c=\Xi_\infty\gamma(f_0),\,z^s=0\}\subset W^u_{r_0(f_0),loc}.\]
Moreover, by the blender property given by Assumption~\ref{blender}, the graph $\Gamma_{j_n,l_n}(f_0)$ intersects $W^s_{x_n(f_0),loc}$ for some point $x_n(f_0)\in\hat\Lambda(f_0)$. We denote by $z_n(f_0)$ this intersection point, and by $y_n(f_0)$ the intersection between $W^u_{r_0(f_0),loc}$ and $W^s_{x_n(f_0),loc}$. Since $\hat\Lambda(f_0)$ is a basic set, the point $y_n(f_0)$ corresponds to a point $\hat y_n(f_0)$ in $\hat\Lambda(f_0)$ which can be followed holomorphically in $\hat\Lambda(f).$ In particular, the same holds for  $y_n(f_0)$ which can be followed in $\Lambda(f)$, and the family $\{f\mapsto y_n(f)\}_{n\geq0}$ is normal.

On the other hand, the distance between $z_n(f_0)$ and $y_n(f_0)$ tends to zero as $n\to\infty$, so the first coordinate of $y_n(f_0)$ converges to $\Xi_\infty\gamma(f_0)$. Hence, on a small neighborhood $M_0$ of $f_0$, the first coordinate of $y_n(f)$ belongs to a fixed compact set $D\Subset\Db^*_{1/4}$ for all sufficiently large $n$ and all $f\in M_0$.

We now conclude with three further observations. First, since the intersection between $W^s_{\Lambda(f),loc}$ and $W^u_{p(f)}$ persists along $M$, the point $z_n(f_0)$ can be followed holomorphically as the intersection point $\Gamma_{j_n,l_n}(f)\cap W^s_{x_n(f),loc}$, as long as $\Gamma_{j_n,l_n}(f)$ is well defined. Moreover, since $\Gamma_{j_n,l_n}(f)$ converges to a subset of $W^u_{r_0(f),loc}$, we have $z_n(f)\to y_n(f)$ locally uniformly.

Finally, assume by contradiction that the family $\{f\mapsto\Xi_{j_n,l_n}(f)\}_{n\geq 0}$ is not normal on $M_0$. Then the same holds for $\{f\mapsto\Xi_{j_n,l_n}(f)\gamma(f)\}_{n\geq 0}$, which implies that for arbitrarily large $n$ there exists $f_n\in M_0$ such that
\[\Xi_{j_n,l_n}(f_n)\gamma(f_n)\in\Db_{1/4}\setminus D',\]
where $D'\Subset\Db_{1/4}$ is an open disc containing $D$ in its interior. This is impossible, since $\Xi_{j_n,l_n}(f_n)\gamma(f_n)$ is arbitrarily close to the first coordinate of $z_n(f_n)$, whose distance to $y_n(f_n)$ tends to zero.
\end{proof}

It follows, exactly as in \cite[Proposition~3.13]{gtv}, that $\chi_p$ and $\chi_{r_0}$ satisfy a fixed relation throughout $M$.
\begin{proposition}\label{prop-relation}
 There exists $\zeta\in\Sb^1$ and $\omega\in\Rb^*_+$ such that for all $f\in M$
\[\chi_{r_0(f)}=\zeta\chi_{p(f)}^{-\omega}.\]
\end{proposition}
At this stage, the strategy in \cite{gtv} was to prove that both $\chi_p$ and $\chi_{r_0}$ are constant on $M$. The argument then concludes through a rather involved step, relying on stability in families in the sense of Berteloot-Bianchi-Dupont and on the equilibrium lamination constructed in \cite{BBD}, to show that all maps $f\in M$ are holomorphically conjugate in a neighborhood of their small Julia sets.

In the present setting, however, we cannot rely on stability in the sense of \cite{BBD} for the family $(f)_{f\in M}$, nor on equilibrium laminations, since the set $\Lambda(f)$ is far from the small Julia set. Fortunately, the results of \cite{t-gorbovickis} considerably simplify the final part of the argument, as it suffices to prove that the derivatives of the multipliers $\chi_{r_i}$ vanish at $f_0$ for all $1\leq i\leq N$.

To this end, we fix $1\leq i\leq N$ and we use Lemma \ref{le-shadowing} with $x=r_0(f_0)$ and $y=r_i(f_0)$ in order to obtain $n_i\geq0$ and $\tilde r_i(f_0)$ in $B_{r_0}\cap W^u_{r_0(f_0),loc},$ but outside $W^{uu}_{r_0(f_0),loc},$ such that $f_0^{n_i}(\tilde r_i(f_0))\in W^s_{r_i,loc}.$ Moreover, possibly after shrinking $M$, we require that for all $f\in M$ the point $\tilde r_i(f)$ is in $B_{r_0},$ and we denote its coordinates in this polydisc by  
\[(R_i^c(f),0,R^u_i(f))\in B_{r_0}.\]
The important data here is its first coordinate $R^c_i(f),$ which is not zero since $\tilde r_i(f)\notin W^{uu}_{r_0(f),loc},$ while the other unstable coordinates $R^u_i(f)$ will not be use.

Since the subgroup $\langle \chi_{p(f_0)}, \chi_{r_0(f_0)} \rangle$ is dense on $\Cb^*$ there exist two increasing sequence $(j_n)$ and $(l_n)$ such that $\Xi_{j_n,l_n}(f_0)$ converges to $t_i(f_0)$ with
$$t_i(f):=\frac{R_i^c(f)}{\gamma(f)},$$
where $\gamma$ is defined in Lemma \ref{le-doublecontrole}.

This has two consequences. The first one is that, if we write, as in Proposition \ref{prop-relation}, $\chi_{r_0(f)}=\zeta\chi_{p(f)}^{-\omega}$ then
\[\Xi_{j_n,l_n}(f_0)=\chi_{p(f_0)}^{j_n-\omega l_n}\zeta^{l_n},\]
thus, the convergence of  $\Xi_{j_n,l_n}(f_0)$  implies
\begin{equation}\label{eq-omegainfty}
j_n-\omega l_n \xrightarrow[n\to\infty]{} \omega_\infty := \frac{\log|t_i(f_0)|}{\log|\chi_{p(f_0)}|}\ \text{ and }\ \zeta^{l_n}\xrightarrow[n\to\infty]{}\frac{t_i(f_0)}{\chi_{p(f_0)}^{\omega_\infty}}.
\end{equation}
Hence, we obtain the convergence on $M$
\begin{equation}\label{eq-Xi}
\Xi_{j_n,l_n}(f)\xrightarrow[n\to\infty]{}\Xi_\infty(f)
:=\frac{\chi_{p(f)}^{\omega_\infty}}{\chi_{p(f_0)}^{\omega_\infty}}t_i(f_0).
\end{equation}
Since no new relation between $W^u_{p(f)}$ and $W^s_{\Lambda(f),loc}$ arises along $M$, this imposes the following constraint on $t_i(f)$.
\begin{lemma}\label{le-Xi}
For all $f\in M$, we have $t_i(f)=\Xi_\infty(f).$ That is, the point $\tilde r_i(f)$ belongs to 
the vertical space
\[\{z^c=\Xi_\infty(f)\gamma(f),\,z^s=0\}.\]
Moreover, possibly by changing the sequence $(j_n,l_n),$ we can assume that $\Xi_{j_n,l_n}(f)\neq t_i(f),$ for all $n$ and all $f\in M.$
\end{lemma}
\begin{proof}
By analytic continuation, it is sufficient to prove the result in a neighborhood of $f_0$. A key observation is that, since the assumption~\eqref{eq-hypdoublecontrole} of Lemma~\ref{le-doublecontrole} is satisfied for $(\Gamma_{j_n,l_n}(f_0))$ for all sufficiently large $n$, and since the family $(\Xi_{j_n,l_n}(f))$ converges, the same assumption~\eqref{eq-hypdoublecontrole} also holds in a neighborhood of $f_0$ for all large $n$. Hence, Lemma~\ref{le-doublecontrole} implies that, for such $f$, the graphs $\Gamma_{j_n,l_n}(f)$ converge to the vertical space
\[\{z^c=\Xi_\infty(f)\gamma(f),\,z^s=0\},\]
which intersects $W^s_{\tilde r_i(f),loc}$, for $f=f_0$, exactly at the point $\tilde r_i(f_0)$. Since these intersections persist along $M$, the first coordinate $R_i^c(f)$ of $\tilde r_i(f)$ in $B_{r_0}$ must equal $\Xi_\infty(f)\gamma(f)$, that is,
\[t_i(f)=\Xi_\infty(f).\]
For the last point in the statement, as $\overline{\langle \chi_{p(f_0)}, \chi_{r_0(f_0)} \rangle}=\Cb^*,$ we can assume that $(\Xi_{j_n,l_n}(f_0))$ converges toward $t_i(f_0)$ while
\[|\Xi_{j_n,l_n}(f_0)|\neq |t_i(f_0)|,\ \text{ i.e., }\ j_n-\omega l_n\neq\omega_\infty,\]
by \eqref{eq-omegainfty}. This implies that, for all $f\in M,$ we have
\[|t_i(f)|=|\Xi_\infty(f)|=|\chi_{p(f)}^{\omega_\infty}|\neq|\chi_{p(f)}^{j_n-\omega l_n}|=|\Xi_{j_n,l_n}(f)|.\]
\end{proof}

This allows us to apply a second time Lemma \ref{le-control2}, from $r_0$ to $r_i.$
\begin{lemma}\label{le-triplecontrole}
There exists a holomorphic function $\delta\colon M\to\Cb^*$  such that, if $(m_n)$ is a sequence then, for all $f\in M$ such that
\begin{equation}\label{eq-hyptriplecontrole}
\tilde\chi^{-2m_n}=o(\Xi_{j_n,l_n}(f)-t_i(f))\ \text{ and }\ \left|\delta(f)\big(\Xi_{j_n,l_n}(f)-t_i(f)\big)\chi_{r_i(f)}^{m_n}\right|\leq 1/4,
\end{equation}
the set $f^{j_n+n_0+l_n+n_i+m_n}(\Gamma(f))$ contains a $u$-graph $\tilde\Gamma_{n}(f)$ in $B_{r_i}$ of the form
\begin{equation}\label{eq-triplecontrole}
\tilde\Gamma_{n}(f)=\left\{\left(\delta(f)\big(\Xi_{j_n,l_n}(f)-t_i(f)\big)\chi_{r_i(f)}^{m_n}\Big[(1,0)+\tilde S_{n}(f,z)\Big],z\right)\,;\,z\in B^u\right\},
\end{equation}
where the remainders $\tilde S_{n}(f,z)$ converges to $0$ locally uniformly.
\end{lemma}
\begin{proof}
The proof is very similar to that of Lemma~\ref{le-doublecontrole}. More precisely, we apply Lemma~\ref{le-control2} to the graphs
\[\Gamma_n(f):=f^{n_i}\big(\Gamma_{j_n,l_n}(f)\big).\]
Observe that the sequence $(\Gamma_n(f))$ converges to the image $\Gamma_\infty(f)$, under $f^{n_i}$, of the vertical space $\{z^c=\Xi_\infty(f)\gamma(f),\,z^s=0\}$ in $B_{r_0}$. By Lemma~\ref{le-Xi}, this limit graph satisfies
\[\Gamma_\infty(f)\cap W^s_{r_i(f),loc}\neq\varnothing\]
for all $f\in M.$ Moreover, for each $n$, the graph $\Gamma_n(f)$ intersects $\mathsf W^{cs}_{r_i(f)}$ at $(b_n^c(f),b_n^s(f))$, where
\[b_n^c(f)=\tilde v^c(f)\gamma(f)\big(\Xi_{j_n,l_n}(f)-t_i(f)\big)+o\big(\Xi_{j_n,l_n}(f)-t_i(f)\big).\]
Here, the coefficient $\tilde v^c(f)$ arises from the transfer from $B_{r_0}$ to $B_{r_i}$ by $f^{n_i}$ and plays the same role as $v^c(f)$ in~\eqref{eq-v}. More precisely, the set $f^{-n_i}(\mathsf W^{cs}_{r_i(f)})$ is locally a manifold passing through $\tilde r_i(f)$, which projects injectively onto a neighborhood $V$ in $\mathsf W^{cs}_{r_0(f)}$ of $(R_i^c(f),0)$. The quantity $\tilde v^c(f)$ is then defined as the first coordinate of
\[D_{(R_i^c(f),0)}\phi(1,0),\]
where $\phi\colon V\to\mathsf W^{cs}_{r_i(f)}$ denotes the composition of the inverse of this projection with $f^{n_i}$.

Applying Lemma~\ref{le-control2} now yields the existence of a function $\beta$, and the desired conclusion follows from this lemma,  by setting
\[\delta(f):=\beta(f)\tilde v^c(f)\gamma(f).\]
\end{proof}
We can now conclude the proof of Theorem \ref{th-point}.
\begin{proof}[Proof of Theorem \ref{th-point}]
By~\eqref{star}, there exists $1\leq i\leq N$ such that $\chi_{r_i(f_0)}'\neq 0$. The key ingredient to contradict this fact is to prove that the family
\[\left((\Xi_{j_n,l_n}(f)-t_i(f))\chi_{r_i(f)}^{m_n}\right)_{n\geq0},\]
is normal for a suitable sequence $(m_n)$. This argument is very similar to that of Lemma~\ref{le-normal}.

More precisely, the graphs $\Gamma_{j_n,l_n}(f)$ converge to the vertical space
\[\{z^c=\Xi_\infty(f)\gamma(f),\,z^s=0\},\]
which intersects $W^s_{\tilde r_i(f),loc}$. Thus, since by Lemma \ref{le-Xi} $\Xi_{j_n,l_n}(f_0)\neq t_i(f_0),$ there exists a sequence $(m_n)_{n\geq0}$ tending to $+\infty$ such that
\begin{equation}\label{eq-controletriple}
\left(\delta(f_0)(\Xi_{j_n,l_n}(f_0)-t_i(f_0))\chi_{r_i(f_0)}^{m_n}\right)_{n\geq0}
\end{equation}
converges in $\Db^*_{1/4}$.

Therefore, by Lemma~\ref{le-triplecontrole}, the iterate $f_0^{j_n+n_0+l_n+n_i+m_n}(\Gamma(f_0))$ contains a $u$-graph $\tilde\Gamma_n(f_0)$ in $B_{r_i}$ that is arbitrarily close to a vertical space in $W^u_{r_i(f_0),loc}$, whose first coordinate is exactly given by \eqref{eq-controletriple}. By the blender property in Assumption~\ref{blender}, the graph $\tilde\Gamma_n(f_0)$ must intersect a local stable manifold $W^s_{x_n(f_0),loc}$. Hence, exactly as in Lemma~\ref{le-normal}, the persistence of the intersections between $W^u_{p(f)}$ and $W^s_{\Lambda(f),loc}$ along $M$ forces the family in~\eqref{eq-controletriple} to be normal.

Since $\delta$ does not vanish on $M$, this implies the normality of the family
\[\left((\Xi_{j_n,l_n}(f)-t_i(f))\chi_{r_i(f)}^{m_n}\right)_{n\geq0},\]
as desired.

To conclude the proof, we compute the derivative of these functions. Recall that, by Proposition~\ref{prop-relation} and Lemma~\ref{le-Xi},
\[\Xi_{j_n,l_n}(f)=\chi_{p(f)}^{j_n-\omega l_n}\zeta^{l_n}
\quad\text{and}\quad
t_i(f)=\Xi_\infty(f)
:=\frac{\chi_{p(f)}^{\omega_\infty}}{\chi_{p(f_0)}^{\omega_\infty}}\,t_i(f_0),\]
where $j_n-\omega l_n\to\omega_\infty$ and $\zeta^{l_n}\to t_i(f_0)/\chi_{p(f_0)}^{\omega_\infty}$ (see~\eqref{eq-omegainfty} and~\eqref{eq-Xi}). Dividing by $t_i(f)$, we obtain the normality of the family
\[\left(e^{(j_n-\omega l_n-\omega_\infty)P(f)+i\theta_n}-1\right)\chi_{r_i(f)}^{m_n},\]
where $P(f):=\mathrm{Log}(\chi_{p(f)})$ and
\[e^{i\theta_n}:=\zeta^{l_n}\chi_{p(f_0)}^{\omega_\infty}/t_i(f_0),\]
which converges to $1$.

Using the Taylor expansion of the exponential function and setting $a_n:=j_n-\omega l_n-\omega_\infty$, there exist constants $0<A<B$ such that, for all $f$ in a sufficiently small neighborhood $V$ of $f_0$ and all $n$,
\begin{equation}\label{eq-borne}
A<\left|(a_nP(f)+i\theta_n)\chi_{r_i(f)}^{m_n}\right|<B.
\end{equation}
Since $P(f_0)$ is not purely imaginary, taking first the real part at $f=f_0$ and then the imaginary part yields
\[a_n=O(\chi_{r_i(f_0)}^{-m_n})
\quad\text{and}\quad
\theta_n=O(\chi_{r_i(f_0)}^{-m_n}).\]
On the other hand, if $\chi_{r_i}$ is not constant on $M$, there exists $f_1\in V$ such that
$|\chi_{r_i(f_1)}|<|\chi_{r_i(f_0)}|$, and therefore
\[\left|(a_nP(f_1)+i\theta_n)\chi_{r_i(f_1)}^{m_n}\right|
=O\left(\frac{|\chi_{r_i(f_1)}|^{m_n}}{|\chi_{r_i(f_0)}|^{m_n}}\right)\]
tends to $0$, contradicting the lower bound in~\eqref{eq-borne}. Hence, $\chi_{r_i}$ must be constant on $M$, which contradicts our initial assumption that $\chi_{r_i(f_0)}'\neq0$.
\end{proof}

\subsection{Relations from subvariety variations}\label{sec-subvar}
So far, we have only used variations coming from the space $\mathrm{End}_d^k$ (or the moduli space $\mathscr{M}_d^k$), working with a fixed family of submanifolds $\Gamma$ contained in the (global) unstable manifold $W^u_p$. The inclination lemma ensures that the results obtained for $\Gamma$ also apply to any family admitting a transverse intersection with the local stable manifold $W^s_{p,loc}$.

We now turn to a different situation. Our goal is to produce new relations in the case where the families we consider, $M\subset \mathcal X_D^u\times\mathrm{End}_d^k$,
are of the form $\mathcal Z\times\{f\}$. In that setting, $W^u_{p(f)}$ is constant in the family. Thus, the relations we seek cannot arise from it, but must instead result from genuine variations of the submanifold $Y\in\mathcal Z$ itself.

If $Y$ has a transverse intersection with $W^s_{p(f),loc},$ by the inclination lemma, we already know that the iterates $f^n(Y)$ have many intersections with $W^s_{\Lambda(f),loc}$. However, this does not guarantee that these intersections are not persistent when $Y$ varies inside the family $\mathcal Z$. In fact, such a degeneracy may occur when $f$ is a product map, obtained as the product of a two-dimensional example with a saddle fixed point. In this situation:
\begin{itemize}
\item the intersection points $Y\cap f^{-n}(W^s_{\Lambda(f),loc})$ are not Zariski dense in $Y$, thus $Y$ may vary while preserving exactly the same intersection points;
\item the image $f^n(Y)$ may ``slide'' along a local stable manifold $W^s_{x(f),loc}$ of $\Lambda(f)$, and therefore does not produce any new relations.
\end{itemize}
Assumptions~\ref{decalage} and~\ref{grille} are precisely designed to break these product features, in the unstable and in the stable directions respectively. Moreover, the homoclinic intersection provided by Assumption~\ref{homoclinic} plays a crucial role in ruling out all possible singularities arising in these problems.

From now on, in order to simplify notation, we denote by $\mathcal Z\subset\mathcal{X}_D^u$ the irreducible family we consider, which we view as lying over a fixed map $f\in\tilde\Omega$. Observe that, by our choice of $\mathcal Y_0$ after Theorem~\ref{th-main-blender}, every $Y\in\mathcal Z$ is reduced and irreducible.

Most objects previously defined as depending on $f$ (such as $\Lambda$, $p$, etc.), will henceforth be written without reference to $f$. In addition to the dynamical assumptions satisfied by $f$, we assume throughout this section that
\begin{equation}\tag{$\dag$}\label{eq-dag}
\text{every } Y\in\mathcal Z \text{ has a transverse intersection with } W^s_{p,loc}.
\end{equation}
Observe that this is an open condition. Thus, after possibly shrinking $\mathcal Z$, it holds for all
$Y \in \mathcal Z$ as soon as it holds for one of them.
Hence, in a sense, condition~\eqref{eq-dag} can be viewed as a reformulation
of assumption~\eqref{eq-pitch} in Theorem~\ref{th-main-blender}
within this more restricted setting.

Another remark is that, under~\eqref{eq-dag}, there are many intersections between $f^n(Y)$ and $W^s_{\Lambda,loc}$, which do not necessarily all occur in regions where $f^n(Y)$ is locally given by a vertical graph. For this reason, we restrict our attention to small connected open subsets $U\subset\Pb^k$ and $\mathcal Z'\subset\mathcal Z$ such that, for every $Y\in\mathcal Z'$, the set
\[U_Y:=U\cap Y\]
is nonempty, smooth, and connected. We then consider the problem restricted to local families of the form $(U_Y)_{Y\in\mathcal Z'}.$

\begin{theorem}\label{th-subvar}
Assume that $\mathcal Z$ satisfies \eqref{eq-dag}. If $\mathcal Z$ is not reduced to a point, then there exist $n\geq0,$ a local family $(U_Y)_{Y\in\mathcal Z'}$ and $x\in\Lambda$ such that $f^n(U_Y)$ is a $u$-graph tangent to $C^{uu}_\eta$ in $\mathcal U_+$ for all $Y\in\mathcal Z'$ and
\[\{\,Y\in\mathcal Z' \; ; \; f^n(U_Y)\cap W^s_{x,loc}\neq\varnothing\,\}\]
is a hypersurface of $\mathcal Z'$. 
\end{theorem}
Note that we restrict ourselves to $\mathcal U_+$,
as Assumption~\ref{grille} only applies in this region.

In order to exclude product-type behavior in the unstable directions, we use Assumption~\ref{decalage} to show that the set of points in $Y$ send to transverse intersections between $f^n(Y)$ and $W^s_{p,loc}$ for some $n\geq0$ is Zariski dense in $Y$.
\begin{lemma}\label{le-zariski}
For all $Y\in\mathcal Z,$ the set of points in $Y$ defined by
\[P:=\left\{w\in Y_{reg}\ ;\ \begin{array}{c}\exists n\geq0,\ f^n(w)\in W^s_{p,loc},\ D_wf^n\ \text{is invertible, and }\\ T_{f^n(w)}W^s_{p,loc}\cap D_wf^n(T_w Y)=\{0\}\end{array}\right\}\]
is Zariski dense in $Y.$
\end{lemma}
\begin{proof}
Fix $Y\in\mathcal Z.$ Let $w_0$ denote a transverse intersection point between $Y$ and $W^s_{p,loc}$ given by~\eqref{eq-dag}. Let $H$ be the Zariski closure of $P$ in $Y$, and assume by contradiction that $H\neq Y$. 

A first observation is that there exists a point $w\in P$ belonging to the regular part of $H$. Therefore, by replacing $(Y,w_0)$ with $(f^n(Y),f^n(w))$ for a suitable $n\geq0$, we may assume that $H$ is smooth at $w_0$.

We now choose coordinates around $p$ such that $p$ is in standard from in $B_p.$ Moreover, by the inclination lemma, for all sufficiently large $n\geq0$, the graph transforms $f^n(Y)$ in $B_p$ are arbitrarily close to $W^u_{p,loc}$. By Assumption~\ref{saddle} and Assumption~\ref{blender}, there exist sequences $(x_n)\subset\Lambda\cap\mathcal U_+$ and $(y_n)\subset B_p$ such that
\[y_n\in f^n(Y)\cap W^s_{x_n,loc},\]
and such that, up to a subsequence, $(y_n)$ converges to a point $y_\infty \in W^u_{p,loc}$ and $(x_n)$ converges to a point $x_\infty \in \Lambda \cap \mathcal U_+$. By continuity of the local stable manifolds, $y_\infty$ also belongs to $W^s_{x_\infty,loc}$. Notice that we do not relabel the subsequence, as this plays no role in what follows.

A key point for what follows is that, by Assumption~\ref{decalage}, none of the coordinates of $y_\infty$ in $W^u_{p,loc}$ vanishes. Finally, Assumption~\ref{homoclinic} combined with the inclination lemma ensures that each $y_n$ is accumulated by points giving transverse intersections between $f^{n+m}(Y)$ and $W^s_{p,loc}.$ In particular, $w_n:=f^{-n}(y_n)$ is in $H.$

In what follows, we consider the image $(\check y_n)$ and $(\check w_n)$ of these sequences under the linearization map $\theta_f$, in order to derive a contradiction.

Observe that the standard form of $p$ allows us, for each $1\leq i\leq u$, to consider the trace of the $i$-th dynamical direction in $T_{w_0}Y$, namely
\[F_i:=T_{w_0}Y\cap\left(T_{w_0}W^s_{p,loc}\oplus\mathcal E^{u,i}_{w_0}\right).\]
Their direct sum $F_1\oplus\cdots\oplus F_u$ equals the full tangent space $T_{w_0}Y$. Therefore, since $H\neq Y$ and since $H$ is smooth at $w_0,$ there exists at least one index $i_0$ such that
\begin{equation}\label{eq-trans}
F_{i_0}\cap T_{w_0}H=\{0\}.
\end{equation}
This allows us to obtain precise local information on the image $\check H$ of $H$ under the linearization $\theta_f$ at $p$. More precisely, since $Y$ is transverse to $W^s_{p,loc}$, there exists a $C^2$ germ $g\colon(\Cb^u,0)\to B^s$ such that locally
\[\check Y:=\theta_f(Y)=\{(g(z^u),z^u)\}.\]
Thus, the transversality condition~\eqref{eq-trans}, together with the fact
that $D_{w_0}\theta_f(\mathcal E_{w_0}^{u,i_0})$ coincides with the
$i_0$-th canonical unstable direction, implies the existence of a
$C^2$ map $h \colon (\Cb^{u-1},0) \to \Cb$ such that, locally near $\check w_0:=\theta_f(w_0)$,
\[\check H\subset
\bigl\{(g(z^u),z^u)\ ;\ z^u_{i_0}=h(z^u_{I_0})\bigr\},\]
where $z^u=(z^u_1,\ldots,z^u_u)$ and $z^u_{I_0}=(z^u_1,\ldots,z^u_{i_0-1},z^u_{i_0+1},\ldots,z^u_u).$

Returning to the sequence $(\check y_n)$ defined above, which converges to a point $y_\infty\in W^u_{p,loc}$ whose coordinates in the unstable manifold are all nonvanishing, we can choose constants $0<a<b$ and an integer $N>0$ such that, for all $n\geq N$,
\begin{equation}\label{eq-y}
a<|\check y^u_{j,n}|<b,
\end{equation}

As observed above, the point $\check w_n$, defined by $\check w_n:=\theta_f\bigl(f^{-n}(y_n)\bigr),$ and hence equal to
\[\check w_n=K_f^{-n}(\check y_n),\]
belongs to $\check H$. Consequently, using that $h$ is of class $C^2$ and satisfies $h(0)=0,$ we must have
\[\check w^u_{i_0,n}=h\bigl(\check w^u_{I_0,n}\bigr)=D_0h(\check w^u_{I_0,n})+O(\|\check w^u_{I_0,n}\|^2).\]
Therefore, since $\check w^u_{i,n}=\rho_i^{-n}\check y^u_{i,n},$ where $\rho_1,\ldots,\rho_u$ are the unstable multipliers of $p$, this gives
$\rho_{i_0}^{-n}\check y^u_{i_0,n}=\sum_{j\neq i_0}c_j\rho_{j}^{-n}\check y^u_{j,n}+O(\tilde\chi^{-6n}),$ and thus
\[\check y^u_{i_0,n}
=\sum_{j\neq i_0}c_j(\rho_{i_0}\rho_{j}^{-1})^n\check y^u_{j,n}+O(\tilde\chi^{-2n}),\]
by the bunching condition in Assumption \ref{saddle}. On the other hand, by \eqref{star}, the moduli $|\rho_{i_0}\rho_j^{-1}|$ are pairwise distinct and different from $1$. Using~\eqref{eq-y} and the bounds $|\check y^u_{i_0,n}|<b$ and $|\check y^u_{j,n}|>a$, we obtain that $c_j=0$ whenever $|\rho_{i_0}\rho_j^{-1}|>1$. Hence, $\check y^u_{i_0,n}$ tends to $0$, which contradicts the lower bound in~\eqref{eq-y}.
\end{proof}
Using the inclination lemma, each point in $P$ gives rise, for sufficiently large $n$, to a $u$-graph contained in $f^n(Y)$ and intersecting $W^s_{x,loc}$ for some $x\in\Lambda$. By the previous result, these intersection points cannot all remain constant as $Y$ varies.

However, when the dimension of $W^s_{x,loc}$ is positive, that is, when $u<k-1$, the intersection $f^n(Y)\cap W^s_{x,loc}$ may persist, moving inside $W^s_{x,loc}$. As already mentioned, such a situation can occur when the stable direction simply comes from a product.

The following two results will allow us to rule out this possibility.

\begin{proposition}\label{prop-subvar}
Fix $Y_0\in\mathcal Z$ and assume that Theorem~\ref{th-subvar} does not hold. That is, $\mathcal Z$ is not reduced to a single point and, for every $n\geq0$, every $x\in\Lambda$, and every local family $(U_Y)_{Y\in\mathcal Z'}$ such that $f^n(U_Y)$ is a $u$-graph in $\mathcal U_+$ tangent to $C^{uu}_\eta$, the set
\[\{\,Y\in\mathcal Z' \; ; \; f^n(U_Y)\cap W^s_{x,loc}\neq\varnothing\,\}\]
is either empty or equal to $\mathcal Z'$.

Then there exist $n_0\geq0$, an open neighborhood $\mathcal Z_0$ of $Y_0$, and a family $(\Gamma_Y)_{Y\in\mathcal Z_0}$ of $u$-graphs in $B_p$ such that $\Gamma_Y\subset f^{n_0}(Y)$ for all $Y\in\mathcal Z_0$, and which intersects $W^s_{p,loc}$ in
\[(b^c,b^s(Y)),\]
where $b^s(Y)$ is not constant on $\mathcal Z_0$.
\end{proposition}

Observe that the key feature is that $b^c$ does not depend on $Y$, while $b^s(Y)$ does, i.e., $\Gamma_Y\cap W^s_{p,loc}$ moves nontrivially in a strong stable leaf.

\begin{proof}
By Lemma~\ref{le-zariski}, the set of points $P$ associated with $Y_0$ is Zariski dense. Since $\mathcal Z$ is not reduced to $\{Y_0\}$, there exists a point $w_{Y_0}\in P$ whose continuation is not constant as $Y$ varies in $\mathcal Z$. By definition of $P$, there exists $n_0\geq0$ such that $f^{n_0}(w_{Y_0})$ is a transverse intersection point between $f^{n_0}(Y_0)$ and $W^s_{p,loc}$.

This intersection point can therefore be followed holomorphically in a small connected open neighborhood $\mathcal Z_0\subset\mathcal Z$ of $Y_0$, yielding a map
\[y_0\colon \mathcal Z_0 \to W^s_{p,loc}\]
which is not constant. Since $f^{n_0}(Y)$ is transverse to $W^s_{p,loc}$ at $y_0(Y)$, possibly after increasing $n_0$ we may assume that $y_0(Y)$ belongs to a submanifold $\Gamma_Y\subset f^{n_0}(Y)$ arbitrarily close to $W^u_{p,loc}$  for every $Y\in\mathcal Z_0$. In particular, it contains a $u$-graph in $\mathcal U_+$ tangent to the cone field $C^{uu}_\eta$.

Let $(b^c(Y),b^s(Y))$ denote the coordinates of $y_0(Y)$ in $W^s_{p,loc}$. Since $y_0$ is not constant, it remains to show that $b^c(Y)$ is constant on $\mathcal Z_0$.

To this end, by Lemma~\ref{le-control1} and Remark~\ref{rk-control1}, we have
\begin{equation}\label{eq-controle1surY}
f^n(\Gamma_Y)=\left\{\left(
\chi_p^{n} b^c(Y)\big[\alpha(f)(1,0)+R_n(Y,z)\big],
\, z
\right)\ ;\ z\in B^u\right\},
\end{equation}
where $R_n(Y,z)$ converges locally uniformly to $0$ with respect to both $Y$ and $z$.

On the other hand, by Assumption~\ref{blender}, there exists $x_n\in\Lambda\cap\mathcal U_+$ such that $f^n(\Gamma_{Y_0})$ intersects $W^s_{x_n,loc}$ at a point $q_n(Y_0)$. By the assumption of the proposition, all such intersections persist, and thus define a map $q_n(Y)$ for $Y\in\mathcal Z_0$.

If $b^c(Y)$ were not constant on $\mathcal Z_0$, then~\eqref{eq-controle1surY} would imply the existence of unit tangent vectors in $T_{q_n(Y_0)}W^s_{x_n,loc}$ arbitrarily close to vectors of the form $(v_n^c,0,v_n^u)$ in $B_p$. This is impossible, since $W^s_{x_n,loc}$ is tangent to the strong stable cone field $C_\eta^{ss}$, whereas vectors of the form $(v_n^c,0,v_n^u)$ in $B_p$ do not belong to these cones by Lemma \ref{le-centrer-unstable}.
\end{proof}
Proposition~\ref{prop-subvar} rules out variations of the
intersections $f^n(Y) \cap W^s_{p,loc}$ that occur transversely to the
strong stable directions. The next result, combined with
Assumption~\ref{grille}, excludes those that are tangent to them.

In order to simplify the notation, we restrict to the case of a holomorphic disc $(Y_t)_{t\in\Db}$, which explains why the family is parametrized by $\Db$ here.

\begin{lemma}\label{le-grille}
Assume that $p$ is in standard form in $B_p$. Let $\Gamma=(\Gamma_t)_{t\in\Db}$ be a family of pairwise disjoint $u$-graphs in $B_p$ intersecting $W^s_{p,loc}$ at $(b^c,b^s(t))$, where the map $t\mapsto b^s(t)$ is not constant. Let $V$ denote the total space
\[V:=\bigcup_{t\in\Db}\Gamma_t.\]
If $(x_n,v_n)$ is a sequence in the tangent space $TV$ with $x_n\in\Gamma_0$ such that for an increasing sequence $(m_n)$, $(f^{m_n}(x_n),D_{x_n}f^{m_n}(v_n))$ converges in $TB_p$ to a pair $(y,w)$, then there exists $1\leq i\leq s$ such that
$$w\in T_yW^u_{p,loc}\oplus\mathcal E^{s,i}_{y},$$
where $\mathcal E^{s,i}$ is the dynamical direction associated to the $i$-th strong-stable multiplier $\lambda_i.$
\end{lemma}

\begin{proof}
A first remark is that, since the $\Gamma_t$ are pairwise disjoint, the set $V$ is a (non-closed) submanifold. Second, the argument relies entirely on the fact that the expansion in the unstable direction near $p$ dominates the contraction in the stable direction by Assumption~\ref{saddle}. Therefore, it is independent of the complex setting, and we will work in smooth linearization coordinates.

Let $(x_n,v_n)$ be a sequence with $x_n\in V$ and $v_n\in T_{x_n}V$ as in the statement. We set
\[\check x_n:=(\check x_n^c,\check x_n^s,\check x_n^u):={\theta_f}(x_n)
\quad\text{and}\quad
\check v_n:=(\check v_n^c,\check v_n^s,\check v_n^u):=D_{x_n}{\theta_f}\,v_n,\]
where ${\theta_f}$ denotes the linearization map associated to $p$.

Since $(\Gamma_t)_{t\in\Db}$ is a family of $u$-graphs in $B_p$, there exists a family of $C^1$ maps $h_t=(h_t^c,h_t^s)$ such that
\[\check\Gamma_t:={\theta_f}(\Gamma_t)
=\{(h_t^c(z),h_t^s(z),z)\ ;\ z\in B^u\}.\]
We denote by $\check V:={\theta_f}(V)$. In particular, for $\check x\in\check\Gamma_0$, the tangent space of $\check V$ at $\check x$ decomposes as
\[T_{\check x}\check V
= T_{\check x}\check\Gamma_0 \oplus \check E_{\check x},\]
where $\check E_{\check x}$ is generated by the derivatives with respect to $t$ (viewed as two real variables) of the map $(t,z)\mapsto h_t(z)$ at $t=0$.

Accordingly, $\check v_n$ can be written as
\[\check v_n=\alpha_n u_n+\beta_n e_n,\]
where $\alpha_n,\beta_n\in\Rb$, $u_n\in T_{\check x_n}\check\Gamma_0$, and $e_n\in\check E_{\check x_n}$, with $\|u_n\|=\|e_n\|=1$.

The key point is that, at $z=0$, we have by assumption
\[h_t(0)=(b^c,b^s(t)).\]
Hence, at the point $\check x_0:=(b^c,b^s(0),0)$, the space $\check E$ is generated (over $\Cb$) by a unit vector $e_0$ in the strong stable direction. As a consequence,
\[e_n=\xi_n e_0+O(\check x_n^u)\]
for some $\xi_n\in\Sb^1$.

Now let
\[\check y:={\theta_f}(y)
\quad\text{and}\quad
\check w:=(\check w^c,\check w^s,\check w^u):=D_y{\theta_f}\,w.\]
Since $y\in W^u_{p,loc}$ and $D{\theta_f}$ preserves the center direction along this manifold, we have that $\check w^c=0$ if and only if the first coordinate of $w$ vanishes.

To prove that $\check w^c=0$, we consider the image of $\check v_n$ under the linearized dynamics $K^{m_n}=(\chi_p^{m_n},A^{m_n},E^{m_n}),$ and set
\[\check w_n:=(\check w_n^c,\check w_n^s,\check w_n^u):=K^{m_n}\check v_n.\]
By assumption, $\check w_n$ converges to $\check w$ as $n\to\infty$.

Let $\tilde\chi$ be the constant defined in Section~\ref{sec-ass}. Since the sequences $(\check w_n^s)$ and $(\check w_n^u)$ are convergent, we obtain by Assumption~\ref{saddle}
\begin{equation*}\label{eq-s}
\check v_n^s=o(\tilde\chi^{3{m_n}})
\quad\text{and}\quad
\check v_n^u=o(\tilde\chi^{-3{m_n}}).
\end{equation*}
Moreover, since $K^{m_n}(\check x_n)$ converges to $\check y$, we have $\check x_n^u=o(\tilde\chi^{-3{m_n}})$, and therefore
\begin{equation*}\label{eq-e}
e_n=\xi_n e_0+o(\tilde\chi^{-3{m_n}}).
\end{equation*}

Projecting $\check v_n$ onto the unstable direction yields
\[\alpha_n=\beta_n\,o(\tilde\chi^{-3{m_n}})+O(\|\check v_n^u\|)=\beta_n\,o(\tilde\chi^{-3{m_n}})+o(\tilde\chi^{-3{m_n}}),\]
while projection onto the strong stable direction gives
\[\beta_n(1+o(\tilde\chi^{-3{m_n}}))+O(\alpha_n)
=\|\check v_n^s\|
=o(\tilde\chi^{3{m_n}}).\]
It follows that
\[\beta_n=o(\tilde\chi^{3{m_n}})
\quad\text{and}\quad
\alpha_n=o(\tilde\chi^{3{m_n}})o(\tilde\chi^{-3{m_n}})=o(1).\]
Finally, returning to $\check w_n$, we obtain
\[\check w_n^c=\chi_p^n\check v_n^c
=\chi_p^n\big(O(\alpha_n)+\beta_n\,o(\tilde\chi^{-3{m_n}})\big)=o(1).\]
Since $\check w_n^c\to\check w^c$, we conclude that $\check w^c=0$, as claimed. It remains to prove that at most one coordinate of $\check w^s$ can be nonzero.

First, recall that
\[\check w_n^s = A^{m_n} \check v_n^s
= \alpha_n A^{m_n} u_n^s + \beta_n A^{m_n} e_n^s.\]
Using the asymptotics of $\alpha_n=o(1)$, $\beta_n=o(\tilde\chi^{3{m_n}})$, and
$e_n^s = \xi_n e_0^s + o(\tilde\chi^{-3{m_n}})$, we obtain
\[\check w_n^s = \beta_n \xi_n A^{m_n} e_0^s + o(1).\]
Suppose that there exists $1 \le i_0 \le s$ such that the
$i_0$-th coordinate $\check w^s(i_0) \neq 0$. Then
\[0\neq\check w^s(i_0)= \lim_{n \to \infty} \check w_n^s(i_0)
= \lim_{n \to \infty}\beta_n \xi_n \lambda_{i_0}^{m_n} e_0^s(i_0),\]
where $\lambda_{i_0}$ denotes the $i_0$-th strong stable multiplier of $p$ given in Assumption~\ref{saddle}.
Since these multipliers have distinct moduli, such a nonzero and finite limit can occur for at most one index $i_0$.
\end{proof}
We can now conclude the proof of Theorem~\ref{th-subvar}.

\begin{proof}[Proof of Theorem~\ref{th-subvar}]
We argue by contradiction and assume that the statement of the theorem does not hold. Fix $Y_0\in\mathcal Z$. By Proposition~\ref{prop-subvar}, there exist an integer $n_0\geq0$, an open neighborhood $\mathcal Z_0$ of $Y_0$ in $\mathcal Z$, and a family $(\Gamma_Y)_{Y\in\mathcal Z_0}$ of $u$-graphs in $B_p$ such that $\Gamma_Y\subset f^{n_0}(Y)$ for all $Y\in\mathcal Z_0$, and such that
\[\Gamma_Y\cap W^s_{p,loc}=(b^c,b^s(Y)),\]
where the map $Y\mapsto b^s(Y)$ is not constant on $\mathcal Z_0$.

We may therefore choose a holomorphic disc $(Y_t)_{t\in\Db}$ contained in $\mathcal Z_0$ and centered at $Y_0$, such that the map $t\mapsto b^s(t)$ is nonconstant. In particular, near any noncritical point of $b^s$, the germs of the polydiscs $\Gamma_t$ near $W^s_{p,loc}$ are pairwise disjoint. Possibly after restricting to a smaller holomorphic disc and taking an iterate, we may assume that the polydiscs $\Gamma_t$ are pairwise disjoint.

On the other hand, by Assumption~\ref{blender}, for every $n\geq0$ the graph transform $f^n(\Gamma_0)$ intersects the local stable manifold $W^s_{x_n,loc}$ of some point $x_n\in\Lambda\cap\mathcal U_+$ at a unique point, which we denote by $y_n$. By our contradiction assumption, all these intersections persist along the family, so that for each $n$ the intersection
\[f^n(\Gamma_t)\cap W^s_{x_n,loc}\]
is given by a holomorphic map $t\mapsto y_n(t)$ with $y_n(0)=y_n$. Moreover, there are subsequences such that $(x_{m_n})$ converges to some $x_\infty\in\Lambda\cap\mathcal U_+$ and $(y_{m_n})$ converges to
\[y_\infty\in W^u_{p,loc}\cap W^s_{x_\infty,loc}.\]
Since the graphs $\Gamma_t$ are pairwise disjoint, these maps are injective. Therefore, differentiating at $t=0$ yields a nonzero tangent vector
\[w_n \in T_{y_{m_n}}W^s_{x_{m_n},loc}.\]
Up to a subsequence, we have that $w_n/\|w_n\|$ converges to $w_\infty\in T_{y_\infty}W^s_{x_\infty,loc}.$
Since $x_\infty\in\Lambda\cap\mathcal U_+,$ $w_\infty$ is the image of some vector in $C_\eta^{ss}$ by the differential of the inverse branche $G_+\colon f(\mathcal U_+)\to\mathcal U_+.$ Hence, Assumption~\ref{grille} implies that $w_\infty$ does not belong to
\[T_{y_\infty} W^u_{p(f),loc} \oplus\mathcal E^{s,i}_{y_\infty}\]
for any $1 \le i \le s$, where $\mathcal E^{s,i}_{y_\infty}$ denotes the dynamical
direction at $y_\infty$ associated with the eigenvalue $\lambda_i$.

We now apply Lemma~\ref{le-grille} to the sequence of points
\[\hat x_n := f^{-m_n}(y_{m_n})\in \Gamma_0\]
and tangent vectors
\[v_n := D_{y_{m_n}}f^{-{m_n}}\left(\frac{w_n}{\|w_n\|}\right) \in T_{\hat x_n}V,\]
where
\[V:=\bigcup_{t\in\Db}\Gamma_t.\]
Lemma~\ref{le-grille} therefore forces
\[w_\infty \in T_{y_\infty} W^u_{p(f),loc} \oplus\mathcal E^{s,i}_{y_\infty}\]
for some $1 \le i \le s$, contradicting the above consequence of Assumption~\ref{grille}.
\end{proof}

\subsection{Proofs of Theorem~\ref{th-main-blender} and
Theorem~\ref{th-blender-poly}}
\label{sec-preuve-main}

Now that Theorem~\ref{th-point} and Theorem~\ref{th-subvar} are
established, the proofs of the two main results of this section are
essentially the same as that of \cite[Theorem~C]{gtv}. We treat only
the case of endomorphisms, as the polynomial case is identical.

\begin{proof}[Proof of Theorem~\ref{th-main-blender}]
Let $(Y',f')\in\mathcal Y_{\rm reg}$ satisfy~\eqref{eq-pitch}.
We fix an arbitrarily small neighborhood $\mathcal Y_0$ of $(Y',f')$
in $\mathcal Y$ on which \eqref{eq-pitch} and the condition
$f\in\Omega_u$ hold. As explained in~\eqref{eq-inter-non-vide}, it
suffices to prove that
\[
\mathcal Y_0\cap\supp(\mu_{\mathcal Y})\neq\varnothing.
\]

As in Section~\ref{sec-main-blender}, after possibly shrinking
$\mathcal Y_0$, Proposition~\ref{prop:local-normalization} allows us to
choose
\[
 (Y_0,f_0)\in\mathcal Y_0\cap\mathcal Y^\circ
 \qquad\text{and}\qquad
 f_0\in\tilde\Omega\subset\Omega_u
\]
such that simultaneous conjugacy identifies $\mathcal Y_0$ with a
complex submanifold of
\[
 \mathcal X_D^u\times\tilde\Omega
\]
whose projection
\[
 \pi_{\tilde\Omega}:\mathcal Y_0\to\tilde\Omega
\]
is a surjective submersion. Recall that
\[
 N:=\dim\tilde\Omega=\dim\mathscr M_d^k,
\]
so the fibers of $\pi_{\tilde\Omega}$ have dimension $\dim\mathcal Y-N$.

We freely shrink $\mathcal Y_0$ and $\tilde\Omega$ in what follows.

We also use all the reductions introduced in
Section~\ref{sec-standard}. In particular, the periodic points
$r_1(f),\ldots,r_N(f)$ are defined on $\tilde\Omega$ and the multiplier
map
\[
f\mapsto
\bigl(\chi_{r_1(f)},\ldots,\chi_{r_N(f)}\bigr)
\]
is a local biholomorphism onto its image.

Roughly speaking, the strategy is the following.
\begin{itemize}
\item Using Assumption~\ref{init}, choose a map $f_0\in\tilde\Omega$
such that $W^u_{p(f_0),loc}$ intersects a preimage of
$W^s_{r_0(f_0),loc}$ and
\[
\overline{\langle\chi_{p(f_0)},\chi_{r_0(f_0)}\rangle}=\Cb^*.
\]

\item This stable/unstable relation defines a hypersurface
$A_1\subset\tilde\Omega$.

\item Using the homoclinic intersection from
Assumption~\ref{homoclinic}, we successively impose new
stable/unstable relations, defining a decreasing sequence of analytic
subsets
\[
A_1\supset A_2\supset\cdots\supset A_N
\]
of $\tilde\Omega$. Theorem~\ref{th-point} ensures that this process can
be continued until $A_N$ reduces to a point.

\item Taking inverse images by the submersion $\pi_{\tilde\Omega},$ these relations lift to corresponding relations in $\mathcal Y_0.$ Once the dynamical parameter is fixed,
Theorem~\ref{th-subvar}, together with~\eqref{eq-pitch}, allows us to
successively impose new relations coming from variations of $Y$ until
we obtain a zero-dimensional intersection.

\item Finally, using the density of periodic points in $\Lambda(f)$,
the inclination lemma, and the persistence of proper intersections, we
replace all these relations by relations satisfying
\ref{it:cond21} and~\ref{it:cond22}.
\end{itemize}

We now give the details.

\bigskip

\noindent\emph{\textbf{Initialization.}}
Using Assumption~\ref{init}, we fix $f_0\in\tilde\Omega$, $m_0\geq0$
and $\tilde r_0(f_0)\in W^s_{\Lambda(f_0),loc}$ such that
\[
\overline{\langle\chi_{p(f_0)},\chi_{r_0(f_0)}\rangle}=\Cb^*,
\qquad
\tilde r_0(f_0)\in W^u_{p(f_0),loc},
\]
and
\[
f_0^{m_0}(\tilde r_0(f_0))
\in W^s_{r_0(f_0),loc}.
\]
By Assumption~\ref{free}, this relation is not persistent. Hence,
possibly after shrinking $\tilde\Omega$,
\[
A_1:=
\left\{
f\in\tilde\Omega\ ;\
f^{m_0}(W^u_{p(f),loc})\cap W^s_{r_0(f),loc}
\neq\varnothing
\right\}
\]
defines a hypersurface of $\tilde\Omega$ containing $f_0$.

On the other hand, Assumption~\ref{homoclinic} provides a local
submanifold contained in the global unstable manifold $W^u_{p(f)}$
and transverse to $W^s_{p(f),loc}$. Using the inclination lemma and
possibly replacing it by an iterate, we obtain a holomorphic family
$(\Gamma(f))_{f\in\tilde\Omega}$ of $u$-graphs contained in
$W^u_{p(f)}$. The iterates of these graphs will be used to produce the
successive stable/unstable relations.

The only technical difficulty is that $A_1$ need not be irreducible
at $f_0$. A non-persistent relation may therefore decrease the
dimension of only one local component and fail to give the proper
intersection needed below.

\bigskip

\noindent\emph{\textbf{Move to a smooth point.}}
To overcome this difficulty, we proceed as in
\textbf{Step~2} of the proof of~\cite[Theorem~C]{gtv}.
The condition
\[
\overline{\langle\chi_{p(f)},\chi_{r_0(f)}\rangle}=\Cb^*
\]
is very generic in the real analytic sense. We may therefore replace
$f_0$ by a smooth point $f_1$ of $A_1$ satisfying the same condition,
and choose a neighborhood $\tilde\Omega_1$ of $f_1$ such that
$A_1\cap\tilde\Omega_1$ is smooth, connected and closed in
$\tilde\Omega_1$.

We now argue inductively. Assume that, for some $1\leq i<N$, we have
constructed a smooth connected analytic submanifold
$A_i\subset\tilde\Omega_i$ of codimension $i$, containing a point
$f_i$ such that
\[
\overline{\langle\chi_{p(f_i)},\chi_{r_0(f_i)}\rangle}=\Cb^*.
\]
If no new stable/unstable relation defines a hypersurface of $A_i$
near $f_i$, Theorem~\ref{th-point} implies that $A_i$ is reduced to a
point. Since $i<N=\dim\tilde\Omega$, this is impossible.

Hence there exist $m_i\geq0$ and $x_i\in\Lambda$ such that
\[
A_{i+1}:=
\left\{
f\in A_i\cap\tilde\Omega_i\ ;\
f^{m_i}(\Gamma(f))\cap W^s_{x_i(f),loc}
\neq\varnothing
\right\}
\]
defines locally a hypersurface of $A_i$. As above, we may move to a
smooth point $f_{i+1}$ of a suitable component of $A_{i+1}$ satisfying
the dense subgroup condition and choose a neighborhood
$\tilde\Omega_{i+1}$ such that $A_{i+1}$ is smooth, connected and
closed there.

Iterating this procedure, we obtain
\[
A_1\supset A_2\supset\cdots\supset A_N,
\qquad
\codim_{\tilde\Omega}A_i=i,
\]
and, after the last reduction,
\[
A_N=\{f_N\}.
\]

\bigskip

\noindent\emph{\textbf{Variations coming from subvarieties.}}
We now use the fact that $\pi_{\tilde\Omega}\colon\mathcal Y_0\to\tilde\Omega$ is a submersion and, for $1\leq i\leq N$, set
\[
\mathcal A_i:=\pi_{\tilde\Omega}^{-1}(A_i)\cap\mathcal Y_0.
\]
After shrinking around the final parameter, we may assume that
$\mathcal A_N$ is a connected open subset of the fiber
\[
\pi_{\tilde\Omega}^{-1}(f_N).
\]
In particular,
\[
\dim\mathcal A_N=\dim\mathcal Y-N.
\]

If $\mathcal A_N$ is not reduced to a point, condition~\eqref{eq-pitch}
and the inclination lemma put us in the setting of
Theorem~\ref{th-subvar}. Hence we can impose a new relation between an
iterate of the varying subvariety $Y$ and the local stable manifold of
a point of $\Lambda(f_N)$, defining a hypersurface of $\mathcal A_N$.

Repeating the argument, we obtain a decreasing sequence
\[
\mathcal A_N\supset\mathcal A_{N+1}\supset\cdots
\supset\mathcal A_{N'},
\qquad
N':=\dim\mathcal Y,
\]
where, for every $N<i\leq N'$, the set $\mathcal A_i$ is locally of
the form
\[
\left\{
(Y,f_N)\in\mathcal A_{i-1}\ ;\
f_N^{m_{i-1}}(U_Y)\cap W^s_{x_{i-1},loc}
\neq\varnothing
\right\}
\]
for some local family $(U_Y)$, some $m_{i-1}\geq0$ and
$x_{i-1}\in\Lambda(f_N)$. Theorem~\ref{th-subvar} ensures that the
procedure can be continued as long as the current family has positive
dimension. Therefore
\[
\mathcal A_{N'}=\{(Y_*,f_N)\}
\]
is reduced to a single point.

\bigskip

\noindent\emph{\textbf{Genuine Misiurewicz relations.}}
The construction above gives $N'=\dim\mathcal Y$ independent
stable/unstable relations. We now replace them by relations involving
periodic points in $\Lambda$ and iterates of $Y$, as required in
\ref{it:cond21} and~\ref{it:cond22}.

Indeed, consider the corresponding intersection in
\[
\mathcal Y_0\times(\Pb^k)^{N'}.
\]
For the first $N$ coordinates, the local $u$-graphs involved are
respectively contained in iterates of
\[
W^u_{p(f),loc},\ \Gamma(f),\ldots,\Gamma(f),
\]
whereas for the remaining coordinates they are contained in iterates
of the varying subvariety $Y$. Let $Z$ denote the analytic set obtained
by putting these $N'$ local $u$-graphs together, and let
\[
\mathcal W^s_{x,loc}:=
\left\{
(Y,f,z_1,\ldots,z_{N'})\ ;\
z_i\in W^s_{x_i(f),loc},\ 1\leq i\leq N'
\right\}.
\]
By construction,
\[
\dim\bigl(Z\cap\mathcal W^s_{x,loc}\bigr)=0.
\]
Moreover,
\[
\dim Z=(u+1)N',
\qquad
\dim\mathcal W^s_{x,loc}=(s+1)N',
\]
whereas
\[
\dim\bigl(\mathcal Y_0\times(\Pb^k)^{N'}\bigr)
=(k+1)N'.
\]
Since $k=s+u+1$, the intersection is proper.

Periodic points are dense in $\Lambda(f)$. Hence, by persistence of
proper intersections, we may replace each $x_i$ by a periodic point
$x_i'$ arbitrarily close to it. Furthermore, by the inclination lemma
and~\eqref{eq-pitch}, sufficiently large iterates of suitable local
pieces of $Y$ can be chosen arbitrarily close to all the $u$-graphs
appearing in the first $N$ relations. After shrinking the corresponding
local pieces, we therefore obtain periodic points
\[
x_1'(f),\ldots,x_{N'}'(f)\in\Lambda(f)
\]
and local families satisfying conditions~\ref{it:cond21}
and~\ref{it:cond22} with
\[
m=N'=\dim\mathcal Y.
\]

In particular, this proves that the maximal integer $m$ introduced in
Section~\ref{sec-main-blender} is equal to $\dim\mathcal Y$.

By Lemma~\ref{le-cond2}, condition~\ref{it:cond2} of
Proposition~\ref{prop-positivevolume} follows. On the other hand,
Assumption~\ref{misu} and Lemma~\ref{le-cond1} give
condition~\ref{it:cond1}. Proposition~\ref{prop-positivevolume}
therefore applies and yields
\[
\mathcal Y_0\cap\supp(\mu_{\mathcal Y})\neq\varnothing.
\]

Since the original neighborhood $\mathcal Y_0$ was arbitrarily small
around $(Y',f')$, and $\supp(\mu_{\mathcal Y})$ is closed, we conclude
that
\[
(Y',f')\in\supp(\mu_{\mathcal Y}).
\]
\end{proof}

The proof of Theorem~\ref{th-blender-poly} is identical, replacing $\mathscr M_d^k$ by $\mathscr P_d^k$ and applying Proposition~\ref{prop:local-normalization} in the polynomial setting. The assumption that the transverse intersection in~\eqref{eq-pitch} lies in $\Ab^k$ allows the construction to be carried out in the affine part, and the polynomial conclusion of Proposition~\ref{prop-positivevolume} gives the result.

\section{Main construction}\label{sec-existence}
In this section, we establish the existence of regions in $\mathrm{End}_d^k$ and $\mathrm{Poly}_d^k$ where the assumptions of Section~\ref{sec-ass} are fulfilled.
\begin{theorem}\label{th-existence}
For each codimension $1\leq u\leq k-1,$ there exist $a,\alpha,\epsilon,\beta,\tilde w(a)\in\Cb,$ $\tau,\gamma\in\Cb^s,$ $\sigma\in\Cb^{u-1}$ and a small perturbation $f\in\mathrm{Poly}_d^k$ of the map $f_0:\mathbb{C}^k\to\mathbb{C}^k$ defined, for $(z,x,y,w)\in\Cb\times\Cb^s\times\Cb^{u-1}\times\Cb$ by
\[f_0(z,x,y,w)=\left(\alpha z+\epsilon w+\beta zw+\tau\cdot x(w-\tilde w(a)),\gamma x,\sigma\left(y+\frac{w-\tilde w(a)}{6}\right),a(w^2-1)\right)\]
such that $f$ admits a neighborhood $\Omega_u$ in $\mathrm{End}_d^k$ (or $\mathrm{Poly}_d^k$) which satisfies all the assumptions described in Section~\ref{sec-ass}. 
\end{theorem}
Here, $\tau\cdot x$ denotes the dot product in $\Cb^s$, while $\gamma x := (\gamma_i x_i)_{1\le i\le s}$ and $\sigma(y+(w-\tilde w(a))/6)=(\sigma_j(y_j+(w-\tilde w(a))/ 6))_{1\le j\le u-1}$. The factor $1/6$ is chosen so that the resulting hyperbolic set $\Lambda(f_0)$ lies in the prescribed polydisc.

In order to explain the structure of this section, observe that the map $f_0$ is a skew-product preserving the
pencil of hyperplanes
\[\{w=\mathrm{cst}\}\]
and acting affinely on each of them. The induced action on this pencil is given by the one-dimensional map
\[q_a(w):=a(w^2-1).\]
The parameter $a$ will be chosen very large. In this case,
$q_a$ has two fixed points, one close to $1$ and the other
close to $-1$, and the point $\tilde w(a)$ corresponds to the latter.

The point $p(f)$ appearing in Section~\ref{sec-ass} will be obtained
as the continuation of the unique fixed point $p(f_0)$ lying on the
hyperplane $\{w=\tilde w(a)\}$. The hyperbolic set $\Lambda(f)$ will
in turn arise from a hyperbolic set of $q_a^2$ located near the two
points $\pm1$.

To construct these objects progressively, we increase the dimension
of the dynamics step by step:
\begin{itemize}
\item we quickly recall some results about the maps $q_a$ in Section~\ref{sec-qa},

\item in Section~\ref{sec-dim2}, we establish the blender property for
\[(z,w)\mapsto\big(\alpha z+\epsilon w+\beta zw,q_a(w)\big),\]
and establish the analogue of Assumption~\ref{init} in this setting,

\item finally, in Section~\ref{sec-dimk}, we add the remaining $s$ stable
directions and the $u-1$ unstable directions, with a shift
depending on whether $w$ is close to $\tilde w(a)$ or not.
\end{itemize}

Observe that the first two steps coincide exactly with
\cite[Sections~4.1 and~4.2]{gtv}. We therefore only recall the
essential results. We follow the notation of~\cite{gtv} as closely as possible, while
choosing to place the variable $w$, which defines the fibration,
in the last position. This explains the somewhat unusual ordering
of the variables $z,x,y,w$.

\subsection{One-dimensional map}\label{sec-qa}
For $|a|$ large enough, $q_a(w)=a(w^2-1)$ is hyperbolic with a Cantor set as Julia set, where the dynamics is conjugate to the full shift via the coding given by the following two inverse branches, which will play a key role in what follows.
\begin{lemma}[{\cite[Lemma~4.3]{gtv}}]\label{le-U}
If $a\in\Cb^*$ satisfies $|a|>10$ then there exist neighborhoods $U_+$ of $1$ and $U_-$ of $-1$ such that $q_{a|U_\pm}$ is a biholomorphism from $U_\pm$ onto $\Db_3$, and
\[ D(\pm1,|a|^{-1})\subset U_\pm\subset D(\pm1,3|a|^{-1}). \]
In particular, $q_a$ admits a unique fixed point $\tilde w(a)\in U_-$.
\end{lemma}

Let $g_+\colon \Db_3\to U_+$ and $g_-\colon \Db_3\to U_-$ be the two inverse branches of $q_a$ associated to Lemma \ref{le-U}. We define
\begin{itemize}
\item $V_-:=g_-(U_+)$ and $V_+:=g_+(U_-)$,
\item $V_1:=q_a^{-2}(V_+)\cap V_+$, $V_2:=q_a^{-2}(V_-)\cap V_+$, $V_3:=q_a^{-2}(V_+)\cap V_-$, and $V_4:=q_a^{-2}(V_-)\cap V_-$.
\end{itemize}
These sets satisfy
\[ q_a^2(V_-)=q_a^2(V_+)=\Db_3
\quad\text{and}\quad
q_a^4(V_i)=\Db_3 \ \text{ for every } \ i\in\{1,2,3,4\}. \]
By construction, the corresponding maximal invariant sets coincide:
\begin{equation}\label{eq-cantor}
E:=\bigcap_{n\geq0}q_a^{-2n}(V_+\cup V_-)
=
\bigcap_{n\geq0}q_a^{-4n}\Bigl(\bigcup_{i=1}^4V_i\Bigr).
\end{equation}
The set $E$ is a Cantor set contained in the Julia set of $q_a$, on which $q_a^2$ is conjugate to the full shift on two symbols.

All the sets introduced in this section depend on $a$. When needed, we write $U_\pm(a)$, $V_\pm(a)$, $V_i(a)$, and $E(a)$ to emphasize this dependence.

\subsection{Two-dimensional map}\label{sec-dim2}

As in \cite[Section~4.2]{gtv}, we consider the skew product
\[ {\mathcal F}_\lambda(z,w)=(\alpha z+\epsilon w+\beta zw,q_a(w)), \]
where $\lambda=(a,\alpha,\beta,\epsilon)$. Several objects denoted with stylized uppercase letters in this section, such as $\mathcal F_\lambda$, ${\mathcal P}(\lambda)$, and ${\mathcal R}(\lambda)$, will correspond to their lowercase analogues in Section~\ref{sec-dimk}, namely $f_\lambda$, $p(\lambda)$, and $r(\lambda)$.

If $\tilde w(a)$ denotes the unique fixed point of $q_a$ in $U_-(a)$, then ${\mathcal F}_\lambda$ has a fixed point
\[
{\mathcal P}(\lambda)=\left(\frac{-\epsilon\tilde w(a)}{\alpha+\beta\tilde w(a)-1},\tilde w(a)\right).
\]
This point is repelling in the vertical direction, while its horizontal multiplier is $\alpha+\beta\tilde w(a)$, which is close to $\alpha-\beta$ when $|a|$ is large. On the other hand, by the choice of the sets $V_+(a)$ and $V_-(a)$, the horizontal component of the second iterate $\mathcal F_\lambda^2$ is essentially a dilation of ratio $\alpha^2-\beta^2$ on $\Cb\times (V_+(a)\cup V_-(a))$. Accordingly, we shall choose $\alpha$ and $\beta$ so that
\begin{itemize}
\item $|\alpha-\beta|<1$, which implies that ${\mathcal P}(\lambda)$ is a saddle point,
\item $|\alpha^2-\beta^2|>1$, which ensures the existence of a repelling hyperbolic set $\Lambda(\lambda)$ for $\mathcal F_\lambda^2$.
\end{itemize}
For suitable choices of $\alpha$, $\beta$, and $\epsilon$, this hyperbolic set will moreover satisfy a blender property, and it will project onto $E(a)$. In order to obtain the transversality properties needed later on, we let $a$ tend to infinity. In this regime, the set $E(a)$ degenerates to $\{-1,1\}$, and the dynamics on $\Lambda(\lambda)$ degenerates to the inverse limit of an iterated function system with two generators.

More precisely, if $\hat\lambda:=(\alpha,\beta,\epsilon)\in\Cb^3$ and $R>2$ are fixed, then for $|a|>10$ large enough the map $\mathcal F_\lambda^2(z,w)$ is arbitrarily close to $(\phi^+_{\hat\lambda}(z),q_a^2(w))$ on $\Db_R\times V_+(a)$ and to $(\phi^-_{\hat\lambda}(z),q_a^2(w))$ on $\Db_R\times V_-(a)$, where $\phi^+_{\hat\lambda}$ and $\phi^-_{\hat\lambda}$ are affine maps defined by
\[
\phi^+_{\hat\lambda}(z):=(\alpha^2-\beta^2)z-\epsilon(\beta+1-\alpha)\ \text{ and }\
\phi^-_{\hat\lambda}(z):=(\alpha^2-\beta^2)z-\epsilon(\beta+\alpha-1).
\]
We then define
\[
\phi_{1,\hat\lambda}:=\phi^+_{\hat\lambda}\circ\phi^+_{\hat\lambda},
\quad
\phi_{2,\hat\lambda}:=\phi^-_{\hat\lambda}\circ\phi^+_{\hat\lambda},
\quad
\phi_{3,\hat\lambda}:=\phi^+_{\hat\lambda}\circ\phi^-_{\hat\lambda},
\quad\text{and}\quad
\phi_{4,\hat\lambda}:=\phi^-_{\hat\lambda}\circ\phi^-_{\hat\lambda},
\]
so that $\mathcal F_\lambda^4(z,w)\simeq(\phi_{j,\hat\lambda}(z),q_a^4(w))$ on $\Db_R\times V_j(a)$.

We now fix a small real number $A>0$ and set $\alpha_0:=\zeta(1+A)$ and $\beta_0:=2A\zeta$, where $\zeta\in\Sb^1$. Then
\[
\alpha_0-\beta_0=\zeta(1-A),
\]
so that the fixed point ${\mathcal P}(a,\alpha_0,\beta_0,\epsilon)$ is a saddle for $|a|$ sufficiently large. Moreover,
\[
\alpha_0^2-\beta_0^2=\zeta^2(1+2A-3A^2),
\]
which has modulus strictly larger than $1$ as soon as $A<2/3$. Following \cite[Lemma~4.4]{dujardin-bif}, we choose
\[
\zeta=e^{i\pi/4}
\]
in order to obtain the blender property for the limiting iterated function system, and hence for $\Lambda(\lambda)$.

The next result provides the building blocks for establishing the blender property.
\begin{lemma}[{\cite[Lemma~4.5]{gtv}}]\label{le-ifs}
Let $\zeta:=e^{i\pi/4}$ and $\epsilon_0=(20(\zeta-1))^{-1}$. Let $A\in(0,1/10]$ be small enough and set $\hat\lambda_0:=(\alpha_0,\beta_0,\epsilon_0)$ where
\[\alpha_0=\zeta(1+A)\ \ \text{ and }\ \ \beta_0=2A\zeta.\]
Then, there exist four open sets $H_j$, $j\in\{1,2,3,4\}$ such that
\[\Db_2=\bigcup_{j=1}^4H_j,\ \overline{\phi_{j,\hat\lambda_0}(H_j)}\subset\Db_2\ \text{ and } \overline\Db\subset H_j.\]
\end{lemma}
Actually, this result allows us to fix to size of the cone fields we will use later. More precisely, since $\overline{\phi_{j,\hat\lambda_0}(H_j)}\subset\Db_2,$ there exists $r>0$ such that
\begin{equation}\label{eq-cover}
\bigcup_{j=1}^4\overline{\phi_{j,\hat\lambda_0}(H_j)}\subset\Db_{2-r}.
\end{equation}
On the other hand, we have 
\[\overline{\Db_{2-r}}\subset\Db_2=\bigcup_{j=1}^4H_j.\]
Let $\rho>0$ be the Lebesgue number of this cover. We choose
$\eta>10$ so that $\eta>4\rho^{-1}$. With this choice, the
vertical projection of any $u$-graph parametrized by a set of
diameter less than $2$ and tangent to the cone field $C_\eta$
has diameter less than $\rho/2$.

Since the conclusion in \eqref{eq-cover} is stable under perturbations of the maps $\phi_{j,\hat\lambda_0}$, they remain valid on the closure of sufficiently small connected neighborhoods of
\[
\hat\lambda_0=(\alpha_0,\beta_0,\epsilon_0).
\]
From now on, we denote by $\hat M$ such a neighborhood, chosen so that for every $\hat\lambda=(\alpha,\beta,\epsilon)\in\hat M$ one has
\begin{equation}\label{def-M}
1/20<|\epsilon|<1/10,
\qquad
|\alpha-\beta|<1,
\ |\alpha-\beta-1|>1/2,\ \text{ and }\
|\alpha^2|-|\beta^2|>1+A.
\end{equation}
In particular, if we set $R:=A^{-1}$, then the maps $\phi^\pm_{\hat\lambda}$ satisfy
\[
\overline{\Db_R}\subset \phi^\pm_{\hat\lambda}(\Db_R).
\]

The next step is to define the hyperbolic set $\Lambda(\lambda)$ with the blender property appearing in Assumptions \ref{hyperbolic} and \ref{blender}. We note that a complete proof of the second statement will be given in the higher-dimensional setting in Lemma \ref{le-blender-selle}.

\begin{lemma}[{\cite[Lemma~4.6]{gtv}}]\label{le-blender}
There exist constants $\eta>100$ and $\delta>100$ such that, if $|a|>\delta$, then for every $(\alpha,\beta,\epsilon)\in\hat M$, the map ${\mathcal F}_{\lambda}$, with $\lambda:=(a,\alpha,\beta,\epsilon)$, satisfies the following properties:
\begin{itemize}
\item On each of the sets $\Db_R\times V_-(a)$ and $\Db_R\times V_+(a)$, the map $\mathcal F^2_\lambda$ is injective, expanding, and contracts the cone field $C_\eta$. Moreover,
\[
\overline{\Db_R\times (U_+(a)\cup U_-(a))}\subset \overline{\Db_R\times \Db_2}
\subset
\mathcal F^2_\lambda(\Db_R\times V_\pm(a)).
\]
In particular, the set
\[
\Lambda(\lambda):=\bigcap_{n\geq0}\mathcal F^{-2n}_\lambda\bigl(\Db_R\times (V_-(a)\cup V_+(a))\bigr)
\]
is a hyperbolic repelling invariant set for $\mathcal F^2_\lambda$.
\item For every $i\in\{1,2,3,4\}$, any vertical graph in $\Db\times V_i(a)$ tangent to the cone $C_\eta$ intersects $\Lambda(\lambda)$.
\end{itemize}
Moreover, both conclusions are stable under small $C^1$-perturbations of ${\mathcal F}_\lambda$.
\end{lemma}
By construction, each point of $\Lambda(\lambda)$ is associated to a unique word $\omega$ in $\Sigma:=\{-1,1\}^{\Nb}$. More precisely, let
\[
g_{\lambda,+}\colon \Db_R\times \Db_2\to \Db_R\times V_+(a)
\quad\text{and}\quad
g_{\lambda,-}\colon \Db_R\times \Db_2\to \Db_R\times V_-(a)
\]
be the two inverse branches of $\mathcal F^2_\lambda$. Identifying the symbol $+$ with $1$ and the symbol $-$ with $-1$, any word $\omega=(\omega_n)_{n\geq0}\in\Sigma$ determines a sequence of contractions
\[
g^n_{\lambda,\omega}:=g_{\lambda,\omega_0}\circ\cdots\circ g_{\lambda,\omega_{n-1}}.
\]
Since the maps $g_{\lambda,\pm}$ are contracting, the sequence $\bigl(g^n_{\lambda,\omega}(y)\bigr)_{n\geq0}$ converges, for every $y\in \Db_R\times \Db_2$, to a limit denoted by $x_\omega(\lambda)$, and this limit is independent of $y$. The hyperbolic set $\Lambda(\lambda)$ is therefore exactly
\[
\Lambda(\lambda)=\{x_\omega(\lambda)\,;\,\omega\in\Sigma\}.
\]
In particular, the repelling point ${\mathcal R}(\lambda)$ appearing in Assumption \ref{repelling} is given by $x_\omega(\lambda)$ for the constant sequence $\omega=(-1)_{n\geq0}$. Observe that $x_\omega(\lambda)$ depends holomorphically on $\lambda$ and continuously on $\omega$ for the product topology on $\Sigma$.

In order to verify Assumption \ref{init}, we also introduce the two central multipliers, $\chi_{{\mathcal P}(\lambda)}$ and $\chi_{{\mathcal R}(\lambda)},$ associated to the fixed saddle point ${\mathcal P}(\lambda)$ and to the $2$-periodic repelling point ${\mathcal R}(\lambda)$ respectively.

The transversality properties required in Assumption \ref{init} then follow by combining the four results \cite[Lemmas~4.8--4.11]{gtv}. We define, for a word $\omega\in\Sigma$ and two constants $\xi\in\Sb^1$, $t\in\Rb$
\[
X_\omega:=\left\{\lambda\in M\,;\,x_\omega(\lambda)\in W^u_{{\mathcal P}(\lambda),loc}\right\}
\]
and
\[
Y_{\xi,t}:=\left\{\lambda\in M\,;\,\chi_{{\mathcal R}(\lambda)}=\xi\,\chi_{{\mathcal P}(\lambda)}^t\right\}.
\]
\begin{lemma}\label{le-init-dim2}
Possibly after shrinking $\hat M$ and increasing the constant $\delta>100$ given by Lemma \ref{le-blender}, we have that, if
\[
D_\infty:=\Cb\setminus\Db_\delta
\quad\text{and}\quad
M:=D_\infty\times \hat M,
\]
then the following hold.
\begin{itemize}
\item For all $\omega\in\Sigma,$ the set $X_\omega$ is a (possibly empty) hypersurface of $M.$
\item For all $\lambda\in M,$ the local unstable manifold $W^u_{{\mathcal P}(\lambda),loc}$ is a vertical graph in $\Db_{1/4}\times \Db_2$, tangent to the cone field $C_\eta$.
\item If $\omega,\omega'\in\Sigma$ and $(\xi,t)\in\Sb^1\times\Rb$ are such that $X_\omega$ and $X_{\omega'}$ have a commun irreducible component $Z_\omega=Z_{\omega'}$ contained in $Y_{\xi,t}$, then $\omega=\omega'.$
\end{itemize}
\end{lemma}
\begin{rk}\label{rk-S}
In Section~\ref{sec-chasse} below, we need the same local construction directly in higher degrees. But all the results of this subsection also hold for
\[
\mathcal F_{a,S}(z,w)=\bigl((\alpha+\beta wS(w))z+\epsilon w,
 \tilde w(a)+a(w^2-\tilde w(a)^2)S(w)\bigr),
\]
where $S$ is a fixed polynomial of degree $D$ sufficiently close to $1$ on $\overline{\Db_3}$, provided $|a|$ is sufficiently large.
Indeed, the second coordinates has $D+2$ inverse branches when $|a|$ is large, $D$ corresponding to the roots of $S$ and $2$ others, $g_{\pm,a}$ close to $w=\pm1.$
They satisfy $g_{\pm,a}(w)=\pm1+O(1/a)$ and $g_{\pm,a}'=O(1/a)$ on fixed large discs. Set
\[
\alpha_S=\alpha+\frac\beta2\bigl(S(1)-S(-1)\bigr),\qquad
\beta_S=\frac\beta2\bigl(S(1)+S(-1)\bigr).
\]
The limiting affine maps are exactly those above with
$(\alpha,\beta)$ replaced by $(\alpha_S,\beta_S)$, and both the
covering argument and the multiplier-transversality argument apply
in the same way, provided $\alpha_S$ and $\beta_S$ satisfy \eqref{def-M}. Thus $S$ may be fixed close to $1$ on $\overline\Db_3$ first and $|a|$ then taken
large.
\end{rk}

\subsection{Higher dimension and degree}\label{sec-dimk}
We now add the remaining $k-2$ dimensions, consisting of $s$
stable directions and $u-1$ unstable ones. The construction is
designed to ensure, in particular,
\begin{itemize}
\item contraction of the cone field $C^{uu}_\eta$ and expansion of
the cone field $C^{ss}_\eta$;
\item the blender property for the associated hyperbolic set
$\Lambda(f)$;
\item the condition that for every prehistory
$\hat x(f)\in\hat\Lambda(f)$,
\[
[W^u_{\hat x(f),loc}]\wedge T_f^{\wedge(u+1)}\neq0;
\]
\item a suitable twisting of the dynamics, different in the $s$
stable directions and the $u-1$ unstable ones, in order to obtain
Assumptions~\ref{decalage} and~\ref{grille}.
\end{itemize}

To this end, we again start from a skew-product acting affinely
on the leaves of a fibration. For $\tau,\gamma\in\Cb^s$ and
$\sigma\in\Cb^{u-1}$, consider the endomorphism of
$\Cb^k=\Cb\times\Cb^s\times\Cb^{u-1}\times\Cb$ defined, as in
Theorem~\ref{th-existence}, by
\begin{equation}\label{eq-f_0}
f_0(z,x,y,w)=\left(\alpha z+\epsilon w+\beta zw+\tau\cdot x(w-\tilde w(a)),\gamma x,\sigma\left(y+\frac{w-\tilde w(a)}{6}\right),q_a(w)\right).
\end{equation}
Observe that his map does not extend to an endomorphism of $\Pb^k$ but is quite easy to study.

At the level of parameters, in order to start verifying the
assumptions of Section~\ref{sec-ass}, we take
\begin{itemize}
\item $(a,\alpha,\beta,\epsilon)$ in the set $M$ defined in
Lemma~\ref{le-init-dim2}, where $(\alpha,\beta,\epsilon)$ is fixed,
and $|a|$ is allowed to increase if needed (and consequently so do
$|\gamma_i|^{-1}$ and $|\sigma_j|$, according to the next conditions);
\item $\gamma=(\gamma_i)_{1\le i\le s}\in\Cb^s$ and
$\sigma=(\sigma_j)_{1\le j\le u-1}\in\Cb^{u-1}$ such that 
\[4|a|^{1/2}<|\gamma_s|^{-1}<\cdots<|\gamma_1|^{-1}<2|\gamma_s|^{-1}<|a|^{3/4}\]
and
\[8|a|^{3/4}<|\sigma_1|<\cdots<|\sigma_{u-1}|<|a|^{4/5};\]
\item there exist integers $1\leq r<\ell$ such that the saddle fixed point
of index $(s+1,u)$
\[p(f_0)=\left(\frac{-\epsilon\tilde w(a)}{\alpha+\beta\tilde w(a)-1},0,0,\tilde w(a)\right)\]
 satisfies conditions~\eqref{eq-bun} and~\eqref{eq-res}, and,
if $a_0,b_0$ are the constants appearing in~\eqref{eq-bun},
\[
b_0^{\ell-r}<a_0^{r+1},
\qquad
|\alpha+\tilde w(a)\beta|^{r+1}<|\gamma_1|.
\]
\item $\tau=(\tau_i)_{1\le i\le s}$ with $\tau_1=\cdots=\tau_s,$ and $0<|\tau_i|<\delta^{-1}$ where $\delta$ is given in Lemma \ref{le-init-dim2}.
\end{itemize}
Observe that $\tau$ twists the stable directions. Since we need
the contraction of the cone field $C^{uu}_\eta$ where $\eta$ is defined after Lemma~\ref{le-ifs},
and we also want to use the same $\eta$ for the stable cone field
$C^{ss}_\eta$, $\tau$ must not be too large. For the sequel, the
important point in verifying Assumption~\ref{grille} is that the
components of $\tau$ do not vanish and do not depend on $a$,
$\gamma$, or $\sigma$.

Since $\tilde w(a)$ is a fixed point of $q_a$, the hyperplane
$\{w=\tilde w(a)\}$ is invariant under $f_0$ and contains a unique
fixed point $p(f_0)$ described above, whose multipliers are $\alpha+\tilde w(a)\beta$, $2a\tilde w(a)$,
and the components of $\gamma$ and $\sigma$. The above choices
ensure that $p(f_0)$ is a saddle fixed point of index $(s+1,u)$ whose dynamics on the stable and unstable manifold are both linearizables, and, by Proposition \ref{prop-dynaslope} (applied to $f_0^{-1}$), there exists stable dynamical directions along $W^u_{p(f_0),loc}$, $\mathcal E^{s,0}_{f_0}$ corresponding to $\alpha+\tilde w(a)\beta$ and  $\mathcal E^{s,i}_{f_0}$ corresponding to $\gamma_i,$ $1\leq i\leq s.$ Moreover, $p(f_0)$ also satisfies the bunching condition in
Assumption~\ref{homoclinic}, although at this stage all its homoclinic
intersections are degenerate.

On the other hand, the plane $\{x=0\}$, of dimension $u+1$, is also invariant
and globally attracting. The induced dynamics there is essentially
the one described in Section~\ref{sec-dim2}, with an additional
factor $\sigma(y+(w-\tilde w(a))/6)$ that is uniformly expanding. More
precisely, set
\[U_\pm^u:=\Bb^{u-1}\times U_\pm(a).\]
Since for each $1\le i\le u-1$ we have $|\sigma_i|>8|a|^{3/4}>100$, the
image of $U_+^u$ (resp.\ $U_-^u$) under the map
\[G\colon(y,w)\mapsto\left(\sigma\left(y+\frac{w-\tilde w(a)}{6}\right),q_a(w)\right)\]
compactly contains $\Bb_2^u$. Observe that these sets $U_\pm^u$ depend on $a$. However, the last property is stable under small
perturbations, so that the sets $U_\pm^u$ appearing in
Assumption~\ref{verti} will correspond to the above ones for a suitable choice of $a$.

Since $U_\pm^u\Subset\Bb_2^u\Subset G(U_\pm^u),$ there exist $V_\pm^u\Subset U_\pm^u$ such that
\[G(V^u_\pm)\subset U^u_\mp\ \text{ and }\ G^2(V^u_\pm)=\Bb_2^u.\]
Observe that by definition of the sets $V_\pm(a)$ in Section~\ref{sec-qa}, we must have that the $w$-coordinate of a point in $V^u_\pm$ belongs to $V_\pm(a).$ In the same way, there exist $V_1^u,V_2^u\subset V_+^u$ and $V_3^u,V_4^u\subset V_-^u$ such that
\[G(V_1^u)=G(V_3^u)=V_+^u\ \text{ and }\ G(V_2^u)=G(V_4^u)=V_-^u,\]
and the $w$-coordinate of a point in $V^u_j$ belongs to $V_j(a).$ In particular, if the constant $\delta>0$ is large enough, then writing
\[f_0^4(z,x,y,w)=\big(F_1(z,w,x),\gamma^4 x,F_2(y,w),q^4_a(w)\big),\]
we have that $F_1(z,w,x)$ is arbitrarily close to $\phi_{j,\hat\lambda}$ (defined in Section~\ref{sec-dim2}) on $\Db_R\times\Bb^s\times V^u_j,$ where $\hat\lambda=(\alpha,\beta,\epsilon).$ Using \eqref{eq-cover}, this gives that
\begin{equation}\label{eq-cover2}
\overline{f_0^4(H_j\times\Bb^s\times V^u_j)}\subset\Db_{2-r}\times\Bb^s_{1/2}\times\overline{\Bb^u_2},
\end{equation}
which will be essential to establish the blender property in Lemma~\ref{le-blender-selle}.

Let $R:=A^{-1}$ as before Lemma~\ref{le-blender}. We can now check Assumptions~\ref{verti} and \ref{hyperbolic} on the sets 
\[
\mathcal U_\pm:=\Db_{2R}\times \Bb^s\times U_\pm^u
\]
and
\[\mathcal V_+:=\Db_R\times \Bb^s\times V_+^u,\ \ \mathcal V_-:=\Db_R\times \Bb^s\times V_-^u\ \text{ and }\ \mathcal V:=\mathcal V_+\cup\mathcal V_-.\]

\begin{lemma}\label{le-ass1et2}
Possibly by increasing the constant $\delta>0$ then, for the choice of parameters as above, the map $f_0$ satisfies Assumptions~\ref{verti} and \ref{hyperbolic} on these sets. In particular,
\[\Lambda(f_0):=\bigcap_{n\in\Zb}f_0^{-2n}(\mathcal V)\]
is a basic set of index $(s,u+1).$ Moreover, $\bigcap_{n\in\Zb}f_0^{-2n}(\mathcal V_-)$ is reduced to a single point $r_0(f_0)$ which is fixed by $f_0^2$ and whose first coordinate belongs to $\Db_{1/2}.$
\end{lemma}
\begin{proof}
Each point is straightforward, and we leave precise details to the reader.

The injectivity of $f_0$ on these sets follows from that of $q_a$ on $U_\pm(a)$, since $f_0$ acts affinely with
full rank on each hyperplane of the pencil, except on
$\{w=-\alpha/\beta\}$. However, $|\alpha/\beta|>20>2$.

The cone field conditions follow from a direct computation of the
differential of $f_0$. The only delicate point concerns the
expansion of the stable cone field: with our choice
$|\gamma_1|<\cdots<|\gamma_s|<2|\gamma_1|$, we need that each component of
$\tau$ satisfies $|\tau_i|<(10\eta s)^{-1}$, which is ensured by
$|\tau_i|<\delta^{-1}$ if $\delta$ is chosen large enough.

The horizontal-like behavior of $f_0$ and $f_0^2$ on $\mathcal U_\pm$ and $\mathcal V$ respectively come from
\[U_\pm^u\Subset\Bb_2^u\Subset G(U_\pm^u),\quad G^2(V^u_\pm)=\Bb_2^u,\]
and the fact that the $w$-coordinates on $U_\pm^u$ is almost equal to $\pm1.$

The locally maximal set
\[\Lambda(f_0)=\bigcap_{n\in\Zb} f_0^{-2n}(\mathcal V)\]
is contained in the invariant plane $\{x=0\}$. On this plane,
$f_0$ is conjugate, via the projection $(z,y,w)\mapsto(z,w)$,
to the map $\mathcal F_\lambda$ studied in
Section~\ref{sec-dim2}. Since $\mathcal F_\lambda$ is expanding
on the corresponding set by Lemma~\ref{le-blender}, and since the
additional factor $\sigma(y+w-\tilde w(a))$ is also expanding,
it follows that $\Lambda(f_0)$ is a repelling hyperbolic set for
$f^2_{0|\{x=0\}}$.

Moreover, on this plane, the fibration by lines
$\{(y,w)=\mathrm{cst}\}$ is invariant. For $(y,w)\in V^u_\pm$,
the differential of $f_0^2$ along these lines is essentially
$\alpha^2-\beta^2$, which is much smaller than the other expansion
rates (of order $|4a^2|$ and $|\sigma_i|^2$). Hence, this
direction defines a central unstable direction for
$\Lambda(f_0)$.

On the other hand, since the plane $\{x=0\}$ is uniformly
attracting, it follows that $\Lambda(f_0)$ is hyperbolic of
index $(s,u+1)$ and the bunching conditions in Assumption~\ref{hyperbolic} are satisfied with $\tilde\chi:=2|a|^{1/4}.$

Finally, a direct computation of $f^2_0|_{\{x=0\}}$ gives that the first coordinate of the fixed point $r_0(f_0)$ for this map is smaller than
\[|\epsilon(\beta+\alpha-1)|<3|\epsilon|<3/10.\]
\end{proof}
In the same way, we have the following.
\begin{lemma}\label{le-ass-saddle}
The saddle fixed point $p(f_0)$ satisfies Assumption~\ref{saddle}.
\end{lemma}
\begin{proof}
The only point which does not directly comes from the choice of the parameters is that $W^u_{p(f_0),loc}$ is a $u$-graph above $\Db$ and not a larger disc. To see this, we use that, since $f_0$ contracts $C^{uu}_\eta,$ this local unstable manifold has to be tangent to it. On the other hand, the first coordinate of $p(f_0)$ is $(\epsilon\tilde w(a))/(\alpha+\beta\tilde w(a)-1)$ whose modulus is smaller that $1/2$ by \eqref{def-M}. Since $\eta>10,$ we have that $W^u_{p(f_0),loc}$ is contained in $\Db\times\Bb^s\times\Bb^u_{3/2}$.
\end{proof}

The next assumption to check is the blender property. It actually follows as in \cite[Lemma~4.6]{gtv} but, as it is important for what follows, we give a full proof of it.
\begin{lemma}\label{le-blender-selle}
Each $u$-graph in $\Db\times \Bb^s\times V_+^u$ (resp. $\Db\times \Bb^s\times V_-^u$)  tangent to $C^{uu}_{\eta}$ intersects the stable set of $\Lambda(f_0)$
\[W^s_{\Lambda(f_0),loc}=\bigcup_{x(f_0)\in\Lambda(f_0)}W^s_{x(f_0),loc}.\]
Moreover, this property is stable under small $C^1$-perturbation of $f_0$ on $\mathcal V.$
\end{lemma}
\begin{proof}
The key ingredient is that, as we have seen, the first factor of $f_0^4$ is arbitrarily close to $\phi_{j,\hat\lambda}$ on $\Db_R\times\Bb^s\times V^u_j,$ which implies the inclusion \eqref{eq-cover2}.

Recall that $\rho>0$ is the Lebesgue number of the cover $\overline{\Db_{2-r}}\subset \bigcup_{j=1}^4H_j$ and that $\eta>4\rho^{-1}.$ Observe that $\rho$ is also the Lebesgue number of the cover
\[\overline{\Db_{2-r}\times\Bb^s_{1/2}}\subset \bigcup_{j=1}^4H_j\times\Bb^s.\]
 Hence, the vertical projection of $u$-graph $\Gamma\subset\Db_{2-r}\times\Bb^s_{1/2}\times\Bb^u_2$ tangent to $C^{uu}_\eta$ must have diameter smaller than $\rho$ and thus, is compactly contained in on of the $H_j\times\Bb^s.$

Now, let $j_0\in\{1,2,3,4\}$ and let $\Gamma_0\subset H_{j_0}\times\Bb^s\times V^u_{j_0}$ be a $u$-graph tangent to $C^{uu}_\eta.$ Its image by $f_0^4$ is a $u$-graph in $\Db_{2-r}\times\Bb^s_{1/2}\times\Bb^u_2$ and by what we have just said, there exists $j_1\in\{1,2,3,4\}$ such that its vertical projection is compactly contained in $H_{j_1}\times\Bb^s.$  Thus, $f_0^4(\Gamma_0)$ contains a $u$-graph $\Gamma_1\subset H_{j_1}\times\Bb^s\times V^u_{j_1}$ tangent to $C^{uu}_\eta.$ By induction, we obtain a sequence of vertical graph $\Gamma_n\subset f_0^{4n}(\Gamma_0)$ in some $H_{j_n}\times\Bb^s\times V^u_{j_n}$ and thus $\Gamma_0$ intersects 
\[\bigcap_{n\geq0} f_0^{-4n}\left(\Db_R\times\Bb^s\times\left(\bigcup_{j=1}^4V^u_j\right)\right).\]
By definition of the sets $V^u_j$ and by \cite[Paper~I, Corollary~2.6]{jonsson-phd}, this intersection is exactly the stable set of $\Lambda(f_0)$.
Moreover, all these arguments are robust under small $C^1$ perturbations.
\end{proof}
We now verify the two assumptions that will be used in
Section~\ref{sec-subvar}. Both are related to the choice of coordinates in $\Cb^k$. We therefore introduce, for
$1\le i\le s$ and $1\le j\le u$, the coordinate vectors
$e^c,e^s_i,e^u_j\in\Cb\times\Cb^s\times\Cb^{u-1}\times\Cb$ by
\[e^c=(1,0,0,0),\qquad e^u_u=(0,0,0,1),\]
 $e^s_i$ denotes the $i$-th coordinate vector in $\Cb^s$,
and $e^u_j$ denotes the $j$-th coordinate vector in
$\Cb^{u-1}$ for $j<u$.

\begin{lemma}\label{le-decalage}
Possibly after increasing $\delta$, Assumption~\ref{decalage}
holds at $f_0$.
\end{lemma}

\begin{proof}
We need a more precise description of the coordinate hyperplanes in the
linearizing coordinates of $W^u_{p(f_0),loc}$.

A first remark is that $W^u_{p(f_0),loc}$ is contained in the plane
$\{x=0\}$. Moreover, the affine subspace
\[\{x=0,\ w=\tilde w(a)\}\]
is invariant, and the restriction of $f_0$ to it is given by
\[(z,y)\mapsto\big(z(\alpha+\beta\tilde w(a))+\epsilon\tilde w(a),\sigma y\big).\]
It follows that, for $1\le i\le u-1$, the curve corresponding in a
linearization chart to the axis associated to the multiplier $\sigma_i$ is precisely the affine
line
\[
p(f_0)+\Cb\cdot e_i^u.
\]
The last unstable direction, corresponding to the multiplier
$2a\tilde w(a)$, is less explicit. However, by our choice of parameters,
this multiplier dominates a power of the other unstable multipliers.
Hence, if $r>0$ is fixed, the cone field
\[
C_r^w:=\left\{(v_1,\ldots,v_{k-1},v_k)\in\Cb^k\ ;\
r\max_{1\le j\le k-1}|v_j|<|v_k|\right\}
\]
is contracted by $f_0^{-1}$ provided $\delta$ is large enough.
Therefore, the strong unstable manifold $W^{uu}_{p(f_0),loc}$ is a curve
arbitrarily close to the affine line
\[
p(f_0)+\Cb\cdot e_u^u.
\]
On the other hand, the hyperbolic set $\Lambda(f_0)$ is also contained in
the plane $\{x=0\}$, and each of its local stable manifolds is transverse
to that plane, intersecting it in a unique point. Thus, for $f_0$, every
intersection point between $W^u_{p(f_0),loc}$ and
$W^s_{\Lambda(f_0),loc}$ actually belongs to $\Lambda(f_0)$, which is a
highly nongeneric situation.

To understand where these points lie, observe that the dynamics of
$f_0^2$ in the $(y_i,w)$-coordinates is given by
\[
H_i\colon (y_i,w)\mapsto\left(
\sigma_i\Big(\sigma_i(y_i+\frac{w-\tilde w(a)}{6})
+\frac{q_a(w)-\tilde w(a)}{6}\Big),\,q_a^2(w)\right).
\]
If $w\in V_+(a)$, then Lemma~\ref{le-U} gives
\[
|w-\tilde w(a)-2|<6|a|^{-1}
\quad\text{and}\quad
|q_a(w)-\tilde w(a)|<6|a|^{-1},
\]
and similarly, with the obvious sign changes, if $w\in V_-(a)$.
Using moreover
\[
8|a|^{3/4}<|\sigma_1|<\cdots<|\sigma_{u-1}|<|a|^{4/5},
\]
we obtain
\[
\left(D(1/3,|a|/2)\cup D(0,|a|/2)\right)\times \Db_2
\subset
H_i\big(D(-1/3,a/\sigma_i^2)\times V_+(a)\big)
\cap
H_i\big(D(0,a/\sigma_i^2)\times V_-(a)\big).
\]
In particular, if $(z,x,y,w)\in\Lambda(f_0)\cap\mathcal U_+$, then
\[
w\in V_+(a)\subset D(1,3|a|^{-1})
\quad\text{and}\quad
y_i\in D(-1/3,|a/\sigma_i^2|),
\]
where $|a/\sigma_i^2|<|a|^{-1/2}$.

Combining the two previous discussions, we conclude that no point of
$W^s_{\Lambda(f_0),loc}\cap\mathcal U_+$ can lie on one of the coordinate
hyperplanes of $W^u_{p(f_0),loc}$. This is exactly Assumption~\ref{decalage}.
\end{proof}
In general, the dynamical directions obtained in
Section~\ref{sec-dir} are not easy to control. However, the very
simple form of $f_0$ allows us to do so in this case.

\begin{lemma}\label{le-grille-2}
Possibly after increasing $\delta$, Assumption~\ref{grille}
holds at $f_0$.
\end{lemma}

\begin{proof}
Fix $1\le i\le s$. The proof follows the same scheme as that of
Lemma~\ref{le-decalage} and relies on the following two claims.

\begin{itemize}
\item[\textbf{(1)}] Let $t\in W^u_{p(f_0),loc}$ and
$v\in\mathcal E^{s,i}_{f_0,t}$ for some $1\le i\le s$. Then
\[
v=v^ce^c+v_ie_i^s,
\qquad
\frac{|v^c|}{|v_i|}<\frac{6|\tau_1|}{|a|^{1/4}}.
\]
\item[\textbf{(2)}] Let
$G_+\colon f_0(\mathcal U_+)\to\mathcal U_+$ be the inverse branch
of $f_0$. If $t\in f_0(\mathcal U_+)$ and $v\in C_\eta^{ss}$, then,
setting $v':=D_tG_+(v)$, its components satisfy
\[
|v'^c|\geq\frac{|\tau_1|}{2}\left|\sum_{j=1}^s v_j'^{s}\right|-\frac{1}{|a|^{1/2}}\sum_{j=1}^s|v_j'^s|.
\]
In particular, if the only non-zero stable component is $v_i'^s$, then
\[
|v'^c|\geq\frac{|\tau_1|}{4}|v_i'^s|.
\]
\end{itemize}

Since $T_tW^u_{p(f_0),loc}\subset\{x=0\}$, these two estimates imply
the desired conclusion.

\bigskip
\noindent\textbf{Proof of (1).} By Proposition~\ref{prop-dynaslope} and
Remark~\ref{rk-dynaslope}, there exists a unique holomorphic
section of the projectivized tangent bundle over
$W^u_{p(f_0),loc}$, invariant under the dynamics and equal to
$e^s_i$ at $p(f_0)$.

Parametrizing $W^u_{p(f_0),loc}$ by the projection onto the
$(y,w)$-coordinates (possibly after shrinking the domain), and
choosing an affine chart in which the coordinate corresponding
to $e^s_i$ is identically equal to $1$, this yields a unique
holomorphic map from $W^u_{p(f_0),loc}$ to
$\Cb\times\Cb^{s-1}\times\Cb^{u-1}\times\Cb$, equivariant under
the differential of $f_0$, of the form
\[(y,w)\mapsto\begin{array}{rl}
\big(A^c(y,w),A^s_1(y,w), \ldots,A^s_{i-1}(y,w), &\\
A^s_{i+1}(y,w),\ldots, & A^s_s(y,w),
A^u_1(y,w),\ldots,A^u_u(y,w)\big).
\end{array}\]

Since the $y_i$-component of $y$ only interacts with itself in
$f_0$, and since the differential of $f_0$ is independent of $y$,
the uniqueness of the section implies that it does not depend on
$y$. Moreover, as the differential of $f_0$ preserves the
subspace $E^{c,i}$ spanned by $e^c$ and $e^s_i$, it follows again
by uniqueness that all the above components vanish except for
$A^c(y,w)$.

We are thus left with studying this last component, which does
not depend on $y.$

The action of the differential on $E^{c,i}$ at $(y,w)$ is
\[\begin{pmatrix}
\alpha+w\beta & \tau_i(w-\tilde w(a))\\
0 & \gamma_i
\end{pmatrix}
=
\begin{pmatrix}
\phi(\hat w) & \tau_i\hat w\\
0 & \gamma_i
\end{pmatrix},\]
where we set $\hat w:=w-\tilde w(a)$. Hence, setting
$A(\hat w):=A^c(y,w)$ and
\[\hat q_a(\hat w):=q_a(\hat w+\tilde w(a))-\tilde w(a)
= a\hat w(\hat w+2\tilde w(a)),\]
the invariance of the (projectivized) section implies that
\[\gamma_i A(\hat q_a(\hat w))
=
\phi(\hat w)A(\hat w)+\tau_i\hat w.\]

Using the inverse branch
\[g\colon \Db_3\to \hat U_-(a)\subset D(0,3|a|^{-1}),\]
of $\hat q_a$, sending $\Db_3$ into a neighborhood of $0$, we obtain
\[A(\hat w) =
\gamma_i^{-1}\big(
\phi(g(\hat w))A(g(\hat w))+\tau_i g(\hat w)
\big),\]
which allows us to define $A$ on $\Db_3$ by iteration.

On the other hand, since $A(0)=0$, there exists a holomorphic
function $B$ such that
\[
A(\hat w)=\hat w\,B(\hat w).
\]
Since $g(0)=0$ and $|g(\hat w)|<3|a|^{-1}$ on $\Db_3$,
Schwarz's lemma gives
\[
\left|\frac{g(\hat w)}{\hat w}\right|\leq |a|^{-1}.
\]
The sup-norm $\|B\|_{\infty,\Db_2}$ of $B$ on $\Db_2$
then satisfies
\[
\|B\|_{\infty,\Db_2}
\leq
\frac{2\|B\|_{\infty,\Db_2}+|\tau_i|}{|a\gamma_i|},
\]
where the factor $2$ comes from $|\phi(\hat w)|<2$.
Thus, using $|\gamma_i|^{-1}<|a|^{3/4}$, we have,
for $|a|$ sufficiently large,
\[
\|B\|_{\infty,\Db_2}
\leq \frac{|\tau_i|}{|a\gamma_i|-2}
\leq \frac{2|\tau_i|}{|a|^{1/4}}.
\]
Since $g(\Db_3)\subset\Db_2$, the same bound holds
for $\|B\|_{\infty,\Db_3}$.
Hence, for $|\hat w|<3$, we obtain
\[
|A(\hat w)|
\leq \frac{2|\tau_i|}{|a|^{1/4}}\,|\hat w|
< \frac{6|\tau_i|}{|a|^{1/4}},
\]
which proves that the central component is negligible compared
to the stable one when $|a|$ is large. This establishes
\textbf{(1)}.

\bigskip
\noindent\textbf{Proof of (2).} This follows directly from the computation of the differential
of $G_+$ and from the cone contraction property. More precisely,
let $v\in\Cb^k$, $t\in f_0(\mathcal U_+)$, and set
$v':=D_tG_+(v)$. Writing $w$ for the last coordinate of $t$, and
denoting by $g_+\colon \Db_3\to U_+(a)$ the inverse branch of
$q_a$, the components of $v'$ satisfy
\[
v'^c=\psi(t)\left(
v^c-\tau_1\big(g_+(w)-\tilde w(a)\big)
\sum_{j=1}^s\gamma_j^{-1}v^s_j
+\phi(t)v^u_u
\right),
\qquad
v'^s_j=\gamma_j^{-1}v^s_j,
\]
where $\phi(t)=O(|a|^{-1})$ uniformly, and $\psi(t)=(\alpha+g_+(w)\beta)^{-1}$ has modulus
bounded between $1/2$ and $2$.

Since $g_+(w)-\tilde w(a)\simeq 2$ on $U_+(a)$, and since
$v\in C^{ss}_\eta$ implies that the sum of the stable components
dominates both $v^c$ and $v^u_u$, the desired estimate follows.
We leave the details to the reader.
\end{proof}

It remains to verify Assumptions~\ref{free}, \ref{init},
\ref{misu} and~\ref{homoclinic}. The first two require us to be in $\mathrm{End}_d^k$ or $\mathrm{Poly}_d^k$,
whereas Assumption~\ref{misu} involves the Green current. Assumption~
\ref{homoclinic} will be obtained after a final arbitrarily small
perturbation.

To treat the case of endomorphisms of $\Pb^k$, we now fix the
parameters of $f_0$ once and for all and consider a sufficiently
small open connected neighborhood $V_0$ of $f_0$ in the space of rational maps
of $\Pb^k$ of degree at most $d$. The case of $\mathrm{Poly}_d^k$ is
handled in the same way by considering a neighborhood of $f_0$
in the space of (not necessarily regular) polynomial endomorphisms
of $\Cb^k$ of degree at most $d$.

A first remark is that, even if $V_0$ is small, the global behavior
of a map $f\in V_0$ on $\Pb^k$ may be very different from that of
$f_0$. However, if $K\subset\Cb^k$ is a fixed compact set, then
$f|_{K}$ is arbitrarily close to ${f_0}|_{K}$ provided $V_0$ is
sufficiently small. In particular, since all the properties established so far are
stable under small perturbations and involve only the dynamics
on a large ball of $\Cb^k$, we may assume that every $f\in V_0$
satisfies all the assumptions of Section~\ref{sec-ass}, except
possibly Assumptions~\ref{free}, \ref{misu}, \ref{homoclinic} and \ref{init}. We denote by $p(f)$, $r_0(f)$ and $\Lambda(f)$ the corresponding
objects for $f\in V_0$. Both $p(f)$ and $r_0(f)$ depend
holomorphically on $f$, and so do their local stable and unstable
manifolds. On the other hand, $\Lambda(f)$ admits a branched
holomorphic motion via its natural extension $\hat\Lambda(f)$
(see \cite[Theorem~B]{jonsson-motion_hyperbolic}).

The first result in this setting shows that
Lemma~\ref{le-init-dim2} extends to $V_0$.

\begin{lemma}\label{le-init}
Possibly after shrinking $V_0$, Assumptions~\ref{free} and
\ref{init} hold on $V_0$.
\end{lemma}
\begin{proof}
The proof is a reduction to Lemma~\ref{le-init-dim2}. More precisely,
if, after shrinking $V_0$, one of the required properties still failed,
then the same failure would persist on a suitable open subset of the
parameter set $M$ appearing in Lemma~\ref{le-init-dim2}, contradicting
that lemma.

We first state two simple facts about $f_0$. The local unstable
manifold of $p(f_0)$ corresponds to the one associated to
$\mathcal P(\lambda)$, where $\lambda=(a,\alpha,\beta,\epsilon)$, for the
map
\[
\mathcal F_\lambda(z,w)=(\alpha z+\epsilon w+\beta zw,q_a(w))
\]
studied in Section~\ref{sec-dim2}, with the additional stable coordinates
$x$ are fixed to $0$, while the unstable coordinates $y$ are free.
Namely,
\[
W^u_{p(f_0),loc}
=
\{(z,x,y,w)\in\Cb\times\Cb^s\times\Cb^{u-1}\times\Cb\ ;\
(z,w)\in W^u_{\mathcal P(\lambda),loc},\ x=0\}.
\]
Moreover, as already observed, $\Lambda(f_0)$ is contained in the plane
$\{x=0\}$, and every intersection between
$W^u_{p(f_0),loc}$ and $W^s_{\Lambda(f_0),loc}$ actually belongs to
$\Lambda(f_0)$.

The same facts hold for the maps $f_{\tilde\lambda}\in V_0$
obtained from $f_0$ by replacing
$\lambda=(a,\alpha,\beta,\epsilon)$ with
$\tilde\lambda=(\tilde a,\tilde\alpha,\tilde\beta,\tilde\epsilon)$. That is
\[
f_{\tilde\lambda}(z,x,y,w)
=
\left(\tilde\alpha z+\tilde\epsilon w+\tilde\beta zw
+\tau\cdot x(w-\tilde w(\tilde a)),
\gamma x,
\sigma\left(y+\frac{w-\tilde w(\tilde a)}{6}\right),
q_{\tilde a}(w)\right).
\]
Let $M_0$ be a connected open subset of the set $M$ of
Lemma~\ref{le-init-dim2} such that $f_{\tilde\lambda}\in V_0$ for all
$\tilde\lambda\in M_0$, and write
\[
\tilde M_0:=\{f_{\tilde\lambda}\ ;\ \tilde\lambda\in M_0\}\subset V_0.
\]

For $f\in V_0$, let $G_\pm\colon f^2(\mathcal V_\pm)\to\mathcal V_\pm$
denote the corresponding inverse branches. Then each word
$\omega\in\Sigma=\{-1,1\}^{\Nb}$ defines a stable manifold of
$\Lambda(f)$ by
\[
W^s_\omega(f):=
\bigcap_{n\geq0}
G_{\omega_0}\circ\cdots\circ G_{\omega_n}
(\Db_R\times\Bb^s\times\Bb^u_{3/2}).
\]
The set $\Db_R\times\Bb^s\times\Bb^u_{3/2}$ could in fact be replaced by
any of its $s$-graphs, and $W^s_\omega(f)$ coincides with the local
stable manifolds constructed in the proof of Proposition~\ref{prop-hyp-graphs}.
Observe that different points of $\Lambda(f)$ may belong to the same
stable manifold.

By Proposition~\ref{prop-hyp-graphs}, the $s$-graph $W^s_\omega(f)$
depend holomorphically on $f\in V_0$ in
$\Db_R\times\Bb^s\times\Bb^u_{3/2}$. The same is true for
$W^u_{p(f),loc}$, viewed as a $u$-graph. Hence, by
Proposition~\ref{prop-relation}, the set
\[
\tilde X_\omega:=\{f\in V_0\ ;\
W^s_\omega(f)\cap W^u_{p(f),loc}\neq\varnothing\}
\]
is either equal to $V_0$ or is a (possibly empty) hypersurface of
$V_0$. The first case is impossible since, for
$\tilde\lambda\in M_0$ the description above gives
\begin{equation}\label{eq-xomega}
\{\tilde\lambda\in M_0\ ;\
W^u_{p(f_{\tilde\lambda}),loc}\cap W^s_\omega(f_{\tilde\lambda})
\neq\varnothing\}
=
M_0\cap X_\omega,
\end{equation}
where $X_\omega$ is the set introduced in Lemma~\ref{le-init-dim2}.
Thus, if $\tilde X_\omega=V_0$, then $M_0\cap X_\omega=M_0$, contradicting
the last point of Lemma~\ref{le-init-dim2}. Therefore each
$\tilde X_\omega$ is a proper analytic subset of $V_0$, which implies Assumption~\ref{free}.

It remains to prove Assumption~\ref{init}, possibly after shrinking
$V_0$. The argument is very close to that of~\cite[Lemma~4.12]{gtv}.

The identity~\eqref{eq-xomega} is a key ingredient, identifying
$\tilde M_0\cap\tilde X_\omega$ with $M_0\cap X_\omega$ via $f_{\tilde\lambda}\mapsto\tilde\lambda.$
Now let $\chi_{p(f)}$ and $\chi_{r_0(f)}$ denote the central
multipliers of the fixed point $p(f)$ and of the $2$-periodic point
$r_0(f)\in\Lambda(f)$, respectively. For $t\in\Rb$ and $\xi\in\Sb^1$,
set
\[
\tilde Y_{\xi,t}:=
\left\{f\in V_0\ ;\ \chi_{r_0(f)}=\xi\chi_{p(f)}^t\right\}.
\]

Let $\Sigma_0\subset\Sigma$ be the set of words eventually equal to $-1$.
A point $x\in W^s_{\Lambda(f),loc}$ is eventually sent to
$W^s_{r_0(f),loc}$ if and only if $x\in W^s_\omega(f)$ for some
$\omega\in\Sigma_0$. Since $\Sigma_0$ is dense in $\Sigma$, and since
the blender property gives
\[
V_0=\bigcup_{\omega\in\Sigma}\tilde X_\omega,
\]
we obtain
\[
V_0\subset
\overline{\bigcup_{\omega\in\Sigma_0}\tilde X_\omega}.
\]
This already gives the first two points of
Assumption~\ref{init}. For the third one, it is enough to show that,
possibly after shrinking $V_0$, every regular point of a hypersurface
$\tilde Y_{\xi,t}$ is crossed by some $\tilde X_\omega$ in a proper way.
A small perturbation of $\omega$ inside $\Sigma_0$ then gives the desired
statement.

We first claim that there exists a neighborhood $\tilde V_0$ of
$\tilde M_0$ such that every irreducible component of some
$\tilde X_\omega$ meeting $\tilde V_0$ also meets $\tilde M_0$.
Otherwise, one could find a sequence $(\omega_n)$, which can be assumed to converge to some $\omega_\infty$ by compactness, and irreducible components
$\tilde Z_{\omega_n}\subset\tilde X_{\omega_n}$, approaching
$\tilde M_0$ without intersecting it. By persistence of proper intersections,
$\tilde X_{\omega_\infty}$ cannot intersect $\tilde M_0$ properly, hence
$\tilde M_0\subset \tilde X_{\omega_\infty}$. Using
\eqref{eq-xomega}, this would imply that $M_0\subset X_{\omega_\infty}$,
again contradicting Lemma~\ref{le-init-dim2}.

Now fix $(\xi,t)\in\Sb^1\times\Rb$ and let
$f\in\tilde V_0\cap\tilde Y_{\xi,t}$ be a regular point of
$\tilde Y_{\xi,t}$. By Lemma~\ref{le-blender-selle},
$W^s_{\Lambda(f),loc}$ intersects every $u$-graph in
$\Db\times\Bb^s\times V_+^u$ and in
$\Db\times\Bb^s\times V_-^u$. Since $W^u_{p(f),loc}$ is a $u$-graph in
$\Db\times\Bb^s\times\Bb^u_{3/2}$, it contains $u$-graphs in both
regions. Therefore there exist at least two different words
$\omega_1,\omega_2\in\Sigma$ such that
\[
f\in \tilde X_{\omega_1}\cap \tilde X_{\omega_2}.
\]

Assume that both intersections
$\tilde X_{\omega_i}\cap\tilde Y_{\xi,t}$ are non-proper at $f$.
Then there exist irreducible components
$\tilde Z_{\omega_i}\subset\tilde X_{\omega_i}$ containing $f$ and
contained in $\tilde Y_{\xi,t}$. Since $f$ is a smooth point of
$\tilde Y_{\xi,t}$, we must have $\tilde Z_{\omega_1}=\tilde Z_{\omega_2}\subset \tilde Y_{\xi,t}.$ By the choice of $\tilde V_0$, this common component meets $\tilde M_0$,
hence, by \eqref{eq-xomega}, the corresponding components in $M_0$
coincide as well. This is impossible by Lemma~\ref{le-init-dim2} since $\omega_1\neq\omega_2$.
Therefore at least one of the two intersections is proper at $f$, which
gives the last point of Assumption~\ref{init}.

This concludes the proof after replacing $V_0$ by $\tilde V_0$.
\end{proof}

The last two assumptions are obtained by a final perturbation. The first
one relies on~\cite{jonsson-skew,dupont-t-int}, while the second one
follows from an explicit deformation of the stable coordinates.

\begin{lemma}\label{le-misu}
There exists a regular polynomial endomorphism $f_1\in V_0$ such that
Assumptions~\ref{misu} and~\ref{homoclinic} hold in a neighborhood of
$f_1$ in $\End_d^k$.
\end{lemma}

\begin{proof}
Fix $a$ and all the parameters chosen above. For $c\in\Cb^*$ sufficiently
small, set
\[
q_c(w):=q_a(w)+cw^d.
\]
The fixed point $\tilde w(a)$ has a holomorphic continuation
$\tilde w_c$ satisfying
\[
q_c(\tilde w_c)=\tilde w_c.
\]
Consider the regular polynomial endomorphism
\[
f_c(z,x,y,w)
=
\left(
\begin{array}{c}
\alpha z+\epsilon w+\beta zw
+\tau\cdot x(w-\tilde w_c)+cz^d\\[0.2em]
(\gamma_i x_i+cx_i^d)_{1\leq i\leq s}\\[0.2em]
\left(\sigma_j\left(y_j+\dfrac{w-\tilde w_c}{6}\right)
+cy_j^d\right)_{1\leq j\leq u-1}\\[0.2em]
q_c(w)
\end{array}
\right).
\]

The plane $\{x=0\}$ is invariant under $f_c$ and contains
$\Lambda(f_c)$ together with all its unstable manifolds. Identifying it
with $\Cb^{u+1}$, the restriction of $f_c$ takes the form
\[
g(z,y,w)
=
\bigl(g_0(z,w),g_1(y_1,w),\ldots,g_{u-1}(y_{u-1},w),g_u(w)\bigr)
\]
and preserves the sequence of fibrations
\[
\Cb^{u+1}\to\Cb^u\to\cdots\to\Cb.
\]
Applying successively~\cite[Corollary~1.2]{dupont-t-int}, all repelling
periodic points of $g$ belong to the support of its equilibrium measure
$\mu_g$. In particular,
\[
[W^u_{\hat t,loc}]\wedge\mu_g\neq0
\qquad
\text{for every }\hat t\in\hat\Lambda(f_c).
\]
Moreover,
\[
(G_{f_c})|_{\{x=0\}}=G_g,
\qquad
\mu_g=(\ddc G_g)^{\wedge(u+1)}.
\]
Thus, writing $T_{f_c}:=\ddc G_{f_c}$,
\[
[W^u_{\hat t,loc}]\wedge T_{f_c}^{\wedge(u+1)}
=
[W^u_{\hat t,loc}]\wedge\mu_g\neq0.
\]
Hence Assumption~\ref{misu} holds in a neighborhood of $f_c$.

For Assumption~\ref{homoclinic}, we now set
\begin{equation}\label{eq-fcA}
f_{c,A}
:=
f_c+A(w-\tilde w_c)^2(1,\mathbf 1_s,0,0),
\qquad
\mathbf 1_s:=(1,\ldots,1),
\end{equation}
where $A\neq0$ will be chosen arbitrarily small. Since this perturbation
vanishes to first order on $\{w=\tilde w_c\}$, the fixed point $p$ and
its multipliers do not depend on $A$. In particular, the bunching conditions in Assumption~\ref{homoclinic}
remain satisfied, and the $z$-axis and the $x_i$-axes remain the coordinate
axes in the linearizing chart of $W^s_{p,loc}$ since the dynamics on it is simply a product dynamics.

Let $g_{c,+}$ and $g_{c,-}$ be the continuations of the inverse branches
$g_+$ and $g_-$ of $q_a$. We set
\[
m_2:=g_{c,+}(\tilde w_c),
\qquad
m_1:=g_{c,-}(m_2).
\]
Hence
\[
q_c(m_1)=m_2,
\qquad
q_c(m_2)=\tilde w_c,
\]
with $m_1,\tilde w_c\in U_-$ and $m_2\in U_+$. By
Lemma~\ref{le-U}, after decreasing $|c|$ if necessary,
\begin{equation}\label{eq-m1m2}
|m_1-\tilde w_c|\leq\frac6{|a|},
\qquad
|m_2-\tilde w_c|\geq2-\frac6{|a|}.
\end{equation}
A direct estimate on $q_c'$ also gives
\begin{equation}\label{eq-gc}
|g_{c,-}(w)-\tilde w_c|
\leq\frac2{|a|}|w-\tilde w_c|
\qquad\text{on }U_-.
\end{equation}

The problem is independent of $y$ so we simply ignore these coordinates in what follows.

The $y$-coordinates play no role in the argument and will be omitted
from now on. Since the $z$-, $x$- and $w$-coordinates of $f_{c,A}$ do
not depend on $y$, the corresponding coordinates of
$W^u_{p,loc}$ depend only on $w$. We denote them by $z_A(w)$ and
$x_A(w)$, and set
\[
\tilde q_A:=\bigl(z_A(m_1),x_A(m_1),m_1\bigr)\in W^u_{p(f_{c,A}),loc}.
\]
Since $q_c^2(m_1)=\tilde w_c$, we write
\[
q_A:=f_{c,A}^2(\tilde q_A)
=
\bigl(z_A^{(2)},x_A^{(2)},\tilde w_c\bigr)\in W^s_{p(f_{c,A}),loc}\cap W^u_{p(f_{c,A})}.
\]

We show that all the stable coordinates of $q_A$ are nonzero. The
$x$-coordinates are simpler, so we only detail the $z$-coordinate.
For $c$ small, the $x$-maps are $1/2$-Lipschitz and the $z$-map is
$2$-Lipschitz on the polydiscs under consideration. We also use
$\|\tau\|_1<1$, which follows from our previous choice of $\delta$.

The invariance equation for the $x$-components and~\eqref{eq-gc}
give
\begin{equation}\label{eq-xA}
\|x_A(w)\|_\infty
\leq
\frac{5|A|}{|a|^2}|w-\tilde w_c|^2.
\end{equation}
Indeed, if
\[
M_A^x:=
\sup_{w\neq\tilde w_c}
\frac{\|x_A(w)\|_\infty}
{|A||w-\tilde w_c|^2},
\]
then
\[
M_A^x
\leq
\frac2{|a|^2}M_A^x+\frac4{|a|^2}.
\]

Set
\[
u_A(w):=z_A(w)-z_0(w)
\]
and
\[
M_A^z:=
\sup_{w\neq\tilde w_c}
\frac{|u_A(w)|}
{|A||w-\tilde w_c|^2}.
\]
Using the invariance equation, \eqref{eq-gc}, \eqref{eq-xA} and
$|w-\tilde w_c|\leq6|a|^{-1}$ on $U_-$, we obtain
\[
M_A^z
\leq
\frac8{|a|^2}M_A^z
+\frac4{|a|^2}
+\frac{240}{|a|^6}.
\]
Since $|a|>100$,
\begin{equation}\label{eq-zA}
M_A^z\leq\frac5{|a|^2}.
\end{equation}

Let $z_A^{(1)},z_A^{(2)}$ and $x_A^{(1)},x_A^{(2)}$ denote the
coordinates after the first and second iterates of $\tilde q_A$.
Equations~\eqref{eq-m1m2}, \eqref{eq-xA} and \eqref{eq-zA} give,
with deliberately coarse constants,
\[
|z_A^{(1)}-z_0^{(1)}|
\leq\frac{40}{|a|^2}|A|,
\qquad
\|x_A^{(1)}\|_\infty
\leq\frac{40}{|a|^2}|A|.
\]
Consequently,
\begin{equation}\label{eq-zA-return}
\left|
z_A^{(2)}-z_0^{(2)}
-A(m_2-\tilde w_c)^2
\right|
\leq
\frac{300}{|a|^2}|A|.
\end{equation}

Let $z_p$ be the $z$-coordinate of $p$. If
$z_0^{(2)}\neq z_p$, then $z_A^{(2)}\neq z_p$ for $A$ sufficiently
small by continuity. If $z_0^{(2)}=z_p$, then
\eqref{eq-m1m2} and~\eqref{eq-zA-return} give
\[
|z_A^{(2)}-z_p|
\geq
|A|
\left[
\left(2-\frac6{|a|}\right)^2
-\frac{300}{|a|^2}
\right]>0.
\]
Thus
\[
z_A^{(2)}\neq z_p.
\]

For every $1\leq i\leq s$, the same computation, without the coupling
term involving $\tau$, gives
\[
|x_{i,A}^{(2)}|
\geq
|A|
\left[
\left(2-\frac6{|a|}\right)^2
-\frac{20}{|a|^2}
\right]>0.
\]

Finally, the orbit segment from $\tilde q_A$ to $q_A$ stays in the
domains where the map is locally biholomorphic, and the cone conditions
give the transversality of
\[
f_{c,A}^2(W^u_{p,loc})
\quad\text{and}\quad
W^s_{p,loc}
\]
at $q_A$. Therefore $q_A$ is the homoclinic intersection required in
Assumption~\ref{homoclinic}.

Since $A$ can be chosen arbitrarily small and Assumption~\ref{misu} is
open, both Assumptions~\ref{misu} and~\ref{homoclinic} hold in a
neighborhood of $f_1:=f_{c,A}$.
\end{proof}

\begin{proof}[Proof of Theorem~\ref{th-existence}]
Let $f_1$ be given by Lemma~\ref{le-misu}. Taking a sufficiently
small neighborhood $\Omega_u$ of $f_1$ in $\End_d^k$ (or in
$\Poly_d^k$), contained in $V_0$, Assumptions~\ref{misu} and
\ref{homoclinic} hold throughout $\Omega_u$.

All the remaining assumptions also hold on $\Omega_u$:
Assumptions~\ref{verti}, \ref{hyperbolic} and~\ref{repelling} by
Lemma~\ref{le-ass1et2}, Assumption~\ref{saddle} by
Lemma~\ref{le-ass-saddle}, Assumption~\ref{blender} by
Lemma~\ref{le-blender-selle}, Assumptions~\ref{free} and~\ref{init} by
Lemma~\ref{le-init}, Assumption~\ref{decalage} by
Lemma~\ref{le-decalage}, and Assumption~\ref{grille} by
Lemma~\ref{le-grille-2}. This completes the proof.
\end{proof}

\section{The case of curves in $\Pb^2$: a degeneration phenomenon}\label{sec-chasse}
This section is devoted to the proof of Theorem~\ref{th-chasse-int}. To that
end, let $\mathcal Z\subset \mathcal X_D^1\times\mathrm{End}_d^2$ be an
irreducible algebraic subvariety whose projection onto $\mathrm{End}_d^2$ is
dominant. Our argument relies on suitable degenerations in the space
$\mathrm{Rat}_{d}^2$ of rational self-maps of $\Pb^2$ of degree $d$. We use an explicit family in
$\mathrm{Rat}_{d}^2\setminus\mathrm{End}_d^2$, whose expression
is simple when $d=2$. The same strategy applies in
higher degree, but requires a more carefully chosen family to control
the subsequent perturbation into $\mathrm{End}_d^2$. We therefore
present the quadratic case in detail first, and describe the
modifications needed for $d\geq3$ in Section~\ref{sec-chasse-higher}. In both cases,
the final argument concerning the small Julia set is deferred
to Section~\ref{sec-chasse-julia}.

For the quadratic case, we start from the map $f_0$ given by Theorem~\ref{th-existence} in the case
$k=2$,
\[
f_0(z,w)
=
\bigl(\alpha z+\epsilon w+\beta zw,\,q_a(w)\bigr),
\qquad
q_a(w):=a(w^2-1),
\]
where the parameters $\alpha,\epsilon,\beta$ satisfy \eqref{def-M} and
$|a|$ is sufficiently large. As in Theorem~\ref{th-existence}, we denote by
$\tilde w(a)$ the fixed point of $q_a$ close to $-1$. We deform $f_0$
by adding terms vanishing on the invariant line
$\{w=\tilde w(a)\}$, given, in homogeneous coordinates, by
\[
\ell_a(w,t):=w-\tilde w(a)t.
\]

The main result of this section is the following statement, which concerns
the case $d=2$ (see Theorem~\ref{th-chasse-d3} for $d\geq3$).

\begin{theorem}\label{th-chasse-d2}
Assume that $d=2$. There exist arbitrarily large values of $a$ and arbitrarily small $A,B\in\Cb$ such that the rational self-map of $\Pb^2$
\[
f_{\lambda}[z:w:t]
=
\bigl[
\alpha zt+\epsilon wt+\beta zw:
a(w^2-t^2)+Az\ell_a(w,t):
t^2+Bw\ell_a(w,t)
\bigr],
\]
with $\lambda=(a,A,B),$ has the following property. There exists $(Z,f)\in\mathcal Z$, with
$f\in\mathrm{End}_2^2$ arbitrarily close to $f_{\lambda}$ in
$\mathrm{Rat}_{2}^2$, such that
\begin{enumerate}[label=(\roman*)]
\item\label{it-chasse-1}
$f$ satisfies all the conditions of Section~\ref{sec-ass} for $k=2$,

\item\label{it-chasse-2}
there exists $n\geq0$ such that $f^n(Z)$ intersects
$W^s_{p(f),\mathrm{loc}}$ transversely, where $p(f)$ is given by
Assumption~\ref{saddle}.
\end{enumerate}
\end{theorem}
The proof involves degeneration issues in establishing both
\ref{it-chasse-1} and \ref{it-chasse-2}. Let us first observe that
\ref{it-chasse-1} already almost holds in a neighborhood of $f_0$.
Indeed, all the results of Section~\ref{sec-dimk} hold in a neighborhood of
$f_0$, except for Lemma~\ref{le-misu}, which yields
Assumptions~\ref{misu} and~\ref{homoclinic}. The latter can be established rather
easily when $k=2$ (see Lemma~\ref{le-homof0} and Remark~\ref{rk-homoclinic}), but it remains to establish Assumption~\ref{misu}, which in
dimension $2$ amounts to showing that the repeller $\Lambda(f)$ is contained in the small Julia set $J_2(f)$.

So far, this issue was handled using maps preserving a fibration, but this is
no longer the case for $f_{\lambda}$ as soon as $\lambda=(a,A,B)$ with $A\neq0$. To overcome this difficulty,
we use a rescaling as $|a|\to\infty$ under which the inverse dynamics of $f_{(a,A,0)}$
degenerates to a correspondence $\Gamma$ on $\Pb^1$ with $d_0(\Gamma)=2$ images and $d_1(\Gamma)=3$ preimages.
The results of Dinh~\cite{dinh-laminar}, Dinh-Sibony~\cite{DS-MT} and
Bharali-Sridharan~\cite{BS-correspondences}, based on a method introduced by Lyubich in \cite{lyubich:equi}, then provide sufficiently many
contracting inverse branches related to $\Lambda(f_{\lambda})$ to ensure, by a criterion of de Th\'elin and Dinh \cite{deThelin05,Dinh07}, that $\Lambda(f)\subset J_2(f)$ for all $f\in\mathrm{End}_2^2$ close enough to $f_{\lambda}.$

\bigskip

The main difficulty is to establish item \ref{it-chasse-2} in Theorem~\ref{th-chasse-d2}. To that end, we consider the degeneration of $\mathcal Z$ above the family $\mathcal F:=(f_{\lambda})_{\lambda\in\Cb^3}$ where it turns out that the statement reduces to a rigid intersection phenomenon for
algebraic curves, which we can rule out by dynamical and monodromy arguments. To be more precise, let still denote by $\mathcal Z$ its closure in $\mathcal X_D^1\times\mathrm{Rat}_{2}^2$ and let $\mathcal Z_{\mathcal F}$ be its restriction above $\mathcal F.$ Our strategy is then as follows. Every map $f$ in this family preserves a line $L_f$ and acts affinely on it. The restriction $f_{|L_f}$ has two fixed points: the indeterminacy point $I$ of $f$, which lies on the line at infinity $L_\infty$, and a second fixed point, corresponding to $p(f)$ in Assumption~\ref{saddle}. As $f_{L_f}$ is affine, when $p(f)$ is a saddle point whose contracting eigendirection
is tangent to $L_f$, its stable manifold is a whole line punctured at the
indeterminacy point
\[
  W^s_{p(f)}=L_f\setminus\{I\}.
\]

Assume now that, for some $(Z,f)\in\mathcal Z_{\mathcal F}$, no forward image of $Z$ meets
$W^s_{p(f)}$. Then, by B\'ezout, 
\begin{equation}\label{eq-max-int}
  f^n(Z)\cap L_f=\{I\}\ \text{ for every }\ n\geq 0
\end{equation}
with a maximal multiplicity. Observe that, throughout this section, images and preimages of subvarieties are understood as strict transforms and that an important difficulty is that some maps \(f\in\mathcal F\) contract curves.

This constraint \eqref{eq-max-int} may be viewed as a dynamical Mordell-Lang-type condition on the backward orbit of $I$, suggesting that $Z$ should be preperiodic. Actually, we first prove that, for a map $f_1$ in this family for which $L_\infty$ is totally
invariant, \eqref{eq-max-int} can only hold for $Z=Z_1:=L_\infty$ (see Proposition \ref{prop-rigidity}). We then show in Proposition \ref{prop-DMM}
that if \eqref{eq-max-int} held for every pair $(Z,f)\in\mathcal Z_{\mathcal F}$ in a
neighborhood of $(Z_1,f_1)$, a dynamical argument would force some iterate of $Z$
to be invariant under $f$. A contradiction then follows from a monodromy argument in Theorem \ref{th-inter}.

Consequently, there exist a pair $(Z,f)\in\mathcal Z_{\mathcal F}$ and an integer $n\geq 0$
such that
\[
  f^n(Z)\cap W^s_{p(f),{loc}}\neq\varnothing.
\]
It then remains to upgrade this intersection into a transverse one, by combining a dynamical argument with results of Fakhruddin~\cite{fakhruddin-generic},
Bedford-Lyubich-Smillie~\cite{BedfordLyubichSmillie2} and Dujardin-Favre-Ruggiero~\cite{DFR}.

\begin{rk}
Although this reduction of \ref{it-chasse-2} may look elementary, our proof is quite specific to
the present setting. So far, we do not know how to extend it to higher dimensions (in
particular to subvarieties of higher codimension), nor to the polynomial locus
$\mathrm{Poly}_d^2$, where the strategy breaks down.
\end{rk}

The rest of this section is organized as follows. We begin by briefly come back to
the model map in Section~\ref{sec-f_0} and establishing key properties of
$f_\lambda$ in Section~\ref{sec-break}. The core of the argument is the proof of
Proposition~\ref{prop-DMM}, given in Section~\ref{sec-rigidity}, which is
then used in Section~\ref{sec-proper} to prove \ref{it-chasse-2} for a generic
element $f_{\lambda}$ of the family. In Section~\ref{sec-transversal}, we
perturb $f_{\lambda}$ inside $\mathrm{End}_2^2$ in order to prove
Theorems~\ref{th-chasse-d2}. We then explain the adequat moditications to dead with higher degrees in Section~\ref{sec-chasse-higher}. In both cases, we use a result about the small Julia set that we prove in Section~\ref{sec-chasse-julia}.

\subsection{Further properties of the model maps}\label{sec-f_0}
We start with the two-dimensional model map from Section~\ref{sec-dim2}. Observe that in what follows, the parameter $\lambda=(a,A,B)$ has not the same role that in that section. Here, we note
\[
 f_a(z,w)=\bigl(\alpha z+\epsilon w+\beta zw,\,q_a(w)\bigr),
 \qquad q_a(w)=a(w^2-1),
\]
where $|a|$ is large and the parameters $\alpha,\beta,\epsilon$ satisfy the conditions \eqref{def-M} imposed in Section~\ref{sec-dim2}.  An important point is that, throughout this section, the parameters \(\alpha,\beta\), and \(\epsilon\) are fixed, whereas \(a\) will vary. In particular, we set
$$
c_a:=\tilde w(a),
$$
to be the fixed point of \(q_a\) close to \(-1\), and put
\[
 \rho_a:=\alpha+\beta c_a.
\]
By our choice of parameters, the point
\[
 p(f_a):=\left(z_a,c_a\right)\ \text{ where }\ z_a:=\frac{\epsilon c_a}{1-\rho_a}
\]
is a saddle fixed point of $f_a$, with
\[
 W^s_{p(f_a)}=\{w=c_a\}.
\]
Indeed, the line $\{w=c_a\}$ is invariant, and the restriction
of $f_a$ to this line is the affine map
\[
 z\mapsto \rho_a z+\epsilon c_a,
\]
while the transverse multiplier is
$q_a'(c_a)=2ac_a$, whose modulus is much larger than $1$.

We shall use the following elementary consequence of the construction
in Section~\ref{sec-dimk} both to promote an intersection to a
transverse one and to obtain the homoclinic intersection required
in Assumption~\ref{homoclinic}.

\begin{lemma}\label{le-homof0}
For $|a|$ large enough, $p(f_a)$ admits a transverse homoclinic satisfying Assumption~\ref{homoclinic} with $r=1$.
\end{lemma}
\begin{proof}
Let $g_-$ be the inverse branch
of $q_a$ fixing $c_a$. Since $g_-(w)=c_a+O(|a|^{-1})$ on $\Db_2$
and $|\rho_a|=|\alpha+\beta c_a|<1$, the local unstable manifold at $p(f_a)$ is given by a graph
\[
 \{(\varphi_a(w),w)\,:\,w\in\Db_2\}\subset W^u_p,\ \text{ with } \varphi_a(w)=z_a+O(|a|^{-1}).
\]
As $q_a(-c_a)=c_a$, the point $x_a=f_a(\varphi_a(-c_a),-c_a)$
belongs to $W^u_{p(f_a)}\cap W^s_{p(f_a)}$. This intersection is transverse
because $q_a'(-c_a)\neq0$. Moreover, the Jacobian of $f_a$ at $(\varphi_a(-c_a),-c_a)$ is $(\alpha-\beta c_a)q_a'(-c_a)\neq0$.
Finally, using
$z_a=\epsilon c_a/(1-\rho_a)$ and $c_a=-1+O(|a|^{-1})$,
we obtain
\[
 \pi_z(x_a)-z_a
 =(\rho_a-2\beta c_a)\varphi_a(-c_a)-\epsilon c_a-z_a
 =\frac{2\epsilon c_a(\alpha-1)}{1-\rho_a}
   +O(|a|^{-1}).
\]
Since $\epsilon(1-\alpha)\neq0$, we have $x_a\neq p(f_a)$
for $|a|$ sufficiently large.
\end{proof}

\begin{rk}\label{rk-homoclinic}\normalfont
Observe that in our setting Assumption~\ref{homoclinic} is a little bit artificial. Indeed, the saddle point $p(f_a)$ has a single stable multiplier. Therefore, the assumptions on the multipliers in Proposition~\ref{prop-zariski-dense-unstable} are automatically satisfied with $r=1$, by Assumption~\ref{saddle}. Moreover, in dimension $2$, the Zariski density of $W^u_{p(f)}$ can in fact be established by a much simpler argument.
\end{rk}

\subsection{Breaking the pencil}\label{sec-break}
We now perturb \(f_a\) inside a family of rational maps which does not, in general, preserve the pencil of lines \(\{w=\mathrm{cst}\}\). The parameter $A$ breaks the pencil, while $B$ varies the induced dynamics on its base and will be used in the final monodromy argument.

Put
\[
 Q_a(w,t):=a(w^2-t^2),\qquad
 \ell_a(w,t):=w-c_at.
\]
For $\lambda=(a,A,B)\in\Cb^3$ we consider
\begin{equation}\label{eq-family}
 f_\lambda[z:w:t]
 =
 \left[
 \alpha zt+\epsilon wt+\beta zw:
 Q_a(w,t)+Az\ell_a(w,t):
 t^2+Bw\ell_a(w,t)
 \right].
\end{equation}
The model maps above correspond to $\lambda_a:=(a,0,0).$ The line
\[
 L_a:=\{\ell_a=0\}
\]
is invariant for every $\lambda=(a,A,B)$, and
\[
 I:=[1:0:0]=L_a\cap L_\infty
\]
is an indeterminacy point for these maps.

We first have the following basic observations.
\begin{lemma}\label{le-stable}
For $\lambda=(a,A,B)$ sufficiently close
to $\lambda_a$, the map $f_\lambda$ has a saddle fixed point
$p(f_\lambda)$ satisfying
\[
 W^s_{p(f_\lambda)}=L_a\setminus\{I\}.
\]
Moreover, $f_\lambda$ preserves the pencil of lines through $I$
when $A=0$.
\end{lemma}
\begin{proof}
The line $L_a$ is invariant, and the restriction of $f_\lambda$ to
$L_a\setminus\{I\}$ is the affine map
\[
 z\mapsto \rho_a z+\epsilon c_a.
\]
Since $0<|\rho_a|<1$, this map has a unique attracting fixed point.
For $\lambda$ sufficiently close to $\lambda_a$, its transverse
multiplier remains of modulus greater than one, giving the first
assertion.

If $A=0$, the last two homogeneous coordinates of $f_\lambda$
depend only on $(w,t)$. Hence $f_\lambda$ induces a rational map
on the base of the pencil of lines through $I$.
\end{proof}

We now introduce, when $A\neq0$, the following coordinates which will be used throughout most of this
section,
\begin{equation}\label{eq-HUT}
 U:=\ell_a=w-c_at,\qquad
 T:=t,\qquad
 H:=Az+a(w+c_at)-c_aBw.
\end{equation}
They depend on $\lambda$, but we suppress this dependence from the
notation unless it is relevant. The choice of $H$ is made so that, using $a(c_a^2-1)=c_a$, we have
\begin{equation}\label{eq-U-global}
 U\circ f_\lambda=UH.
\end{equation}
Since $A\neq0$, the linear forms $(H,U,T)$ are homogeneous coordinates
on $\Pb^2$.  If $F_\lambda$ corresponds to $f_\lambda$ in these coordinates then we have
\begin{equation}\label{eq-F-global}
 F_\lambda[H:U:T]
 =
 \left[
 R_\lambda(H,U,T):
 UH:
 T^2+B(U+c_aT)U
 \right],
\end{equation}
where $R_\lambda$ is a homogeneous quadratic form.  Except in the special
case $B=0$, its explicit expression will not be needed.

Actually, we need a better understanding of these maps when $B=0.$ For what follows, we denote
$$F:=F_\lambda\ \text{ with }\ \lambda=(a,A,0),$$
where $A$ will take a special value $A_*$ in Section~\ref{sec-rigidity}.
\begin{lemma}\label{le-fixed}
The line $L_\infty$ is totally invariant and $I$ is the only indeterminacy point. Blowing up $I$
in the source resolves $F$ as a morphism
\[
 \pi:X=\mathrm{Bl}_I\Pb^2\to\Pb^2,
 \qquad
 \hat F:X\to\Pb^2.
\]
The morphism $\hat F$ is finite and maps the exceptional divisor
isomorphically onto $L_\infty$. In particular, $F$ contracts no curve. Its topological degree
is $3$, and, for every irreducible curve $C\subset\Pb^2$,
\begin{equation}\label{eq-degree-curve}
 e_C\deg F(C)=2\deg C-\mathrm{mult}_I(C),
 \qquad e_C=\deg(F_{|C}).
\end{equation}
The restriction to $L_\infty\setminus\{I\}$ extends to the homography
\begin{equation}\label{eq-infinity-d2}
 [H:U]\mapsto[(a+\beta)H-a\beta U:H].
\end{equation}
Its fixed points $q=[\beta:1:0]$ and $s=[a:1:0]$ are respectively
saddle and attracting for $F$.
\end{lemma}

\begin{proof}
Since $B=0$, the last homogeneous coordinate of $F$ is $T^2$.
Thus there is no indeterminacy in $\Cb^2$, and $L_\infty$ is totally
invariant. On $T=0$, the first two coordinates are
\[
 U\bigl((a+\beta)H-a\beta U\bigr)
 \qquad\text{and}\qquad UH.
\]
As $a\beta\neq0$, their only common zero on $L_\infty$ is $I$.

In the chart $H=1$ centered at $I=[1:0:0]$, a first-order expansion
of \eqref{eq-family} in the coordinates \eqref{eq-HUT} gives
\[
 F[1:U:T]
 =
 \left[
 (a+\beta)U+\rho_a T+O\bigl(\|(U,T)\|^2\bigr):
 U:T^2
 \right].
\]
On the blowup, writing $U=\sigma u$ and $T=\sigma v$ and cancelling
the common factor $\sigma$, we obtain
\[
 \left[
 (a+\beta)u+\rho_a v+O(\sigma):
 u:\sigma v^2
 \right].
\]
Consequently, the exceptional divisor $E=\pi^{-1}(I)$ is mapped by
\[
 [u:v]\mapsto[(a+\beta)u+\rho_a v:u:0].
\]
Since $\rho_a\neq0$, this is a bijection onto $L_\infty.$ Thus no
indeterminacy remains after this single blowup of the source.

Now, let \(\mathcal L\) denote the class of a line in \(\Pb^2\), and set $\mathcal H=\pi^*\mathcal L.$ Since the homogeneous coordinates of \(F\) have degree \(2\), their pullbacks define sections of the class \(2\mathcal H\). The preceding computation shows that these sections have a common exceptional factor \(\tau\), vanishing to order one along \(E=\{\tau=0\}\). After cancelling this common factor, the resulting class is \(2\mathcal H-E\). Hence
\begin{equation}\label{eq-resolution-class}
\hat F^*\mathcal L=2\mathcal H-E.
\end{equation}

On the other hand, let $C\subset\Pb^2$ be an irreducible curve, and write
$D=\deg C$ and $m=\mathrm{mult}_I(C)$.
The total transform formula
\[
 \pi^*C=\hat C+mE
\]
gives $\hat C=D\mathcal H-mE$.
Using
\[
 \mathcal H^2=1,\qquad \mathcal H\cdot E=0,\qquad E^2=-1,
\]
we obtain
\[
 \hat F^*\mathcal L\cdot\hat C
 =(2\mathcal H-E)\cdot(D\mathcal H-mE)
 =2D-m\geq D>0,
\]
since $m\leq D.$ By the projection formula (see e.g. \cite[\S1.3]{DF01}),
\[
 e_C\deg F(C)
 =\mathcal L\cdot\hat F_*\hat C
 =\hat F^*\mathcal L\cdot\hat C
 =2\deg C-\mathrm{mult}_I(C).
\]
Thus, $\hat F$ contracts no curve different from $E$ and we have already seen that $E$ is sent on $L_\infty.$ Therefore, $\hat F$ is finite and the same holds for $F.$

Moreover, by \cite[Corollary~3.4]{DF01}, the topological degree is
\begin{equation}\label{eq-topological-degree}
(\hat F^*\mathcal L)^2
 =(2\mathcal H-E)^2=3.
\end{equation}

Finally, cancelling the factor $U$ in the restriction to $T=0$
gives~\eqref{eq-infinity-d2}. The tangent multipliers at $q$
and $s$ are $a/\beta$ and $\beta/a$, respectively, while both
transverse multipliers are zero. Since $|a|>|\beta|$, the points
$q$ and $s$ are respectively saddle and attracting for $F$.
\end{proof}

\begin{rk}
The total invariance of $L_\infty$ does not by itself imply that the blow-up
of the map at an indeterminacy point contained in $L_\infty$ is equal to
$L_\infty$. This is already illustrated by the case $A=B=0$.
\end{rk}
\begin{rk}\label{rk-contract}
Since the set of degree-$d$ rational maps that do not contract any curve is Zariski open, Lemma~\ref{le-fixed} implies that a generic map $f_\lambda$ in our family does not contract any curve.
\end{rk}

From this, we want to understand the set of subvarieties $Z$ that satisfy \eqref{eq-max-int} with $F$, i.e., such that
$$\text{for every }\ n\geq0,\ F^n(Z)\cap L_a=\{I\}.$$
To that end, we need a better understanding of the preimages of $I$ and of the strict transforms $F^{-n}(L_a).$ We then set $\Delta_0:=L_a$ and for $n\geq1$, if
\[
 H_n:=H\circ F^{n-1},
\]
then we denote by $\Delta_n:=\{H_n=0\}$. A case of particular importance is \(n=1\), for which \(\Delta_1\) is a line. Since it will often be used together with \(L_a\), we set
$$
\tilde L_\lambda:=\Delta_1.
$$
As $F$ contracts no curve, no curve
is mapped into its indeterminacy set, and no cancellation occurs in the
iterates. Hence
\begin{equation}\label{eq-pullback}
 U\circ F^n=U\prod_{j=1}^nH_j,
 \qquad
 (F^n)^*L_a=\Delta_0+\Delta_1+\cdots+\Delta_n.
\end{equation}

On the other hand, the fixed points in Lemma~\ref{le-fixed},
\[
 q=[\beta:1:0],\qquad s=[a:1:0].
\]
which are respectively repelling and  attracting on $L_\infty$ are saddle and attracting in $\Pb^2$, with $W^u_{q,{loc}}\subset L_\infty.$ Let
\[
 x_n=(F_{|L_\infty})^{-n}(I),\qquad n\geq1.
\]
The points $x_n$ are distinct and converge to $q$. Moreover, the point
\[
 x_1=[0:1:0]= \tilde L_\lambda\cap L_\infty
\]
which depends also on $A$, will play a special role in what follows. We have that $x_n\in\Delta_n\cap L_\infty$ but, although we will not use it, observe that for $n\geq2$, intersection is more complicated than one might expect and actually satisfies, set-theoretically,
$$\Delta_n\cap L_\infty=\{I,x_1,\ldots,x_{n-2},x_n\}.$$
The next lemma deals with the components of $\Delta_n$.
\begin{lemma}\label{le-delta}
If $C$ is an irreducible component of $\Delta_n$, then for every
$m>n$ it has a nonempty proper intersection with $\Delta_m$ in $\Cb^2$.
\end{lemma}
\begin{proof}
We have $F^n(C)=L_a$ and $F^{m-1}(C)=L_a\neq \tilde L_\lambda$, so $C$ is not
a component of $\Delta_m$. Let $\xi=L_a\cap \tilde L_\lambda=[0:0:1]$ and
$N=m-n-1$. The affine action on $L_a$ is an automorphism. Choose
$y\in L_a\cap\Cb^2$ with $F^N(y)=\xi$. The finite proper map
$F^n:C\cap\Cb^2\to L_a\cap\Cb^2$ is onto, so $y$ has an affine
preimage $x\in C$. Then $F^{m-1}(x)=\xi\in \tilde L_\lambda$, as required.
\end{proof}
The fact that $F$ contracts no curve is important for what follow but we will also use in a crucial way the following consequence of Lemma~\ref{le-fixed}.
\begin{corollary}\label{cor-fixed}
For any $n\geq1,$ the set of $n$-periodic points of $F$ is finite.
\end{corollary}
\begin{proof}
Let $C$ be an irreducible curve of $n$-periodic points. Set $C_j=F^j(C)$ and $D_j=\deg Y_j$, for $0\leq j\leq n$.
Since $F^n_{C}$ is the identity, we have that $e_{C_j}$ is equal to $1.$ Hence, by~\eqref{eq-degree-curve},
\[
 D_{j+1}=2D_j-\mathrm{mult}_I(C_j)\geq D_j.
\]
By periodicity, we must have equality, hence $\mathrm{mult}_I(C)=\deg C$. Therefore, $C$ has to be a line through $I.$ However, $\tilde L_\lambda$ intersects any such line and contains no periodic points.\end{proof}

\subsection{Rigidity by finiteness.}\label{sec-rigidity}
The first goal of this section is to show that a curve $Z$ whose
forward images under $F$ all avoid $W^s_{p(F)}$ must coincide with
$L_\infty$. To this end, we seek three distinct lines through $I$
avoided by every $F^n(Z\setminus L_\infty)$. Projection from $I$
will then allow us to use the finiteness of nonconstant regular maps from $Z$, punctured above $L_\infty$,
to a fixed thrice-punctured $\Pb^1$.

Two such lines are already available, $L_a$ by the above dynamical assumption and $L_\infty$, since it is totally invariant. In degrees $d\geq3$, a third
line will arise directly from the construction. In degree $2$,
however, we choose a special parameter $\lambda_*=(a,A_*,0)$
for which the conic $\Delta_2$ splits into two lines, one of which
passes through $I$. This line is mapped into $\tilde L_a$,
and hence into $L_a$ after one further iteration, so it must
also be avoided.

Indeed, for $B=0$, the divisor $\Delta_2$ is defined in the original
affine coordinates $(z,w)$ by
\[
 A\bigl((a+\beta)w+\alpha-ac_a\bigr)z
 +a^2(w^2-1)+A\epsilon w+ac_a=0.
\]
For $A\neq0$, requiring the line $D_r:=\{w=rt\}$ to be a component
amounts to imposing
\[
 (a+\beta)r+\alpha-ac_a=0,
 \qquad
 a^2(r^2-1)+A\epsilon r+ac_a=0.
\]
We therefore fix $\lambda_*=(a,A_*,0)$ by setting
\begin{equation}\label{eq-special-A}
 r=\frac{ac_a-\alpha}{a+\beta}
   =c_a-\frac{\rho_a}{a+\beta},
 \qquad
 A_*=-\frac{a^2(r^2-1)+ac_a}{\epsilon r}.
\end{equation}
Since $\rho_a\neq0$, we have $r\neq c_a$, so $L_\infty$, $L_a$,
and $D_r$ are three distinct lines of the pencil. Moreover, for $|a|$ sufficiently large,
\begin{equation}\label{eq-open}
 A_*\rho_a(a+\beta)r\neq0,
 \qquad A_*\epsilon=2a(\alpha-\beta-1)+O(1).
\end{equation}
In particular, $A_*$ is not assumed small.

Put as above $F=f_{\lambda_*}$,
using either the original coordinates or those in \eqref{eq-HUT},
and $\tilde L:=\tilde L_{\lambda_*}$, and recall that we have
\begin{equation}\label{eq-third-value}
 D_r\mapsto \tilde L\mapsto L_a.
\end{equation}

\begin{proposition}\label{prop-rigidity}
Let $Z\subset\Pb^2$ be an irreducible curve. Suppose that for every
$n\geq0$ and every irreducible component $D$ of $\Delta_n$, the curves
$Z$ and $D$ have no intersection in $\Cb^2$. Then $Z=L_\infty$.
Equivalently, the same conclusion holds if all strict forward images by $F^n$
of $Z$ avoid $L_a\cap\Cb^2$.
\end{proposition}
\begin{proof}

As $Z$ has no intersection in $\Cb^2$ with the components of $\Delta_n,$ $n\geq0,$ it cannot be one of these components.

Now, assume $Z\neq L_\infty$. Let $\nu\colon C\to Z$ be its normalization,
$\Sigma=\nu^{-1}(Z\cap L_\infty)$, and write
$(h_m,u_m)=F^m\circ\nu$ on $C\setminus\Sigma$, in the affine
coordinates $(H/T,U/T)$. By~\eqref{eq-third-value},
\[
 u_m:C\setminus\Sigma\to
      \Pb^1\setminus\{0,r-c,\infty\}.
\]
If one $u_m$ is constant then $F^m(Z\cap\Cb^2)$ is an entire horizontal line, which is impossible since in this case it meets $\tilde L$ in $\Cb^2$.
The finiteness of the set of holomorphic maps from an algebraic curve to  $\Pb^1\setminus\{0,r-c,\infty\}$ implies that $\{u_m\ ;\ m\geq0\}$ is
finite. Since $u_{m+1}=u_mh_m$, the pairs $(h_m,u_m)$ also form a
finite family. Hence $F^{m+p}\circ\nu=F^m\circ\nu$ for some
$m\geq0$, $p\geq1$. The curve $C=F^m(Z)$ is consequently fixed
pointwise by $F^p$ and thus is a curve of $p$-periodic curve. This is impossible by Corollary~\ref{cor-fixed} and gives the desired contradiction.
\end{proof}

To rule out this last case, we use the parameter \(B\) to perturb the map \(F=f_{\lambda_*}\) considered above so that \(L_\infty\) is no longer totally invariant under \(f_\lambda\). The argument relies on Proposition~\ref{prop-rigidity} (where \(A_*\neq 0\)) together with a monodromy argument along the subfamily preserving the pencil, namely \(\lambda=(a,0,B)\) with \(B\) varying. We are therefore led to consider the full family in which both \(A\) and \(B\) vary.

Thus, let $\mathcal Z\subset\mathcal X_D^1\times\mathrm{Rat}_{2}^2$ be an irreducible subvariety
that dominates $\mathrm{Rat}_{2}^2$. Let $\mathcal Z_{\mathcal F}$
be an irreducible component of its restriction over $\mathcal F$
that dominates $\mathcal F$. We now assume that for $|a|$ large enough, there exists a neighborhood $V_a$ of $\lambda_a=(a,0,0)$ in $\Cb^3$, on which $p(f_\lambda)$ and the transverse homoclinic intersection from Lemma~\ref{le-homof0} persist, such that for all $(Z,f_\lambda)\in\mathcal Z_{\mathcal F}$ such that $\lambda\in V_a$ we have
\begin{equation}\label{eq-proper-int}
\forall m\geq0,\text{ there is no proper intersection between }f_\lambda^m(Z)\text{ and }L_a\setminus\{I\}. 
\end{equation}

By Lemma~\ref{le-fixed} and Remark~\ref{rk-contract}, there
exists a Zariski open subset $\mathcal V_0\subset\Cb^3$
containing all $(a,A,0)$ with $A\neq0$ such that $f_\lambda$
does not contract any curve for $\lambda\in\mathcal V_0$. Therefore, for each fixed $m\geq0$, the condition
$L_a\subset f_\lambda^m(Z)$ is Zariski closed over
$\mathcal V_0$, as is its union with the condition
$f_\lambda^m(Z)\cap L_a=\{I\}$.
This union contains every pair over $V_a\cap\mathcal V_0$,
hence, by irreducibility, every pair over $\mathcal V_0$.
Consequently, each $(Z,f_\lambda)\in\mathcal Z_{\mathcal F}$
with $\lambda\in\mathcal V_0$ satisfies
$$\mathbf{(C_1)}\ f_\lambda^m(Z)\cap L_a=\{I\}\text{ for each }m\geq0,\ \text{ or } \mathbf{(C_2)}\ L_a\subset f_\lambda^m(Z)\text{ for some }m\geq0.$$

%
%

After shrinking $V_*\subset\mathcal V_0$, we may assume that the saddle fixed point $q(f_{\lambda_*})$ and the attracting fixed point $s(f_{\lambda_*})$ given by Lemma~\ref{le-fixed} admit holomorphic continuations over $V_*$. For $|a|$ sufficiently large, the following two propositions describe the two cases of the preceding dichotomy.
\begin{proposition}\label{prop-DMM}
Assume that $\mathbf{(C_2)}$ fails for some $(Z,f_{\lambda'})\in\mathcal Z_{\mathcal F}$, with $\lambda'=(a,A',0)\in V_*$ and $A'\neq0$. After possibly shrinking $V_*$, there exists an integer $k\geq0$ such that, for every $(Z,f_\lambda)\in\mathcal Z_{\mathcal F}$ with $\lambda\in V_*$, the reduced strict image $f_\lambda^k(Z)$
\begin{itemize}
    \item is irreducible,
    \item is invariant under $f_\lambda$,
    \item contains $s(f_\lambda)$ and
    $W^u_{q(f_\lambda),loc}$,
    \item is independent of $Z$, i.e.,
    $f_\lambda^k(Z)=f_\lambda^k(Y)$ whenever
    $(Y,f_\lambda)\in\mathcal Z_{\mathcal F}$.
\end{itemize}
\end{proposition}

\begin{proof}
We first show that every member of the fiber over $F=f_{\lambda_*}$ is supported on $L_\infty$. Suppose that such a member $Z$ has a component different
from $L_\infty$. By Proposition~\ref{prop-rigidity} and Lemma~\ref{le-delta}, $Z$ has a nonempty proper intersection with some $\Delta_m$ in $\Cb^2$.
Indeed, if $Z$ shares a component with some $\Delta_n$, we choose $m$ larger than all the corresponding indices and apply Lemma~\ref{le-delta}. Otherwise, Proposition~\ref{prop-rigidity} applies directly.

This intersection persists for nearby pairs $(Y,f_\lambda)$, and its image under $f_\lambda^m$ lies in $L_a\setminus\{I\}$. By~\eqref{eq-proper-int}, we must therefore have $L_a\subset f_\lambda^m(Y)$ on a nonempty open subset of $\mathcal Z_{\mathcal F}$. For this fixed $m$, the containment condition is Zariski closed over $\mathcal V_0$. By irreducibility, it holds throughout the family over $\mathcal V_0$, contradicting the hypothesis.

Thus the fiber over $\lambda_*$ consists only of $L_\infty$. The restriction of $F$ to $L_\infty$ extends to a homography for which $q(F)$ is repelling and $s(F)$ is attracting. Hence $W^u_{q(F),loc}\subset L_\infty$ contains points in the basin of $s(F)$, while $W^s_{q(F),loc}$ intersects $L_\infty$ properly.

Thus, possibly by reducing $V_*,$ we can assume that for all $(Z,f_\lambda)\in\mathcal Z$ with $f_\lambda\in V_*$, every irreducible component $Z'$ of $Z$ satisfies
\begin{itemize}
\item $Z'\sim L_\infty,$
\item $Z'\cap W^s_{q(f_\lambda),loc}$ is proper,
\item $Z'$ intersects the basin of attraction of $s(f_\lambda).$
\end{itemize}
From this, there are two possibilities for each pair $(Z', f_\lambda)$. The first is that there exists an integer $k \geq 0$ such that $f_\lambda^k(Z')$ contains $W^u_{q(f_\lambda),loc}$. Since its image contains the same germ and $f_\lambda^k(Z')$ is irreducible, it is $f_\lambda$-invariant. In this case, since $Z'$ intersects the basin of $s(f_\lambda)$, the iterate $f_\lambda^k(Z')$ must contain this point. Observe that if the same property holds for $f_\lambda^l(Y')$ for an irreducible component $Y'$ of $Y$ with $(Y,f_\lambda)\in\mathcal Z_{\mathcal F}$, then necessarily $f_\lambda^l(Y') = f_\lambda^k(Z')$, since both contain $W^u_{q(f_\lambda), loc}$. 

Therefore, if this case applies to all $(Z, f_\lambda) \in \mathcal{Z}_{\mathcal F}$ over $V_*$ and every irreducible component $Z' \subset Z$, a density argument implies that we may choose the integer $k \geq 0$ uniformly, and thus the conclusion of the lemma follows.

Otherwise, there exists a pair $(Z', \lambda)$ such that $f_\lambda^N(Z')$ does not contain $W^u_{q(f_\lambda), loc}$ for any $N \geq 0$. A graph transform argument  then implies that if $C$ is a curve intersecting $W^u_{q(f_\lambda), loc}$ properly at a point $x\neq q(f_\lambda)$, we have
\begin{equation}\label{eq-inter}
\#\left( \bigcup_{N \geq 0} f_\lambda^N(Z') \cap C \right) = +\infty.
\end{equation}
Since the sequence $(x_n)_{n\geq1}$ defined above converges to $q(f_{\lambda_*})$, we may choose $n\geq1$ such that $x_n\neq q(f_{\lambda_*})$ belongs to $W^u_{q(f_{\lambda_*}),loc}$. Let $D_n(\lambda)$ be the continuation of the local branch of $\Delta_n$ through $x_n$. After shrinking $V_1$, this branch intersects $W^u_{q(f_\lambda),loc}$ properly, and the restriction of $f_\lambda^n$ to it is nonconstant. Applying \eqref{eq-inter} with $C=D_n(\lambda)$ gives infinitely many intersection points. Only finitely many of them are mapped to $I$ by $f_\lambda^n$, so there exist $N\geq0$ and
\[
x\in f_\lambda^N(Z')\cap f_\lambda^{-n}(L_a)
\]
such that $f_\lambda^n(x)\neq I$. In other words, $x$ is a proper intersection point between $f_\lambda^N(Z')$ and $f_\lambda^{-n}(W^s_{p(f_\lambda)})$.
\end{proof}
The other alternative is simpler to obtain.
\begin{proposition}\label{prop-delta_n}
Assume that, for every \((Z',f_{\lambda'})\in\mathcal Z_{\mathcal F}\) with \(\lambda'=(a,A',0)\in V_*\), condition $\mathbf{(C_2)}$ holds. Then, the same is true for all \((Z,f_\lambda)\in\mathcal Z_{\mathcal F}\) when \(\lambda=(a,A,0)\) and \(A\neq0\) is arbitrarily small.
\end{proposition}
\begin{proof}
For each fixed $m\geq0$, the condition $L_a\subset f_\lambda^m(Z)$ is Zariski closed over $\mathcal V_0$. Moreover, if it holds for some $m$, then it holds for every $m'\geq m$. Therefore, it holds for a fixed $m\geq0$ for every

$$
(Z',f_{\lambda'})\in\mathcal Z_{\mathcal F},
\qquad \lambda'=(a,A',0)\in V_*.
$$

By Zariski closedness, the same condition then holds on every irreducible component of $\mathcal Z_{\mathcal F}$ over $\mathcal A:=\{f_\lambda:\lambda=(a,A,0),\ A\in\mathbb C\}$ that dominates $\mathcal A$. The remaining components are vertical, hence lie over finitely many parameters $A$. We may therefore choose $A\neq0$, arbitrarily close to $0$, outside this finite exceptional set. For such $A$, every point of the fiber lies on a component dominating $\mathcal A$, and the conclusion follows.
\end{proof}

\subsection{A proper intersection from monodromy}\label{sec-proper}
In this section, we assume that the hypotheses of Proposition~\ref{prop-DMM} hold.
We remain in degree $2$. Fix $a$ as in the preceding subsection, put
$c:=c_a=\tilde w(a)$, and set
\[
 P_a:=\{a\}\times\Cb^2,\qquad
 \lambda_a=(a,0,0),\qquad \lambda_*=(a,A_*,0).
\]
We write $\lambda=(a,A,B)$ and, in what follows, we take all parameter neighborhoods in $P_a$. We use $V_a$ for the neighborhood
of $\lambda_a$ from the preceding subsection, and $V_*$ for the
neighborhood of $\lambda_*$. A monodromy argument will then yield Theorem~\ref{th-inter} below. To carry it out, we shall use the following elementary consequences of the pencil structure. When $A=0$, the map $f_{(a,0,B)}$ preserves the pencil of lines
through $I$, and the induced map on its base is
\[
g_B:\Pb^1\to\Pb^1,
\qquad
[w:t]\mapsto
[a(w^2-t^2):t^2+Bw(w-ct)],
\]
With this notation, $g_B$ has degree $2$ except for
\[
B=-\frac1{1-c}
\qquad\text{or}\qquad
B=-\frac1{1+c}.
\]

\begin{lemma}\label{le-pinceau}
Assume that $g_B$ has degree $2$ and
\[
B\neq \frac{2a}{c}.
\]
Then $L_a$ is not totally invariant. If $Z\subset\Pb^2$ is an algebraic
curve such that, for every $n\geq0$, the strict preimage
$f_{(a,0,B)}^{-n}(L_a)$ has no proper intersection with $Z$ outside $I$, then
$Z$ is a union of lines through $I$.
\end{lemma}

\begin{proof}
The line $L_a$ is totally invariant if and only if the fixed point
$[c:1]$ is totally invariant under $g_B$. Using $a(c^2-1)=c$, the equation
for its preimages factorizes as
\[
a(w^2-t^2)-c\bigl(t^2+Bw(w-ct)\bigr)
=(w-ct)\bigl((a-cB)w+ac\,t\bigr).
\]
When $g_B$ has degree $2$, the second factor defines the same point as
$w-ct$ if and only if $B=2a/c$. Thus, under our assumptions, the backward
orbit of $[c:1]$ is infinite, and the successive strict preimages of $L_a$
contain infinitely many distinct lines of the pencil.

Let $D:=\deg Z$, and choose affine coordinates centered at $I$,
with the line at infinity of this chart chosen not to be a component of
$Z$. Write
\[
h=h_0+\cdots+h_D
\]
for an equation of $Z$, with $h_j$ homogeneous of degree $j$. For infinitely many lines $L$ of the pencil, we have $Z\cap L=\{I\}$. Hence, the
restriction of $h$ to each such line vanishes at $I$ with multiplicity $D$.
Therefore, if $[u:v]$ is the direction of $L$,
\[
h_j(u,v)=0,\qquad j=0,\ldots,D-1.
\]
Since this holds for infinitely many directions, $h_j=0$ for every $j<D$.
Thus $h=h_D$ is homogeneous, and $Z$ is a union of lines through $I$.
\end{proof}
We can now establish this weak version of \ref{it-chasse-2} in
$\mathrm{Rat}_{2}^2$.

\begin{theorem}\label{th-inter}
Let $a$ and $V_a$ be as above. Let
$\mathcal Z \subset \mathcal X_D^{1} \times \mathrm{Rat}_{2}^2$
be an algebraic subvariety that surjects onto $\mathrm{Rat}_{2}^2$, and assume
that we are in the case covered by Proposition~\ref{prop-DMM}.
Then there exist $Z \subset \mathbb{P}^2$ and $\lambda\in V_a\cap P_a$
with $(Z,f_\lambda)\in\mathcal Z$ such that, for some $n\geq0$,
$Z$ and $f_\lambda^{-n}(W^s_{p(f_\lambda)})$ intersect properly.
\end{theorem}
\begin{rk}\label{rk-inter}
Actually, this gives a contradiction with our assumptions, i.e., \eqref{eq-proper-int} and the one of Proposition~\ref{prop-DMM}. To be more precise, an intersection as in Theorem~\ref{th-inter} is persistent. Therefore, it also holds for some $(Z',f_{\lambda'})\in\mathcal Z$ with $\lambda'\in\mathcal V_0\cap V_a$. If \eqref{eq-proper-int} is verified then this implies that $f_{\lambda'}^n(Z')$ contains $L_a$ in a persistent way. This contradicts the assumption of Proposition~\ref{prop-DMM} that we also used.  At the end, the conclusion of Theorem~\ref{th-inter} is that either \eqref{eq-proper-int} does not hold or Proposition~\ref{prop-delta_n} is satisfied.
\end{rk}
\begin{proof}
In order to apply Proposition~\ref{prop-DMM}, consider
\[
\mathcal Z_{P_a}
:=
\{(Z,\lambda)\in\mathcal X_D^1\times P_a\,;\,
  (Z,f_\lambda)\in\mathcal Z\},
\]
from which we remove the irreducible components that do not surject onto
$P_a$.

Choose one irreducible component $\mathcal T$ coming from the dominating
family to which Proposition~\ref{prop-DMM} applies. We denote by
$\mathcal T_{V_a}$ and $\mathcal T_{V_*}$ its restrictions over
$V_a\cap P_a$ and $V_*\cap P_a$, respectively.

By Proposition~\ref{prop-DMM}, there exists $k\geq0$ such that, for $\lambda\in V_*\cap P_a$, the reduced strict image $f_\lambda^k(Z)$ is irreducible, invariant, contains $s(f_\lambda)$ and $W^u_{q(f_\lambda),loc}$, and depends only on $\lambda$. In particular, it equals $L_\infty$ when $B=0$ and $A\neq0$.

We restrict to a dense Zariski open subset of $\mathcal T$ where these images have constant degree, and denote again by $\mathcal Z_{P_a}$ the closure of their family in the corresponding Chow variety. We may thus assume that $k=0$ generically, while allowing the special fibers to have several components or multiplicities. The monodromy argument below will exploit the fact that this image family is generically single-valued.

We also need to control its fibers over $B=0$. Near $\lambda_*=(a,A_*,0)$, the same blow-up of $I$ resolves all nearby maps in our family, as in Lemma~\ref{le-fixed}. When taking limits, strict transforms may acquire additional components on the exceptional divisor. At $B=0$, these components map into $L_\infty$, which is invariant. Thus any additional components arising in limits of the iterated images remain supported on $L_\infty$.

On the other hand, Proposition~\ref{prop-DMM} gives $f_\lambda^k(Z)=L_\infty$ for $\lambda=(a,A,0)$ near $\lambda_*$. Every limiting image cycle over these parameters is therefore supported on $L_\infty$. By algebraicity, the same holds for generic $A$ along the line $B=0$.

In this situation, the projection
\[
\pi:\mathcal Z_{P_a}\to P_a,
\qquad
(Z,\lambda)\mapsto\lambda,
\]
is then generically one-to-one.

Let $E_0$ be the set of bad values of $B$ given by Lemma~\ref{le-pinceau} (and above), i.e.,
\[
E_0:=\left\{-\frac1{1-c},-\frac1{1+c},\frac{2a}{c}\right\}.
\]
For $B\notin E_0$, the map $g_B$ has degree $2$ and $L_a$ is not totally invariant. If a cycle $Z$ in the fiber of $\pi$ over $\lambda=(a,0,B)$ had a proper intersection outside $I$ with a strict preimage of $L_a$, this intersection would persist for nearby generic parameters. By invariance of the generic image curves, this would give an intersection with $L_a\setminus\{I\}$, a contradiction. Lemma~\ref{le-pinceau} therefore shows that every cycle in this fiber is supported on lines of the pencil through $I$.

Moreover, the invariance at a generic parameter gives at $\lambda$
\[
 f_\lambda(Z\setminus\{I\})\subset Z.
\]
Since no point outside $I$ is mapped to $I$, even a contracted line is mapped to a point outside $I$, which belongs to a line of the support. 
Thus the component directions form a finite forward-invariant set for $g_B$, and are preperiodic. There are only countably many such directions, hence only countably many possible cycles of the prescribed degree. Hence, the fiber of $\pi$ are finite above such $\lambda=(a,0,B).$

Consequently, since $\pi$ is generically one-to-one, there exists a neighborhood $M\subset P_a$ of
\[
 \{(a,0,B):B\notin E_0\}
\]
on which $\pi$ is an isomorphism over $M$. We may therefore write its fibers as $(Z_\lambda,\lambda)$, with $Z_\lambda$ depending holomorphically on $\lambda\in M$.

By the preceding control of the limiting cycles along
$B=0$, we have $Z_{\lambda_a}=L_\infty.$ Since infinity is superattracting for $g_0$, it is still contained in an attracting domain for $g_B$ when
$B$ is small. Thus, all component directions of $Z_{(a,0,B)}$ lie in this disk. They form a nonempty finite
forward-invariant set, so they contain a periodic point.
The only such point in the disk is the fixed point
continued from infinity. Its line is therefore a
component of $Z_{(a,0,B)}$.

It remains to show that this is incompatible with
monodromy. Parametrize the lines of the pencil different
from $\{w=0\}$ by
\[
 D_\tau:=\{t=\tau w\},
\]
so that $L_\infty=D_0$.
The fixed directions satisfy
\[
 (1-c\tau)\bigl(B-a\tau+(1-ac)\tau^2\bigr)=0.
\]
The first factor corresponds to $L_a$. Since
$1-ac=-a/c$, the other two fixed directions satisfy
\begin{equation}\label{eq-fixed-lines-B}
 \tau^2+c\tau-\frac ca B=0.
\end{equation}
At $B=0$, they correspond to $L_\infty$ and $D_{-c}$.
They merge at
\[
 B_*=-\frac{ac}{4},
\]
which lies outside $E_0$ for $|a|$ sufficiently large.
A loop based at $0$, avoiding $E_0\cup\{B_*\}$ and winding
once around $B_*$, exchanges these two directions.

The condition
\[
 D_\tau\subset Z_{(a,0,B)}
\]
is algebraic and closed on the curve
\eqref{eq-fixed-lines-B}, over the good parameters.
It holds along the local branch through $(B,\tau)=(0,0)$,
and hence remains valid after continuation along the loop.
Since the family $Z_{(a,0,B)}$ is single-valued, we obtain
\[
 D_{-c}\subset Z_{\lambda_a}=L_\infty,
\]
a contradiction.
\end{proof}

\subsection{The transverse intersection}\label{sec-transversal}
As explained in Remark~\ref{rk-inter}, there are now two possibilities. Either, for some $\lambda\in V_a\mathcal V_0$, the curve $f_\lambda^n(Z)$ intersects $L_a\setminus\{I\}$ properly, or, by Proposition~\ref{prop-delta_n}, we have $L_a\subset f_\lambda^n(Z)$ for every $\lambda=(a,A,0)\in V_a$ with $A\neq0$. In the first case, it remains to promote the proper intersection to a transverse one. In the second case, we perturb $f_\lambda$ within $\mathrm{Poly}_2^2$ to make the intersection proper, thereby reducing the problem to a setting similar to that of the first case.

\begin{proposition}\label{prop-trans}
Let  $\lambda$ be in $V_a\cap\mathcal V_0$ and let $U$ be a small neighborhood of $f_\lambda$ in $\mathrm{Rat}_d^2$ such that Lemma~\ref{le-homof0} still holds on $U.$ If for some $(Z,g)\in\mathcal Z$, with $g\in U,$ we have that for some $n\geq0$
\[
 Z\cap g^{-n}(W^s_{p(g),loc})
\]
is proper, then there exist $(Z',f)\in\mathcal Z$ arbitrarily close to $(Z,g)$,
with $f\in\mathrm{End}_d^2$, and $N\geq n$ such that
$f^N(Z')$ intersects $W^s_{p(f),loc}$ transversely.
\end{proposition}
\begin{proof}
There are two difficulties: the forward image could contain the
stable manifold, or their intersection could be proper but not
transverse. We use Fakhruddin's
\cite[Theorem~1.2]{fakhruddin-generic} to exclude the first
possibility and the local argument of Bedford-Lyubich-Smillie
\cite[Lemma~6.4]{BedfordLyubichSmillie2} to rule out the second.

Choose $(Z',f)\in\mathcal Z$ sufficiently close to $(Z,g)$, with
$f\in\mathrm{End}_d^2$ very general. Persistence of proper
intersections and continuity of the local
stable manifold give a proper intersection of $Z'$ with
$f^{-n}(W^s_{p(f),loc})$. By Fakhruddin's theorem, $f$ has no
invariant algebraic curves. Consequently, the curve
$C:=f^n(Z')$ cannot contain $W^s_{p(f),loc}$ and thus $C$
and $W^s_{p(f),loc}$ intersect properly.

Near $p(f)$, the distinct local stable plaques of $E(f)$ form
a perfect family of pairwise disjoint holomorphic disks.
By continuity and persistence of proper intersections, nearby
plaques still meet $C$. Since $C$ has only finitely many singular
points, one of these plaques meets its regular locus. If the
intersection is tangent, \cite[Lemma~6.4]{BedfordLyubichSmillie2}
gives a transverse intersection with a nearby distinct plaque.
By openness of transversality and density of periodic points
in $E(f)$, we may choose such a plaque based at a periodic point
$q$ homoclinically related to $p(f)$. The inclination lemma then
gives a transverse intersection of $f^N(Z')$ with
$W^s_{p(f),loc}$ for some $N\geq n$.
\end{proof}

\begin{proposition}\label{prop-DFR}
For every parameter $\lambda=(a,A,0)\in V_a$ to which
Proposition~\ref{prop-delta_n} applies, there exist arbitrarily small
$\delta,\theta\in\Cb^*$ and a pair $(Z,f)\in\mathcal Z$, with
\begin{equation}\label{eq-DFR}
f\colon(z,w)\mapsto
\bigl(\alpha z+\epsilon w+\beta zw+\delta z^2,\,
q_a(w)+Az(w-c_a)+\theta\bigr),
\end{equation}
such that $f^n(Z)$ intersects $W^s_{p(f),\mathrm{loc}}$ properly
for some $n\geq0$.
\end{proposition}

\begin{proof}
By assumption, when $\lambda=(a,A,0)$ with $A\neq0$ small enough, for every pair
$(Z_\lambda,f_\lambda)\in\mathcal Z_{\mathcal F}$, there exists
$m\geq0$ such that $L_a\subset f_\lambda^m(Z_\lambda)$.
In particular, since $W^u_{p(f_\lambda),\mathrm{loc}}$ is Zariski
dense by Lemma~\ref{le-homof0}, $f_\lambda^m(Z_\lambda)$ intersects
this local manifold properly. Therefore, for sufficiently small
$\delta,\theta\in\Cb^*$ and $f$ as in the statement, there exists
$(Z,f)\in\mathcal Z$ such that $f^m(Z)$ intersects
$W^u_{p(f),\mathrm{loc}}$ properly near $p(f)$.

There are now two possibilities. If this intersection always
consists only of $p(f)$, then $f^m(Z)$ intersects
$W^s_{p(f),\mathrm{loc}}$, and this intersection must be proper
by Lemma~\ref{le-WsZar}.

Otherwise, the intersection does not always consist only of
$p(f)$. After possibly passing to a ramified cover of the
parameter space, we can follow it holomorphically, denoting it by $h(f)$. We then consider a
linearization $\psi_f\colon\Cb\to W^u_{p(f)}$ of the dynamics,
depending holomorphically on $f$, such that $\psi_f(0)=p(f)$
and $\psi_f(1)$ is the homoclinic intersection given by
Lemma~\ref{le-homof0}. For sufficiently small
$\delta,\theta\in\Cb^*$, the point $h(f)$ lies in the image of
a neighborhood of $0$ on which $\psi_f$ is invertible.
We therefore let $\tilde h(f)$ be the unique point near $0$
such that $\psi_f(\tilde h(f))=h(f)$. If $\rho(f)$ denotes the
unstable multiplier of $p(f)$, then
\[
f^n(h(f))=\psi_f\bigl(\rho(f)^n\tilde h(f)\bigr).
\]

On the other hand, the family
$\{f\mapsto\rho(f)^n\tilde h(f)\}_{n\geq0}$ is not normal.
Hence there exist arbitrarily small $\delta,\theta\in\Cb^*$
such that, for the corresponding map $f$ and some $n\geq0$,
$\rho(f)^n\tilde h(f)$ equals either $1$ or $\rho(f)^{-1}$.
Therefore, $f^{m+n}(Z)$ or $f^{m+n+1}(Z)$ intersects
$W^s_{p(f),\mathrm{loc}}$, and this intersection must be proper,
as above.
\end{proof}
In dimension $2$, the preceding homoclinic argument provides a
simple criterion for the Zariski density of the unstable manifold.
For the stable manifold, we use the local analysis of
Dujardin-Favre-Ruggiero to treat the family of
Proposition~\ref{prop-DFR}.

\begin{lemma}\label{le-WsZar}
Fix $\lambda=(a,A,0)$, with $A\neq0$, as in
Proposition~\ref{prop-DFR}. For all sufficiently small
$\delta,\theta\in\Cb^*$, the map $f$ defined by
\eqref{eq-DFR} has a Zariski-dense local stable manifold
$W^s_{p(f),\mathrm{loc}}$.
\end{lemma}

\begin{proof}
Write $c=c_a$. For $\delta\neq0$ sufficiently small, we have
$a\delta-\beta A\neq0$, so $f$ extends to an endomorphism of
$\Pb^2$ of degree $2$.

Suppose that $W^s_{p(f),\mathrm{loc}}$ is not Zariski dense,
and let $Y$ be its irreducible algebraic closure. Then $Y$ is
$f$-invariant. It intersects $L_\infty$ in a totally invariant set for its normalisation. Thus, this set has at most $2$ points. As it also possesses a non-critical attracting fixed point, there is only one exceptional point, thus $Y\cap L_\infty$ consists of a single point $\xi,$ which must be fixed.

In the coordinate $u=w/z$, the restriction to $L_\infty$ is
\[
f_\infty(u)=\frac{u(A+au)}{\delta+\beta u}.
\]
Its three fixed points, $\infty,$ $(\delta-A)/(a-\beta)$ and $0.$ The first $2$ are perturbations of the points $s$ and $q$ in Lemma~\ref{le-fixed} respectively, thus they are (non-critical) attracting and repelling. For $0$, which corresponds to the indeterminacy point before perturbation, the multiplier is $A/\delta\neq0.$

By \cite[Theorem~2.1]{DFR}, there are local coordinates
$(x,y)$ centered at $\xi$, with $L_\infty=\{y=0\}$, in which
\[
f(x,y)=\bigl(xU(x,y),\,y^2V(x,y)\bigr),
\qquad U(0,0)V(0,0)\neq0.
\]
Here $\{x=0\}$ is the smooth invariant curve transverse to
$L_\infty$ constructed by Dujardin-Favre-Ruggiero.
The unique branch of $Y$ at $\xi$ must coincide with this curve.
Indeed, the induced map on its normalization has local degree
$2$, so orbits on the branch converge super-exponentially to
$\xi$. On the other hand, $U(0,0)\neq0$ implies
$|x_n|\geq\kappa^n|x_0|$ for some $\kappa>0$.
Thus $x_0=0$, and the branch coincides with the smooth
transverse curve $\{x=0\}$. Hence $Y$ is smooth and transverse to
$L_\infty$ at its unique point at infinity. By Bézout,
$\deg Y=1$.

It remains to exclude a line containing
$W^s_{p(f),\mathrm{loc}}$. Its direction is the stable
eigendirection at $p(f)$, which converges to the horizontal
direction as $(\delta,\theta)\to(0,0)$. Thus, for sufficiently
small parameters, this line cannot correspond to the fixed points $\infty$ and $(\delta-A)/(a-\beta)$ since $A\neq0.$ The remaining case corresponds to horizontal lines $\{w=cst\}$ which cannot be invariant by the choice of \eqref{eq-DFR}, as soon as $\theta\neq0.$
\end{proof}

Assuming Theorem~\ref{th-chasse-julia} on the small Julia set,
we can now complete the proof of Theorem~\ref{th-chasse-d2}.

\begin{proof}[Proof of Theorem~\ref{th-chasse-d2}]
As explained above, Theorem~\ref{th-inter} gives the following
dichotomy. Either, for some $\lambda\in V_a\cap\mathcal V_0$,
an iterate $f_\lambda^n(Z)$ intersects $L_a\setminus\{I\}$
properly, or, by Proposition~\ref{prop-delta_n}, we have $L_a\subset f_\lambda^n(Z)$ for every $\lambda=(a,A,0)\in V_a$ with $A\neq0$.

In either case, fix $\lambda_0=(a,A,0)$ with $a^{-1}$ and
$A\neq0$ sufficiently small. By Theorem~\ref{th-chasse-julia}, there is a neighborhood
$\Omega$ of $f_{\lambda_0}$ in $\mathrm{Rat}_2^2$ such that
every map in $\Omega\cap\mathrm{End}_2^2$ satisfies all the
assumptions of Section~\ref{sec-ass}.

In the second case, Proposition~\ref{prop-DFR} provides a pair
$(Z,f)\in\mathcal Z$, with $f\in\Omega\cap\mathrm{End}_2^2$,
such that some iterate $f^N(Z)$ intersects
$W^s_{p(f),\mathrm{loc}}$ properly. By
Proposition~\ref{prop-trans}, a further perturbation within
$\Omega$ makes this intersection transverse.

In the first case, the locus of such proper intersections
is nonempty and Zariski open, so we may choose the corresponding
map $f_{\lambda'}$ to be very close to $f_\lambda$ and thus in $\Omega$. Proposition~\ref{prop-trans}
again gives the required transverse intersection after
perturbation into $\Omega\cap\mathrm{End}_2^2$.
\end{proof}
\subsection{Higher degrees}\label{sec-chasse-higher}
We now want to adapt the argument to higher degrees. We cannot
simply add terms of degree $d$ to the quadratic maps for which
a proper intersection has been obtained. Indeed, a perturbation
from degree $2$ to degree $d$ implies a blow-up along a subvariety,
and it is natural to expect that the worst family $\mathcal Z$
will degenerate to this subvariety as we approach the quadratic
locus in that direction. Thus, we lose all control during the
perturbation.

Hence, we need to carry out the proof in degree $d$. There are
several natural generalizations of our quadratic family, but none
leads to simple computations at every step. We therefore adopt
the following compromise.

\bigskip

Fix $d\geq3$ and set
\[
 S_\eta(w)=1+\eta w^{d-2},\qquad
 S_{\eta,h}(w,t)=t^{d-2}+\eta w^{d-2}.
\]
For $\eta\neq0$, the polynomial $S_\eta$ has degree $d-2$, and
$S_\eta\to1$ locally uniformly as $\eta\to0$. We keep
$c_a=\tilde w(a)$, with $a(c_a^2-1)=c_a$ and $c_a\sim-1$,
and the nonzero horizontal parameters $\alpha,\beta,\epsilon$
of the local construction. We always have $S_\eta(0)=1$ and,
for $\eta$ sufficiently small and $|a|$ sufficiently large,
$S_\eta(c_a)$ and $S_\eta(\pm1)$ are also close to $1$.

We now fix a small $\eta\in\Cb^*$ and define, for $a,A\in\Cb$,
the polynomial map
\[
 \widetilde f_{(a,A,0)}(z,w):=
 \bigl(\alpha z+\epsilon w+\beta zwS_\eta(w),
 c_a+(w-c_a)S_\eta(w)(Az+a(w+c_a))\bigr).
\]
Observe that in the degenerate case $\eta=0$, this is exactly
the map $f_{(a,A,0)}$ in the quadratic case. However, for fixed
$\eta\neq0$ and large $|a|$, it is not a small perturbation of
this quadratic map. The important point is that the way $S_\eta$
appears in both coordinates simplifies the computations for the
dynamics on $L_\infty$ and for the Jacobian.

The full family in homogeneous coordinates is then, for
$\lambda:=(a,A,B)\in\Cb^3$,
\begin{equation}\label{eq-higher-family}
\begin{aligned}
 \widetilde f_\lambda[z:w:t]=[\,&
 \alpha zt^{d-1}+\epsilon wt^{d-1}+\beta zwS_{\eta,h}(w,t):\\
 &c_at^d+(w-c_at)S_{\eta,h}(w,t)\bigl(Az+a(w+c_at)\bigr):\\
 &t^d+Bw^{d-1}(w-c_at)\,].
\end{aligned}
\end{equation}

\begin{theorem}\label{th-chasse-d3}
Let $d\geq3$ and fix $\eta\neq0$ sufficiently small. The conclusion
of Theorem~\ref{th-chasse-d2} holds with $\mathrm{End}_d^2$ and
$\mathrm{Rat}_d^2$ in place of $\mathrm{End}_2^2$ and
$\mathrm{Rat}_2^2$, for the family \eqref{eq-higher-family}.
\end{theorem}

The first observation is that, as in the quadratic case, the line
$L_a=\{w=c_a\}$ is invariant by construction. On the other hand,
when $B=0$, its preimage still contains
\[
 \widetilde L_\lambda:=\{Az+a(w+c_a)=0\}.
\]
But it also has other components, namely the lines $D_r=\{w=r\}$
where $S_\eta(r)=0$. Observe that they are far from $L_a$ since
$\eta\neq0$ is small, so they do not interact with the dynamics
in the blender region. However, they will greatly simplify the
counterpart of Proposition~\ref{prop-rigidity}, as they directly
provide the third line, without the need to choose a special
parameter $A=A_*$.

We now explain step by step how the quadratic proof adapts to
this context.

\bigskip

\noindent\emph{\textbf{Local dynamics.}}
As we have said, the line $L_a=\{w=c_a\}$ is invariant.
The action of $\widetilde f_\lambda$ on it is given by
$z\mapsto\rho_a z+\epsilon c_a$, where
$\rho_a=\alpha+\beta c_aS_\eta(c_a)$. Its fixed point is
\[
 p(\widetilde f_\lambda):=(z_a,c_a),\qquad
 z_a:=\frac{\epsilon c_a}{1-\rho_a}.
\]
Moreover, as explained in Remark~\ref{rk-S}, all results of the
two-dimensional blender construction in Section~\ref{sec-dim2}
remain valid when $A=B=0$, $\eta\neq0$ is fixed and sufficiently
small, and $|a|$ is then taken sufficiently large.

The nontrivial homoclinic intersection is obtained by the same
argument as in Lemma~\ref{le-homof0}, except that the horizontal
displacement from $p(\widetilde f_\lambda)$ is now given by
\[
 \pi_z(x_a)-z_a
 =
 \frac{\epsilon c_a\bigl(2(\alpha-1)
       +\beta c_a(S_\eta(c_a)-S_\eta(-c_a))\bigr)}
      {1-\rho_a}
 +O(|a|^{-1}),
\]
whose leading term is nonzero for sufficiently small $\eta$.

\medskip
\noindent\emph{\textbf{Resolution, degrees and rigidity.}}
Fix any $A\neq0$ and put $F=\widetilde f_{(a,A,0)}$. We now
discuss the counterparts of Lemma~\ref{le-fixed} and
Proposition~\ref{prop-rigidity}. We could use exactly the same coordinates $(H,U,T)$ as in the quadratic
case, but we prefer to keep the original coordinates $(z,w,t)$,
in which the blow-up calculation is more direct.

The only indeterminacy point of $F$ is $I=[1:0:0]$. In the chart $z=1$ near $I$, the lowest-degree terms of the first
two coordinates have degree $d-1$ and are
\[
 P_h(w,t)=\alpha t^{d-1}+\beta wS_{\eta,h}(w,t),\qquad
 Q_h(w,t)=A(w-c_at)S_{\eta,h}(w,t).
\]
They have no common zero on $\Pb^1$ since at $S_{\eta,h}=0$ the first
is $\alpha t^{d-1}\neq0$, at $w=c_at$ it is
$\rho_a t^{d-1}\neq0$, and at $t=0$ both are nonzero.
Thus blowing up $I$ in the source resolves $F$ as a morphism
$\widehat F:X\to\Pb^2$. Its restriction to the exceptional
divisor is
\[
[w:t]\longmapsto[P_h(w,t):Q_h(w,t):0],
\]
which has degree $d-1$. With the notation of Lemma~\ref{le-fixed},
and with $\widehat L_\infty$ denoting the strict transform of
$L_\infty$, we obtain
\begin{equation}\label{eq-higher-resolution}
 \widehat F^*\mathcal L=d\mathcal H-(d-1)E,\qquad
 \widehat F^*L_\infty=d\widehat L_\infty+E.
\end{equation}
For a curve of degree $D$ and multiplicity $m$ at $I$, the
intersection of its strict transform with $\widehat F^*\mathcal L$
is $dD-(d-1)m\geq D>0$. Since $E$ is also not contracted,
$\widehat F$ is finite. We also have
\begin{equation}\label{eq-higher-degrees}
 \begin{gathered}
 \deg_{\mathrm{top}}F=d^2-(d-1)^2=2d-1,\\
 e_C\deg F(C)=d\deg C-(d-1)\mathrm{mult}_I(C),
 \qquad e_C=\deg(F_{|C}).
 \end{gathered}
\end{equation}

Moreover, the preimage of $L_a$ is given by
\[
 F^{-1}(L_a)
 =L_a\cup\widetilde L_\lambda\cup\bigcup_{S_\eta(r)=0}D_r,
 \qquad D_r=\{w=r\}.
\]
Choose a root $r$ of $S_\eta$, so that $r\neq c_a$. Let $\nu:C\to Z$ be the normalization
and put $\Sigma=\nu^{-1}(Z\cap L_\infty)$. On $C\setminus\Sigma$, write $(z_m,w_m)=F^m\circ\nu$. If all strict forward images of an irreducible curve $Z\neq L_\infty$
avoid $L_a\cap\Cb^2$ then each $w_m$ omits $c_a,r,\infty$. As in the quadratic case, they have to be nonconstant thus they form a finite family. Since every $w_m$ also avoids all roots of $S_\eta$,
the identity
\[
 Az_m=\frac{w_{m+1}-c_a}{(w_m-c_a)S_\eta(w_m)}-a(w_m+c_a)
\]
also implies the finiteness of $\{z_m\}_{m\geq0}.$ The one of Lemma~\ref{le-delta} follows from the fact that $\tilde L_\lambda$ intersects all the other components of $F^{-1}(L_a)$ in $\Cb^2.$

\medskip

\noindent\emph{\textbf{Dynamics at infinity and monodromy.}}
The restriction at infinity is, when $w\neq0$,
$$ \widetilde f_{\lambda}[z:w:0]=[\beta z:Az+aw:(B/\eta)w].$$
At $B=0$, this gives to the same homography as in the quadratic
case. Its fixed points are $q=[\beta-a:A:0]$ and $s=[0:1:0]$,
with tangent multipliers $a/\beta$ and $\beta/a$, respectively,
and transverse multipliers zero. The preimages $x_n=(f_{\lambda|L_\infty})^{-n}(I)$ converges towards $q$ and the local picture is the same as for $d=2.$

The monodromy argument requires a few additional computations.
When $A=0$, the induced map on the base of the pencil is
\[
 g_{B}(w)=
 \frac{c_a+a(w^2-c_a^2)S_\eta(w)}
      {1+Bw^{d-1}(w-c_a)}.
\]
Outside a finite set of values of $B$
not containing $0$, this map has degree $d$ and its fixed point
$c_a$ is noncritical, since
\[
 g_{B}'(c_a)=2ac_aS_\eta(c_a)-Bc_a^d.
\]
Thus $c_a$ has an infinite backward orbit, and the pencil argument
of Lemma~\ref{le-pinceau} applies.

After removing the marked fixed direction $c_a$, the fixed-point
equation is 
\begin{equation}\label{eq-higher-monodromy}
 Bw^d+P(w)=0,\ \text{ where }\ P(w)=1-a(w+c_a)S_\eta(w).
\end{equation}
Since $P(0)\neq0$, choosing a local holomorphic $d$-th root of $P$
near $0$ gives that the monodromy around $B=\infty$ is a $d$-cycle.

From this,  the proof of Theorem~\ref{th-inter}, as well as Proposition~\ref{prop-trans} is exactly the same.

\medskip
\noindent\emph{\textbf{A polynomial perturbation.}}
It remains to choose a suitable polynomial perturbation to apply the construction of Dujardin-Favre-Ruggiero~\cite{DFR} and obtain
the analogue of Lemma~\ref{le-WsZar}. A first idea would be to use the same type of perturbation as in degree $2$. However, in this
case, the indeterminacy point $I$ splits into $d-1$ fixed points, one of which is superattracting. We therefore use the following perturbation
\begin{equation}\label{eq-higher-DFR}
\begin{aligned}
 f(z,w)=\bigl(&
 (\alpha+\beta wS_\eta(w))z+\epsilon w+\delta z^d,\\
 &c_a+(w-c_a)\bigl[S_\eta(w)(Az+a(w+c_a))
                  +\delta z^{d-1}\bigr]+\theta\bigr).
\end{aligned}
\end{equation}
For sufficiently small $\delta,\theta$, the fixed points $s$ and $q$ persist as an attracting fixed point and a saddle fixed point, respectively. When $\delta\neq0$, the map $f$ extends to an endomorphism of $\Pb^2$, and it is easy to check that $I=[1:0:0]$ becomes a parabolic fixed point of multiplicity $d-1$ for $f_{|L_\infty}$. In particular, $f_{|L_\infty}$ has no superattracting fixed point. For $\theta\neq0$, the proofs of Lemma~\ref{le-WsZar} and then Proposition~\ref{prop-DFR} proceed as when $d=2.$

\medskip
\noindent\emph{\textbf{Parameter choices and the final perturbation.}} To prove Theorem~\ref{th-chasse-d3}, first fix $\eta\neq0$ and the parameters $\alpha,\beta,\epsilon$ as above. Then, for $|a|$ sufficiently large and $A\neq0$ sufficiently small, Lemma~\ref{le-homof0}, Theorem~\ref{th-chasse-julia} and the results of Section~\ref{sec-existence} (see in particular Lemma~\ref{le-init} and \cite[Lemma~4.12]{gtv}) ensure that all the assumptions of Section~\ref{sec-ass} hold in a neighborhood of $\widetilde f_{(a,A,0)}$.

\subsection{Rescaling toward correspondences}\label{sec-chasse-julia}
We now prove the small Julia condition simultaneously for all $d\geq2$.
Throughout this subsection, put
\[
 S(w)=\begin{cases}
 1,&d=2,\\
 1+\eta w^{d-2},&d\geq3,
 \end{cases}
\]
where, for $d\geq3$, $\eta\neq0$ is fixed and sufficiently small.
On a first reading, the reader may only consider the case $d=2$ where all the ideas already appear.

The parameter $B$ is not considered in this section so we write $F_{a,A}=f_{(a,A,0)}$ for $d=2$ and
$F_{a,A}=\widetilde f_{(a,A,0)}$ for $d\geq3$. Both families then have the affine expression
\begin{equation}\label{eq-all-degrees-affine}
 F_{a,A}(z,w)=\bigl((\alpha+\beta wS(w))z+\epsilon w,\,
 c_a+(w-c_a)S(w)(Az+a(w+c_a))\bigr),
\end{equation}
where $a(c_a^2-1)=c_a$ and $c_a\sim-1$ as $|a|\to\infty$.

We start with the following elementary result.
\begin{lemma}\label{le-rplus}
Let $f\in\mathrm{End}_d^2$ satisfy Assumptions~\ref{verti}
and~\ref{hyperbolic} with $s=0$ and $u=1$. With the notation
of these assumptions, $f$ has a unique fixed point $r_+(f)$
in $\mathcal U_+$. Moreover, if $r_+(f)\in J_2(f)$ and
$r_+(f)\in\Db_R\times U_+$, then
\[
 \Lambda(f)\subset J_2(f).
\]
\end{lemma}
\begin{proof}
By Assumption~\ref{verti}, $f_{|\mathcal U_+}$ is a horizontal-like map of degree $1$ and index $(0,2)$.
Its inverse branch $G_+$ is defined near $\overline{\mathcal U_+}$ and satisfies
\[
 G_+(\overline{\mathcal U_+})\Subset\mathcal U_+.
\]
It therefore has a unique fixed point $r_+(f)$, which is also the unique fixed point of $f$ in $\mathcal U_+$.

By Assumption~\ref{hyperbolic}, the two maps $f^2_{|\mathcal V_\pm}$ have inverse branches $g_\pm$ on
the common domain $\mathcal D=\Db_R\times\Db_{3/2}$,
with images relatively compact in $\mathcal D$.
These branches define the limit set $\Lambda(f)$.
If $r_+(f)\in J_2(f)\cap(\Db_R\times U_+)$, then
$r_+(f)\in\mathcal D$, and all iterated compositions of
$g_+$ and $g_-$ applied to $r_+(f)$ belong to $J_2(f)$
by total invariance. Every point of $\Lambda(f)$ is a limit
of such points, so closedness gives $\Lambda(f)\subset J_2(f)$.
\end{proof}

In order to prove that $r_+(f)\in J_2(f)$, we will use the following
entropy criterion of de Th\'elin \cite[\S2.2]{deThelin05} and
Dinh \cite[Theorem~2.1]{Dinh07}.
\begin{theorem}\label{th-dtd-entropy}
Let $f$ be in $\mathrm{End}_d^2$. If $\nu$ is an invariant ergodic
measure on $\Pb^2$ of entropy larger than $\log d$, then
$\supp(\nu)\subset J_2(f)$.
In particular, if $f^N$ admits $M>d^N$ inverse branches on a
bidisk $D$, with pairwise disjoint images relatively compact
in $D$, then the associated limit set $K$ is contained in $J_2(f)$.
\end{theorem}
\begin{proof}
The first point is simply the criterion of de Th\'elin and Dinh.
For the second, let $g_1,\ldots,g_M$ be inverse branches of $f^N$
on $D$, with pairwise disjoint images satisfying
$\overline{g_i(D)}\subset D$.
They automatically contract the Kobayashi metric of $D$, and
\[
 K=\bigcap_{n\geq1}\ \bigcup_{(i_1,\ldots,i_n)\in\{1,\ldots,M\}^n}
 \overline{g_{i_1}\circ\cdots\circ g_{i_n}(D)}.
\]
This set is invariant under $f^N$ and supports a Bernoulli measure
of entropy $\log M>N\log d$. Thus,
$K\subset J_2(f^N)=J_2(f)$.
\end{proof}
We now use this criterion to establish the following result.
\begin{theorem}\label{th-chasse-julia}
Fix $d\geq2$ and $S$ as above, and let $A\neq0$ be sufficiently
small. For $|a|$ sufficiently large, every $f\in\mathrm{End}_d^2$
sufficiently close to $F_{a,A}$ satisfies
\[
 \Lambda(f)\subset J_2(f).
\]
\end{theorem}
\begin{rk}
The neighborhood may depend on $a,A,S$. In particular, this makes
no assertion about arbitrary perturbations of the map with $A=0$.
For $d\geq3$, we keep $\eta\neq0$ fixed throughout the proof.
Together with the local construction, the theorem supplies the
remaining small Julia condition in Section~\ref{sec-ass} for $k=2$.
\end{rk}
\begin{proof}
The proof has three main steps. First, we find a suitable rescaling under which the inverse dynamics converges to a correspondence.
Then, we construct sufficiently many inverse branches to apply Theorem~\ref{th-dtd-entropy} in a neighborhood of $F_{a,A}$.
Finally, we verify that the fixed point $r_+(f)$ lies in the required region.

\medskip

\noindent\emph{\textbf{The limiting correspondence.}}
The preimage of the invariant line
$L_a=\{w=c_a\}$ consists of $L_a$,
$\widetilde L_\lambda=\{Az+a(w+c_a)=0\}$, where
$\lambda=(a,A,0)$, and the additional lines $\{w=r\}$ with
$S(r)=0$ when $d\geq3$.

For $A\neq0$, the intersection of $L_a$ and
$\widetilde L_\lambda$ is $(-2ac_a/A,c_a)$, a critical point
of $F_{a,A}$ which escapes to infinity as $A\to0$.
Moreover, for small $A$ and large $|a|$, the
corresponding domains $\mathcal U_-$ and $\mathcal U_+$ in
Assumption~\ref{verti} lie in tubular neighborhoods of radius
$O(|a|^{-1})$ around $L_a$ and $\widetilde L_\lambda$,
respectively.

To keep this geometry for $A\neq0$ and $a$ finite, we rescale
coordinates and parameters
\[
 \tau=1/a,\qquad v=Az/a,\qquad \varepsilon_\tau=A\epsilon/a.
\]
Viewed as a function of $\tau$, the fixed point $c_a$ extends holomorphically to $\tau=0$,
with value $-1$. In these coordinates, the two lines become $\{w=c_a\}$ and
$\{w=-c_a-v\}$, meeting at $(-2c_a,c_a)$.
As $\tau\to0$, they converge respectively to $\{w=-1\}$ and
$\{w=1-v\}$, whose intersection is $(v,w)=(2,-1)$.
This explains the special role of $v=2$ in what follows.

The inverse equations for the conjugated map are
\begin{equation}\label{eq-rescaling-inverse}
 \begin{cases}
 (\alpha+\beta wS(w))v+\varepsilon_\tau w-v'=0,\\
 \tau c_a+(w-c_a)S(w)(w+c_a+v)-\tau w'=0.
 \end{cases}
\end{equation}
At $(\tau,\varepsilon_\tau)=(0,0)$, the second equation becomes
\[
 (w+1)S(w)(w-1+v)=0.
\]
We keep the roots coming from $w=-1$ and $w=1-v$, whereas
those corresponding to $S=0$ will not be needed in what follows.
The retained inverse dynamics is then described by
\[
 \Gamma=\mathrm{gr}(\mathcal L)+\mathrm{gr}(\mathcal P),\qquad
 \mathcal L(v)=\chi_-v,\quad
 \mathcal P(v)=v\bigl[\alpha+\beta(1-v)S(1-v)\bigr],
\]
where $\chi_-:=\alpha-\beta S(-1)$. This multiplier plays a
special role below, as does
$\chi_+:=\mathcal P'(0)=\alpha+\beta S(1)$.
By our choice of parameters, $0<|\chi_-|<1<|\chi_+|$.
The fact that $\mathcal L$ is contracting is not an obstacle,
since we only need sufficiently many contracting inverse branches
of long iterates.

Since $\deg\mathcal P=d$, the correspondence has degrees
\begin{equation}\label{eq-julia-correspondence-degrees}
 d_0(\Gamma)=2,\qquad d_1(\Gamma)=d+1=:M>d.
\end{equation}

By the equidistribution theorem of Dinh--Sibony~\cite{DS-MT},
in the form \cite[Fact~4.3]{BS-correspondences}, there is a
probability measure $\nu$, charging no polar set, such that
\begin{equation}\label{eq-julia-equidistribution}
 \Gamma^*\nu=M\nu,\qquad
 M^{-n}(\Gamma^n)^*\delta_x\longrightarrow\nu
\end{equation}
outside a polar exceptional set of targets $x$. Since
$\mathcal L^*\nu+\mathcal P^*\nu=M\nu$, positivity gives
\[
 \mathcal P^{-1}(\supp\nu)\subset\supp\nu.
\]
This closed set has more than two points, so it contains
$J(\mathcal P)$. In particular, $0\in\supp\nu$, since $0$ is a
repelling fixed point of $\mathcal P$, with multiplier $\chi_+$.
Moreover,
\begin{equation}\label{eq-julia-distinguished-symbol}
 M^{-n}\mathcal P^*(\Gamma^{n-1})^*\delta_x
 \longrightarrow\nu^{\mathcal P}:=M^{-1}\mathcal P^*\nu.
\end{equation}
This measure gives positive mass to every neighborhood of $0$ and
counts itineraries whose first forward symbol is $\mathcal P$.

Let $\iota$ be the local inverse of $\mathcal P$ fixing $0$.
It gives a degenerate inverse branch in dimension $2$.
Indeed, choose a small bidisk $\mathcal N$ about $(0,1)$ such that
\[
 H_0(v',w')=(\iota(v'),1-\iota(v')),\qquad
 H_0(\mathcal N)\Subset\mathcal N,
\]
and a disk $V$ about $0$ whose lifted graph
$\{(v,1-v):v\in\overline V\}$ lies in $\mathcal N$.
This map has a unique fixed point $r_+:=(0,1)$, which will be useful
to connect the Julia set $J_2(f)$ to $\Lambda(f)$ after small
perturbations using Lemma~\ref{le-rplus}. Observe that, in these coordinates, the set
$\Lambda(F_{a,A})$ from Assumption~\ref{hyperbolic} degenerates
to the two points $\{(0,1),(0,-1)\}$, but the fact that $(0,1)$
also corresponds to $r_+$ cannot be used directly. What will be
useful is the persistence of inverse branches, i.e., for small
$\tau\neq0$ and $\varepsilon_\tau=A\epsilon\tau$, there is a
small perturbation of $H_0$ corresponding to an inverse branch
of the conjugated map $F_{a,A}$.

In order to construct more inverse branches and lift them to
dimension $2$, we consider regular inverse itineraries of the
correspondence avoiding the finite set
\[
 E=\{2,\infty\}\cup\mathrm{CV}(\mathcal P)
                  \cup\{1-r:S(r)=0\},
\]
where $\mathrm{CV}(\mathcal P)$ is the set of critical values
of $\mathcal P$. Observe that the source Jacobian determinants in
\eqref{eq-rescaling-inverse}, on the two retained sheets, are
\begin{equation}\label{eq-julia-lifting-jacobians}
 \chi_-(v-2)S(-1),\qquad
 (2-v)S(1-v)\mathcal P'(v).
\end{equation}
Regular inverse itineraries avoiding $E$ can therefore be lifted
to dimension $2$.

\medskip
\noindent\emph{\textbf{Returning inverse branches.}}
We now follow the inverse-branch construction in
\cite{BS-correspondences}.
In our case, we avoid a larger finite set $E$. The arguments
remain the same, but we recall them for convenience. Fix a disk $V_0\Subset V$ about $0$, choose
\[
 0<\xi<\frac1{10}\min\{1,\nu^{\mathcal P}(V_0)\},
\]
and put $e=\#E$. Take $l$ sufficiently large that
\[
 \frac{2e}{M-2}\left(\frac2M\right)^l<\xi.
\]
Choose bounded simply connected planar domains
\[
 W_0\Subset W\Subset U\Subset
 \Cb\setminus\bigcup_{j=0}^l\Gamma^j(E),
 \qquad \nu(W_0)>1-\xi.
\]
To obtain them, cut along parallel rays of zero $\nu$-mass from the
finitely many excluded points to infinity, and exhaust the simply
connected complement by conformal disks. Such a direction exists,
since $\nu$ has no atoms.

We consider inverse itineraries on $U$ which, at every stage, are regular and have images avoiding $E$.
We call them \emph{admissible} and denote their number at depth
$k$ by $G_k$. Then, by the choice of $U$,
\[
 G_l=M^l,\qquad G_{k+1}\geq MG_k-e2^{k+1}\quad(k\geq l).
\]
Indeed, each admissible itinerary has $M$ regular extensions,
since its initial image is simply connected and avoids
$\mathrm{CV}(\mathcal P)$, before excluding source images
meeting $E$. For each of the $2^{k+1}$ labelled words, distinct
inverse branches have disjoint source images, so at most $e$
extensions are excluded. Consequently,
\begin{equation}\label{eq-julia-branch-count}
 G_n\geq\left[1-\frac{2e}{M-2}\left(\frac2M\right)^l\right]M^n
 >(1-\xi)M^n\qquad(n\geq l).
\end{equation}
For the same reason, the sum of the spherical areas of the initial
images over all admissible itineraries on $U$ is at most $2^n$,
normalizing the area of $\Pb^1$ to one. Discarding at most $\xi M^n$
branches makes each remaining area at most $\xi^{-1}(2/M)^n$.
The area--diameter estimate \cite[Appendix]{briendduval2} then gives,
for large $n$,
\[
 \operatorname{diam}v_0(W)
 \leq C_{\xi,U,W}\left(\frac2M\right)^{n/2}.
\]
Apply~\eqref{eq-julia-equidistribution} and
\eqref{eq-julia-distinguished-symbol} at a point of $W_0$ outside
the polar exceptional set. The two discards have total relative
size less than $2\xi$. For large $n$, the remaining itineraries
therefore include $N_n\geq c_0M^n$ returns, with $c_0>0$ independent
of $n$, of the form
\[
 v_n(x)=x,\quad v_{i+1}=R_i(v_i),\quad
 R_i\in\{\mathcal L,\mathcal P\},\qquad v_0(W)\Subset W.
\]
These returns also include one with $R_0=\mathcal P$ and
$v_0(W)\Subset V$. Indeed, $\nu^{\mathcal P}\leq\nu$, so
\[
 \nu^{\mathcal P}(V_0\cap W_0)
 >\nu^{\mathcal P}(V_0)-\xi>9\xi.
\]
The shrinking diameters, together with $W_0\Subset W$ and
$V_0\Subset V$, give the stated compact inclusions.

Since $M=d+1>d$, fix $n$ and these finitely many itineraries so that
\begin{equation}\label{eq-julia-entropy-threshold}
 N_n>d^n.
\end{equation}

\medskip
\noindent\emph{\textbf{Lifting and perturbing.}}
At the limiting parameter $(\tau,\varepsilon_\tau)=(0,0)$,
lift an itinerary by setting $w_i=-1$ when $R_i=\mathcal L$
and $w_i=1-v_i$ when $R_i=\mathcal P$, for $0\leq i<n$.
Choose $\rho_1$ sufficiently large and put $D=W\times\Db_{\rho_1}$.
The selected returns then give $N_n$ lifted maps $g_{j,0}$ with
$g_{j,0}(\overline D)\Subset D$, including a map $g_{0,0}$
satisfying $g_{0,0}(\overline D)\Subset\mathcal N$.
Composing this map with iterates of $H_0$ gives inverse
itineraries converging to $(0,1)$.

Two distinct branches cannot give
the same sequence of lifted points $(v_i,w_i)_{0\leq i\leq n}$.
Indeed, different symbols give different lifted points at the
corresponding step, since the two lines meet only at $v=2\in E$.
On the other hand, distinct branches with the same symbolic
itinerary have disjoint source images.

By~\eqref{eq-julia-lifting-jacobians}, these finitely many
itineraries and the local branch $H_0$ persist for small
$(\tau,\varepsilon_\tau)$. By compactness over $\overline D$,
distinct branches still give distinct sequences of lifted points
after continuation. For $\tau\neq0$, equality of two initial
points would force equality of their forward orbits, so the
continued return branches have pairwise disjoint images.

Taking $\varepsilon_\tau=A\epsilon\tau$, we obtain these branches
for $F_{a,A}$ when $A\neq0$ is small and $|a|$ is large, and then
for all $f\in\mathrm{End}_d^2$ sufficiently close to $F_{a,A}$.
We denote the continued branches by $g_{j,f}$ and $H_f$,
using the same rescaled coordinates.

By~\eqref{eq-julia-entropy-threshold} and
Theorem~\ref{th-dtd-entropy}, the return limit set satisfies
$K_f\subset J_2(f)$. The connecting inverse branch $g_{0,f}$ gives
\[
 \varnothing\neq g_{0,f}(K_f)\subset\mathcal N\cap J_2(f).
\]
The iterates of $H_f$ converge to its unique fixed point,
which we identify below with $r_+(f)$. Backward invariance
and closedness therefore give $r_+(f)\in J_2(f)$.

\medskip
\noindent\emph{\textbf{Identification of the positive fixed point.}}
It remains to check the hypotheses of Lemma~\ref{le-rplus} for $r_+(f)$. Write $(z_+,w_+)$ for the original
coordinates of the fixed point just constructed for $F_{a,A}$.
Since $w_+$ is close to $1$ and $\chi_+\neq1$, its fixed-point
equations give
\[
 z_+=\frac{\epsilon w_+}{1-\alpha-\beta w_+S(w_+)},\qquad
 S(w_+)\bigl(w_++c_a+A\tau z_+\bigr)=\tau,
\]
where the second identity follows by dividing by $w_+-c_a\neq0$.
Since $S(1)\neq0$, these equations yield, uniformly for bounded
small $A$,
\[
 w_+=1+O(\tau),\qquad
 z_+=\frac{\epsilon}{1-\chi_+}+O(\tau).
\]
By uniqueness in the local construction, this is the positive
fixed point $r_+(F_{a,A})$. Thus its continuation $r_+(f)$
belongs to $J_2(f)$ for nearby $f$.

Moreover, by our choice of parameters,
$|\epsilon|<1/10$ and $|1-\chi_+|>1/2$, so $|z_+|<1/4$
for large $|a|$. Since this point belongs to $\mathcal U_+$,
it lies in $\Db_R\times U_+$, and the same holds after a
sufficiently small perturbation. Lemma~\ref{le-rplus} now gives
$\Lambda(f)\subset J_2(f)$.
All inverse itineraries were chosen for fixed $S$ before taking
$|a|$ large. No uniformity across the degree drop $\eta=0$ is used.\end{proof}

\section{Size of the exceptional set}\label{sec-size}
This section is devoted to the proof of Theorem~\ref{th-mono-gorbo}. The main ingredient is the transitivity of the monodromy action on eigendirections, proved by the second author jointly with Igors Gorbovickis in~\cite{t-gorbovickis}.

\begin{proof}[Proof of Theorem~\ref{th-mono-gorbo}]
Choose an irreducible component of $\mathcal Z$ whose incidence map dominates $\Pb^k\times M$, and still denote it by
$\mathcal Z$. Suppose that the conclusion does not hold.

Take a general relative linear section
$\mathcal T\subset\mathcal Z$ such that the map
\[
 \rho\colon\mathcal I:=\{(Y,x,f):(Y,f)\in\mathcal T,\ x\in Y\}
 \longrightarrow\Pb^k\times M,\qquad
 (Y,x,f)\longmapsto(x,f),
\]
is generically finite and dominant. More precisely, if $q=\dim\mathcal Z-\dim M$, we take a general linear section of codimension $q-(k-u)$ in the Chow factor. There exists  a dense Zariski open subset $U\subset\Pb^k\times M$ on which $\rho$ is a finite covering, and such that $x$ is a regular point of $Y$ for all $(Y,x,f)\in\mathcal I,$ with $(x,f)\in U.$

Choose a very general $f_0\in\Omega_u$ such that
$U_{f_0}:=\{x\in\Pb^k\,:\,(x,f_0)\in U\}$ is nonempty, $D_{p(f_0)}f_0$ only has simple eigenvalues and $p(f_0)$ is not
postcritical. In particular, the backward orbit of $p(f_0)$ is Zariski dense and thus, intersects $U_{f_0}.$

We may therefore choose $m\geq0$ and $(Y_0,x_0,f_0)$ over $U$ such that
\[
 f_0^m(x_0)=p(f_0).
\]
Since $D_{x_0}f_0^m$ is invertible and $Y_0$ is smooth at $x_0,$ we have that
\[
 T_0:=D_{x_0}f_0^m(T_{x_0}Y_0)
\]
is a subspace of dimension $u$ of the tangent space of $p(f_0).$ Moreover, we can follow $x_0$ as a preimage $x_f$ of $p(f)$ by $f^m$ for $f$ in a small neighborhood $V\subset\Omega_u$ of $f_0.$ As $\rho$ is a covering over $(x_0,f_0)\in U$, there exists also $Y_f$ depending holomorphically on $f\in V$ such that $x_f\in Y_f$ and $Y_{f_0}=Y_0.$ This defines a subspace $T_f:=D_{x_f}f^m(T_{x_f}Y_f)$ of dimension $u$ at $p(f).$

If, for some $f\in V$, $T_f$ is complementary to $T_{p(f)}W^s_{p(f),loc}$, we are done. In fact, a monodromy
argument shows that this can always be achieved, possibly after changing the preimage $x_f$ and the sheet $Y_f$ of the
covering $\rho$. More precisely, even if $T_0$ is not transverse to $T_{p(f_0)}W^s_{p(f_0),loc}$, we can extend a basis of $T_0$
to a basis of the full tangent space by adding eigenvectors of $D_{p(f_0)}f_0$. Let $E_0$ denote the $(s+1)$-dimensional
subspace spanned by these eigenvectors.

In the proof of \cite[Theorem~1.7]{t-gorbovickis}, the authors show that any two distinct eigendirections $v_1(f_0)$ and
$v_2(f_0)$ at a fixed point $p(f_0)\in\Cb^k$ with simple eigenvalues can be exchanged by continuation along a loop
in $M$ based at $f_0$, while the fixed point returns to itself. This holds for both $M=\mathrm{End}_d^k$ and $M=\mathrm{Poly}_d^k$.

Combining such loops, we obtain a loop $\gamma=(f_t)_{t\in[0,1]}$ along which the continuations $p_t$ and $E_t$ satisfy
\[
 p_1=p(f_0),\qquad E_1=T_{p(f_0)}W^s_{p(f_0),loc}.
\]
Near $t=0$, this loop lifts to a path $(Y_t,x_t,f_t)$ with $Y_t=Y_{f_t}$, $x_t=x_{f_t}$ and $f_t^m(x_t)=p_t$.
The lift extends as long as $(x_t,f_t)\in U$ and $D_{x_t}f_t^m$ is invertible.

Recording also the continued eigendirections gives a finite algebraic covering of a suitable Zariski open subset of $M$. On the component containing our initial data, the condition
\[
 D_xf^m(T_xY)\oplus E=T_p\Pb^k
\]
is Zariski open and nonempty, since it holds initially. Its complement therefore projects to a proper algebraic subvariety of $M$. After slightly perturbing $f_0$ within $V$
and deforming $\gamma$ without changing its monodromy, we may avoid this subset and the loci where the lift ceases to be defined.

The lifted path is then defined for all $t\in[0,1]$ and preserves the above complementarity. At its endpoint, $x_1\in(Y_1)_{\mathrm{reg}}$, $f_0^m(x_1)=p(f_0)$, and
\[
 D_{x_1}f_0^m(T_{x_1}Y_1)
 \oplus T_{p(f_0)}W^s_{p(f_0),loc}
 =T_{p(f_0)}\Pb^k,
\]
as required.
\end{proof}

\section{Proofs of the main results}\label{sec-mainresult}

We first record the arithmetic argument common to the proofs.
We then prove Theorem~\ref{th-sparsity} and deduce
Theorems~\ref{th-special-hypersurface} and~\ref{th-demarco2}.
The proof of Theorem~\ref{th-unif} uses the same ingredients,
followed by induction on the dimension.

\subsection{The common arithmetic argument}

Let $M$ be either $\mathrm{End}_d^k$ or $\mathrm{Poly}_d^k$,
and let $G$ be the corresponding conjugacy group,
$\mathrm{PGL}(k+1)$ or $\mathrm{Aff}(\Ab^k)$.
Observe that specialness of a pair $(Z,f)$ is preserved by conjugacy. Moreover, their Zariski closure defines subvarieties in $\mathcal X_D^u\times M$ which are defined over $\overline{\Qb}.$

\begin{lemma}\label{lem-closure-Qbar}
The Zariski closure of the set of special pairs in $\mathcal X_D^u\times M$ is defined over $\Qb$. Each of its irreducible components is defined over $\overline{\Qb}$ and contains a Zariski dense set of special pairs defined over $\overline{\Qb}$.
\end{lemma}
\begin{proof}
Let $(Y,f)$ be a special pair, and let $Z\supseteq Y$, $m\geq1$
and $g\colon Z\to Z$ be as in the definition, with
$g^*\mathcal A\simeq\mathcal A^{\otimes q}$ for some ample line
bundle $\mathcal A$ on $Z$ and some $q\geq2$.
By \cite[Corollary~2.2]{Fakhruddin}, there exist an embedding
$i\colon Z\hookrightarrow\Pb^N$ and an endomorphism $G$ of
$\Pb^N$ of degree at least $2$ such that $G\circ i=i\circ g$.

Hence $(Y,f)$ is special if and only if there exist such data
$(Z,g,i,G)$ and integers $m\geq1$, $a>b\geq0$ with
\[
Y\subseteq Z,\quad f^m(Z)=Z,\quad G\circ i=i\circ g,\quad
g\circ f^m=f^m\circ g\ \text{on } Z,\quad g^a(Y)=g^b(Y).
\]

We record $g$ and $i$ by their graphs, which are points of
Chow varieties of $\Pb^k\times\Pb^k$ and
$\Pb^k\times\Pb^N$, respectively, and $G$ by its coefficients.
Fix $m,a,b,N$, the degree of $G$, the dimension and degree
of $Z$, and the degrees of these graphs.
Observe that all the conditions are algebraic in these parameters, with rational coefficients but some, as $G$ being an endomorphism, are open. Therefore, the tuples $(Y,f,Z,g,i,G)$ satisfying the above conditions form a merely constructible set $\mathcal S$ defined over $\Qb$.
Hence the points of $\mathcal S$ defined over
$\overline{\Qb}$ are Zariski dense in $\mathcal S$, and
they project to special pairs defined over $\overline{\Qb}$.

As these discrete data vary, the projections of the countably
many sets $\mathcal S$ cover all special pairs. Therefore
the Zariski closure of the set of special pairs coincides
with that of the set of special pairs defined over
$\overline{\Qb}$.
Moreover, since these parameter sets and projections
are defined over $\Qb$, the special pairs defined over
$\overline{\Qb}$ are invariant under the Galois action.
Their Zariski closure is therefore defined over $\Qb$.
Its irreducible components are therefore defined over
$\overline{\Qb}$, and the preceding density statement shows
that each contains a Zariski dense set of special pairs
defined over $\overline{\Qb}$.
\end{proof}

Every irreducible component of an $f$-special subvariety is
itself $f$-special. We may therefore pass to families of
irreducible components whenever necessary.

\begin{lemma}\label{le-arithmetic-transverse}
Let $\mathcal Z\subset\mathcal X_D^u\times M$ be an
irreducible algebraic family defined over a number field,
invariant under conjugacy and dominant over $M$.
Assume that its general member is irreducible and that some
$(Y,f)\in\mathcal Z$, with $f\in\Omega_u$, satisfies
\eqref{eq-pitch}. Then the special pairs are not Zariski
dense in $\mathcal Z$.

If $M=\mathrm{Poly}_d^k$, there also exist an integer $N\geq1$
and a dense Zariski open subset $\mathcal Z^\circ\subset\mathcal Z$,
defined over $\overline{\Qb}$ and invariant under
conjugacy, such that, for every
$(Y,f)\in\mathcal Z^\circ(\overline{\Qb})$, there is a proper
subvariety $W_{Y,f}\subsetneq Y$ of degree at most $N$ with
\[
 \mathrm{Preper}(f)\cap Y\subseteq W_{Y,f}.
\]
\end{lemma}

\begin{proof}
The transverse intersection persists on a nonempty open subset
of $\mathcal Z$. Let $\mathcal S\subset M$ be the algebraic
family given by Lemma~\ref{good-family-end} or
Lemma~\ref{good-family-poly}, according to $M$.
Since $\mathcal Z$ is invariant under conjugacy, we can choose
an irreducible component
\[
 \mathcal Y\subset
 \mathcal Z\cap(\mathcal X_D^u\times\mathcal S)
\]
whose conjugacy translates are Zariski dense in $\mathcal Z$.
This family is dynamically full and maximally varying:
the projection of $\mathcal S$ to the moduli space is finite,
and the stabilizers of maps are finite.

The conjugacy translates of $\mathcal Y$ cover a dense
Zariski open subset of $\mathcal Z$. After conjugating
$\mathcal Y$ by an element of $G(\overline{\Qb})$, we may
therefore assume that it meets the transverse open subset.

After restricting to the open set of maps used in
Proposition~\ref{prop-special}, Theorem~\ref{th-main-blender} and that proposition
give a nonempty open subset $\Omega\subset\mathcal Y$
contained in $\supp(\mu_{\mathcal Y})$ and containing no
special pair. If special pairs were Zariski dense, choose a
generic sequence of such pairs defined over $\overline{\Qb}$,
and let $F_j$ be their Galois orbits over a number field of
definition of the family. Theorem~\ref{tm:small-distrib} gives
\[
 \frac{1}{\#F_j}\sum_{t\in F_j}\delta_t
 \longrightarrow
 \frac{\mu_{\mathcal Y}}{\mu_{\mathcal Y}(\mathcal Y)}.
\]
Every measure on the left gives zero mass to $\Omega$,
whereas the limit has positive mass there. This is a
contradiction.

Now suppose that $M=\mathrm{Poly}_d^k$.
Theorem~\ref{th-blender-poly} gives a nonempty open subset
$\Omega\subset\supp(\mu_{\mathcal Y}^{\mathrm{pol}})$.
The general member of $\mathcal Y$ is not contained in
$H_\infty$, since the transverse intersection given by \eqref{eq-pitch} lies in $\Ab^k$.
The function $G_{\mathcal Y}$ is continuous and nonnegative,
and
\[
 \mu_{\mathcal Y}^{\mathrm{pol}}
 =(dd^cG_{\mathcal Y})^{\wedge\dim\mathcal Y}.
\]
The function $G_{\mathcal Y}$ cannot vanish identically on
$\Omega$, since the measure
would then vanish there as well. Consequently,
\[
 \int_{\mathcal Y(\Cb)}G_{\mathcal Y}\,
             \mu_{\mathcal Y}^{\mathrm{pol}}>0.
\]
Theorem~\ref{prop-empty-interior} now gives a dense Zariski
open subset $\mathcal Y^\circ$ and a uniform degree bound
for a proper subvariety containing the preperiodic
points of each algebraic pair in $\mathcal Y^\circ$.

The Zariski interior of the union of the conjugacy translates of $\mathcal Y^\circ$ gives the desired invariant open subset
$\mathcal Z^\circ$. 
\end{proof}

\subsection{Sparsity and the exceptional locus}
\begin{proof}[Proofs of Theorems~\ref{th-sparsity}
and~\ref{th-special-hypersurface}]
Theorem~\ref{th-sparsity} is a direct consequence of Lemma~\ref{le-arithmetic-transverse}, applied to each irreducible component $\mathcal Z$ of $\mathcal X_D^u\times M$, possibly after decomposing $\mathcal Z$ into finitely many components of smaller degree, as explained in Section~\ref{sec-chow}.

For Theorem~\ref{th-special-hypersurface}, let $\mathcal Z_D^u$ be the Zariski closure of the special
pairs in $\mathcal X_D^u\times M$ and we do not yet assume that
it is proper. Since $\mathcal X_D^u$ is projective, the incidence map
\[
 \rho\colon
 \mathcal I_D^u:=
 \{(Y,x,f):(Y,f)\in\mathcal Z_D^u,\ x\in Y\}
 \longrightarrow\Pb^k\times M,\qquad
 (Y,x,f)\longmapsto(x,f),
\]
has a closed image $\mathcal H_D^u$ whose fiber over $f\in M$ is
\[
 \mathcal H_D^u(f):=\bigcup_{\substack{Y\in\mathcal X_D^u\\
                         (Y,f)\in\mathcal Z_D^u}}Y.
\]
Suppose that this image is all of $\Pb^k\times M$, and
choose an irreducible component of $\mathcal Z_D^u$ whose
incidence is dominant.

Passing to families of irreducible components of its members
if necessary, we obtain an irreducible family $\mathcal Z$
with generically irreducible members, dominant incidence
and a Zariski dense set of special pairs. This family is
defined over a number field and invariant under conjugacy.

By Theorem~\ref{th-mono-gorbo}, some pair in $\mathcal Z$
satisfies \eqref{eq-pitch}. This contradicts
Lemma~\ref{le-arithmetic-transverse}.
Thus $\mathcal H_D^u$ is a proper closed subset of
$\Pb^k\times M$, and $\mathcal H_D^u(f)$ is a proper subvariety of
$\Pb^k$, possibly empty, for general $f\in M$. Put
\[
 \mathcal H_D:=\bigcup_{u=1}^{k-1}\mathcal H_D^u.
\]
By the preceding argument, its fiber $\mathcal H_D(f)$ is a
proper algebraic subvariety of $\Pb^k$ for general $f\in M$.

Since $\mathcal H_D$ is algebraic, it is contained in a hypersurface of $\Pb^k\times M$ which is dominant over $M$ and thus, we can replace $\mathcal H_D$ by this hypersurface. Therefore, there exists a dense Zariski open subset
$\mathcal U_D\subset M$, such that the fiber $\mathcal H_D(f)$ is a hypersurface of constant degree $B.$ This defines a morphism
\[\mathcal U_D\to\mathcal X_B^{k-1}.\]
By construction, every $f$-special subvariety of degree
at most $D$ and dimension between $1$ and $k-1$ is contained
in $\mathcal H_D(f)$.
This proves Theorem~\ref{th-special-hypersurface}.
\end{proof}

\begin{rk}\label{rk-fini}
In codimension one, the generic fibers of $\mathcal Z_D^{k-1}$ are finite. Indeed, every irreducible hypersurface in such a fiber is an irreducible component
of $\mathcal H_D(f)$, so there are at most $B$ of them. More generally, every support in this fiber is a union of these components, leaving only finitely many possibilities.
The same holds if cycle multiplicities of total degree at most $D$ are retained.
\end{rk}

\subsection{Proof of Theorem~\ref{th-demarco2}}

\begin{proof}
Let $M=\mathrm{End}_d^2$ and let $\mathcal Z_D^1$ be the minimal exceptional locus constructed above. Suppose that one of its irreducible components dominates $M$.
Passing to a family of irreducible components of its members if necessary, Theorem~\ref{th-chasse-int} gives a pair satisfying \eqref{eq-pitch}. This again contradicts Lemma~\ref{le-arithmetic-transverse}.

Thus no irreducible component of $\mathcal Z_D^1$ dominates $M$. Since $\mathcal X_D^1$ is projective, the projection $V_D$ of $\mathcal Z_D^1$ onto $M$ is closed, and it is proper. Every map admitting a special curve of degree at most $D$ belongs to $V_D$, as required.
\end{proof}
\subsection{Proof of Theorem~\ref{th-unif}}\label{sec-unif}

\begin{proof}
Throughout, $M=\mathrm{Poly}_d^k$. We first work over $\overline{\Qb}$ and argue by induction on the dimension, for all degree bounds at once. A density argument then
extends the result to $\Cb$. For now, we only seek a counting bound outside a proper closed subset of $\Pb^k\times M$. We will replace this exceptional set by
a family of hypersurfaces at the end. All the sets constructed below are defined over $\overline{\Qb}$.

In dimension zero $u=0$, the degree bounds directly the number of points. Fix $u\geq1$ and $D\geq1$.

\medskip
\noindent\emph{\textbf{Reduction to smaller subvarieties.}}
We first find a proper closed subset $\mathcal E\subset\Pb^k\times M$ and an integer $N\geq1$ such that, for every pair $(Y,f)$ defined over
$\overline{\Qb}$, with $Y$ irreducible of dimension $u$ and degree at most $D$, either $Y\subset\mathcal E(f)$, or
\begin{equation}\label{eq-main-degree-reduction}
 \mathrm{Preper}(f)\cap Y\subseteq W_{Y,f} 
\end{equation}
for some proper subvariety $W_{Y,f}\subsetneq Y$ of degree at most $N$. Here, for a subvariety of mixed dimension, we add the degrees of its irreducible components.

Reduced irreducible cycles form open subsets of the $\mathcal X_e^u$, since their complements are finite unions of images of the proper maps~\eqref{eq-chow}. For $1\leq e\leq D$, we consider the irreducible components $\mathcal Z$ of the corresponding loci of pairs. They are defined over $\overline{\Qb}$ and invariant under conjugacy, since the conjugacy group is connected.

If the incidence map of $\mathcal Z$ is not dominant over $\Pb^k\times M$, we put the closure of its image in $\mathcal E$. Otherwise, the argument of Theorem~\ref{th-mono-gorbo} gives a pair in $\mathcal Z$ satisfying~\eqref{eq-pitch}. Lemma~\ref{le-arithmetic-transverse}, applied to $\overline{\mathcal Z}$, then gives \eqref{eq-main-degree-reduction}, with a uniform degree bound, on a dense invariant Zariski open subset $\mathcal Z^\circ\subset\mathcal Z$ defined over $\overline{\Qb}$. We remove this open subset and repeat the argument on the components of $\mathcal Z\setminus\mathcal Z^\circ$.

The dimension drops at each step, so the process ends. The finitely many proper closed sets collected along the way give $\mathcal E$, and we choose $N\geq1$ at least as large as all the resulting degree bounds.

\medskip
\noindent\emph{\textbf{Counting the points.}}
Enlarge $\mathcal E$ by adding $H_\infty\times M$ and the exceptional sets already obtained in dimensions $1,\ldots,u-1$ for the degree bound $N$. Their union $\mathcal H$ is still proper.

Let $(Y,f)$ be a pair defined over $\overline{\Qb}$, with $Y$ irreducible of dimension $u$ and degree at most $D$. If $Y\subset\mathcal H(f)$, there is nothing to count.
Otherwise, its preperiodic points outside $\mathcal H(f)$ are affine, hence lie in $W_{Y,f}$. This subvariety has at most $N$ components, each of dimension less than $u$ and degree at most $N$. The bounds already obtained in smaller dimensions therefore bound the number of these points. Summing over the at most $D$ components of an arbitrary cycle gives the desired bound in dimension $u$.

\medskip
\noindent\emph{\textbf{Conclusion over $\overline\Qb$.}}
For the given $D$, take the union of the exceptional sets obtained for $1\leq u\leq k-1$. As in the proof of Theorem~\ref{th-special-hypersurface}, over a dense Zariski open subset $\mathcal U_D\subset M$, its fibers lie in an algebraically varying hypersurface of constant degree. Choose $B$ at least as large as this degree and all the counting bounds. Adding a suitable multiple of $H_\infty$ gives a morphism
\[
 \mathcal H_D\colon\mathcal U_D\longrightarrow
 \mathcal X_B^{k-1},
\]
defined over $\overline{\Qb}$, satisfying the conclusion for pairs defined over $\overline{\Qb}$.

\medskip
\noindent\emph{\textbf{Passing to $\Cb$ by density.}}
Finally, suppose that the bound fails for some complex pair $(Y_0,f_0)$ with $f_0\in\mathcal U_D$. Choose $B+1$ distinct preperiodic points $x_1,\ldots,x_{B+1}$ in $Y_0\setminus\mathcal H_D(f_0)$, and fix integers $n_i>m_i\geq0$ such that $f_0^{n_i}(x_i)=f_0^{m_i}(x_i)$.

The set $X$ of tuples $(Y,f,z_1,\ldots,z_{B+1})
 \in\mathcal X_D^u\times\mathcal U_D\times(\Pb^k)^{B+1}$ such that
\[z_i\in Y,\ \ f^{n_i}(z_i)=f^{m_i}(z_i)
 \text{ for every }i
\]
is a closed subvariety that is defined over $\overline{\Qb}$. Let $\Delta\subset X$ be the closed subset where two of the points $z_i$ coincide or one of them belongs to $\mathcal H_D(f)$.
This subset is also defined over $\overline{\Qb}$. Our chosen tuple belongs to $(X\setminus\Delta)(\Cb)$, so this open set is nonempty and therefore contains a point defined over $\overline{\Qb}$. Such a point gives a pair over $\overline{\Qb}$ with $B+1$ distinct preperiodic points outside the exceptional hypersurface, contradicting the bound already proved.
\end{proof}

\bibliographystyle{short}
\bibliography{biblio}

@book{kollar-rational,
  author    = {Koll{\'a}r, J{\'a}nos},
  title     = {Rational curves on algebraic varieties},
  series    = {Ergebnisse der Mathematik und ihrer Grenzgebiete. 3. Folge.
               A Series of Modern Surveys in Mathematics},
  volume    = {32},
  publisher = {Springer-Verlag},
  address   = {Berlin},
  year      = {1996},
  pages     = {viii+320},
  isbn      = {3-540-60168-6},
  mrclass   = {14-02 (14C05 14E05 14F17 14J45)},
  mrnumber  = {1440180},
  mrreviewer = {Yuri G. Prokhorov},
  doi       = {10.1007/978-3-662-03276-3},
  url       = {https://doi.org/10.1007/978-3-662-03276-3}
}

@article{ght,
  author  = {Ghioca, Dragos and Tucker, Thomas J.},
  title   = {A reformulation of the dynamical {M}anin--{M}umford conjecture},
  journal = {Bull. Aust. Math. Soc.},
  volume  = {103},
  year    = {2021},
  pages   = {154--161},
}

@article{Xie-algebraicity,
  author   = {Xie, Junyi},
  title    = {Algebraicity criteria, invariant subvarieties
              and transcendence problems from arithmetic dynamics},
  journal  = {Peking Math. J.},
  fjournal = {Peking Mathematical Journal},
  volume   = {7},
  year     = {2024},
  pages    = {345--398},
  doi      = {10.1007/s42543-022-00059-9}
}

@misc{Zhong-periodic-curves,
  author        = {Zhong, Xiao},
  title         = {Polynomial endomorphisms of {$\mathbb{A}^2$}
                   with many periodic curves},
  year          = {2025},
  note          = {Preprint, arXiv:2508.13873},
  eprint        = {2508.13873},
  archivePrefix = {arXiv},
  primaryClass  = {math.DS},
  doi           = {10.48550/arXiv.2508.13873},
  url           = {https://arxiv.org/abs/2508.13873}
}

@article{Li-Turaev-cycles,
  author  = {Li, Dongchen and Turaev, Dmitry},
  title   = {Persistence of heterodimensional cycles},
  journal = {Invent. Math.},
  volume  = {236},
  year    = {2024},
  pages   = {1413--1504},
  doi     = {10.1007/s00222-024-01255-3}
}

@article {bonatti-diaz-tan,
    AUTHOR = {Bonatti, Christian and D\'{\i}az, Lorenzo J.},
     TITLE = {Abundance of {$C^1$}-robust homoclinic tangencies},
   JOURNAL = {Trans. Amer. Math. Soc.},
  FJOURNAL = {Transactions of the American Mathematical Society},
    VOLUME = {364},
      YEAR = {2012},
    NUMBER = {10},
     PAGES = {5111--5148},
      ISSN = {0002-9947},
   MRCLASS = {37C20 (37C05 37C29 37C70 37D05)},
  MRNUMBER = {2931324},
MRREVIEWER = {Leonardo E. Mora},
       DOI = {10.1090/S0002-9947-2012-05445-6},
       URL = {https://doi.org/10.1090/S0002-9947-2012-05445-6},
}

@article {berger-parablender,
    AUTHOR = {Berger, Pierre},
     TITLE = {Generic family with robustly infinitely many sinks},
   JOURNAL = {Invent. Math.},
  FJOURNAL = {Inventiones Mathematicae},
    VOLUME = {205},
      YEAR = {2016},
    NUMBER = {1},
     PAGES = {121--172},
      ISSN = {0020-9910},
   MRCLASS = {37C20 (26A21 37C25)},
  MRNUMBER = {3514960},
MRREVIEWER = {Raquel Ribeiro},
       DOI = {10.1007/s00222-015-0632-6},
       URL = {https://doi.org/10.1007/s00222-015-0632-6},
}

@article{t-blender,
    AUTHOR = {J. Taflin},
     TITLE = {Blenders near polynomial product maps of {$\mathbb{C}^2$}},
   JOURNAL = {J. Eur. Math. Soc. (JEMS)},
  FJOURNAL = {Journal of the European Mathematical Society (JEMS)},
    VOLUME = {23},
      YEAR = {2021},
    NUMBER = {11},
     PAGES = {3555--3589},
      ISSN = {1435-9855},
   MRCLASS = {37F80 (32H50 37F46)},
  MRNUMBER = {4310812},
       DOI = {10.4171/jems/1076},
       URL = {https://doi.org/10.4171/jems/1076},
}

@article {biebler-hetero,
    AUTHOR = {Biebler, S\'{e}bastien},
     TITLE = {Robust complex heterodimensional cycles},
   JOURNAL = {Math. Ann.},
  FJOURNAL = {Mathematische Annalen},
    VOLUME = {394},
      YEAR = {2026},
    NUMBER = {4},
     PAGES = {Paper No. 104, 18},
      ISSN = {0025-5831,1432-1807},
   MRCLASS = {37F80 (37D05)},
  MRNUMBER = {5045181},
       DOI = {10.1007/s00208-026-03429-1},
       URL = {https://doi.org/10.1007/s00208-026-03429-1},
}

@article{Zhang-adelic,
  author   = {Zhang, Shou-Wu},
  title    = {Small points and adelic metrics},
  journal  = {J. Algebraic Geom.},
  fjournal = {Journal of Algebraic Geometry},
  volume   = {4},
  number   = {2},
  year     = {1995},
  pages    = {281--300},
  mrnumber = {1311351}
}

@article{dinh-laminar,
  author  = {Dinh, Tien-Cuong},
  title   = {Suites d'applications m{\'e}romorphes multivalu{\'e}es
             et courants laminaires},
  journal = {J. Geom. Anal.},
  volume  = {15},
  number  = {2},
  year    = {2005},
  pages   = {207--227},
  doi     = {10.1007/BF02922193}
}

@article{DS-MT,
  author  = {Dinh, Tien-Cuong and Sibony, Nessim},
  title   = {Distribution des valeurs de transformations
             m{\'e}romorphes et applications},
  journal = {Comment. Math. Helv.},
  volume  = {81},
  number  = {1},
  year    = {2006},
  pages   = {221--258},
  doi     = {10.4171/CMH/50}
}

@article{BS-correspondences,
  author  = {Bharali, Gautam and Sridharan, Shrihari},
  title   = {Holomorphic correspondences related to finitely
             generated rational semigroups},
  journal = {Internat. J. Math.},
  volume  = {28},
  number  = {14},
  year    = {2017},
  pages   = {1750108},
  note    = {25 pp.},
  doi     = {10.1142/S0129167X17501087}
}

@article{deThelin05,
  author  = {de Th{\'e}lin, Henry},
  title   = {Un ph{\'e}nom{\`e}ne de concentration de genre},
  journal = {Math. Ann.},
  volume  = {332},
  number  = {3},
  year    = {2005},
  pages   = {483--498},
  doi     = {10.1007/s00208-005-0631-9}
}

@article{Dinh07,
  author  = {Dinh, Tien-Cuong},
  title   = {Attracting current and equilibrium measure for
             attractors on {$\mathbb{P}^k$}},
  journal = {J. Geom. Anal.},
  volume  = {17},
  number  = {2},
  year    = {2007},
  pages   = {227--244},
  doi     = {10.1007/BF02930722}
}

@article{DF01,
  author  = {Diller, Jeffrey and Favre, Charles},
  title   = {Dynamics of bimeromorphic maps of surfaces},
  journal = {Amer. J. Math.},
  volume  = {123},
  number  = {6},
  year    = {2001},
  pages   = {1135--1169},
  doi     = {10.1353/ajm.2001.0038}
}

@article{DFR,
  author        = {Dujardin, Romain and Favre, Charles
                   and Ruggiero, Matteo},
  title         = {On the dynamical {Manin--Mumford} conjecture
                   for plane polynomial maps},
  journal       = {J. Eur. Math. Soc. (JEMS)},
  year          = {2026},
  note          = {To appear},
  eprint        = {2312.14817},
  archivePrefix = {arXiv},
  primaryClass  = {math.DS},
  url           = {https://arxiv.org/abs/2312.14817}
}

@article{przytycki-anosov,
  author  = {Przytycki, Feliks},
  title   = {Anosov endomorphisms},
  journal = {Studia Math.},
  volume  = {58},
  number  = {3},
  pages   = {249--285},
  year    = {1976},
  doi     = {10.4064/sm-58-3-249-285}
}

@misc{zhang-iterated-centralizer,
  author        = {Zhang, Yugang},
  title         = {A general endomorphism of {$\mathbb{P}^k$}
                   has trivial iterated centralizer},
  year          = {2026},
  note          = {Preprint, arXiv:2609.09858},
  eprint        = {2609.09858},
  archivePrefix = {arXiv},
  primaryClass  = {math.DS},
  doi           = {10.48550/arXiv.2609.09858},
  url           = {https://arxiv.org/abs/2609.09858}
}

@article {BDR-continuity-super-pot,
    AUTHOR = {Bianchi, Fabrizio and Dinh, Tien-Cuong and Rakhimov, Karim},
     TITLE = {H\"{o}lder continuity and laminarity of the {G}reen currents
              for {H}\'{e}non-like maps},
   JOURNAL = {J. Inst. Math. Jussieu},
  FJOURNAL = {Journal of the Institute of Mathematics of Jussieu. JIMJ.
              Journal de l'Institut de Math\'{e}matiques de Jussieu},
    VOLUME = {25},
      YEAR = {2026},
    NUMBER = {2},
     PAGES = {1363--1392},
      ISSN = {1474-7480,1475-3030},
   MRCLASS = {37F80 (32H50 32U05)},
  MRNUMBER = {5024924},
       DOI = {10.1017/S1474748025101588},
       URL = {https://doi.org/10.1017/S1474748025101588},
}

@article {HP,
    AUTHOR = {Hubbard, John H. and Papadopol, Peter},
     TITLE = {Superattractive fixed points in {${\bf C}^n$}},
   JOURNAL = {Indiana Univ. Math. J.},
  FJOURNAL = {Indiana University Mathematics Journal},
    VOLUME = {43},
      YEAR = {1994},
    NUMBER = {1},
     PAGES = {321--365},
      ISSN = {0022-2518,1943-5258},
   MRCLASS = {32H50 (32C30)},
  MRNUMBER = {1275463},
MRREVIEWER = {Gregery\ T.\ Buzzard},
       DOI = {10.1512/iumj.1994.43.43014},
       URL = {https://doi.org/10.1512/iumj.1994.43.43014},
}

@incollection {fs-montreal,
    AUTHOR = {Forn{\ae}ss, John Erik and Sibony, Nessim},
     TITLE = {Complex dynamics in higher dimensions},
 BOOKTITLE = {Complex potential theory ({M}ontreal, {PQ}, 1993)},
    SERIES = {NATO Adv. Sci. Inst. Ser. C: Math. Phys. Sci.},
    VOLUME = {439},
     PAGES = {131--186},
      NOTE = {Notes partially written by Estela A. Gavosto},
 PUBLISHER = {Kluwer Acad. Publ., Dordrecht},
      YEAR = {1994},
      ISBN = {0-7923-3005-6},
   MRCLASS = {32H50 (58F23)},
  MRNUMBER = {1332961},
MRREVIEWER = {Eric\ Bedford},
}

@article{t-attractor,
	Author = {J. Taflin},
	  JOURNAL = {Invent. Math.},
  FJOURNAL = {Inventiones Mathematicae},
	Pages = {83--137},
	Volume = {213},
	Title = {Attracting Currents and Equilibrium Measures for Quasi-attractors of {$\mathbb P^k$}},
	Year = {2018}}

@article {dujardin-directions,
    AUTHOR = {Dujardin, Romain},
     TITLE = {Fatou directions along the {J}ulia set for endomorphisms of
              {$\mathbb{CP}^k$}},
   JOURNAL = {J. Math. Pures Appl. (9)},
  FJOURNAL = {Journal de Math\'{e}matiques Pures et Appliqu\'{e}es.
              Neuvi\`eme S\'{e}rie},
    VOLUME = {98},
      YEAR = {2012},
    NUMBER = {6},
     PAGES = {591--615},
      ISSN = {0021-7824,1776-3371},
   MRCLASS = {37F10 (32H50 32U40 58A25)},
  MRNUMBER = {2994694},
MRREVIEWER = {Clinton\ Paul\ Curry},
       DOI = {10.1016/j.matpur.2012.05.004},
       URL = {https://doi.org/10.1016/j.matpur.2012.05.004},
}

@article {takens,
    AUTHOR = {Takens, Floris},
     TITLE = {Abundance of generic homoclinic tangencies in real-analytic
              families of diffeomorphisms},
   JOURNAL = {Bol. Soc. Brasil. Mat. (N.S.)},
  FJOURNAL = {Boletim da Sociedade Brasileira de Matem\'{a}tica. Nova
              S\'{e}rie},
    VOLUME = {22},
      YEAR = {1992},
    NUMBER = {2},
     PAGES = {191--214},
      ISSN = {0100-3569},
   MRCLASS = {58F12 (58F13 58F22)},
  MRNUMBER = {1179485},
MRREVIEWER = {Romeo\ F.\ Thomas},
       DOI = {10.1007/BF01232942},
       URL = {https://doi.org/10.1007/BF01232942},
}

@article {bonatti-diaz-blender,
    AUTHOR = {C. Bonatti and L. {D\'\i az}},
     TITLE = {Persistent nonhyperbolic transitive diffeomorphisms},
   JOURNAL = {Ann. of Math. (2)},
  FJOURNAL = {Annals of Mathematics. Second Series},
    VOLUME = {143},
      YEAR = {1996},
    NUMBER = {2},
     PAGES = {357--396},
      ISSN = {0003-486X},
   MRCLASS = {58F12 (58F10 58F15)},
  MRNUMBER = {1381990},
MRREVIEWER = {Marcelo Viana},
       DOI = {10.2307/2118647},
       URL = {http://dx.doi.org/10.2307/2118647},
}

@article {DF-dmm,
    AUTHOR = {Dujardin, R. and Favre, C.},
     TITLE = {The dynamical {M}anin-{M}umford problem for plane polynomial
              automorphisms},
   JOURNAL = {J. Eur. Math. Soc. (JEMS)},
  FJOURNAL = {Journal of the European Mathematical Society (JEMS)},
    VOLUME = {19},
      YEAR = {2017},
    NUMBER = {11},
     PAGES = {3421--3465},
      ISSN = {1435-9855,1435-9863},
   MRCLASS = {37P30 (32H50 37P05 37P35)},
  MRNUMBER = {3713045},
MRREVIEWER = {Dragos\ Ghioca},
       DOI = {10.4171/JEMS/743},
       URL = {https://doi.org/10.4171/JEMS/743},
}

@incollection {raynaud2,
    AUTHOR = {Raynaud, M.},
     TITLE = {Sous-vari\'{e}t\'{e}s d'une vari\'{e}t\'{e} ab\'{e}lienne et
              points de torsion},
 BOOKTITLE = {Arithmetic and geometry, {V}ol. {I}},
    SERIES = {Progr. Math.},
    VOLUME = {35},
     PAGES = {327--352},
 PUBLISHER = {Birkh\"{a}user Boston, Boston, MA},
      YEAR = {1983},
      ISBN = {3-7643-3132-1},
   MRCLASS = {14K05 (32K10)},
  MRNUMBER = {717600},
MRREVIEWER = {Daniel\ Coray},
}

@article {raynaud1,
    AUTHOR = {Raynaud, M.},
     TITLE = {Courbes sur une vari\'{e}t\'{e} ab\'{e}lienne et points de
              torsion},
   JOURNAL = {Invent. Math.},
  FJOURNAL = {Inventiones Mathematicae},
    VOLUME = {71},
      YEAR = {1983},
    NUMBER = {1},
     PAGES = {207--233},
      ISSN = {0020-9910,1432-1297},
   MRCLASS = {14H25 (14G25 14K15)},
  MRNUMBER = {688265},
MRREVIEWER = {F.\ Oort},
       DOI = {10.1007/BF01393342},
       URL = {https://doi.org/10.1007/BF01393342},
}

@article {DTZ-dmm,
    AUTHOR = {Ghioca, Dragos and Tucker, Thomas J. and Zhang, Shouwu},
     TITLE = {Towards a dynamical {M}anin-{M}umford conjecture},
   JOURNAL = {Int. Math. Res. Not. IMRN},
  FJOURNAL = {International Mathematics Research Notices. IMRN},
      YEAR = {2011},
    volume = {22},
     PAGES = {5109--5122},
      ISSN = {1073-7928,1687-0247},
   MRCLASS = {37P15 (11G10 37P55)},
  MRNUMBER = {2854724},
MRREVIEWER = {Joseph\ H.\ Silverman},
       DOI = {10.1093/imrn/rnq283},
       URL = {https://doi.org/10.1093/imrn/rnq283},
}

@article{udmm,
 author = {DeMarco, Laura and Mavraki, Niki Myrto},
 title = {The geometry of preperiodic points in families of maps on {{\(\mathbb{P}^N\)}}},
 fjournal = {L'Enseignement Math{\'e}matique. 2e S{\'e}rie},
 journal = {Enseign. Math. (2)},
 issn = {0013-8584},
 volume = {72},
 number = {3-4},
 pages = {305--335},
 year = {2026},
 language = {English},
 doi = {10.4171/lem/1102},
 zbMATH = {8257262}
}

@article {GNY-split-2,
    AUTHOR = {Ghioca, D. and Nguyen, K. D. and Ye, H.},
     TITLE = {The dynamical {M}anin-{M}umford conjecture and the dynamical
              {B}ogomolov conjecture for split rational maps},
   JOURNAL = {J. Eur. Math. Soc. (JEMS)},
  FJOURNAL = {Journal of the European Mathematical Society (JEMS)},
    VOLUME = {21},
      YEAR = {2019},
    NUMBER = {5},
     PAGES = {1571--1594},
      ISSN = {1435-9855,1435-9863},
   MRCLASS = {37P05 (37P30)},
  MRNUMBER = {3941498},
MRREVIEWER = {Joseph\ H.\ Silverman},
       DOI = {10.4171/JEMS/869},
       URL = {https://doi.org/10.4171/JEMS/869},
}

@article {dujardin-bif,
    AUTHOR = {R. Dujardin},
     TITLE = {Non-density of stability for holomorphic mappings on {$\mathbb
              P^k$}},
   JOURNAL = {J. \'Ec. polytech. Math.},
  FJOURNAL = {Journal de l'\'Ecole polytechnique. Math\'ematiques},
    VOLUME = {4},
      YEAR = {2017},
     PAGES = {813--843},
      ISSN = {2429-7100},
   MRCLASS = {32H50 (37F75)},
  MRNUMBER = {3694096},
}

@article {jonsson-skew,
    AUTHOR = {M. Jonsson},
     TITLE = {Dynamics of polynomial skew products on {$\mathbb C^2$}},
   JOURNAL = {Math. Ann.},
  FJOURNAL = {Mathematische Annalen},
    VOLUME = {314},
      YEAR = {1999},
    NUMBER = {3},
     PAGES = {403--447},
      ISSN = {0025-5831},
   MRCLASS = {32H50 (37F10 37F15)},
  MRNUMBER = {1704543},
MRREVIEWER = {Jean-Yves Briend},
       DOI = {10.1007/s002080050301},
       URL = {http://dx.doi.org/10.1007/s002080050301},
}

@article {dupont-t-int,
    AUTHOR = {C. Dupont and J. Taflin},
     TITLE = {Dynamics of fibered endomorphisms of {$\mathbb P^k$}},
   JOURNAL = {Ann. Sc. Norm. Super. Pisa Cl. Sci. (5)},
  FJOURNAL = {Annali della Scuola Normale Superiore di Pisa. Classe di
              Scienze. Serie V},
    VOLUME = {22},
      YEAR = {2021},
    NUMBER = {1},
     PAGES = {53--78},
      ISSN = {0391-173X},
   MRCLASS = {37F80 (32U40)},
  MRNUMBER = {4288649},
MRREVIEWER = {Fabrizio Bianchi},
}

@unpublished{jonsson-phd,
	Author = {M. Jonsson},
	Note = {PhD thesis, Royal Institute of Technology},
	Title = {Dynamical Studies in Several Complex Variables},
	Year = {1997}}

@article {jonsson-motion_hyperbolic,
    AUTHOR = {Jonsson, Mattias},
     TITLE = {Holomorphic motions of hyperbolic sets},
   JOURNAL = {Michigan Math. J.},
  FJOURNAL = {Michigan Mathematical Journal},
    VOLUME = {45},
      YEAR = {1998},
    NUMBER = {2},
     PAGES = {409--415},
      ISSN = {0026-2285,1945-2365},
   MRCLASS = {32H50 (58F15 58F23)},
  MRNUMBER = {1637690},
MRREVIEWER = {Thomas\ Ransford},
       DOI = {10.1307/mmj/1030132192},
       URL = {https://doi.org/10.1307/mmj/1030132192},
}

@phdthesis{dujardin-thesis,
  TITLE = {{Dynamique d'applications non polynomiales et courants laminaires}},
  AUTHOR = {Dujardin, Romain},
  URL = {https://theses.hal.science/tel-00004028},
  SCHOOL = {{Universit{\'e} Paris Sud - Paris XI}},
  YEAR = {2002},
  MONTH = Dec,
  TYPE = {Theses},
  HAL_ID = {tel-00004028},
  HAL_VERSION = {v1},
}

@article {dns,
    AUTHOR = {Dinh, Tien-Cuong and Nguy\^{e}n, Vi\^{e}t-Anh and Sibony,
              Nessim},
     TITLE = {Dynamics of horizontal-like maps in higher dimension},
   JOURNAL = {Adv. Math.},
  FJOURNAL = {Advances in Mathematics},
    VOLUME = {219},
      YEAR = {2008},
    NUMBER = {5},
     PAGES = {1689--1721},
      ISSN = {0001-8708,1090-2082},
   MRCLASS = {32U40 (32H50 37B40 37F10)},
  MRNUMBER = {2458151},
MRREVIEWER = {Ma\l gorzata\ Stawiska},
       DOI = {10.1016/j.aim.2008.07.006},
       URL = {https://doi.org/10.1016/j.aim.2008.07.006},
}

@article {dujardin-hl,
    AUTHOR = {Dujardin, Romain},
     TITLE = {H\'{e}non-like mappings in {$\mathbb{C}^2$}},
   JOURNAL = {Amer. J. Math.},
  FJOURNAL = {American Journal of Mathematics},
    VOLUME = {126},
      YEAR = {2004},
    NUMBER = {2},
     PAGES = {439--472},
      ISSN = {0002-9327,1080-6377},
   MRCLASS = {32H50 (37A35 37F10)},
  MRNUMBER = {2045508},
MRREVIEWER = {Charles\ Favre},
       URL =
              {http://muse.jhu.edu/journals/american_journal_of_mathematics/v126/126.2dujardin.pdf},
}

@article {sell,
    AUTHOR = {Sell, George R.},
     TITLE = {Smooth linearization near a fixed point},
   JOURNAL = {Amer. J. Math.},
  FJOURNAL = {American Journal of Mathematics},
    VOLUME = {107},
      YEAR = {1985},
    NUMBER = {5},
     PAGES = {1035--1091},
      ISSN = {0002-9327},
   MRCLASS = {58F19 (34C20)},
  MRNUMBER = {805804},
MRREVIEWER = {Carmen Chicone},
       DOI = {10.2307/2374346},
       URL = {https://doi.org/10.2307/2374346},
}

@unpublished{t-gorbovickis,
      title={Independence of multipliers in several variables complex dynamics}, 
      author={I. Gorbovickis and J. Taflin},
      year={2024},
      note={arXiv:2411.12856},
      archivePrefix={arXiv},
      primaryClass={math.DS},
      url={https://arxiv.org/abs/2411.12856}, 
}

@unpublished{dujardin-homoclinic,
      title={Degenerate homoclinic bifurcations in complex dimension 2}, 
      author={Romain Dujardin},
      year={2023},
      Note = {arXiv:2306.08160},
      archivePrefix={arXiv},
      primaryClass={math.DS},
      url={https://arxiv.org/abs/2306.08160}, 
}

@misc{demarco-cirm,
  author       = {Laura DeMarco},
  title        = {The (Algebraic) Geometry of Preperiodic Points in $\mathbb{P}^N$, in Three Lectures},
  howpublished = {Lecture notes, CIRM, France, January 2025},
  year         = {2025},
  url          = {https://people.math.harvard.edu/~demarco/CIRM_Jan2025_LectureNotes.pdf},
  note         = {}
}

@article {Lang-Division,
    AUTHOR = {Lang, Serge},
     TITLE = {Division points on curves},
   JOURNAL = {Annali di Matematica},
  FJOURNAL = {Annali di Matematica Pura ed Applicata},
    VOLUME = {70},
      YEAR = {1965},
    NUMBER = {1},
     PAGES = {229--234}
}

@article{GTV,
     author = {Thomas Gauthier and Johan Taflin and Gabriel Vigny},
     title = {Sparsity of postcritically finite maps of $\mathbb{P}^k$ and beyond: {A} complex analytic approach},
     journal = {Publications Math\'ematiques de l'IH\'ES},
     pages = {1--96},
     year = {2026},
     publisher = {IHES},
     volume = {143},
     doi = {10.5802/pmihes.1},
     language = {en},
     url = {https://pmihes.centre-mersenne.org/articles/10.5802/pmihes.1/}
}

@ARTICLE{Poineau-BFT,
 author = {Poineau, J{\'e}r{\^o}me},
 title = {Dynamique analytique sur {{\(\mathbf{Z} \)}}. {II}: {{\'E}cart} uniforme entre {Latt{\`e}s} et conjecture de {Bogomolov}-{Fu}-{Tschinkel}},
 fjournal = {Journal f{\"u}r die Reine und Angewandte Mathematik},
 journal = {J. Reine Angew. Math.},
 issn = {0075-4102},
 volume = {824},
 pages = {89--135},
 year = {2025},
 language = {French},
 doi = {10.1515/crelle-2025-0019},
 zbMATH = {8060553},
 Zbl = {1584.37107}
}

@book{YZ-adelic,
 author = {Yuan, Xinyi and Zhang, Shou-Wu},
 title = {Adelic line bundles on quasi-projective varieties},
 fseries = {Annals of Mathematics Studies},
 series = {Ann. Math. Stud.},
 volume = {221},
 isbn = {978-0-691-27172-9; 978-0-691-27173-6; 978-0-691-27870-4},
 year = {2026},
 publisher = {Princeton, NJ: Princeton University Press},
 language = {English},
 zbMATH = {8063230}
}

@article{GGK,
     author = {Ziyang Gao and Tangli Ge and Lars K\"uhne},
     title = {The {Uniform} {Mordell{\textendash}Lang} {Conjecture}},
     journal = {Publications Math\'ematiques de l'IH\'ES},
     pages = {189--235},
     year = {2026},
     publisher = {IHES},
     volume = {143},
     doi = {10.5802/pmihes.26},
     language = {en},
     url = {https://pmihes.centre-mersenne.org/articles/10.5802/pmihes.26/}
}

@article{Yuan-Bogomolov,
author = {Xinyi Yuan},
title = {{Arithmetic bigness and a uniform Bogomolov-type result}},
volume = {203},
journal = {Annals of Mathematics},
number = {1},
publisher = {Department of Mathematics of Princeton University},
pages = {15 -- 119},
year = {2026},
doi = {10.4007/annals.2026.203.1.2},
URL = {https://doi.org/10.4007/annals.2026.203.1.2}
}

@article{Dimitrov-Gao-Habegger,
 author = {Dimitrov, Vesselin and Gao, Ziyang and Habegger, Philipp},
 title = {Uniformity in {Mordell}-{Lang} for curves},
 fjournal = {Annals of Mathematics. Second Series},
 journal = {Ann. Math. (2)},
 issn = {0003-486X},
 volume = {194},
 number = {1},
 pages = {237--298},
 year = {2021},
 language = {English},
 doi = {10.4007/annals.2021.194.1.4},
 zbMATH = {7582363},
 Zbl = {1514.11040}
}

@article{Good-height,
 author = {Gauthier, Thomas},
 title = {Good height functions on quasi-projective varieties: equidistribution and applications in dynamics},
 fjournal = {Annales de la Facult{\'e} des Sciences de Toulouse. Math{\'e}matiques. S{\'e}rie VI},
 journal = {Ann. Fac. Sci. Toulouse, Math. (6)},
 issn = {0240-2963},
 volume = {35},
 number = {1},
 pages = {57--94},
 year = {2026},
 language = {English},
 doi = {10.5802/afst.1841},
 zbMATH = {8144876}
}

@misc{Kuhne,
      title={Equidistribution in Families of Abelian Varieties and Uniformity}, 
      AUTHOR = {K{\"u}hne, Lars},
      year={2021},
      eprint={2101.10272},
      archivePrefix={arXiv},
      primaryClass={math.NT}
}

@article {DKY1,
    AUTHOR = {DeMarco, Laura and Krieger, Holly and Ye, Hexi},
     TITLE = {Uniform {M}anin-{M}umford for a family of genus 2 curves},
   JOURNAL = {Ann. of Math. (2)},
  FJOURNAL = {Annals of Mathematics. Second Series},
    VOLUME = {191},
      YEAR = {2020},
    NUMBER = {3},
     PAGES = {949--1001},
      ISSN = {0003-486X},
   MRCLASS = {14G40 (11G50 14H40 37F44 37P50)},
  MRNUMBER = {4088354},
MRREVIEWER = {Yu Yasufuku},
       DOI = {10.4007/annals.2020.191.3.5},
       URL = {https://doi-org.revues.math.u-psud.fr/10.4007/annals.2020.191.3.5},
}

@article {BedfordLyubichSmillie2,
	AUTHOR = {Bedford, Eric and Lyubich, Mikhail and Smillie, John},
	TITLE = {Polynomial diffeomorphisms of {${\bf C}^2$}. {IV}. {T}he
	measure of maximal entropy and laminar currents},
	JOURNAL = {Invent. Math.},
	FJOURNAL = {Inventiones Mathematicae},
	VOLUME = {112},
	YEAR = {1993},
	NUMBER = {1},
	PAGES = {77--125},
	ISSN = {0020-9910},
	MRCLASS = {32H50 (32C30 58F11 58F23)},
	MRNUMBER = {1207478},
	MRREVIEWER = {Yasuichiro Nishimura},
	DOI = {10.1007/BF01232426},
	URL = {https://doi.org/10.1007/BF01232426},
}

@article {AGMV,
    AUTHOR = {Astorg, Matthieu and Gauthier, Thomas and Mihalache, Nicolae and Vigny, Gabriel},
     TITLE = {Collet, {E}ckmann and the bifurcation measure},
   JOURNAL = {Invent. Math.},
  FJOURNAL = {Inventiones Mathematicae},
    VOLUME = {217},
      YEAR = {2019},
    NUMBER = {3},
     PAGES = {749--797},
      ISSN = {0020-9910},
   MRCLASS = {37F45 (14D20 14E05 30D05 32G15)},
  MRNUMBER = {3989253},
       DOI = {10.1007/s00222-019-00874-5},
       URL = {https://doi.org/10.1007/s00222-019-00874-5},
}

@article{GV_Northcott,
 author = {Gauthier, Thomas and Vigny, Gabriel},
 title = {The geometric dynamical {Northcott} and {Bogomolov} properties},
 fjournal = {Annales Scientifiques de l'{\'E}cole Normale Sup{\'e}rieure. Quatri{\`e}me S{\'e}rie},
 journal = {Ann. Sci. {\'E}c. Norm. Sup{\'e}r. (4)},
 issn = {0012-9593},
 volume = {58},
 number = {1},
 pages = {231--273},
 year = {2025},
 language = {English},
 doi = {10.24033/asens.2605},
 zbMATH = {8048371}
}

@article {Tan-similarity,
    AUTHOR = {Tan, Lei},
     TITLE = {Similarity between the {M}andelbrot set and {J}ulia sets},
   JOURNAL = {Comm. Math. Phys.},
  FJOURNAL = {Communications in Mathematical Physics},
    VOLUME = {134},
      YEAR = {1990},
    NUMBER = {3},
     PAGES = {587--617},
      ISSN = {0010-3616},
   MRCLASS = {58F23 (30D05)},
  MRNUMBER = {1086745},
MRREVIEWER = {M. Rees},
       URL = {http://projecteuclid.org/euclid.cmp/1104201823},
}

@incollection {zhang,
    AUTHOR = {Zhang, Shou-Wu},
     TITLE = {Distributions in algebraic dynamics},
 BOOKTITLE = {Surveys in differential geometry. {V}ol. {X}},
    SERIES = {Surv. Differ. Geom.},
    VOLUME = {10},
     PAGES = {381--430},
 PUBLISHER = {Int. Press, Somerville, MA},
      YEAR = {2006},
   MRCLASS = {32H50 (11G50 32J27 37F10)},
  MRNUMBER = {2408228},
MRREVIEWER = {Matthew H. Baker},
       DOI = {10.4310/SDG.2005.v10.n1.a9},
       URL = {https://doi.org/10.4310/SDG.2005.v10.n1.a9},
}

@article {northcott,
	AUTHOR = {Northcott, D. G.},
	TITLE = {Periodic points on an algebraic variety},
	JOURNAL = {Ann. of Math. (2)},
	FJOURNAL = {Annals of Mathematics. Second Series},
	VOLUME = {51},
	YEAR = {1950},
	PAGES = {167--177},
	ISSN = {0003-486X},
	MRCLASS = {14.0X},
	MRNUMBER = {34607},
	MRREVIEWER = {B. Segre},
	DOI = {10.2307/1969504},
	URL = {https://doi.org/10.2307/1969504},
}

@article {GNKY-Maninmumford,
    AUTHOR = {Ghioca, Dragos and Nguyen, Khoa D. and Ye, Hexi},
     TITLE = {The dynamical {M}anin-{M}umford conjecture and the dynamical
              {B}ogomolov conjecture for endomorphisms of {$(\mathbb{P}^1)^n$}},
   JOURNAL = {Compos. Math.},
  FJOURNAL = {Compositio Mathematica},
    VOLUME = {154},
      YEAR = {2018},
    NUMBER = {7},
     PAGES = {1441--1472},
      ISSN = {0010-437X},
   MRCLASS = {37P05 (37P30)},
  MRNUMBER = {3826461},
MRREVIEWER = {Lukas Pottmeyer},
       DOI = {10.1112/s0010437x18007157},
       URL = {https://doi.org/10.1112/s0010437x18007157},
}

@article {szpiro,
    AUTHOR = {Petsche, Clayton and Szpiro, Lucien and Tepper, Michael},
     TITLE = {Isotriviality is equivalent to potential good reduction for
              endomorphisms of {$\mathbb{P}^N$} over function fields},
   JOURNAL = {J. Algebra},
  FJOURNAL = {Journal of Algebra},
    VOLUME = {322},
      YEAR = {2009},
    NUMBER = {9},
     PAGES = {3345--3365},
      ISSN = {0021-8693},
   MRCLASS = {14E05 (12J25 14H05 37P15)},
  MRNUMBER = {2567424},
MRREVIEWER = {Valmecir A. Bayer},
       DOI = {10.1016/j.jalgebra.2008.11.027},
       URL = {https://doi.org/10.1016/j.jalgebra.2008.11.027},
}

@article {fakhruddin-generic,
    AUTHOR = {Fakhruddin, Najmuddin},
     TITLE = {The algebraic dynamics of generic endomorphisms of
              {$\mathbb{P}^n$}},
   JOURNAL = {Algebra Number Theory},
  FJOURNAL = {Algebra \& Number Theory},
    VOLUME = {8},
      YEAR = {2014},
    NUMBER = {3},
     PAGES = {587--608},
      ISSN = {1937-0652},
   MRCLASS = {37P55 (37F10)},
  MRNUMBER = {3218803},
       DOI = {10.2140/ant.2014.8.587},
       URL = {https://doi.org/10.2140/ant.2014.8.587},
}

@article {Fakhruddin,
    AUTHOR = {Fakhruddin, Najmuddin},
     TITLE = {Questions on self maps of algebraic varieties},
   JOURNAL = {J. Ramanujan Math. Soc.},
  FJOURNAL = {Journal of the Ramanujan Mathematical Society},
    VOLUME = {18},
      YEAR = {2003},
    NUMBER = {2},
     PAGES = {109--122},
      ISSN = {0970-1249},
   MRCLASS = {14G05 (37A45)},
  MRNUMBER = {1995861},
MRREVIEWER = {Yuri Tschinkel},
}

@ARTICLE{gauthier-abscont,
 author = {Gauthier, Thomas},
 title = {Dynamical pairs with an absolutely continuous bifurcation measure},
 fjournal = {Annales de la Facult{\'e} des Sciences de Toulouse. Math{\'e}matiques. S{\'e}rie VI},
 journal = {Ann. Fac. Sci. Toulouse, Math. (6)},
 issn = {0240-2963},
 volume = {32},
 number = {2},
 pages = {203--230},
 year = {2023},
 language = {English},
 doi = {10.5802/afst.1735},
 zbMATH = {7779473},
 Zbl = {1533.37108}
}

@article {CS-height,
    AUTHOR = {Call, Gregory S. and Silverman, Joseph H.},
     TITLE = {Canonical heights on varieties with morphisms},
   JOURNAL = {Compositio Math.},
  FJOURNAL = {Compositio Mathematica},
    VOLUME = {89},
      YEAR = {1993},
    NUMBER = {2},
     PAGES = {163--205},
      ISSN = {0010-437X},
   MRCLASS = {11G35 (14G25)},
  MRNUMBER = {1255693},
MRREVIEWER = {Takeshi Ooe},
       URL = {http://www.numdam.org/item?id=CM_1993__89_2_163_0},
}

@article {lyubich:equi,
    AUTHOR = {Ljubich, M. Ju.},
     TITLE = {Entropy properties of rational endomorphisms of the {R}iemann
              sphere},
   JOURNAL = {Ergodic Theory Dynam. Systems},
  FJOURNAL = {Ergodic Theory and Dynamical Systems},
    VOLUME = {3},
      YEAR = {1983},
    NUMBER = {3},
     PAGES = {351--385},
      ISSN = {0143-3857},
   MRCLASS = {58F11 (28D20 54H20)},
  MRNUMBER = {741393 (85k:58049)},
MRREVIEWER = {W. R. Utz},
       DOI = {10.1017/S0143385700002030},
       URL = {http://dx.doi.org/10.1017/S0143385700002030},
}

@article{favre,
title={Degeneration of endomorphisms of the complex projective space in the hybrid space},
journal={Journal of the Institute of Mathematics of Jussieu},
publisher={Cambridge University Press},
author={Favre, {C}harles},
pages={1-43},
}

@book{Silvermandiophantine,
	Address = {New York},
	Author = {Hindry, Marc and Silverman, Joseph H.},
	Doi = {10.1007/978-1-4612-1210-2},
	Isbn = {0-387-98975-7; 0-387-98981-1},
	Mrclass = {11Gxx (11-02 11G10 11G30 11G50 14G25)},
	Mrnumber = {1745599 (2001e:11058)},
	Mrreviewer = {Dino J. Lorenzini},
	Note = {An introduction},
	Pages = {xiv+558},
	Publisher = {Springer-Verlag},
	Series = {Graduate Texts in Mathematics},
	Title = {Diophantine geometry},
	Url = {http://dx.doi.org/10.1007/978-1-4612-1210-2},
	Volume = {201},
	Year = {2000}}

@article{Article1,
	Author = {{Gauthier}, Thomas},
	Fjournal = {Annales Scientifiques de l'\'Ecole Normale Sup\'erieure. Quatri\`eme S\'erie},
	Journal = {Ann. Sci. \'Ec. Norm. Sup\'er. (4)},
	Number = {6},
	Pages = {947--984},
	Title = {Strong bifurcation loci of full {H}ausdorff dimension},
	Volume = {45},
	Year = {2012}}

@article {BBD,
    AUTHOR = {Berteloot, Fran\c{c}ois and Bianchi, Fabrizio and Dupont,
              Christophe},
     TITLE = {Dynamical stability and {L}yapunov exponents for holomorphic
              endomorphisms of {$\mathbb{P}^k$}},
   JOURNAL = {Ann. Sci. \'{E}c. Norm. Sup\'{e}r. (4)},
  FJOURNAL = {Annales Scientifiques de l'\'{E}cole Normale Sup\'{e}rieure. Quatri\`eme
              S\'{e}rie},
    VOLUME = {51},
      YEAR = {2018},
    NUMBER = {1},
     PAGES = {215--262},
      ISSN = {0012-9593},
   MRCLASS = {32H50 (37D25 37F45)},
  MRNUMBER = {3764042},
MRREVIEWER = {Marco Abate},
       DOI = {10.24033/asens.2355},
       URL = {https://doi.org/10.24033/asens.2355},
}

@article{BB1,
	Author = {Bassanelli, Giovanni and Berteloot, Fran{\c{c}}ois},
	Coden = {JRMAA8},
	Fjournal = {Journal f\"ur die Reine und Angewandte Mathematik. [Crelle's Journal]},
	Issn = {0075-4102},
	Journal = {J. Reine Angew. Math.},
	Mrclass = {32U40 (32H50 37F10 37F50)},
	Mrnumber = {MR2339474 (2008g:32052)},
	Mrreviewer = {Laura G. DeMarco},
	Pages = {201--235},
	Title = {Bifurcation currents in holomorphic dynamics on {$\mathbb{P}^k$}},
	Volume = {608},
	Year = {2007}}

@article{BL,
	Author = {Berteloot, Fran{\c{c}}ois and Loeb, Jean-Jacques},
	Coden = {BSMFAA},
	Fjournal = {Bulletin de la Soci\'et\'e Math\'ematique de France},
	Issn = {0037-9484},
	Journal = {Bull. Soc. Math. France},
	Mrclass = {32U40 (32H50)},
	Mrnumber = {MR1871293 (2002k:32063)},
	Mrreviewer = {Mattias Jonsson},
	Number = {2},
	Pages = {175--188},
	Title = {Une caract\'erisation g\'eom\'etrique des exemples de {L}att\`es de {${\mathbb{ P}}^k$}},
	Volume = {129},
	Year = {2001}}

@article{B,
	Author = {Berteloot, Fran{\c{c}}ois},
	Doi = {10.1142/S0129167X94000322},
	Fjournal = {International Journal of Mathematics},
	Issn = {0129-167X},
	Journal = {Internat. J. Math.},
	Mrclass = {32M05 (32F15)},
	Mrnumber = {1297410 (95i:32040)},
	Mrreviewer = {Eric Bedford},
	Number = {5},
	Pages = {619--634},
	Title = {Characterization of models in {$\mathbf{C}^2$} by their automorphism groups},
	Url = {http://dx.doi.org/10.1142/S0129167X94000322},
	Volume = {5},
	Year = {1994}}

@incollection{buffepstein,
	Address = {Wellesley, Massachussets},
	Author = {Buff, Xavier and Epstein, Adam L.},
	Booktitle = {Complex dynamics : families and friends / edited by Dierk Schleicher},
	Isbn = {978-1-56881-450-6},
	Pages = {491--512},
	Publisher = {A K Peters, Ltd.},
	Title = {Bifurcation Measure and Postcritically Finite Rational Maps},
	Year = {2009}}

@article{briendduval2,
	Author = {Briend, Jean-Yves and Duval, Julien},
	Doi = {10.1007/s10240-001-8190-4},
	Fjournal = {Publications Math\'ematiques. Institut de Hautes \'Etudes Scientifiques},
	Issn = {0073-8301},
	Journal = {Publ. Math. Inst. Hautes \'Etudes Sci.},
	Mrclass = {32H50 (37A35 37F10)},
	Mrnumber = {1863737 (2002k:32027)},
	Mrreviewer = {Mattias Jonsson},
	Volume = {93},
	Pages = {145--159},
	Title = {Deux caract\'erisations de la mesure d'\'equilibre d'un endomorphisme de {${\rm P}^k(\mathbf{C})$}},
	Url = {http://dx.doi.org/10.1007/s10240-001-8190-4},
	Year = {2001}}

@book{Chirka,
	Address = {Dordrecht},
	Author = {Chirka, E. M.},
	Isbn = {0-7923-0234-6},
	Mrclass = {32C25 (32-02)},
	Mrnumber = {MR1111477 (92b:32016)},
	Note = {Translated from the Russian by R. A. M. Hoksbergen},
	Pages = {xx+372},
	Publisher = {Kluwer Academic Publishers Group},
	Series = {Mathematics and its Applications (Soviet Series)},
	Title = {Complex analytic sets},
	Volume = {46},
	Year = {1989}}

@article{DeMarco-Mavraki,
 author = {DeMarco, Laura and Mavraki, Niki Myrto},
 title = {Dynamics on {{\(\mathbb{P}^1\)}}: preperiodic points and pairwise stability},
 fjournal = {Compositio Mathematica},
 journal = {Compos. Math.},
 issn = {0010-437X},
 volume = {160},
 number = {2},
 pages = {356--387},
 year = {2024},
 language = {English},
 doi = {10.1112/S0010437X23007546},
 zbMATH = {7793888},
 Zbl = {1535.37114}
}

@article{Shishikura2,
	Author = {Shishikura, Mitsuhiro},
	Coden = {ANMAAH},
	Doi = {10.2307/121009},
	Fjournal = {Annals of Mathematics. Second Series},
	Issn = {0003-486X},
	Journal = {Ann. of Math. (2)},
	Mrclass = {37F35 (30D05 37F45 37F50)},
	Mrnumber = {MR1626737 (2000f:37056)},
	Mrreviewer = {Hartje Kriete},
	Number = {2},
	Pages = {225--267},
	Title = {The {H}ausdorff dimension of the boundary of the {M}andelbrot set and {J}ulia sets},
	Url = {http://dx.doi.org/10.2307/121009},
	Volume = {147},
	Year = {1998}}

@book{Silverman,
	Address = {New York},
	Author = {Silverman, Joseph H.},
	Isbn = {978-0-387-69903-5},
	Mrclass = {11-02 (11-01 11G05 11G07 11G50 37-02 37F10)},
	Mrnumber = {MR2316407 (2008c:11002)},
	Mrreviewer = {Thomas Ward},
	Pages = {x+511},
	Publisher = {Springer},
	Series = {Graduate Texts in Mathematics},
	Title = {The arithmetic of dynamical systems},
	Volume = {241},
	Year = {2007}}

\end{document}